\pdfoutput=1 

\documentclass[12pt,a4paper,twoside]{report} 
\usepackage[latin1]{inputenc}
\usepackage[matrix,arrow,ps,color,line,curve,frame,all]{xy} 
\usepackage{hyperref}
\SelectTips{cm}{}
\usepackage[left=3cm,right=3.6cm,top=2.7cm,bottom=3.7cm,twoside,includehead]{geometry}

\usepackage{titlesec}
\titleformat{\chapter}[display]{\center\bfseries}{\Large Chapter \thechapter\\ \rule{0.3\textwidth}{0.4pt}}{5pt}{\Huge}

\usepackage[english]{babel}
\usepackage{graphicx}
\usepackage[usenames,dvipsnames]{xcolor}
\usepackage{amsfonts}
\usepackage{amssymb}
\usepackage{amsmath}
\usepackage{amsthm}
\usepackage{aliascnt}
\usepackage{parskip}
\usepackage{verbatim}
\usepackage{dsfont}
\usepackage{stmaryrd}
\usepackage{setspace}

\renewcommand{\marginpar}[2][]{}

\usepackage[nottoc,notlot,notlof]{tocbibind}

\usepackage{etoolbox}
\apptocmd{\sloppy}{\hbadness 10000\relax}{}{}

\usepackage{fancyhdr}
\newtheoremstyle{standard}{10pt}{3pt}{\itshape}{}{\bfseries}{.}{.5em}{}
\theoremstyle{standard}

\newaliascnt{thm}{theorem}  
\newtheorem{thm}[thm]{Theorem}  
\aliascntresetthe{thm}

\newaliascnt{lemma}{theorem}  
\newtheorem{lemma}[lemma]{Lemma}  
\aliascntresetthe{lemma}

\newaliascnt{prop}{theorem}  
\newtheorem{prop}[prop]{Proposition}  
\aliascntresetthe{prop}

\newaliascnt{cor}{theorem}  
\newtheorem{cor}[cor]{Corollary}  
\aliascntresetthe{cor}

\newtheoremstyle{definition}{10pt}{3pt}{}{}{\bfseries}{.}{.5em}{}
\theoremstyle{definition}

\newaliascnt{defi}{theorem}  
\newtheorem{defi}[defi]{Definition}
\aliascntresetthe{defi}

\newaliascnt{nota}{theorem}  
\newtheorem{nota}[nota]{Notation}
\aliascntresetthe{nota}

\newaliascnt{ex}{theorem}  
\newtheorem{ex}[ex]{Example}
\aliascntresetthe{ex}

\newaliascnt{rem}{theorem}  
\newtheorem{rem}[rem]{Remark}
\aliascntresetthe{rem}

\newtheoremstyle{standard}{10pt}{3pt}{\itshape}{}{\bfseries}{.}{.5em}{}
\theoremstyle{standard}

\newaliascnt{thmi}{theoremi}  
\newtheorem{thmi}[thmi]{Theorem}  
\aliascntresetthe{thmi}

\newaliascnt{propi}{theoremi}  
\newtheorem{propi}[propi]{Proposition}  
\aliascntresetthe{propi}

\DeclareMathOperator{\Q}{\mathsf{Qcoh}}
\DeclareMathOperator{\M}{\mathsf{Mod}}
\DeclareMathOperator{\tfM}{\mathsf{tfMod}}
\DeclareMathOperator{\CoM}{\mathsf{CoMod}}
\DeclareMathOperator{\grM}{\mathsf{grMod}}

\DeclareMathOperator{\Rep}{\mathsf{Rep}}

\DeclareMathOperator{\gr}{\mathsf{gr}}
 
\DeclareMathOperator{\Ind}{\mathsf{Ind}}
\DeclareMathOperator{\Desc}{\mathsf{Desc}}

\DeclareMathOperator{\Cat}{\mathsf{Cat}}
\DeclareMathOperator{\Ab}{\mathsf{Ab}}

\DeclareMathOperator{\CAlg}{\mathsf{CAlg}}

\DeclareMathOperator{\Alg}{\mathsf{Alg}}
\DeclareMathOperator{\Gpd}{\mathsf{Gpd}}

\DeclareMathOperator{\Sch}{\mathsf{Sch}}
\DeclareMathOperator{\FinSet}{\mathsf{FinSet}}
\DeclareMathOperator{\Stack}{\mathsf{Stack}}
\DeclareMathOperator{\SPEC}{\mathsf{Spec}}

\DeclareMathOperator{\Tors}{\mathsf{Tors}}
\DeclareMathOperator{\Set}{\mathsf{Set}}
\DeclareMathOperator{\CMon}{\mathsf{CMon}}
\DeclareMathOperator{\sSet}{\mathsf{sSet}}
\DeclareMathOperator{\Top}{\mathsf{Top}}

\DeclareMathOperator{\Spec}{Spec}
\DeclareMathOperator{\End}{End}
\DeclareMathOperator{\Hom}{Hom}
\DeclareMathOperator{\Isom}{Isom}
\DeclareMathOperator{\BiHom}{BiHom}
\DeclareMathOperator{\HOM}{\underline{\Hom}}
\DeclareMathOperator{\Aut}{Aut}
\DeclareMathOperator{\coker}{coker}
\DeclareMathOperator{\colim}{colim}
\DeclareMathOperator{\op}{op}
\DeclareMathOperator{\coeq}{coeq}
\DeclareMathOperator{\eq}{eq}
\DeclareMathOperator{\id}{id}
\DeclareMathOperator{\pr}{pr}
\DeclareMathOperator{\GL}{GL}
\DeclareMathOperator{\im}{im}
\DeclareMathOperator{\sgn}{sgn}
\DeclareMathOperator{\lax}{lax}
\DeclareMathOperator{\rad}{rad}
\DeclareMathOperator{\Sym}{Sym}
\DeclareMathOperator{\ASym}{ASym}
\DeclareMathOperator{\Proj}{Proj}
\DeclareMathOperator{\supp}{supp}
\DeclareMathOperator{\Der}{Der}
\renewcommand{\Im}{\mathrm{Im}}
\DeclareMathOperator{\Tor}{Tor}
\DeclareMathOperator{\Fix}{Fix}

\DeclareMathOperator{\Bl}{Bl}

\DeclareMathOperator{\Pic}{Pic}
\DeclareMathOperator{\epi}{\twoheadrightarrow}

\let\para\S
\renewcommand{\L}{\mathcal{L}}
\newcommand{\K}{\mathcal{K}}
\renewcommand{\P}{\mathds{P}}
\newcommand{\C}{\mathcal{C}}
\newcommand{\A}{\mathcal{A}}
\newcommand{\D}{\mathcal{D}}
\newcommand{\I}{\mathcal{I}}
\newcommand{\Grass}{\mathrm{Grass}}
\renewcommand{\O}{\mathcal{O}}
\newcommand{\E}{\mathcal{E}}
\newcommand{\G}{\mathds{G}}
\newcommand{\B}{\mathcal{B}}
\newcommand{\N}{\mathds{N}}
\newcommand{\Z}{\mathds{Z}}
\newcommand{\QQ}{\mathds{Q}}
\newcommand{\R}{\mathds{R}}
\newcommand{\F}{\mathds{F}}
\renewcommand{\S}{\Sigma}
\newcommand{\m}{\mathfrak{m}}

\newcommand{\e}{\varepsilon}
\newcommand{\fp}{\mathrm{fp}}

\begin{document}


\hypersetup{pageanchor=false}
\pagenumbering{gobble}

\begin{titlepage}
\pagestyle{empty}
\phantom{-} \vspace{1cm}
\begin{center}
\huge{\textbf{Tensor categorical foundations}}\\ \medskip
\huge{\textbf{of algebraic geometry}}

\vspace{16mm}

{\large Martin Brandenburg\\ \medskip
-- 2014 --}

\vspace{2cm}

\end{center}

\begin{center}
\textbf{Abstract}
\end{center}

Tannaka duality and its extensions by Lurie, Schäppi et al. reveal that many schemes as well as algebraic stacks may be identified with their tensor categories of quasi-coherent sheaves. In this thesis we study constructions of cocomplete tensor categories (resp. cocontinuous tensor functors) which usually correspond to constructions of schemes (resp. their morphisms) in the case of quasi-coherent sheaves. This means to \emph{globalize} the usual local-global algebraic geometry. For this we first have to develop basic commutative algebra in an arbitrary cocomplete tensor category. We then discuss tensor categorical globalizations of affine morphisms, projective morphisms, immersions, classical projective embeddings (Segre, Plücker, Veronese), blow-ups, fiber products, classifying stacks and finally tangent bundles. It turns out that the universal properties of several  moduli spaces or stacks translate to the corresponding tensor categories.

This is a slightly expanded version of the author's PhD thesis.

\end{titlepage}



\clearpage
\titlespacing*{\chapter}{0pt}{0pt}{40pt}
\pagestyle{empty}
\tableofcontents

\newpage
\thispagestyle{empty}
\vspace*{\fill}
\begin{quote}
\begin{onehalfspace}
\emph{``According to Peter Freyd: 'Perhaps the purpose of categorical algebra is to show that which is trivial is trivially trivial.' That was written early on, in 1966. I prefer an update of that quote: ``Perhaps the purpose of categorical algebra is to show that which is formal is formally formal''. It is by now abundantly clear that mathematics can be formal without being trivial. Categorical algebra allows one to articulate analogies and to perceive unexpected relationships between concepts in different branches of mathematics.''} \bigskip \\  \mbox{} \hfill Peter May, \cite{May01}
\end{onehalfspace}
\end{quote}
\vspace*{\fill}

\pagestyle{fancy}

\thispagestyle{empty}

\hypersetup{pageanchor=true}
\pagenumbering{arabic}

\chapter{Introduction}



\section{Background}

It is a common principle of mathematics to study geometric objects by considering categories of algebraic objects acting on them. For example in algebraic geometry one tries to understand a scheme $X$ via its category of quasi-coherent sheaves $\Q(X)$. Do we really understand $X$ by looking at $\Q(X)$?

This has been answered in 1962 by Gabriel (\cite{Gab62}) who proved his famous Reconstruction Theorem: A noetherian scheme $X$ can be reconstructed from $\Q(X)$. The idea is to associate to an abelian category $\A$ a ringed space $\Spec(\A)$, its \emph{spectrum}, and then prove $X \cong \Spec\bigl(\Q(X)\bigr)$. Later this paved the way for the vision of \emph{noncommutative algebraic geometry} in which abelian categories are regarded as noncommutative schemes (\cite{Ros95}, \cite{AZ94}). In 1998 A. L. Rosenberg used a different spectrum construction to give a short proof of the Reconstruction Theorem for arbitrary schemes (\cite{Ros96}, \cite{Ros98}), then in 2004 another proof for quasi-separated schemes (\cite{Ros04}). The proof has been corrected and simplified by Gabber (\cite{Bra13}). In 2013 Antieau obtained a generalization of the Reconstruction Theorem to twisted quasi-coherent sheaves on quasi-compact quasi-separated schemes (\cite{Ant13}). With a completely new approach Calabrese and Groechenig have proven a Reconstruction Theorem for quasi-compact separated algebraic spaces (\cite{Cal13}). Although a smooth variety $X$ cannot be reconstructed from its derived category of quasi-coherent sheaves $D^b(\Q(X))$, Bondal and Orlov have shown in 2001 that this works at least if the canonical bundle of $X$ is ample or anti-ample (\cite{Bon01}).

However, $\Q(X)$ is actually a \emph{tensor category}, so it is reasonable to use this additional structure. In 2002 Balmer has reconstructed a noetherian scheme from its tensor triangulated category of perfect complexes (\cite{Bal02}). This lead to the development of \emph{tensor triangular geometry} (\cite{Bal10}). Balmer's result has been generalized to quasi-compact quasi-separated schemes by Buan, Krause and Solberg (\cite{Bua07}).

What about reconstruction for algebraic stacks? Classical Tannaka duality gives an affirmative answer in the case of classifying stacks since it reconstructs an affine group scheme from its tensor category of representations (\cite{Del90}). More generally, Lurie has reconstructed an arbitrary geometric stack from its tensor category of quasi-coherent sheaves (\cite{Lur04}).

In the case of algebraic stacks the tensor structure becomes essential, since for example the discrete scheme with two points and the classifying stack of a group of order $2$ have the same category of quasi-coherent sheaves -- but not the same tensor category of quasi-coherent sheaves.

Obviously very much has been done in order to recover the object of interest $X$ from $\Q(X)$. However, it seems that little has been done towards the reconstruction of \emph{morphisms} of schemes, which is certainly useful for recovering properties of $X$ (such as smoothness) from $\Q(X)$. One might ask which functors
\[\Q(X) \to \Q(Y)\]
are induced by morphisms $Y \to X$ via the pullback construction.

A minimal requirement is that these functors preserve colimits and tensor products (up to specified isomorphisms). Therefore, let us call a scheme (or stack) $X$ \emph{tensorial} if colimit-preserving tensor functors $\Q(X) \to \Q(Y)$ are induced by morphisms $Y \to X$. Then for example affine and projective schemes are tensorial (\cite{Bra11}). More generally, in 2012 Chirvasitu and the author showed that any quasi-compact quasi-separated scheme is tensorial (\cite{BC14}). In the same year, Sch\"appi proved that Adams stacks (i.e. geometric stacks with enough locally free sheaves) are tensorial (\cite{Sch12a}). He even gave a classification of those tensor categories which are equivalent to $\Q(X)$ for some Adams stack $X$ (\cite{Sch12a},\cite{Sch13}).

It seems to be an open problem if every algebraic stack (satisfying mild finiteness conditions) is tensorial and to characterize those tensor categories which are equivalent to $\Q(X)$ for some algebraic stack $X$. \marginpar{I've added the remark about finiteness conditions.}

This thesis aims to expand the relationship between algebraic geometry and cocomplete tensor categories. We will find constructions with cocomplete tensor categories which correspond to constructions with schemes or stacks under $\Q(-)$. This leads to a theory which might be called \emph{tensor categorical algebraic geometry} or \emph{$2$-algebraic geometry} which has been already introduced by Chirvasitu and Johnson-Freyd (\cite{CJF13}). In some sense this is like Rosenberg's noncommutative algebraic geometry with additional tensor products (which makes it commutative) or Balmer's tensor triangular geometry without triangulations. The precise results will be explained in the next section.


\section{Results}
 
\subsection*{Tensoriality}

A cocomplete tensor category is a symmetric monoidal category with all colimits which commute with the tensor product in each variable (\autoref{deftens}). These constitute a $2$-category $\Cat_{c\otimes}$. The Hom-categories of cocontinuous tensor functors are denoted by $\Hom_{c\otimes}(-,-)$. We often restrict to linear categories over some fixed commutative ring. We have the following adjunction (\autoref{adju}) of $2$-categories:
\[\xymatrix@C=40pt{\Stack \ar@/_1pc/[r]_{\Q} \ar@{}[r]|{\top} & \Cat_{c\otimes}^{\op} \ar@/_1pc/[l]_{\SPEC}}\]
Here, the $2$-functor $\Q : \Stack \to \Cat_{c\otimes}^{\op}$ associates to a stack in the fpqc topology its cocomplete tensor category of quasi-coherent modules and to a morphism $f$ its pullback functor $f^*$. Its right adjoint $\SPEC : \Cat_{c\otimes}^{\op} \to \Stack$ is defined by
\[\SPEC(\C)(X) := \Hom_{c\otimes}\bigl(\C,\Q(X)\bigr).\]
We call the fixed points of this adjunction \emph{tensorial stacks} resp. \emph{stacky tensor categories}. Hence, we get an equivalence of $2$-categories
\[\{\text{tensorial stacks}\} \simeq \{\text{stacky tensor categories}\}^{\op}.\]
Specifically, a stack $X$ is tensorial if
\[X \simeq \SPEC\bigl(\Q(X)\bigr),\]
i.e. for every scheme $Y$ the pullback functor
\[\Hom(Y,X) \to \Hom_{c\otimes}\bigl(\Q(X),\Q(Y)\bigr),~ f \mapsto f^*\]
is an equivalence of categories.
 
The following results have already appeared, but shorter proofs will be byproducts of our theory (\autoref{tensoriality}):
 
\begin{thmi} \noindent
\begin{enumerate}
\item \textup{\cite{Lur04}} If $X$ is geometric, then
\[\Hom(Y,X) \to \Hom_{c\otimes \cong}\bigl(\Q(X),\Q(Y)\bigr),~ f \mapsto f^*\]
is fully faithful. The image consists of those cocontinuous tensor functors which preserve faithfully flatness. 
\item \textup{\cite{Sch12a}} Every Adams stack is tensorial.
\item \textup{\cite{BC14}} Every quasi-compact quasi-separated scheme is tensorial.
\end{enumerate}
\end{thmi}

The $2$-category of tensorial stacks embeds fully faithfully into $\Cat_{c\otimes}^{\op}$. If we imagine cocomplete tensor categories as $2$-semirings (not $2$-rings) with addition $\oplus$ and multiplication $\otimes$ (\autoref{categorif}), this says that tensorial stacks become $2$-affine. This leads to the following question, partially answered in \autoref{Cons}:
 
\begin{center}
How is algebraic geometry reflected under this embedding?
\end{center}
 
We can even go further and imagine a given cocomplete tensor category as the category of quasi-coherent modules on some ``imaginary'' scheme or stack. This point of view is quite common for abelian categories in noncommutative algebraic geometry. In particular, we will often write $\O_\C$ for the unit object of a cocomplete tensor category $\C$. This can be combined with some kind of element notation in \autoref{Element} in order to simplify calculations.

\subsection*{Commutative algebra}
 
Many local notions and constructions for modules over some commutative ring, or more generally quasi-coherent modules over some scheme, can actually be carried out in an arbitrary (linear) cocomplete tensor category. In category theory such a  process is called \emph{internalization}. Examples include algebras and modules over them (\autoref{algmod}), ideals, prime ideals and affine schemes (\autoref{affineschemes}), symmetric and exterior algebras (\autoref{symmetric-power}), invertible objects (\autoref{invob}), locally free objects of some given rank (\autoref{localfree}), (faithfully) flat objects (\autoref{flat}), localizations (\autoref{locsec}), torsors (\autoref{rep}), differentials (\autoref{deriv}), descent theory (\autoref{descent-theory}) and even cohomology (\autoref{cohomo}). This is mainly the subject of \autoref{Comm}. For some of these notions this generalization is straightforward since the usual definitions only require the tensor structure. But for example flat as well as locally free quasi-coherent modules are usually defined locally -- actually they have global replacements. Most of the material of \autoref{Comm} is probably known to the experts -- but I do not know of a complete treatment. Besides, \autoref{Comm} develops the language which will be needed throughout in \autoref{Cons}.

Here is a selection of topics we discuss in \autoref{Comm}.

\begin{enumerate}
\item In \autoref{affineschemes} we define a  ringed space $\Spec_\C(A)$ of prime ideals of a commutative algebra object $A$ of a locally presentable tensor category $\C$ (more concretely than \cite{Mar09},\cite{ToVa09}) and prove for noetherian schemes $X$ the reconstruction result $X \cong \Spec_{\Q(X)}(\O_X)$.
\item In \autoref{rham} we construct a de Rham complex
\[B=\Omega^0_{B/A} \to \Omega^1_{B/A} \to \Omega^2_{B/A} \to \dotsc\]
for a homomorphism $A \to B$ of commutative algebras in a  cocomplete $R$-linear tensor category with $2 \in R^*$.
\item We generalize the classical Euler sequence of the tangent bundle of a projective space to arbitrary projective schemes (\autoref{euler}).
\item A \emph{line object} $\L$ is an invertible object which is also \emph{symtrivial} (\autoref{symtrivial}). This seems to be the correct replacement of line bundles for tensor categories. It turns out that if $s : E \to \L$ is a regular epimorphism (``global sections generating $\L$''), then it is automatically the coequalizer of the two evident morphisms $E^{\otimes 2} \otimes \L^{\otimes -1} \rightrightarrows E$ (\autoref{goodepi}). This will be needed for the construction of projective tensor categories (\autoref{projcat}).
\item If $V$ is a quasi-coherent module on a scheme such that $\Lambda^d V$ is invertible, then $V$ is locally free of rank $d$ (\autoref{locally-free-sch}) -- this motivates a global definition of locally free objects of rank $d$ in a cocomplete linear tensor category (\autoref{lokfreedef}). A generalization of Cramer's rule shows that locally free objects are dualizable (\autoref{lokdual}).
\end{enumerate}

\subsection*{Globalization} 
 
Many constructions for schemes or algebraic stacks can be translated to the realm of cocomplete tensor categories. We call this process \emph{globalization}, basically because the usual local constructions and proofs have to be replaced by global ones (\autoref{globalization}). This will be the subject of \autoref{Cons}. The term ``globalization'' was already used by Grothendieck for the generalization of commutative algebra and algebraic geometry to topoi (which may be seen as cartesian tensor categories).
 
For example, let $F : \Sch \to \Sch$ be a functor, which thus constructs a scheme from a given scheme. A \emph{globalization} of $F$ is a $2$-functor $F_{\otimes} : \Cat_{c\otimes} \to \Cat_{c\otimes}$ such that $F_{\otimes}(\Q(X)) \simeq \Q(F(X))$, as in the following diagram:

\[\xymatrix@=40pt{\Sch \ar[d]_{\Q} \ar@{~>}[r]^{F} & \Sch  \ar[d]^{\Q} \\ \Cat_{c\otimes}^{\op} \ar@{~>}[r]^{F_{\otimes}} & \Cat_{c\otimes}^{\op}}\]
 
Besides, if $F$ has some universal property which is definable in the language of cocomplete tensor categories (consisting of objects, morphisms, colimits, tensor products), we require that $F_{\otimes}$ satisfies the corresponding universal property, but within \emph{all} cocomplete tensor categories.

We can also look at globalizations for more general constructions by allowing additional inputs such as
 
\begin{itemize}
\item quasi-coherent modules or algebras on the given scheme,
\item more than one scheme,
\item morphisms of schemes,
\item algebraic stacks instead of schemes.
\end{itemize}

If we have no parameter at all, this means that given a scheme with a universal property or a moduli stack $X$ we want $\Q(X)$ to become a corresponding ``moduli tensor category''. This is quite similar to the notion of a classifying topos (\cite[Chapter VIII]{Mac92}). \marginpar{I've added the classifying-topos remark}

 As a general rule, globalized constructions of schemes preserve tensorial schemes. For the author this was the original motivation for globalization, because it made possible to prove that projective schemes are tensorial (\cite{Bra11}).
 
Let us list a variety of examples.
 
\subsection*{Basic examples}

\begin{propi}[Globalization of the terminal scheme, \autoref{initial}]
If $R$ is a commutative ring, then $\Spec(R)$ is a terminal object in the category of $R$-schemes. Its globalization is the cocomplete tensor category of $R$-modules $\M(R) \simeq \Q(\Spec(R))$ which is in fact a $2$-initial object in $\Cat_{c\otimes/R}$.
\end{propi}
  
\begin{propi}[Globalization of Spec, \autoref{module-categories}]
Let $S$ be a scheme and $A$ be a quasi-coherent algebra on $S$. The spectrum $\Spec_S(A)$ globalizes to $\M(A)$:
\begin{itemize}
\item There is an equivalence of cocomplete tensor categories
\[\Q(\Spec_S(A)) \simeq \M(A).\]
\item For schemes $T$ there is a bijection
\[\begin{array}{c}
\Hom(T,\Spec_S(A))) \medskip \\ 
||\wr \medskip \\ 
\{(f,h) : f \in \Hom(T,S),\, h \in \Hom_{\CAlg(T)}(f^*(A),\O_T)\}.
\end{array}\]
\item If $\C$ is a cocomplete tensor category and $A$ is a commutative algebra in $\C$, then $\M(A)$ enjoys the universal property 
\[\begin{array}{c}
\Hom_{c\otimes}(\M(A),\D) \medskip \\ 
|\wr \medskip \\ 
\{(F,h) : F \in \Hom_{c\otimes}(\C,\D),\, h \in \Hom_{\CAlg(\D)}(F(A),\O_\D)\}.
\end{array}\]
for cocomplete tensor categories $\D$.
\end{itemize}
\end{propi}

This leads to a globalization of affine schemes (\autoref{aff-glob}) and of closed immersions (\autoref{closed-UE}). It also leads to the notion of an \emph{affine tensor functor} (\autoref{affinefunk}). Affine tensor functors these are closed under composition and cobase changes (\autoref{mod-BW}).

\begin{propi}[Globalization of open subschemes, \autoref{local}]
Let $i : U \hookrightarrow X$ be a quasi-compact open immersion of schemes. Then a morphism $f : Y \to X$ factors through $i$ if and only if $f^* \to f^* i_* i^*$ is an isomorphism. This globalizes as follows: Let $i^* : \C \to \D$ be a cocontinuous tensor functor with a fully faithful right adjoint $i_* : \D \to \C$. Then there is an equivalence of categories between cocontinuous tensor functors $\D \to \E$ and cocontinuous tensor functors $F : \C \to \E$ with the property that $F \to F i_* i^*$ is an isomorphism.
\end{propi}

\begin{propi}[Globalization of Coproducts I, \autoref{qprod}]
If $(X_i)$ is a family of schemes, then their disjoint union $\coprod_i X_i$ is a coproduct. If $(\C_i)$ is a family of cocomplete tensor categories, we can construct their product $\prod_i \C_i$ in the obvious way. If $\C_i = \Q(X_i)$ for schemes $X_i$, then there is an equivalence of cocomplete tensor categories $\prod_i \Q(X_i) \simeq \Q(\coprod_i X_i)$.
\end{propi}

\begin{propi}[Globalization of Coproducts II, \autoref{prodUE}]
If $(X_i)$ is a family of schemes, their disjoint union $\sqcup_i X_i$ satisfies the following universal property: A morphism $Y \to \sqcup_i X_i$ corresponds to a decomposition $Y = \sqcup_i Y_i$ and a family of morphisms $f_i : Y_i \to X_i$. We globalize this: If $(\C_i)$ is a family of cocomplete tensor categories, then their product $\prod_i \C_i$ has the universal property that the category of cocontinuous tensor functors $\prod_i \C_i \to \D$ is equivalent to the category of decompositions $\D \simeq \prod_i \D_i$ (which may  be classified via systems of orthogonal idempotents) and a family of cocontinuous tensor functors $\C_i \to \D_i$.
\end{propi}

This implies that tensorial stacks are closed under arbitrary coproducts (\autoref{coprod-tens}).

\subsection*{Projective geometry}

Let $S$ be a scheme and $A$ be some $\N$-graded quasi-coherent algebra on $S$ such that $A_1$ is of finite presentation over $S$ and $A$ is generated by $A_1$. Then $\Proj_S(A)$ is a scheme with the universal property that $\Hom(T,\Proj_S(A))$ identifies with the set of triples $(f,\L,s)$ consisting of a morphism $f : T \to S$, a line bundle $\L$ on $T$ and a surjective homomorphism of graded algebras $s : f^*(A) \to \bigoplus_n \L^{\otimes n}$ (see \cite[\href{http://stacks.math.columbia.edu/tag/01O4}{Tag 01O4}]{stacks-project}).

\begin{thmi}[Globalization of the Proj construction, \autoref{proj-def}]
Let $\C$ be a cocomplete tensor category, $A$ a commutative $\N$-graded algebra in $\C$ which is generated by $A_1$. We define $\Proj^{\otimes}_\C(A)$ as a cocomplete tensor category with the following universal property: If $\D$ is a cocomplete tensor category, then $\Hom_{c\otimes}(\Proj^{\otimes}_\C(A),\D)$ is naturally equivalent to the category of triples $(F,\L,s)$, where $F: \C \to \D$ is a cocontinuous tensor functor, $\L \in \D$ is a line object and $s : F(A) \to \bigoplus_n \L^{\otimes n}$ is a homomorphism of graded algebras in $\C$ such that $s_1 : F(A_1) \to \L$ is a regular epimorphism in $\D$.
\begin{enumerate}
\item If $\C$ is locally presentable, then $\Proj^{\otimes}_\C(A)$ exists (\autoref{proj-ex}).
\item If $\C=\Q(S)$ for some scheme $S$, then (\autoref{proj-compare})
\[\Proj^{\otimes}_{\Q(S)}(A) \simeq \Q\bigl(\Proj(A)\bigr).\]
\end{enumerate}
\end{thmi}

In the special case $A=\Sym(E)$ for some (locally free) $E \in \C$ one gets a globalization of projective bundles $\P^\otimes_\C(E)$ relative to $\C$ (\autoref{projbundle}), in particular a projective $n$-space $\P^n_\C := \P^{\otimes}_\C(\O_\C^{n+1})$. Projective tensor functors are closed under cobase change (\autoref{proj-BW}).

\begin{thmi}[Globalization of blow-ups, \autoref{blow-ex}]
Let $\C$ be a locally presentable tensor category and let $I \subseteq \O_\C$ be an ideal. The blow-up of $\C$ at $I$ is a $2$-initial locally presentable tensor category $\Bl_I(\C)$ equipped with a cocontinuous tensor functor $P : \C \to \Bl_I(\C)$ such that $P(I) \cdot \O_{\Bl_I(\C)}$ is a line object. It exists and is given by
\[\Bl_I(\C) = \Proj^\otimes_\C\bigl(\bigoplus_{n \geq 0} I^n\bigr).\]
If $\C=\Q(X)$ for some scheme $X$ and $I$ is of finite presentation, we have
\[\Q\bigl(\Bl_I(X)\bigr) \simeq \Bl_I \bigl(\Q(X)\bigr).\]
\end{thmi}
  
\begin{thmi}[Globalization of the Segre embedding, \autoref{segre-emb}]
Let $\C$ be a cocomplete tensor category and $E_1,E_2 \in \C$. Then there is an equivalence of cocomplete tensor categories (using element notation)
\[\P^\otimes_\C(E_1) \sqcup_\C \P^\otimes_\C(E_2) \simeq \P^\otimes_\C(E_1 \otimes E_2) / ((a \otimes b) \cdot (c \otimes d) = (a \otimes d) \cdot (c \otimes b))_{a,c \in E_1,\, b,d \in E_2}.\]
\end{thmi}

\begin{thmi}[\autoref{veronese}, Globalization of the Veronese embedding]
Let $\C$ be a cocomplete tensor category, $E \in \C$ and $d \in \N^+$. Then there is an equivalence cocomplete tensor categories (using element notation)
\[\P_\C^\otimes(E) \simeq \P_\C^\otimes(\Sym^d(E))/((a_1 a_2 \dotsc a_d) \cdot (b_1 b_2 \dotsc b_d) = (b_1 a_2 \dotsc a_d) (a_1 b_2 \dotsc b_d)).\]
\end{thmi}

\begin{thmi}[Globalization of Grassmannians and their Plücker embedding, \autoref{grassex}]
Let $R$ be a $\QQ$-algebra, $\C$ be an $R$-linear cocomplete tensor category and $E \in \C$. Then $\Grass_d^\otimes(E)$ is a cocomplete tensor category with the following universal property: If $\D$ is an $R$-linear cocomplete tensor category, then $\Hom_{c\otimes/R}(\Grass_d^\otimes(E),\D)$ is equivalent to the category of triples $(F,V,t)$, where $F : \C \to \D$ is a cocontinuous tensor functor, $V \in \D$ is locally free of rank $d$ and $t : F(E) \to V$ is a regular epimorphism.
\begin{enumerate}
\item When $\C$ is locally presentable, then $\Grass_d^\otimes(E)$ exists and is given by
\[\P_\C^\otimes(\Lambda^d(E))/(0=\sum_{k=0}^{d} (-1)^k (a_1 \wedge \dotsc \wedge a_{d-1} \wedge b_k) \cdot (b_0 \wedge \dotsc \wedge \widehat{b_k} \wedge \dotsc \wedge b_d)).\]
\item If $\C=\Q(X)$ for some $R$-scheme $X$ and if $E$ is of finite presentation, then \marginpar{I've added R- to the scheme.}
\[\Grass_d^\otimes(E) \simeq \Q\bigl(\Grass_d(E)\bigr).\]
\end{enumerate}
\end{thmi}

\subsection*{Further examples}

\begin{thmi}[Globalization of fiber products I, \autoref{quasiproj}] 
Let $S$ be a quasi-compact quasi-separated base scheme. Let $X \to S$ be a quasi-projective morphism and $Y \to S$ be any quasi-compact quasi-separated morphism. Then
\[\xymatrix{\Q(S) \ar[r] \ar[d] & \Q(Y) \ar[d] \\ \Q(X) \ar[r] & \Q(X \times_S Y)}\]
is a $2$-pushout in $\Cat_{c\otimes}$.
\end{thmi}

The proof is a combination of the globalization of projective morphisms and immersions. The following result is an attempt to get rid of the quasi-projective assumption.

\begin{thmi}[Globalization of fiber products II, \autoref{prod1}]
Let $X,Y$ be schemes of finite type over a field $K$. If $\C$ is a cocomplete $K$-linear tensor category, then the canonical functor
\[\Hom_{c\otimes/K}(\Q(X \times_K Y),\C) \to \Hom_{c\otimes/K}(\Q(X),\C) \times \Hom_{c\otimes/K}(\Q(Y),\C)\]
is fully faithful. The essential image consists of those pairs $(F,G)$ with the following property: Given an exact sequence of external tensor products
\[A' \boxtimes B' \to A \boxtimes B \to C \boxtimes D \to 0\]
of coherent sheaves, then the induced sequence
\[F(A') \otimes G(B') \to F(A) \otimes G(B) \to F(C) \otimes G(D) \to 0\]
is also exact in $\C$.
\end{thmi}

\begin{propi}[Globalization of $B \G_m$, \autoref{Gm-global}]
The classifying stack $B\G_m$ over some base $S$ has the following universal property: A morphism $T \to B\G_m$ corresponds to a morphism $T \to S$ together with a line bundle on $T$. This globalizes as follows: If $\C$ is a cocomplete tensor category, then the category of $\Z$-graded objects $\gr_{\Z}(\C)$ is a cocomplete tensor category with the following universal property: The category of cocontinuous tensor functors $\gr_{\Z}(\C) \to \D$ is equivalent to the category of cocontinuous tensor functors $F : \C \to \D$ equipped with a line object $\L \in \D$. If $\C = \Q(S)$ for some scheme $S$, then $\gr_{\Z}(\Q(S)) \simeq \Q(B\G_m)$.
\end{propi}

\begin{thmi}[Globalization of classifying stacks, \autoref{BG-global}]
If $G$ is a finite group, then the classifying stack $BG$ has the universal property
\[\Hom(X,BG) = \Tors(G,X) \simeq \Tors(G,\Q(X)).\]
This globalizes to $\Q(BG) \simeq \Rep(G)$ (representations of $G$), in fact we have
\[\Hom_{c\otimes}(\Rep(G),\C) \simeq \Tors(G,\C).\]
for cocomplete tensor categories $\C$.
\end{thmi}

The following is probably the most non-trivial result in this thesis.
 
\begin{thmi}[Globalization of tangent bundles, \autoref{tangent-proj}]
If $S$ is a base scheme, the tangent bundle functor $T : \Sch/S \to \Sch/S$ is right adjoint to the thickening functor $\Sch/S \to \Sch/S,~ X \mapsto X[\e]/\e^2$. We globalize this by defining $T^{\otimes} : \Cat_{c\otimes/\C} \to \Cat_{c\otimes/\C}$ for a linear cocomplete tensor category $\C$ by the adjunction
\[\Hom_{c\otimes/\C}(T^{\otimes}(\D/\C),\E) \simeq \Hom_{c\otimes/\C}(\D,\E[\e]/\e^2).\]
If $\D$ is projective over $\C$, say $\D=\P_\C^\otimes(E)$, then $T^{\otimes}(\D/\C)$ exists and is given by $\M(\Sym \Omega^1)$, where $\Omega^1 \in \D$ is defined via the Euler sequence
\[\Lambda^3(E) \otimes \O(-3) \to \Lambda^2(E) \otimes \O(-2) \to \Omega^1 \to 0.\]
Hence, if $f : X \to S$ is a projective morphism, we have
\[T^{\otimes}\bigl(\Q(X)/\Q(S)\bigr) \simeq \Q\bigl(T(X/S) \bigr).\]
\end{thmi}

As an application, we define \emph{formally unramified} cocontinuous tensor functors and prove that a projective morphism $f : X \to S$ is formally unramified if and only if $f^* : \Q(S) \to \Q(X)$ is formally unramified (\autoref{fr}).

\subsection*{More on cocomplete tensor categories}

We may even abstract from algebraic geometry and study cocomplete tensor categories in their own right.

In \autoref{free} we discuss basic free constructions of (cocomplete) tensor categories. For example, if $I$ is a tensor category and $\C$ is a cocomplete tensor category, we may endow the category of functors $I^{\op} \to \C$ with the \emph{Day convolution} \marginpar{Added q to the coend.}
\[F \otimes G := \int^{p,q \in I} F(p) \otimes G(q) \otimes \Hom(-, p \otimes q)\]
and get a cocomplete tensor category $\widehat{I}_\C$ satisfying the universal property (\autoref{hat})
\[\Hom_{c\otimes}(\widehat{I}_\C,\D) \simeq \Hom_{c\otimes}(\C,\D) \times \Hom_{\otimes}(I,\D).\]
This includes the free cocompletion of a tensor category (\autoref{freecoco}), but also gradings (\autoref{gradings}). For $I=\N$ we see that $\gr_\N(\M(R))$ is the free cocomplete $R$-linear tensor category containing a symtrivial object -- this surely has no pendant in algebraic geometry in contrast to $\gr_\Z(\M(R))$.
  
According to James Dolan every well-behaved cocomplete tensor category should be made up from the initial one ($\Set$) by (transfinite) iterations of a) free cocomplete tensor categories, b) directed colimits and c) module categories over ``coherent'' symmetric monoidal monads. We study the latter in \autoref{monoidalmonads} on a fixed cocomplete tensor category $\C$ (this has some overlap with \cite{Sea13} which appeared at the same time). The special case $\C=\Set$ has already been considered by Durov in his thesis \cite{Dur07} who defined generalized commutative rings as finitary symmetric monoidal monads on $\Set$. Therefore we might think of coherent symmetric monoidal monads on $\C$ as generalized commutative algebras in $\C$. They produce module categories:
 
\begin{thmi}[\autoref{maschine}]
If $T$ is a coherent symmetric monoidal monad on a cocomplete tensor category $\C$, then $\M(T)$ carries the structure of a cocomplete tensor category and the free functor $F : \C \to \M(T)$ becomes a cocontinuous tensor functor.
\end{thmi}

The coherence condition simplifies in the closed case (\autoref{closed-maschine}). It is even automatic when $T$ preserves reflexive coequalizers (\autoref{refcoherent}). In that case the following universal property holds:

\begin{thmi}[\autoref{monadUE1}]
Let $T$ be a symmetric monoidal monad on a cocomplete tensor category $\C$, which preserves reflexive coequalizers. Then $\M(T)$ is a cocomplete tensor category and $F : \C \to \M(T)$ is a cocontinuous tensor functor equipped with a lax symmetric monoidal right action $FT \to F$, inducing for every cocomplete tensor category $\D$ an equivalence of categories between $\Hom_{c\otimes}(\M(T),\D)$ and the category of pairs $(G,\rho)$, where $G \in \Hom_{c\otimes}(\C,\D)$ and $\rho : GT \to G$ is a lax symmetric monoidal right action.
\end{thmi}

This universal property offers a simplification if we work with left adjoint tensor functors throughout (\autoref{monadUE2}).

The cocomplete tensor categories of commutative monoids, abelian groups, $R$-modules, pointed sets or spaces, and even complete partial orders, as well as their universal properties, are some examples of how these theorems can be applied.

The case of idempotent monads is connected  to localization theory which we develop in \autoref{localization-section}. It is mainly motivated by the construction of projective tensor categories. A typical example is the construction of the tensor category of sheaves from the tensor category of presheaves: One sheafifies the underlying tensor product. In general the sheafification is replaced by a reflector which is constructed as a transfinite composition of a suitable endofunctor.

\begin{thmi}[Localization, \autoref{localization}]
Let $S = (s_i : M_i \to N_i)_{i \in I}$ be a family of morphisms in a locally presentable tensor category $\C$. Then
\[\C[S^{-1}] := \{A \in \C : \HOM(N_i,A) \xrightarrow{s_i^*} \HOM(M_i,A) \text{ is an isomorphism for all } i\}\]
is reflective in $\C$, say with reflector $R : \C \to \C[S^{-1}]$. Then $\C[S^{-1}]$ becomes a locally presentable tensor category with unit $R(\O_\C)$ and tensor product
\[M \otimes_{\C[S^{-1}]} N := R(M \otimes N).\]
Besides, $R : \C \to \C[S^{-1}]$ is a universal cocontinuous tensor functor with the property that $R(s_i)$ is an isomorphism for all $i \in I$.
\end{thmi}

Localization allows us to solve many other universal problems for cocomplete tensor categories (\autoref{universal-problem}), for example making two given morphisms equal, making a morphism an epimorphism, making a diagram right exact, making an object zero or making it invertible (\autoref{makeinv}). We discuss many examples in \autoref{examplesloc}. It also shows that every \emph{ideal} of $\C$ (a full subcategory closed under certain operations) is actually the kernel of a cocontinuous tensor functor (\autoref{concrete-is-abstract}).

\clearpage
\section{Acknowledgements}

First of all I would like to thank my advisor Christopher Deninger for his cordial support during my graduate studies. He gave me as much freedom as I needed in order to complete this project.

For various fruitful discussions connected to the topics of this thesis I would like to express my thanks to Leo Alonso, Peter Arndt,  \marginpar{I've added Peter Arndt and David Rydh.} David Ben-Zvi, Alexandru Chirvasitu,  Anton Deitmar, Christopher Deninger, James Dolan, Torsten Ekedahl, Ofer Gabber, Theo Johnson-Freyd, Anders Kock, Zhen Lin Low, Jacob Lurie, Laurent Moret-Bailly, David Rydh, Daniel Sch\"appi,  Jakob Scholbach, Gavin Seal, Tobias Sitte, Jason Starr, Neil Strickland, Georg Tamme, Todd Trimble, Angelo Vistoli, Stefan Vogel and Dimitri Wegner.

For their attentive and meticulous proofreading of this thesis I am much obliged to Oskar Braun, Erik Friese, Eva H\"{o}ning and David Zabka.

I am indebted to Siegfried Bosch for introducing me to algebraic geometry and to Ferit Deniz for introducing me to category theory several years ago.
 
This work was partially supported by the SFB 878 ``Groups, Geometry \& Actions''.


\clearpage
\chapter{Preliminaries} \label{Preliminaries}


\section{Category theory} \label{colimits}

\textbf{1.} We assume that the reader is familiar with some basics of category theory (see \cite{ML98}, \cite{Bor94a}). The Yoneda Lemma will be used all the time. As for the notation, we will denote categories by $\C,\D,\E,\dotsc$, objects by $A,B,C,\dotsc$ (or $M,N,\dotsc$ in case of module-like categories) and functors by $F,G,H,\dotsc$. We will often abbreviate $\id_A$ by $A$.

\textbf{2.} Let us recall some well-known notions which are important for us: A category is \emph{cocomplete} if it admits all small colimits. Usually we omit the word ``small''. Two types of colimits are especially important, since they generate all other ones, namely coproducts, which we will call direct sums and denote by $\oplus$, and coequalizers. A diagram of morphisms of the type
\[A \rightrightarrows B \to C\]
is called \emph{exact}, or a \emph{coequalizer diagram}, if it exhibits $B \to C$ as a coequalizer of the two morphisms $A \rightrightarrows B$. In other words, for every $T \in \C$, which we call a \emph{test object}, the diagram of sets
\[\Hom(C,T) \to \Hom(B,T) \rightrightarrows \Hom(A,T)\]
is exact i.e. an equalizer diagram of sets.

\textbf{3.} If $A$ is an object of a cocomplete category and $I$ is a set, we will often write $I \otimes A$ for the direct sum $\bigoplus_{i \in I} A$.

\textbf{4.} A functor $F : \C \to \D$ is called \emph{cocontinuous} if it preserves all colimits, i.e. for every diagram $\{A_i\}_{i \in I}$ in $\C$ the canonical morphism
\[\colim_{i \in I} F(A_i) \to F(\colim_{i \in I} A_i)\]
is an isomorphism. If $F$ preserves at least finite colimits, we say that $F$ is \emph{finitely cocontinuous} or \emph{right exact}. Observe that a right exact functor preserves epimorphisms, since $A \to B$ is an epimorphism if and only if
\[\xymatrix{A \ar[r] \ar[d] & B \ar[d] \\ B \ar[r] & B}\]
is a pushout square. By duality we obtain the notion of \emph{lext exact} functors. A functor is called \emph{exact} if it is left exact and right exact. A functor is \emph{finitary} when it preserves directed colimits.

\textbf{5.} We will denote any initial object of a category by $0$. Of course a cocontinuous functor preserves initial objects.
 
\textbf{6.} For cocomplete categories $\C,\D$ the category of cocontinuous functors $\C \to \D$ with morphisms of functors is denoted by $\Hom_c(\C,\D)$. 

\textbf{7.} In an equivalence of categories, we do not speak of quasi-inverse functors, but just of inverse functors, since it is understood that interesting functors are almost never inverse to each other on the nose.

\textbf{8.} We will also need some bits of $2$-category theory (see \cite[Section 9]{St96}). We will change the terminology a bit and call bicategories $2$-categories (not assumed to be strict) and pseudo-functors $2$-functors. By adjoints between $2$-functors we mean of course $2$-adjoints. Bilimits (resp. bicolimits) are called $2$-limits (resp. $2$-colimits).

\textbf{9.} A category is called \emph{discrete} if it only contains identity morphisms. A category equivalent to a discrete category is called \emph{essentially discrete}. This means that every morphism is an isomorphism and that every two parallel morphisms agree.

\textbf{10.} Let us recall the \emph{tensor product of functors} (see \cite[2.4]{CJF13} or \cite[IX.6]{ML98}). Given functors $\otimes : \C \times \D \to \E$, $F : I^{\op} \to \C$ and $G : I \to \D$, we may define the tensor product $F \otimes_I G \in \E$ as an object which represents the functor $\E \to \Set$ which maps $T \in \E$ to the set of all families of maps $(F(i) \otimes G(i) \to T)_{i \in I}$ which are \emph{dinatural}: For every morphism $i \to j$ in $I$ the diagram
\[\xymatrix{F(j) \otimes G(i) \ar[r] \ar[d] & F(j) \otimes G(j) \ar[d] \\ F(i) \otimes G(i) \ar[r] & T}\]
commutes. We have an equivalent expression as a \emph{coend}
\[F \otimes_I G := \int^{i \in I} F(i) \otimes G(i).\]
If $I$ is small and $\E$ is cocomplete, we may construct $F \otimes_I G$ by taking the coequalizer of the obvious morphisms
\[\bigoplus_{i \to j} F(j) \otimes G(i) \rightrightarrows \bigoplus_i F(i) \otimes G(i).\]
Todd Trimble and Anton Fetisov have given very enlightening introductions to coends at \url{http://mathoverflow.net/questions/114703}.

   
\textbf{11.} Although we do not need general enriched category theory (\cite{Kel05}), we will consider the following special case: Let $R$ be a commutative ring. An \emph{$R$-linear category} is an $\M(R)$-enriched category, i.e. a category whose hom-sets are actually $R$-modules and for which the composition is $R$-bilinear. An \emph{$R$-linear functor} between $R$-linear categories $F : \C \to \D$ is a functor between the underlying categories such that  the map $\Hom(A,B) \to \Hom(F(A),F(B))$ is $R$-linear for all objects $A,B \in \C$. We obtain the category of $R$-linear functors $\Hom_{/R}(\C,\D)$. Its subcategory of cocontinuous $R$-linear functors is denoted by $\Hom_{c/R}(\C,\D)$.

\textbf{12.} In an $R$-linear category, an object $A$ is initial if and only if $\id_A = 0$. Since the latter property is invariant under dualization, it follows that $A$ is initial if and only if $A$ is terminal. In particular, every cocomplete $R$-linear category has a zero object $0$.

\textbf{13.} If $A,B$ are objects of an $R$-linear category, then their direct sum $A \oplus B$ (if it exists) is actually a \emph{biproduct} and biproducts are preserved by any $\Z$-linear functor. Conversely, since the sum of two morphisms $f,g : A \to B$ is given by
\[A \xrightarrow{\Delta} A \times A \cong A \oplus A \xrightarrow{(f,g)} B,\]
we see that each functor between $\Z$-linear categories with direct sums which preserves direct sums is $\Z$-linear.
 
\textbf{14.} Even if no zero object exists, a sequence of morphisms in a linear category
\[A \xrightarrow{i} B \xrightarrow{p} C \to 0\]
is called \emph{exact} if $p$ is a cokernel of $i$. Dually,
\[0 \to A \xrightarrow{i} B \xrightarrow{p} C\]
is called \emph{exact} if $i$ is a kernel of $p$. Finally, we call a sequence of the form
\[0 \to A \xrightarrow{i} B \xrightarrow{p} C \to 0\]
\emph{exact} if $i$ is a kernel of $p$ and $p$ is a cokernel of $i$. Remark that these definitions make sense in any linear category -- we do not have to restrict to abelian categories (\cite[Definition 4.1.5]{BB04}).

\textbf{15.} Every cocomplete $R$-linear category $\C$ is \emph{tensored over $\M(R)$}, i.e. for every $M \in \M(R)$ and $A \in \C$ there is some $M \otimes_R A \in \C$ satisfying the adjunction
\[\Hom_\C(M \otimes_R A,B) \cong \Hom_{\M(R)}(M,\Hom_\C(A,B)).\]
Since left adjoints may be defined objectwise (\cite[IV.1, Corollary 2]{ML98}) and are cocontinuous, and every $R$-module is a cokernel of a map between free $R$-modules, it suffices to take $R^{\oplus X} \otimes_R A = A^{\oplus X}$ for every set $X$. See also \cite[II.1.5.1]{SR72}.
  
\textbf{16.} Finally let us mention that the \emph{tensor product} $\C \otimes_R \D$ of two $R$-linear categories $\C,\D$ has objects $(A,B)$ with $A \in \C$ and $B \in \C$ and hom modules
\[\Hom_{\C \otimes_R \D}((A,B),(C,D)) := \Hom_\C(A,C) \otimes_R \Hom_\D(B,D).\]
Composition and identity are defined in an evident manner. Notice that $R$-linear functors $\C \otimes_R \D \to \E$ correspond to $R$-bilinear functors $\C \times \D \to \E$, i.e. to functors which are $R$-linear in each variable.

\section{Algebraic geometry}

\textbf{1.} Since this thesis is aimed at translating some foundational aspects of algebraic geometry to cocomplete tensor categories, we will not need much algebraic geometry which goes beyond the basics as developed in the first chapters of EGA (\cite{EGAI},\cite{EGAII}) and descent theory for fpqc morphisms and quasi-coherent modules, for which we refer to \cite{Vis05}. Note that the definition of fqpc coverings there is slightly weaker than faithfully flat and quasi-compact, thereby including arbitrary Zariski coverings. We refer to \cite{Vis05} for the theory of stacks in general and to \cite{LMB00} for algebraic stacks (also known as Artin stacks).

\textbf{2.} Many results for quasi-coherent modules on schemes directly translate to algebraic stacks, because quasi-coherent modules satisfy fpqc descent and an algebraic stack is really just a stack which is (in particular) fpqc-locally a scheme. In fact, one may generalize the notion of an algebraic stack and only require a flat instead of a smooth cover from an affine scheme (\cite[Section 2]{Nau07}). Therefore, we sometimes also include algebraic stacks into our statements. Those readers not interested in algebraic stacks might just read over this. 
 
\textbf{3.} We will use the following abbreviations.
\begin{center}
	\begin{tabular}{ll}
	qc & quasi-compact \\ qs & quasi-separated \\ ss & semi-separated	(\cite[B.7]{TT90})
	\end{tabular}
\end{center}
Note that the class of qc qs schemes is the smallest class of schemes which contains the affine schemes and is closed under finite gluings -- this explains why this finiteness condition works so well for the general theory. For example if $f : X \to Y$ is qc qs, then the induced direct image functor $f_*$ preserves quasi-coherent modules and commutes with directed colimits -- the same is even true for all higher direct images, in particular for cohomology (\cite[Proposition 6.7.1]{EGAI}, \cite[B.6]{TT90}).

\textbf{4.} Following Lurie we also say \emph{geometric} instead of qc ss (\cite{Lur04}). Thus, an algebraic stack $X$ is geometric if there is a smooth (or just flat) surjective \emph{affine} morphism $\pi : P \to X$ such that $P$ is an \emph{affine} scheme. We refer to \cite[Section 3]{Nau07} for an equivalent description via flat groupoids in affine schemes, or equivalently flat Hopf algebroids in commutative rings, in which quasi-coherent sheaves correspond to comodules. See also \cite{Alo13} for the corresponding formalism of pushforward and pullback functors.

\textbf{5.} The category of quasi-coherent modules (we will not say sheaves in order to emphasize the global point of view) on $X$ will be denoted by $\Q(X)$, its full subcategory of quasi-coherent sheaves of finite presentation (\cite[Chapitre 0, 5.2.5]{EGAI}) by $\Q_{\fp}(X)$. Thus for noetherian $X$ this is the category of coherent sheaves.

The following observations are probably well-known.

\begin{lemma}[Quasi-coherence of Hom sheaves] \label{qchom}
Let $X$ be a scheme and $M,N$ be two $\O_X$-modules. If $M$ is of finite presentation and $N$ is quasi-coherent, then the $\O_X$-module $\HOM(M,N)$ (``sheaf hom'') is quasi-coherent.
\end{lemma}
 
See also \cite[Proposition 9.1.1]{EGAI-alt} and \cite[Corollaire 5.3.7]{EGAI} for a special case.
 
\begin{proof}
We may work locally on $X$ and therefore assume that there is an exact sequence
\[\mathcal{O}_X^p \to \mathcal{O}_X^q \to M \to 0.\]
Applying $\HOM(-,N)$ gives an exact sequence
\[0 \to \HOM(M,N) \to N^q \to N^p.\]
The claim follows since $\Q(X) \subseteq \M(X)$ is closed under kernels and direct sums.
\end{proof}

The following Lemma generalizes \cite[5.3.10]{EGAI}.

\begin{lemma}[Pairings] \label{pairing}
Let $M,N,P$ be three quasi-coherent modules on a scheme $X$ and $\psi : M \otimes N \to P$ a pairing. Assume we are given a quasi-coherent submodule $Q \subseteq P$. Define a submodule $M' \subseteq M$ by
\[\Gamma(U,M') = \bigl\{m \in \Gamma(U,M) : \forall V \subseteq U,\ \forall n \in \Gamma(V,N),\ \psi(m|_V \otimes n) \in \Gamma(V,Q)\bigr\}.\]
In other words, $M'$ is the largest submodule of $M$ whose pairing with $N$ by means of $\psi$ lands inside $Q\subseteq P$. If $N$ is of finite type, then $M'$ is quasi-coherent.
\end{lemma}
 
\begin{proof}
We may work locally on $X$ and therefore assume that $N$ is a quotient of $\O_X^n$ for some $n \in \N$. But then we may clearly replace $\psi$ by the pairing $M \otimes \O_X^n \to M \otimes N \to P$ and therefore assume $N=\O_X^n$. Then $\psi$ corresponds to a homomorphism $M \to P^n$ and $M'$ is the kernel of $M \to P^n \to P^n / Q^n$. Hence, $M'$ is quasi-coherent.
\end{proof}

\begin{lemma}[Pullbacks of closed subschemes] \label{pullback-closed}
Let $f : X \to Y$ be a morphism of schemes and let $I \subseteq \O_Y$ be a quasi-coherent ideal. Consider the induced ideal $J=f^* I \cdot \O_X \subseteq \O_X$, the image of $f^* I \to f^* \O_Y \cong \O_X$. Then $f^*(\O_Y/I) \cong \O_X/J$ and we have an equality of sets $f^{-1}(V(I))=V(J)$.
\end{lemma}
 
\begin{proof}
The exact sequence $f^* I \to f^* \O_Y \to f^* (\O_Y/I) \to 0$ yields the first statement. Using \cite[Chap. 0, 5.2.4.1]{EGAI}, it follows
\[f^{-1}(V(I))=f^{-1}(\supp(\O_Y/I))=\supp(f^*(\O_Y/I))=\supp(\O_X/J)=V(J).\]
\end{proof}

\begin{lemma}[Universal property of immersions] \label{imm}
Let $i : Y \to X$ be an immersion of schemes. Then a morphism $f : Z \to X$ factors (uniquely) through $i$ if and only if $f^* : \Q(X) \to \Q(Z)$ maps every homomorphism, which is an isomorphism on $Y$, to an isomorphism.
\end{lemma}

\begin{proof} If $f=ig$ for some $g : Z \to Y$ and $\alpha$ is a homomorphism on $X$ such that $i^* \alpha$ is an isomorphism, then $g^* i^* \alpha \cong f^* \alpha$ also is an isomorphism. For the converse, it suffices to treat two cases:

1. Assume that $i$ is a closed immersion, say $Y=V(I)$ for some quasi-coherent ideal $I \subseteq \O_X$. Then $\O_X \to \O_X/I$ is an isomorphism on $V(I)$, hence $f^*$ maps it to an isomorphism. This implies that $f^* I \to f^* \O_X = \O_Z$ vanishes, or equivalently that the adjoint
\[I \to \O_X \xrightarrow{f^\#} f_* \O_Z\]
vanishes, i.e. $I \subseteq \ker(f^\#)$. But this means that $f$ factors uniquely through $i$.

2. Assume that $i$ is an open immersion. Let $I$ be the vanishing ideal of the closed complement $X \setminus Y \subseteq X$. Then $\O_X / I$ vanishes on $Y$. Hence $f^*(\O_X/I)=0$ vanishes. It follows
\[\emptyset = \supp f^*(\O_X/I) = f^{-1}(\supp(\O_X/I)) = f^{-1}(X \setminus Y),\]
so that $f$ factors through $Y$ as a set map. Since $i$ is an open immersion, this means that $f$ factors through $i$.
\end{proof}

Morphisms of schemes are determined by their pullback functors:

\begin{lemma} \label{volltreufunk}
Let $f,g : X \to Y$ be two morphisms of schemes such that there is an isomorphism of functors  $f^* \cong g^* : \Q(Y) \to \Q(X)$. Then the underlying maps of $f$ and $g$ coincide. If $Y$ is quasi-separated, we even have $f=g$.
\end{lemma}

\begin{proof}
For a closed subset $Z  \subseteq Y$ with vanishing ideal $I \subseteq \O_Y$ we observe  $f^{-1}(Z) = \supp(f^*(\O_Y/I)) = \supp(g^*(\O_Y/I)) = g^{-1}(Z)$. Since $Y$ is a T$_0$-space, every singleton $\{y\}$ is an intersection of subsets which are open or closed. Therefore, $f^{-1}(\{y\}) = g^{-1}(\{y\})$. This proves $f=g$.

Now assume that $X,Y$ are affine. Then $f,g$ correspond to ring homomorphisms $f^\#, g^\# : R \to S$ such that the induced functors $\M(R) \to \M(S)$ are isomorphic. Evaluation at the $R$-module $R$ yields an isomorphism of $(R,S)$-bimodules $_{f^\#} S_{\id_S} \cong$ $ _{g^\#} S_{\id_S}$. This implies $f^\#=g^\#$.

Let $Y$ be quasi-separated. Let $x \in X$ and $y:=f(x)=g(x)$. We want to prove that $f^\#_x,g^\#_x : \O_{Y,y} \to \O_{X,x}$ are equal. By the affine case, it suffices to prove this for the pullback functors. Consider the morphism $i_y : \Spec(\O_{Y,y}) \to Y$. Since $Y$ is quasi-separated, $(i_y)_*$ preserves quasi-coherence. Thus, if $M \in \M(\O_{Y,y})$, then we may construct $N:=(i_y)_* M \in \Q(Y)$ with $N_y = M$. It follows
\[(f^\#_x)^* M \cong (f^\#_x)^* N_y \cong (f^* N)_x \cong (g^* N)_x \cong (g^\#_x)^* N_y \cong (f^\#_x)^* M.\]
\end{proof}

\begin{cor} \label{isofunk}
Let $f : X \to Y$ be a morphism of quasi-separated schemes. If $f^*$ is an equivalence of categories, then $f$ is an isomorphism.
\end{cor}

\begin{proof}
By \cite[Theorem 5.4]{Bra11} there is an isomorphism $g : X \to Y$ such that $f^* \cong g^*$. But then $f=g$ by \autoref{volltreufunk}.
\end{proof}

Probably one can prove \autoref{isofunk} more directly without using Rosenberg's Reconstruction Theorem. The following is Definition 6.1.1 in \cite{Sch12a}.

\begin{defi}[Strong resolution property]
A scheme $X$ (or more generally, an algebraic stack) is said to have the \emph{strong resolution property} if the locally free $\O_X$-modules constitute a generator of $\Q(X)$.
\end{defi}

Here, \emph{locally free} always means locally free of finite rank. \marginpar{I've added the remark about local freeness.}  Many schemes have the strong resolution property, for example divisorial schemes (\cite{Bor63}) -- including projective schemes and any separated noetherian locally factorial scheme (\cite[Exp. II, Proposition 2.2.7]{SGA6}), as well as any separated algebraic surface (\cite{Gro12}). For further results we refer to \cite{Tot04} and \cite{Gro10}. If $X$ is geometric, then $X$ has the strong resolution property if and only if $X$ is an Adams stack, see \autoref{density}.

\section{Local Presentability}

Let $\lambda$ be a regular cardinal. We refer to \cite{AR94} for the theory of locally $\lambda$-presentable categories \marginpar{I've changed ``presentable'' to ``locally presentable'' everywhere!}. Thus, a \emph{locally $\lambda$-presentable category} is a cocomplete category which has a \emph{set} of $\lambda$-presentable objects such that every object is a $\lambda$-directed colimit of these $\lambda$-presentable objects. Here, an object $A$ is called \emph{$\lambda$-presentable} if $\Hom(A,-)$ preserves $\lambda$-directed colimits. For $\lambda=\aleph_0$ one also speaks of finitely presentable objects and locally finitely presentable categories. In the category of modules over some ring, an object is $\lambda$-presentable if and only if it can be written in terms of $<\lambda$ many generators and $<\lambda$ many relations (\cite[Theorem 3.12]{AR94}). In particular, a module is finitely presentable if and only if it is of finite presentation in the usual sense. More generally, we have the following result:

\begin{prop}[Quasi-coherent modules of finite presentation] \label{Qcohfin}
Let $X$ be a qc qs scheme and $M$ be a quasi-coherent module on $X$. If $M$ is of finite presentation in the usual sense of algebraic geometry, then $M$ is a finitely presentable object. The converse also holds.
\end{prop}

\begin{proof}
If $M$ is of finite presentation and $\{N_i\}$ is a directed diagram of quasi-coherent modules on $X$, then we have to show that
\[\colim_i \Hom(M,N_i) \to \Hom(M,\colim_i N_i)\]
is an isomorphism. Since this is the map of global sections associated to the homomorphism of sheaves
\[\colim_i \HOM(M,N_i) \to \HOM(M,\colim_i N_i),\]
we may work locally on $X$. But we have already discussed the affine case above. Conversely, let $M$ be a finitely presentable object in $\Q(X)$. Since $X$ is qc qs, we have $M = \colim_i M_i$ for a directed system of quasi-coherent modules of finite presentation (\cite[Corollaire 6.9.12]{EGAI}). It follows that the identity of $M$ factors through some $M_i$, i.e. that $M$ is a direct summand of $M_i$. Clearly then $M$ is also of finite presentation.
\end{proof}

\begin{rem}
If $X$ is any scheme\marginpar{arbitrary scheme -> any scheme.}, then $\Q(X)$ is locally $\lambda$-presentable for some suitable large regular cardinal $\lambda$. This follows from a result by Gabber \cite[\href{http://stacks.math.columbia.edu/tag/077K}{Tag 077K}]{stacks-project} and \cite[Theorem 1.70]{AR94}. If $X$ is a geometric stack, then $\Q(X)$ is locally presentable (which can be seen from the equivalence to comodules). It is a delicate question when $\Q(X)$ is actually locally \emph{finitely} presentable:
\end{rem}

\begin{prop} \label{Qcohfin2}
Let $X$ be a geometric stack. Each of the following conditions implies that every quasi-coherent module on $X$ is a directed colimit of quasi-coherent modules of finite presentation. Also, then $\Q(X)$ is locally finitely presentable and  the finitely presentable objects coincide with the quasi-coherent modules of finite presentation.
\begin{enumerate}
\item $X$ is Deligne-Mumford.
\item $X$ has the strong resolution property.
\item $X$ is noetherian.
\item $X$ is a scheme.
\end{enumerate}
\end{prop}

\begin{proof}
1. follows from \cite[Theorem A, Proposition 2.9]{Ryd13}, 2. is a consequence of \autoref{lokfreicolim} below, 3. is easy since coherent modules are stable under submodules (\cite[Lemma 3.9]{Lur04}) and 4. follows from \autoref{Qcohfin}. The equivalence between finitely presentable and of finite presentation follows as in \autoref{Qcohfin}.
\end{proof}

\begin{prop}[Properties of locally presentable categories] \label{saft}
A locally $\lambda$-presentable category $\C$ enjoys the following properties:
\begin{enumerate}
\item $\C$ is complete and cocomplete.
\item $\C$ is wellpowered and co-wellpowered.
\item $\C$ has $($regular epi,mono$)$-factorizations.
\item $\lambda$-directed colimits are exact.
\item Any cocontinuous functor $\C \to \D$ is a left adjoint.
\end{enumerate}
\end{prop}

\begin{proof}
For 1. and 2. see \cite[Remark 1.56, Theorem 1.58]{AR94}. For 3. see \cite[Proposition 1.61]{AR94}. For 4. see \cite[Proposition 1.59]{AR94}. 5. Since $\C$ is locally presentable, $\C$ has a generating set and is co-wellpowered, so that we may invoke Freyd's Special Adjoint Functor Theorem (\cite[V.8]{ML98}).
\end{proof}
 
\section{Density and Adams stacks} \label{density}

We refer to \cite[Chapter X]{ML98} for the notions of Kan extensions and density for usual categories. Let $\C,\D$ be cocomplete $R$-linear categories, $S \subseteq \C$ be a full subcategory, which we assume to be small. Actually it suffices to demand that $S$ is essentially small, i.e. equivalent to a small category. Let $F : S \to \D$ be an $R$-linear functor. Then its left Kan extension  \emph{as an $R$-linear functor} (\cite[Chapter 4]{Kel05}) is the $R$-linear functor $\tilde{F}: \C \to \D$ given as
\[\tilde{F}(M) = \int^{A \in S} \Hom(A,M) \otimes_R F(A).\]
It satisfies $\tilde{F}|_S \cong F$. In general, it may differ from the left Kan extension of the underlying functor of $F$ (which forgets about the $R$-linear structure), which is given by
\[M \mapsto \colim_{A \to M,\, A \in S} F(A).\]
In fact, the latter functor has no reason to be $R$-linear at all! But in some cases these two left Kan extensions agree, as is shown in the following Lemma.
 
\begin{lemma} \label{denselin}
Let $\C,\D$ be cocomplete $R$-linear categories and $S \subseteq \C$ be a full essentially small subcategory with the following property:
\begin{enumerate}
\item For all $M \in \C$ and all $A,B \in S/M$, there is some $C \in S/M$ with morphisms $A \to C$ and $B \to C$ in $S/M$, as in the following commutative diagram:
\[\xymatrix@C=10pt@R=10pt{ & A \ar[dl] \ar[dr] & \\ C \ar[rr] && M \\ & B \ar[ul] \ar[ur] & }\]
\end{enumerate}
For example, this holds when $S$ is closed under direct sums in $\C$. Then, for every $R$-linear functor $F : S \to \D$ the underlying functor of the left Kan extension to $\C$ of $F$ as an $R$-linear functor coincides with the left Kan extension of the underlying functor of $F$.

In other words, we have an isomorphism 
\[\colim_{A \to M,\, A \in S} F(A) \cong \int^{A \in S} \hom(A,M) \otimes_R F(A).\]  
In particular: $S$ is dense in $\C$ as an $R$-linear subcategory if and only if $S$ is dense in $\C$ as a subcategory.
\end{lemma}

\begin{proof}
A morphism from the colimit on the left to some test object $T$ corresponds to a family of morphisms $\overline{\sigma} : F(A) \to T$ induced by $\sigma : A \to M$, $A \in S$, with the compatibility condition $\overline{\sigma} \circ F(\alpha) = \overline{\sigma \circ \alpha}$ for morphisms $\alpha : A' \to A$ in $S$. A morphism from the coend to $T$ corresponds to such a family with the additional property that $\sigma \mapsto \overline{\sigma}$ is $R$-linear. But this is automatic: If $r \in R$, the compatibility condition for the corresponding morphism $r : A \to A$ shows that $\overline{r \cdot  \sigma} = r \cdot \overline{\sigma}$. For additivity, choose $\sigma,\,\sigma' : A \to B$. By 1. there is some $\tau : C \to M$ with $C \in S$ and morphisms $\alpha,\,\beta : A \to C$ such that $\tau \circ  \alpha = \sigma$ and $\tau \circ \beta = \sigma'$. It follows
\[\overline{\sigma} + \overline{\sigma'} = \overline{\tau} \circ  F(\alpha) + \overline{\tau} \circ  F(\beta) = \overline{\tau} \circ  F(\alpha + \beta) = \overline{\tau \circ  (\alpha + \beta)} = \overline{\sigma + \sigma'}. \]
\end{proof}

\begin{lemma} \label{qdense}
Let $X$ be a scheme (or more generally an algebraic stack). Let $S \subseteq \Q(X)$ be a full subcategory. Assume that the following holds:
\begin{enumerate}
\item The condition of the previous Lemma.
\item $S$ generates $\Q(X)$: For all $M \in \Q(X)$ there is a family of objects $A_i \in S$ and an epimorphism $\bigoplus_i A_i \to M$.
\end{enumerate}
Then $S$ is dense in $\Q(X)$, both in the usual and the $R$-linear sense.
\end{lemma}
 
\begin{proof}
We have to prove for $M \in \Q(X)$ that the canonical homomorphism
\[\theta : \colim_{A \to M,\, A \in S} A \to M\]
is an isomorphism on $X$. By 2. it is an epimorphism. Choose an fpqc covering $\{X_i \to X\}$ with affine schemes $X_i$. It suffices to prove that each $\theta|_{X_i}$ is a monomorphism in $\Q(X_i)$, or equivalently that it is injective on sections on $X_i$. We need two observations first:
   
A. If $\sigma : A \to M$ is some homomorphism with $A \in S$ and $a \in \Gamma(X_i,A)$, we will denote the corresponding section in the colimit by $[\sigma,a]$. Clearly they generate the colimit. But actually, every section of the colimit already has this form. To prove this, it suffices to check that they are closed under sums. So assume that $[\sigma',a']$ is another section, with $\sigma' : B \to M$ and $a' \in \Gamma(X_i,A)$. Because of 1. there is some $\tau : C \to M$ and homomorphisms $\alpha : A \to C$, $\beta : B \to C$ such that $\tau \alpha = \sigma$ and $\tau \beta = \sigma'$. Then, we have
\[ [\sigma,a] + [\sigma',a'] = [\tau,\alpha(a)] + [\tau,\beta(a')] = [\tau,\alpha(a)+\beta(a')].\]

B. If $s \in \Gamma(X_i,M)$, then there is some $A \to M$ with $A \in S$, such that $s$ has a preimage in $\Gamma(X_i,A)$. In fact, since there is an epimorphism $\bigoplus_j A_j \to M$ with $A_j \in S$, we see that $s$ has a preimage in $\Gamma(X_i,\bigoplus_j A_j) = \bigoplus_j \Gamma(X_i,A_j)$. Here, already finitely many $j$ suffice, so let us restrict to these. By applying 1. and induction, there is some $A \to M$ with $A \in S$ such that each $A_j \to M$ factors as $A_j \to A \to M$. It follows that $s$ has a preimage in $\Gamma(X_i,A)$.
 
Now we can show injectivity of $\theta$: If two sections of the colimit, say $[\sigma,a]$ and $[\sigma',a']$ by A., have the same image, i.e. $\sigma(a)=\sigma'(a')$ in $\Gamma(X_i,M)$, then $(a,a')$ is a section of the pullback $A \times_M A'$ on $X_i$. By B. there is some $C \in S$ and a homomorphism $C \to A \times_M A'$, such that $(a,a')$ has a preimage $c$. If we denote the induced homomorpism $C \to M$ by $\tau$, then we get
\[ [\sigma,a] = [\tau,c] = [\sigma',a']. \]
\end{proof}
  
\begin{cor} \label{lokfreicolim}
If $X$ is a scheme (or more generally an algebraic stack), which has the strong resolution property, then locally free $\O_X$-modules constitute a dense subcategory of $\Q(X)$: For every $M \in \Q(X)$ we have
\[M \cong \colim_{V \to M,\, V \text{ locally free}} V.\] \hfill $\square$
\end{cor}

Recall that the Theorem of Lazard-Govorov (\cite[Theorem 4.34]{Lam99}) states that a flat module over a ring is a \emph{directed} colimit of finitely generated free modules. The following result is a variant of this Theorem for quasi-coherent modules. The proof is inspired by the proof of \cite[Theorem 1.3.1]{Sch12a}. See \cite{EAO13} for another variant.
   
\begin{prop}[Variant of Lazard-Govorov] \label{lazard}
Let $X$ be a scheme (or more generally an algebraic stack), which has the strong resolution property and let $M \in \Q(X)$ be  flat. Assume that $\Spec(\Sym(M))$ (or just $\Spec(M)$ if $M$ is already a quasi-coherent $\O_X$-algebra) is an affine scheme over $\Z$. Then the category of morphisms $V \to M$ with $V$ locally free, is directed. Hence, $M$ is a directed colimit of locally free modules.
\end{prop} 
 
\begin{proof}
We have to check three conditions. 1. The category is non-empty, because it contains $0 \to M$. 2. Given two objects $V \to M$ and $W \to M$, these induce an object $V \oplus W \to M$ and the inclusions induce morphisms from $V \to M$ resp. $W \to M$ to $V \oplus W \to M$. 3. Since the category is anti-equivalent to the category of homomorphisms $\O_X \to M \otimes V$, where $V$ is locally free, it suffices to prove: Given two homomorphisms $\sigma : \O_X \to M \otimes V$ and $\tau : \O_X \to M \otimes W$ (with $V,W$ locally free) and two homomorphisms $\alpha,\beta: V \to W$ such that $(M \otimes \alpha) \sigma = \tau = (M \otimes \beta) \sigma$, then there is some homomorphism $\rho : \O_X \to M \otimes U$ (with $U$ locally free) and a homomorphism $\gamma : U \to V$ such that $(M \otimes \gamma) \rho = \sigma$ and $\alpha \gamma = \beta \gamma$.

Let $i : E \to V$ be an equalizer of $\alpha,\beta$. Since $M$ is flat, $M \otimes i : M \otimes E \to M \otimes V$ is an equalizer of $M \otimes \alpha,M \otimes \beta$. Since $\sigma$ equalizes these morphisms, there is a unique morphism $\delta : \O_X \to M \otimes E$ such that $(M \otimes i) \delta  = \sigma$.

Consider the affine morphism $p : P := \Spec(\Sym(M)) \to X$. If $M$ is already an algebra, just take $p : \Spec(M) \to X$. Then the composition
\[\O_X \xrightarrow{\delta} M \otimes E \hookrightarrow \Sym(M) \otimes E \cong p_* p^* E\]
corresponds to a homomorphism $\O_P \to p^* E$. By what we have seen in the proof of part B. in Lemma \ref{qdense} and by using that $P$ is affine, there is some morphism $\epsilon : U \to E$ with $U$ locally free, such that $\O_P \to p^* E$ factors as $\O_P \to p^* U \to p^* E$. This corresponds to a factorization of the morphism $\O_X \to \Sym(M) \otimes E$ as $\O_X \to \Sym(M) \otimes U \to \Sym(M) \otimes E$. Composing with $\Sym(M) \twoheadrightarrow M$ gives a factorization $\O_X \xrightarrow{\rho} M \otimes U \xrightarrow{M \otimes \epsilon} M \otimes E$ of $\delta$. Then $\gamma := i\epsilon : U \to V$ and $\rho$ satisfy the requirements. The conclusion thus follows from \autoref{lokfreicolim}.
\end{proof}

The following result is due to Sch\"appi (\cite[Theorem 1.3.1]{Sch12a}).

\begin{prop}[Characterization of the strong resolution property] \label{strong-resolution}
Let $X$ be a geometric stack. Choose some smooth surjective affine morphism $\pi : P \to X$ for some affine scheme $P$.
The following are equivalent:
\begin{enumerate}
\item $\pi_* \O_P$ is a directed colimit of locally free modules.
\item $X$ has the strong resolution property.
\item If $M \in \Q(X)$ and $s : \O_P \to \pi^* M$ is a homomorphism, then there is some locally free module $V$ and some homomorphism $V \to M$, such that $s$ factors as $\O_P \to \pi^* V \to \pi^* M$.
\end{enumerate}
\end{prop}

\begin{proof}
$1. \Rightarrow 2.$ Let $\pi_* \O_P = \colim_i V_i$ with a directed system of locally free modules $V_i$. Then their duals $V_i^*$ generate $\Q(X)$: Assume that $M \to N$ is a homomorphism such that all $V_i^* \to M \to N$ vanish. This means that all $V_i \otimes M \to V_i \otimes N$ vanish on global sections. In the colimit, we see that $\pi_* \O_P \otimes M \to \pi_* \O_P \otimes N$ vanishes on global sections. But this is isomorphic to $\pi_* \pi^* M \to \pi_* \pi^* N$ (since $\pi$ is affine), so that $\pi^* M \to \pi^* N$ vanishes on all global sections. Since $P$ is affine, it then vanishes. Since $\pi$ is a descent morphism, this means that $M \to N$ vanishes.

We have seen $2. \Rightarrow 3.$ in the proof of part B. in Lemma \ref{qdense}.

$3. \Rightarrow 2.$ Let $f : M \to N$ be a homomorphism which vanishes when composed with every $V \to M$, where $V$ is locally free. In order to show $f=0$, it suffices to prove $\pi^*(f) = 0$ by descent. Since $P$ is affine, we only have to prove $\pi^*(f) s = 0$ for every $s : \O_P \to \pi^* M$. By 3. there is some locally free module $V$ and some homomorphism $\alpha : V \to M$ such that $s$ factors as through $\pi^* \alpha$. But then $\pi^*(f) s$ factors through $\pi^* (f\alpha)=0$.

$2. \Rightarrow 1.$ follows from Proposition \ref{lazard} since $\Spec \pi_* \O_P = P$ is affine.
\end{proof}

\begin{rem}
The condition 1. means that $X$ is an \emph{Adams stack} (\cite[Definition 6.5]{Goe08}). Hence, Adams stacks coincide with geometric stacks with the strong resolution property. By 2. this does not depend on the chosen presentation. This affirmatively answers a question initially posed by Hovey in the language of Hopf algebroids (\cite[Question 1.4.12]{Hov03}). Our proof is similar to Sch\"appi's. The main difference is that we do not use the language of Hopf algebroids. One advantage of the geometric language is that $\pi^*$ is not treated as a forgetful functor and not omitted from the notation, as is often done for the corresponding forgetful functor from $\Gamma$-modules to $A$-modules for a flat Hopf algebroid $(A,\Gamma,\dotsc)$.
\end{rem}

\section{Extension result}

The following general extension result will be needed in \autoref{prodschemes} and might also be of independent interest. We do not know if this is really new but do not have a reference for it.

\begin{prop}[Extension of functors] \label{extend2} \noindent
Let $\C$ be an \emph{abelian} $R$-linear category and $\D$ any finitely cocomplete $R$-linear category. Let $\C'$ be a full replete subcategory of $\C$ with the following properties:
\begin{enumerate}
\item $0 \in \C'$.
\item For every $M \in \C$ there is some $A \in \C'$ and an epimorphism $A \twoheadrightarrow M$.
\item The condition of \autoref{denselin}: For all objects $M \in \C$ and all morphisms $A \to M \leftarrow B$ with $A,B \in \C'$ there is some $C \in \C'$ which fits into a commutative diagram
\[\xymatrix@C=10pt@R=10pt{ & A \ar[dl] \ar[dr] & \\ C \ar[rr] && M. \\ & B \ar[ul] \ar[ur] & }\]
\end{enumerate}
Let $F : \C' \to \D$ be an $R$-linear functor with the property $(\star)$ that for any every exact sequence
\[A \to B \to C \to 0\]
in $\C$ with objects $A,B,C \in \C'$, the induced sequence
\[F(A) \to F(B) \to F(C) \to 0\]
is also exact in $\D$.

Then $F$ extends to a right exact $R$-linear functor $\overline{F} : \mathcal{C} \to \mathcal{D}$. In fact, this construction implements an equivalence of categories between $R$-linear functors $\C' \to \D$ satisfying $(\star)$ and right exact $R$-linear functors $\C \to \D$.
\end{prop}

\begin{proof}[Proof.]
Let $F : \C' \to \D$ satisfy $(\star)$. Every object $M \in \C$ fits into a right exact sequence
\[A' \xrightarrow{\alpha} A \to M \xrightarrow{} 0\]
with $A',A \in \C'$. We call this a \emph{presentation} of $M$. We fix one for each $M$. We have no choice but to define $\overline{F}(M)$ to be the cokernel of $F(\alpha)$, so that
\[F(A') \xrightarrow{F(\alpha)} F(A) \to \overline{F}(M) \to 0\]
is exact in $\D$. With this ad hoc definition the proof that $\overline{F}$ actually becomes a functor and is therefore independent of the choices of the presentations will require some effort. On the other hand, abstract definitions of $\overline{F}$ (for example as the left Kan extension of $F$, see \autoref{denselin}) would have the drawback of less computability. We divide the proof into 9 Steps.

\textbf{1. Step: Weak pullbacks.} We call a commutative diagram in $\C$
\[\xymatrix{A' \ar[r] \ar[d] & A \ar[d] \\ B' \ar[r] & B}\] 
a \emph{weak pullback} if the induced morphism $A' \to A \times_B B'$ is an epimorphism. If $B' \to B$ is an epimorphism, it follows that $A' \to A$ is also an epimorphism. Besides, the sequence
\[A' \to A \oplus B' \to B  \to 0\]
is exact, which means that the given square is actually a pushout square.

\textbf{2. Step: $F$ preserves epis and pushouts.} Given any epimorphism $B \to C$ in $\C$ with $B,C \in \C'$, we denote by $A$ the kernel and obtain an exact sequence $A \to B \to C \to 0$ in $\C$. Hence, $F(A) \to F(B) \to F(C) \to 0$ is also exact. In particular, $F(B) \to F(C)$ is an epimorphism.
 
Given a pushout square in $\C$ with objects in $\C'$
\[\xymatrix{A' \ar[r] \ar[d] & A \ar[d] \\ B' \ar[r] & B}\] 
in which $A' \to B'$ (and hence $A \to B$) is an epimorphism in $\C$, then we claim that
\[\xymatrix{F(A') \ar[r] \ar[d] & F(A) \ar[d] \\ F(B') \ar[r] & F(B)}\] 
is also a pushout in $\D$. Notice that for $B'=0$ this is precisely the property $(\star)$.

The pushout property means that $A' \to A \oplus B' \to B \to 0$ is exact. Here, $A' \to A \oplus B' \cong A \times B'$ is induced by $A' \to A$ and $A' \to B'$. Besides, $A \oplus B' \to B$ is induced by $A \to B$ and the additive inverse of $B' \to B$. Find a commutative diagram
\[\xymatrix@C=10pt@R=10pt{ & A \ar[dl] \ar[dr] & \\ C \ar[rr] && A \oplus B' \\ & B' \ar[ul] \ar[ur] & }\]
with $C \in \C'$. Then $C \to A \oplus B'$ is an epimorphism. Consider the pullback $C \times_{A \oplus B'} A'$ and choose an epimorphism from some $D \in \C'$, so that
\[\xymatrix{D \ar[r] \ar[d] & C \ar[d] \\ A' \ar[r] & A \oplus B'}\] 
is a weak pullback, hence a pushout by Step 1. Then $D \to C \to B \to 0$ is also exact in $\C$. By assumption, it follows that $F(D) \to F(C) \to F(B) \to 0$ is also exact in $\D$.

Now, a commutative diagram
\[\xymatrix{F(A') \ar[r] \ar[d] & F(A) \ar[d] \\ F(B') \ar[r] & T}\] 
in $\D$ yields a morphism
\[F(C) \to F(A) \times F(B') \cong F(A) \oplus F(B') \to T\]
which vanishes when composed with $F(D) \to F(C)$, hence extends uniquely to a morphism $F(B) \to T$. It extends $F(B') \to T$ and $F(A) \to T$.

\textbf{3. Step: Action of $F$ on morphisms.} Let $M \to N$ be a morphism in $\C$. We choose presentations
\[\xymatrix{A' \ar[r] & A \ar[r] & M \ar[r] \ar[d] & 0 \\ B' \ar[r] & B \ar[r] & N \ar[r] & 0.}\]
We would like to fill this diagram. It is a bit too optimistic to find a commutative diagram
\[\xymatrix{A' \ar[r] \ar[d]  & A \ar[r] \ar[d]  & M \ar[r] \ar[d] & 0 \\ B' \ar[r] & B \ar[r] & N \ar[r] & 0}\]
which would make the definition of $\overline{F}(M) \to \overline{F}(N)$ easy. It is a bit more complicated than that: Let $P$ be a pullback of $A \to M \to N$ with $B \to N$. Since $B \twoheadrightarrow N$ is an epimorphism, the same is true for $P \twoheadrightarrow A$. Choose some epimorphism $C \twoheadrightarrow P$ with $C \in \C'$. Then $C \to A$ is an epimorphism and the following diagram commutes:
\[\xymatrix@R=20pt{A' \ar[r] & A \ar[r] & M \ar[r] \ar[dd] & 0 \\ & C \ar[d] \ar@{->>}[u] & \\ B' \ar[r] & B \ar[r] & N \ar[r] & 0}\]
The rectangle is a weak pullback diagram by construction. Let's do the same trick again: Let $Q = A' \times_A C \times_B B'$ and choose an epimorphism $C' \to Q$ with $C' \in \C'$. Then a diagram chase shows that $C' \to A' \times_A C$ is an epimorphism, so that
\[\xymatrix{A' \ar[r] & A \\ C' \ar[u] \ar[r] & C \ar[u]}\]
is a weak pullback. Since $C \to A$ is an epimorphism, the same is true for $C' \to A'$ and the square is actually a pushout by Step 1. We arrive at the following commutative diagram in which the two rows are exact and the square on the upper left is a pushout:
\[\xymatrix@R=20pt{A' \ar[r] & A \ar[r] & M \ar[r] \ar[dd] & 0 \\C' \ar[r] \ar@{->>}[u] \ar[d] & C \ar[d] \ar@{->>}[u] & \\ B' \ar[r] & B \ar[r] & N \ar[r] & 0.}\]
Using that $C' \to A' \times_A C$ is an epimorphism, another diagram chase shows that
\[C' \to C \to M \to 0\]
is exact. Notice that this is another presentation of $M$ (not the given one) which actually \emph{has} a morphism to the given presentation of $N$.

Now let us apply $F$ to the diagram above:
\[\xymatrix@R=20pt{F(A') \ar[r] & F(A) \ar[r] & \overline{F}(M) \ar[r]  & 0 \\F(C') \ar[r] \ar@{->>}[u] \ar[d] & F(C) \ar[d] \ar[r] \ar@{->>}[u]  \ar[d] & P \ar[r] \ar[u] \ar[d] & 0\\ F(B') \ar[r] & F(B) \ar[r] & \overline{F}(N) \ar[r] & 0}\]
Here $F(C) \to P$ is any cokernel of $F(C') \to F(C)$. Besides, the morphisms $P \to \overline{F}(M)$ and $P \to \overline{F}(N)$ are defined via the universal property of the cokernel. It follows by Step 2 that the square on the upper left is a pushout. The rows are exact by definition. Using these properties, we observe that $P \to \overline{F}(M)$ is an isomorphism. Now we can define $\overline{F}(M) \to \overline{F}(N)$ via the commutative diagram
\[\xymatrix{ & P \ar[dr] \ar[dl]_{\cong} &  \\ \overline{F}(M) \ar[rr] && \overline{F}(N).}\]
Finally, here is a more concise description of this morphism: Given any commutative diagram
\[\xymatrix@R=10pt{ & A \ar[r] & M \ar[r] \ar[dd] & 0 \\ C \ar[dr] \ar@{->>}[ur] & & \\ & B \ar[r] & N \ar[r] & 0}\]
with the property that $C \in \C'$ and the pentagon is a weak pullback, i.e. the induced morphism $C \to A \times_N B$ is an epimorphism, there is a unique morphism $\overline{F}(M) \to \overline{F}(N)$ which renders the diagram
\[\xymatrix@R=10pt{& F(A) \ar[r] & \overline{F}(M) \ar[r] \ar[dd] & 0 \\ F(C) \ar[dr] \ar@{->>}[ur] & \\ & F(B) \ar[r] & \overline{F}(N) \ar[r] & 0}\]
commutative. We have showed that it exists.

Notice that if $M \to N$ is an epimorphism, then the same is true for $C \to B$ and it follows easily that $\overline{F}(M) \to \overline{F}(N)$ is also an epimorphism. Observe that $\overline{F}(M)$ does not depend on the presentation of $M$: If $A' \to A \to M \to 0$ and $B' \to B \to M \to 0$ are two presentations, we find a commutative diagram
\[\xymatrix@R=20pt{A' \ar[r] & A \ar[r] & M \ar[r] \ar@{=}[dd] & 0 \\C' \ar[r] \ar@{->>}[u] \ar@{->>}[d] & C \ar@{->>}[d] \ar@{->>}[u] & \\ B' \ar[r] & B \ar[r] & M \ar[r] & 0.}\]
It follows as above that the cokernels of $F(A') \to F(A)$ and of $F(B') \to F(B)$ are isomorphic to the cokernel of $F(C') \to F(C)$.

\textbf{4. Step: Action on morphisms is well-defined.} We have to check that the definition above does not depend on the chosen commutative diagram (containing $C$). So let us choose two diagrams
\[\xymatrix@R=10pt{& A  \ar[rr] && M \ar[r] \ar[dd] & 0 \\ C \ar[dr]  \ar[ur] & & D \ar[ul] \ar[dl] & \\ & B \ar[rr] && N \ar[r] & 0}\]
such that the induced morphisms from $C$ and $D$ to $A \times_N B$ are epimorphisms. It follows that $C \times_{A \times_N B} D \to C$ is an epimorphism, too. Choose an epimorphism $E \to C \times_{A \times_N B} D$ with $E \in \C'$. Then we obtain a commutative diagram
\[\xymatrix@R=13pt{& A  \ar[rr] && M \ar[r] \ar[dd] & 0 \\ C \ar[dr]  \ar[ur] & E \ar@{->>}[l] \ar[r] & D \ar[ul] \ar[dl] & \\ & B \ar[rr] && N \ar[r] & 0.}\]
Since $F(E) \to F(C)$ is an epimorphism, it follows readily that the morphisms $\overline{F}(M) \to \overline{F}(N)$ defined either via $F(C)$ or via $F(D)$ are the same.

\textbf{5. Step: $\overline{F}$ is a functor.} Given morphisms $M \xrightarrow{f} N \xrightarrow{g} K$, we claim that $\overline{F}(g \circ f) = \overline{F}(g) \circ \overline{F}(f)$. Choose a commutative diagram of weak pullbacks
\[\xymatrix@R=10pt{& & A \ar[r] & M \ar[r] \ar[dd] & 0 \\ & C \ar[dr] \ar@{->>}[ur] & & \\ G \ar[dr] \ar@{->>}[ur] & & B \ar[r] & N \ar[r] \ar[dd] & 0 \\& D \ar@{->>}[ur] \ar[dr] & & \\ & & E \ar[r] & K \ar[r] & 0}\]
with $A,B,C,D,E,G \in \C'$. Then we may use the commutative diagram
\[\xymatrix@R=10pt{& A \ar[r] & M \ar[r] \ar[dd] & 0 \\ G \ar[dr] \ar[ur] && \\ & E \ar[r] & K \ar[r] & 0}\]
for the composition; in fact $G \to A \times_K E$ is epi by a diagram chase which we leave to the reader. The commutative diagram
\[\xymatrix@R=10pt{& & F(A) \ar[r] & \overline{F}(M) \ar[r] \ar[dd] & 0 \\ & F(C) \ar[dr] \ar[ur] & & \\ F(G) \ar[dr] \ar[ur] & & F(B) \ar[r] & \overline{F}(N) \ar[r] \ar[dd] & 0 \\ & F(D) \ar[ur] \ar[dr] & & \\ & & F(E) \ar[r] & \overline{F}(K) \ar[r] & 0}\]
now finishes the proof of the claim.

The identity of an object $M \in \C$ is mapped by the result above to an idempotent endomorphism of $\overline{F}(M)$. On the other hand, it is an epimorphism (see Step 3), so that it must be the identity.

\textbf{6. Step: $\overline{F}$ extends $F$.} For $M \in \C'$ we have $\overline{F}(M) = F(M)$ by $(\star)$. For a morphism $f : M \to N$ in $\C'$ and $C \twoheadrightarrow A \times_N B$ as in the 3. Step, the commutative diagram
\[\xymatrix@R=10pt{& F(A) \ar[r] & F(M) \ar[r] \ar[dd] & 0 \\ F(C) \ar[dr] \ar[ur] & \\ & F(B) \ar[r] & F(N) \ar[r] & 0}\]
shows that $\overline{F}(f) = F(f)$. This proves $\overline{F}|_{\C'} = F$. To be more precise, in the definition of $\overline{F}$ we have made (canonical) choices of cokernels and on $\C'$ we may choose them in such a way that the equation becomes true. In any case, $\overline{F}|_{\C'}$ is canonically isomorphic to $F$.

\textbf{7. Step: $\overline{F}$ is $R$-linear.} Since $0 \in \C'$ we have $\overline{F}(0)=0$. Hence $\overline{F}$ preserves zero morphisms. Now let $f,g : M \to N$ be two morphisms. In order to show that $\overline{F}(f)+\overline{F}(g) = \overline{F}(f+g)$, we may precompose with $F(A) \twoheadrightarrow \overline{F}(M)$ for any $A \twoheadrightarrow M$ with $A \in \C'$. Thus, we may assume that $M=A$. Given $f,g : A \to N$, we may find a commutative diagram
\[\xymatrix@C=13pt@R=13pt{ & A \ar[dl]_{i} \ar[dr]^{f} & \\ C \ar[rr]^{h} && N. \\ & A \ar[ul]^{g} \ar[ur]_{j} & }\]
with $C \in \C'$. It follows:
\[\overline{F}(f)+\overline{F}(g) = \overline{F}(h) F(i) + \overline{F}(h) F(j) = \overline{F}(h) (F(i)+F(j))\]
Since $F : \C' \to \D$ is linear, we have $F(i)+F(j)=F(i+j)$ and therefore
\[\overline{F}(f)+\overline{F}(g) = \overline{F}(h) F(i+j) = \overline{F}(h(i+j)) = \overline{F}(f+g).\]

Now let $\lambda \in R$ and $M \in \C$ with presentation $A' \xrightarrow{i} A \xrightarrow{p} M \to 0$. We denote by $\lambda : M \to M$ also the morphism which multiplies with $\lambda$. One checks that
\[\xymatrix@R=10pt{ & A \ar[r]^{p} & M \ar[r] \ar[dd]^{\lambda} & 0 \\ A \oplus A' \ar[dr]_{(\lambda,i)} \ar[ur]^{\pr_A} & & \\ & A \ar[r]^{p} & M \ar[r] & 0}\]
is a weak pullback. Although $A \oplus A'$ has no reason to be contained in $\C'$, we may choose an epimorphism from some object in $\C'$ and use the already proven additivity of $\overline{F}$ on morphisms and the $R$-linearity of $F$ to conclude that
\[\xymatrix@R=10pt{ & F(A) \ar[r]^{\overline{F}(p)} & \overline{F}(M) \ar[r] \ar[dd]^{F(\lambda)} & 0 \\ F(A) \oplus F(A') \ar[dr]_{(\lambda,F(i))} \ar[ur]^{\pr_{F(A)}} & & \\ & F(A) \ar[r]^{\overline{F}(p)} & \overline{F}(M) \ar[r] & 0}\]
is commutative. Since this diagram commutes also with $F(\lambda)$ replaced by $\lambda$, we infer that $F(\lambda)=\lambda$. 

\textbf{8. Step: $\overline{F}$ is right exact.} We already know that $\overline{F}$ preserves direct sums from Step 7, so that it suffices to prove that $\overline{F}$ preserves cokernels. We choose an exact sequence $M \to N \to K \to 0$ in $\C$. We will construct the following commutative diagram with exact rows and columns and $A,B,B',C,C' \in \C'$.
\[\xymatrix{0 & 0 & 0 & \\ M \ar[r] \ar[u] & N \ar[r] \ar[u] & K \ar[r] \ar[u] & 0 \\ A \ar[r] \ar[u] & B \ar[r] \ar[u] & C \ar[u] \ar[r] & 0 \\ & B' \ar[r] \ar[u] & C' \ar[r] \ar[u] & 0}\]
Choose a presentation $C' \to C \to K \to 0$. Choose some $B \in \C'$ with an epimorphism $B \twoheadrightarrow N \times_K C$. Choose some $A \in \C'$ with $A \twoheadrightarrow \ker(B \to C) \times_N M$. Then $A \to B \to C \to 0$ is exact by a diagram chase. Another diagram chase shows that $\ker(B \to C) \to \ker(N \to K)$ is an epimorphism, which in turn can be used to see that $A \to M$ is an epimorphism. Finally, the remaining square containing $B'$ is constructed via symmetry.

Now we apply $\overline{F}$ and obtain the following commutative diagram with exact columns and rows -- except perhaps for the first row.
\[\xymatrix{0 & 0 & 0 & \\ \overline{F}(M) \ar[r] \ar[u] &  \overline{F}(N) \ar[r] \ar[u] &  \overline{F}(K) \ar[r] \ar[u] & 0 \\ F(A) \ar[r] \ar[u] & F(B) \ar[r] \ar[u] & F(C) \ar[u] \ar[r] & 0 \\ & F(B') \ar[r] \ar[u] & F(C') \ar[r] \ar[u] & 0}\]
We show exactness of the first row as follows: Let $ \overline{F}(N) \to T$ be a morphism which vanishes on $ \overline{F}(M)$. This corresponds to a morphism $F(B) \to T$ which vanishes on $F(A)$ and $F(B')$, hence to a morphism $F(C) \to T$ which vanishes on $F(B')$. Since $F(B') \to F(C')$ is an epimorphism, it already vanishes on $F(C')$, which means that it corresponds to a morphism $ \overline{F}(K) \to T$.
 
This finally finishes the construction of the right exact extension $\overline{F} : \C \to \D$.

\textbf{9. Step: Equivalence of categories.} As for the conclusion, we have an obvious restriction functor and only have to construct an inverse. Given a functor $F : \C' \to \D$ satisfying $(\star)$, we have constructed above an extension $\overline{F} : \C \to \D$. Given a morphism $\tau : F \to G$ of functors $F,G : \C' \to \D$ satisfying $(\star)$, we define $\overline{\tau} : \overline{F} \to \overline{G}$ as follows: If $M \in \C$ with a presentation $A' \to A \to M \to 0$, then $\overline{\tau}_M : \overline{F}(M) \to \overline{F}(M)$ is the unique morphism which makes the diagram
\[\xymatrix{F(A') \ar[d]^{\tau_{A'}} \ar[r] & F(A) \ar[d]^{\tau_A} \ar[r] & \overline{F}(M) \ar[r] \ar[d]^{\overline{\tau}_M} & 0 \\ F(A') \ar[r] & F(A) \ar[r] & F(M) \ar[r] & 0}\]
commutative. Now the rest is easy to check.
\end{proof}

\begin{ex}
If $X$ is a noetherian Adams stack, then $\C=\Q_{\fp}(X)$ is an abelian category and the full subcategory $\C'$ of locally free sheaves satisfies the assumptions of \autoref{extend2}. Hence, right exact functors on $\Q_{\fp}(X)$ only have to be defined on locally free sheaves. This has been independently proven by Bhatt (\cite[Corollary 3.2]{Bha14}).
\marginpar{I've added a reference to Bhatt's paper.}
\end{ex}

\clearpage
\chapter[Introduction to cocomplete tensor categories]{Introduction to cocomplete\\tensor categories} \label{Intro}

\section{Definitions and examples} \label{deftens}

In this section we define the main objects of our study, namely cocomplete tensor categories, as well as cocontinuous tensor functors between them.
 
\begin{defi}[Definition of cocomplete tensor categories] \noindent
\begin{enumerate}
\item 
For us a \emph{tensor category} is a symmetric monoidal category. It consists of a category $\C$, a functor $\C \times \C \to \C$, $(A,B) \mapsto A \otimes B$ (tensor product), an object $1 \in \C$ (unit) and natural isomorphisms $\alpha_{A,B,C} : A \otimes (B \otimes C) \cong (A \otimes B) \otimes C$ (associativity constraints), $l_A : 1 \otimes A \cong A$, $r_A : A \otimes 1 \cong A$ (unit constraints) and $S_{A,B} : A \otimes B \cong B \otimes A$ (symmetries) for $A,B,C \in \C$ subject to certain coherence conditions (\cite[I.2.4]{SR72}). Often we abuse notation and also write $\C$ for the underlying category.
\item
If $\C,\D$ are tensor categories, a \emph{tensor functor} $F : \C \to \D$ is a strong symmetric monoidal functor, i.e. a functor of the underlying categories together with natural isomorphisms $F(A \otimes B) \cong F(A) \otimes F(B)$ for $A,B \in \C$ and an isomorphism $F(1) \cong 1$ subject to certain compatibility conditions (\cite[I.4.1.1, 4.2.4]{SR72}).
\item
If $F,G : \C \to \D$ are two tensor functors between tensor categories, a morphism of tensor functors $\eta : F \to G$ is a morphism of the underlying functors which is compatible with the unit constraints and the tensor product (\cite[I.4.4.1]{SR72}).
\item A \emph{cocomplete tensor category} is a tensor category whose underlying category is cocomplete such that for every object $A$ the endofunctor $A \otimes -$ (and hence also $- \otimes A$) is cocontinuous. A \emph{cocontinuous tensor functor} between cocomplete tensor categories is a tensor functor between the underlying tensor categories whose underlying functor is cocontinuous. A morphism between cocontinuous tensor functors is just a morphism between the underlying tensor functors.
\item Restricting to finite colimits, we obtain the notions of \emph{finitely cocomplete tensor categories} and \emph{right exact tensor functors}. By an \emph{abelian tensor category} we mean a finitely cocomplete tensor category whose underlying category is abelian.
\item If $\C,\D$ are (cocomplete) tensor categories, we will denote the category of (resp. cocontinuous, resp. finitely cocontinuous) tensor functors from $\C$ to $\D$ by $\Hom_{\otimes}(\C,\D)$ (resp. $\Hom_{c\otimes}(\C,\D)$, resp. $\Hom_{fc\otimes}(\C,\D)$).
\item We obtain the $2$-category of tensor categories $\Cat_\otimes$ and the $2$-category of cocomplete tensor categories $\Cat_{c\otimes}$.
\end{enumerate}
\end{defi}

\begin{rem}[Categorification]
Let $\C$ be a cocomplete tensor category. Since $A \otimes -$ is cocontinuous for each $A \in \C$, this functor preserves in particular direct sums: The canonical morphism
\[\bigoplus_i (A \otimes B_i) \to A \otimes \bigoplus_i B_i\]
is an isomorphism. This is a categorified distributive law. For the empty direct sum it becomes
\[A \otimes 0 = 0.\]
Whereas tensor categories categorify commutative monoids, cocomplete tensor categories categorify commutative rings or more precisely commutative \emph{semirings} (also known as \emph{rigs}), since for objects $A \neq 0$ there is no object $B$ with $A \oplus B = 0$. In fact, Baez and Dolan suggest this definition of $2$-rigs (\cite[2.3]{BD98b}). Chirvasitu and Johnson-Freyd use the terminology of $2$-rings (\cite{CJF13}) under the additional assumption of presentability (\autoref{presentable-tensor}). See \autoref{categorif} for more on categorification. One should admit that rig categories as studied by Laplaza (\cite{Lap72}) provide a more adequate categorification of rigs: Here $\oplus$ is not assumed to be the direct sum, but rather part of another monoidal structure.
\end{rem}

\begin{rem}[Coherence] \label{coherence}
Let $\C$ be a tensor category. The Coherence Theorem for symmetric monoidal categories (\cite[XI.1. Theorem 1]{ML98}) roughly says that between any two words in $\C$ built up inductively out of objects using the tensor structure there is at most one isomorphism which is built up out of the associativity and unit constraints of $\C$. In particular, there is no harm in identifying $A \otimes (B \otimes C)$ with $(A \otimes B) \otimes C$ or for example $B \otimes (C \otimes A)$ and just write $A \otimes B \otimes C$ for this object, since there is a \emph{canonical} isomorphism between these objects. The Coherence Theorem implies that ``all reasonable diagrams'' commute. This saves us diagram chases and justifies that we treat (compositions of) the associativity and unit constraints as identities. Since colimits are defined via universal properties, we obtain a similar Coherence Theorem for cocomplete tensor categories.
\end{rem}

\begin{ex}[Closed categories] \label{closedcat}
Let $\C$ be a tensor category whose underlying category is cocomplete and which is \emph{closed}, i.e. internal homs (\cite[I.5.1.1]{SR72}) exist in $\C$. Then $- \otimes A$ is left adjoint to $\HOM(A,-)$, hence cocontinuous. Therefore $\C$ is a cocomplete tensor category. Most examples ``in nature'' arise this way. In particular, every (cocomplete) cartesian closed category is a (cocomplete) tensor category with $\otimes = \times$. This gives lots of examples.
\end{ex}

\begin{ex}[Sets and spaces] \label{sets}
The most basic example is the cartesian closed category of (small) sets $\Set$. Observe that for every cocomplete tensor category $\C$ there is essentially a unique cocontinuous tensor functor $\Set \to \C$, mapping $X$ to $1^{\oplus X}$. In fact, $\Set$ is a $2$-initial object in $\Cat_{c\otimes}$. Another example is the cartesian closed category of simplicial sets $\sSet$. The category of topological spaces $\Top$ is not cartesian closed, however any \emph{convenient category of spaces} (\cite{St67}) such as the category of compactly generated (weak) Hausdorff spaces is closed and therefore may be regarded as a cocomplete tensor category.
\end{ex}

\begin{ex}[Quantales] \label{quantale-def}
A tensor category whose underlying category is a (small) preorder, for short tensor preorder, is a preorder whose underlying set also carries the structure of a commutative monoid in such a way that the product is increasing in each variable. A tensor functor between tensor preorders is a increasing monoid homomorphism. A morphism between such tensor functors $F,G : \C \to \D$ is unique if it exists and it exists if and only if $F(c) \leq G(c)$ for all $c \in \C$.\\
A cocomplete tensor preorder is a tensor preorder such that the underlying preorder has all suprema and such that multiplication preserves suprema in each variable. These structures are also known as \emph{quantales} (\cite{Mul02}), which we assume here to be unital and commutative. A morphism of quantales is an increasing monoid homomorphism which preserves all suprema. Although quantales do not seem to play a role in algebraic geometry yet, they allow us to view some constructions for cocomplete tensor categories more concretely. Three specific examples of quantales are:
\begin{itemize}
\item The unit interval $[0,1]$ with the usual monoid and order structure.
\item The quantale of ideals of a commutative ring $R$ with the ideal product and the inclusion relation.
\item The quantale of open subsets of a topological space with the intersection and the inclusion relation.
\end{itemize}
For objects $a,b$ of a quantale the internal hom $[b:a] := \HOM(a,b)$ exists; it is given by the largest element with the property that $[b:a] \cdot a \leq b$. It is also known as Heyting implication and written $(a \longrightarrow b)$. For the quantale of ideals of a commutative ring we get the usual ideal quotient, which also explains our notation.
\end{ex}
 
\begin{defi}[Linear tensor categories]
Let $R$ be a commutative ring. An \emph{$R$-linear tensor category} $\C$ is a  tensor category whose underlying category is $R$-linear such that the tensor product functor $\otimes : \C \times \C \to \C$ is $R$-linear in both variables. A tensor functor is called $R$-linear if its underlying functor is $R$-linear. The category of $R$-linear tensor functors $\C \to \D$ is denoted by $\Hom_{\otimes/R}(\C,\D)$. We obtain the $2$-category of $R$-linear tensor categories $\Cat_{\otimes/R}$. Similarly, we obtain the $2$-category of cocomplete $R$-linear tensor categories $\Cat_{c\otimes/R}$ and $\Hom_{c\otimes/R}(\C,\D)$ denotes the category of cocontinuous $R$-linear  tensor functors $\C \to \D$.
\end{defi}

\begin{rem}
For a tensor category $\C$ the monoid $\End_\C(1)$ is commutative by the Eckmann-Hilton argument (\cite[I.1.3.3.1]{SR72}). In case of an $R$-linear tensor category this is even a commutative $R$-algebra. This provides a $2$-functor
\[\Cat_{c\otimes/R} \to \CAlg(R), ~ \C \mapsto \End_\C(1)\]
which turns out to be right adjoint to $\M(-)$ (\autoref{affin-global}).
\end{rem}

Our main example is the following:

\begin{ex}[Modules] \noindent
\begin{enumerate}
\item Let $R$ be a commutative ring. Then the category of right $R$-modules $\M(R)$ is a cocomplete $R$-linear tensor category with the usual tensor product $\otimes_R$ and unit $R$.
\item More generally, let $X$ be a ringed space. Then the category of (sheaves of) $\O_X$-modules $\M(X)$ is a cocomplete tensor category with the usual tensor product $\otimes_{\O_X}$ and unit $\O_X$. If $f : X \to Y$ is a morphism of ringed spaces, then the associated pullback functor $f^* : \M(Y) \to \M(X)$ is a tensor functor. It is cocontinuous because it is left adjoint to $f_*$.
\item If $X$ is a scheme, then its category of quasi-coherent (sheaves of) modules $\Q(X) \subseteq \M(X)$ is closed under colimits, tensor products and pullbacks. It follows that $\Q(X)$ is a cocomplete tensor category and that every morphism of schemes $f : X \to Y$ induces a cocontinuous tensor functor $f^* : \Q(Y) \to \Q(X)$.
\item If $X$ is defined over some commutative ring $R$, then $\Q(X)$ is a cocomplete $R$-linear  tensor category. If $f : X \to Y$ is a morphism of $R$-schemes, then $f^* : \Q(Y) \to \Q(X)$ is an $R$-linear cocontinuous tensor functor.
\item The same statements apply to algebraic stacks, but also to monoid schemes or \emph{schemes over $\mathds{F}_1$} (\cite{Dei05}, \cite{Con10}, \cite{Cor11}, \cite{PL09}), even to generalized schemes in the sense of Durov (\cite{Dur07}). For example, the unique morphism $\Spec(\Z) \to \Spec(\F_1)$ induces via pullback the cocontinuous tensor functor $\Set_* \to \Ab$ which is left adjoint to the forgetful functor, i.e. it maps a pointed set $(X,x_0)$ to the abelian group $\{f \in \Z^{\oplus X} : f(x_0)=0\}$.
\end{enumerate}
\end{ex}

\begin{rem}[Structure sheaf]
Since this is our main example and we will ``model'' algebraic geometry on $\C$, we often write $\O_\C$ or just $\O$ for the unit object of $\C$. We also treat morphisms $\O \to M$ as \emph{global sections} of $M$. It is well understood that an object $M$ is not determined by its global sections, but rather by its Hom functor $\Hom(-,M)$.
\end{rem}

\begin{ex}[Chain complexes] \label{chain}
Let $R$ be a commutative ring. The category of chain complexes of $R$-modules $\mathrm{Ch}(R)$, endowed with the usual total tensor product and the ``twisted'' symmetry defined by $a \otimes b \mapsto (-1)^{|a||b|} b \otimes a$, becomes a cocomplete $R$-linear  tensor category. The unit is $R$ concentrated in degree $0$. Looking at those chain complexes with zero differentials, we get the  cocomplete $R$-linear tensor category of graded $R$-modules $\tilde{\gr}_\Z(R)$.
\end{ex}

\begin{rem}[Internalization]
Apart from  ``atomic'' examples of cocomplete tensor categories such as the ones above, there are various constructions for cocomplete tensor categories which produce new ones. This will be the subject of \autoref{Cons}. For instance, in the previous example, $\M(R)$ may be replaced by any  cocomplete $R$-linear tensor category $\C$. We can still construct a cocomplete $R$-linear  tensor category $\mathrm{Ch}(\C)$ of chain complexes with values in $\C$. This kind of \emph{internalization} is one of the central themes in this thesis. Here is another example:
\end{rem}

\begin{ex}[Representation categories] \label{funcat}
Let $I$ be a small category and let $\C$ be a (cocomplete) tensor category. Then the category of functors $I \to \C$ carries the structure of a (cocomplete) tensor category (the tensor product is simply defined objectwise), denoted by ${}^I \C = \Hom(I,\C)$. If $G$ is a group, considered as a category with one object, we obtain the tensor category ${}^G \C = \Rep_\C(G)$ of left actions or representations of $G$ on $\C$. Its objects are objects $V \in \C$ equipped with a homomorphism of groups $\rho : G \to \Aut(V)$. The tensor product $(V,\rho) \otimes (V',\rho')$ is $(V \otimes V',\rho \otimes \rho')$. Similarly, we have the category of right actions $\C^G = {}^{G^{\op}} \C$. The cases $\C=\Set$ and $\C=\M(R)$ are well-known, in the latter case one also writes $\Rep_R(G)$. Notice that there is a difference between the tensor categories $\Rep_R(G)$ and $\M(R[G])$, only the underlying categories coincide.
\end{ex}


\begin{rem} \label{initial}
If $R$ is a commutative ring, then the module category $\M(R)$ is a $2$-initial object of $\Cat_{c\otimes/R}$. In fact, for every $\C \in \Cat_{c\otimes/R}$ the external tensor product $\M(R) \to \C, ~ M \mapsto M \otimes_R \O_\C$ is a cocontinuous tensor functor which is essentially unique. For details, see \autoref{aff-glob}.
\end{rem}
 
\begin{defi}
If $\C$ is a cocomplete tensor category, then we have the $2$-category $\Cat_{c\otimes/\C}$ of cocomplete tensor categories \emph{over} $\C$. Objects are cocontinuous tensor functors $\C \to \D$. If  $P : \C \to \D$ and $Q : \C \to \E$ are  two objects, we define $\Hom_{c\otimes/\C}(\D,\E)$ to be the category whose objects are cocontinuous tensor functors $F : \D \to \E$ equipped with an isomorphism $FP \cong Q$. For example, we have $\Cat_{c\otimes/\Set} = \Cat_{c\otimes}$ and $\Cat_{c\otimes/\M(R)} = \Cat_{c\otimes/R}$. We have chosen the terminology ``over $\C$'' instead of the more appropriate ``under $\C$'' in order to incorporate the analogy to algebraic geometry. Namely, we think of $\Cat_{c\otimes/\Q(S)}$ as a ``globalization'' of $\Sch/S$.
\end{defi}

\begin{defi}[Locally presentable tensor categories] \label{presentable-tensor}
A \emph{locally presentable tensor category} is a cocomplete tensor category $\C$ whose underlying category is locally presentable.  If there is a regular cardinal $\lambda$ such that $\C$ is locally $\lambda$-presentable, $1 \in \C$ is $\lambda$-presentable and the tensor product of two $\lambda$-presentable objects is $\lambda$-presentable, then $\C$ is called a locally $\lambda$-presentable tensor category. If there is such a $\lambda$, we call $\C$ a strongly locally presentable tensor category.
\end{defi}

\begin{rem}
We remark that any locally presentable tensor category is closed. This is because the functors $A \otimes - : \C \to \C$ are cocontinuous, hence have a right adjoint (\autoref{saft}). We do not know if every locally presentable tensor category is strongly locally presentable. In fact, we do not even know if $\Q(X)$ is strongly locally presentable for an arbitrary scheme $X$. If $X$ is qc qs, then we know that $\Q(X)$ is locally $\aleph_0$-presentable (\autoref{Qcohfin}).
\end{rem}

\begin{ex}
The following example of a cocomplete cartesian quantale which is \emph{not} closed is due to Todd Trimble. Let $V$ be some Grothendieck universe. Endow $V$ with the partial order $\subseteq$. Then $V$ has all (small) colimits and all non-empty limits, but there is no terminal object. So let us just adjoin one and consider $\overline{V} := V \cup \{\infty\}$ such that $A \subseteq \infty$ for $A \in V$. Then the functor $(X \times -) = (X \cap -) : \overline{V} \to \overline{V}$ is cocontinuous for all $X \in V$, but it has no right adjoint. Otherwise, for every $Y \in \overline{V}$ with $X \not\subseteq Y$ there would be a largest set $Z \in \overline{V}$ such that $Z \cap X \subseteq Y$, which is absurd since we can simply adjoin random elements of $V \setminus X$ to $Z$, contradicting maximality.
\end{ex}

\begin{ex}
If $X$ is a scheme, then $\Q_{\fp}(X) \subseteq \Q(X)$ is a finitely cocomplete tensor category. Let us mention at this point that large parts of the theory of cocomplete tensor categories may be also carried out for finitely cocomplete tensor categories. But often infinite colimits will be useful. Besides, quasi-coherent sheaves of finite presentation usually are not stable under direct images.
\end{ex}
 
The following basic result will be used very often:
 
\begin{lemma}[Tensor product of presentations] \label{coeq-tensor}
Let $\C$ be a finitely cocomplete tensor category.
\begin{enumerate}
\item The tensor product of two epimorphisms in $\C$ is again an epimorphism.
\item Assume that $A_i \rightrightarrows B_i \to P_i$ for $i=1,2$ are two coequalizer diagrams. Then we obtain a coequalizer diagram
\[(A_1 \otimes B_2) \oplus (B_1 \otimes A_2) \rightrightarrows B_1 \otimes B_2 \to P_1 \otimes P_2.\]
In particular, the tensor product of two regular epimorphisms is again a regular epimorphism.
\end{enumerate}
\end{lemma}

\begin{proof}
1. Let $f : A \to A'$ and $g : B \to B'$ be two epimorphisms. Since $A \otimes -$ is cocontinuous, it preserves epimorphisms, hence $A \otimes g : A \otimes B \to A \otimes B'$ is an epimorphism. By the same reason $f \otimes B' : A \otimes B' \to A' \otimes B'$ is an epimorphism. Hence, the composition $f \otimes g = (f \otimes B') (A \otimes g)$ is an epimorphism.
 
2. Let $f_i,g_i : A_i \to B_i$ and $p_i : B_i \to P_i$, so that $p_i$ is a coequalizer of $f_i$ and $g_i$. Define $\tilde{f} : (A_1 \otimes B_2) \oplus (B_1 \otimes A_2) \to B_1 \otimes B_2$ by $\tilde{f}|_{A_1 \otimes B_2} = f_1 \otimes B_2$ and $\tilde{f}|_{B_1 \otimes A_2} = B_1 \otimes f_2$. Similarly we define $\tilde{g}$. We claim that $p_1 \otimes p_2$ is a coequalizer of $\tilde{f}$ and $\tilde{g}$. Clearly it coequalizes these morphisms. Let $h : B_1 \otimes B_2 \to T$ be a test morphism such that $h \tilde{f} = h \tilde{g}$, i.e. $h (f_1 \otimes B_2) = h (g_1 \otimes B_2)$ and $h(B_1 \otimes f_2) = h(B_1 \otimes g_2)$. Since $B_1 \otimes p_2$ is a coequalizer of $B_1 \otimes f_2$ and $B_1 \otimes g_2$, it follows that $h = h' (B_1 \otimes p_2)$ for some $h' : B_1 \otimes P_2 \to T$.

We claim $h'(f_1 \otimes P_2) = h' (g_1 \otimes P_2)$. Since $p_2$ and hence $A_1 \otimes p_2$ is an epimorphism, it suffices to prove $h'(f_1 \otimes p_2) = h'(g_1 \otimes p_2)$. But this follows from the equation $h(f_1 \otimes B_2) = h(g_1 \otimes B_2)$. Since $p_1 \otimes P_2$ is a coequalizer of $f_1 \otimes P_2$ and $g_1 \otimes P_2$, it follows that $h'=h''(p_1 \otimes P_2)$ for some $h'' : P_1 \otimes P_2 \to T$. Hence, $h$ factors through $p_1 \otimes p_2$. The factorization is unique since $p_1 \otimes p_2$ is an epimorphism.
\end{proof}

\begin{prop}[Tensorial extension] \label{extend}
If $\C$ is an $R$-linear abelian tensor category, $\D$ is any finitely cocomplete $R$-linear tensor category and $\C'$ is a full replete tensor subcategory of $\C$ with the properties as in \autoref{extend2}, then there is an equivalence of categories between $R$-linear tensor functors $\C' \to \D$ satisfying the exactness property $(\star)$ and right exact $R$-linear tensor functors $\C \to \D$.
\end{prop}

\begin{proof}
Given an $R$-linear tensor functor $F : \C' \to \D$, we endow the $R$-linear right exact extension $\overline{F} : \C \to \D$ (\autoref{extend2}) with the structure of a tensor functor as follows: We have $\overline{F}(1) \cong F(1) \cong 1$. If $M \in \C'$, then we have an isomorphism between the right exact functors $\overline{F}(M \otimes - )$ and $\overline{F}(M) \otimes \overline{F}(-)$ when restricted to $\C'$. By \autoref{extend2} this extends to an isomorphism on $\C$. Similarly, we may generalize to the case that $M \in \C$. Now the rest follows formally from \autoref{extend2}.
\end{proof}

\begin{defi}[Lax and oplax] \label{lax}
A \emph{lax tensor functor} $F : \C \to \D$ between tensor categories $\C,\D$ is a functor between the underlying categories equipped with natural morphisms $1_\D \to F(1_\C)$ and $F(A) \otimes F(B) \to F(A \otimes B)$ subject to certain compatibility conditions (\cite[XI.2]{ML98}). Dually one defines \emph{oplax tensor functors}. There is an obvious notion of morphisms between lax/oplax tensor functors.
\end{defi}

Thus, a tensor functor is the same as a lax tensor functor whose structure morphisms are isomorphisms. Some authors call ``tensor functors'' what we call lax tensor functors and ``strong tensor functors'' what we just call tensor functors.

\begin{ex}[Adjunctions] \label{adj-lax}
If $F : \C \to \D$ is an oplax tensor functor, whose underlying functor has a right adjoint $G : \D \to \C$, then $G$ becomes a lax tensor functor as follows: The morphism
\[F(G(A) \otimes G(B)) \to F(G(A)) \otimes F(G(B)) \to A \otimes B\]
corresponds by adjunction to a morphism
\[G(A) \otimes G(B) \to G(A \otimes B).\]
Similarly, the morphism $F(1_\C) \to 1_\D$ corresponds by adjunction to a morphism $1_\C \to G(1_\D)$. One can check that $F$ is, indeed, a lax tensor functor (\cite[1.4]{Kel74}). If $F$ is a tensor functor, then by construction both the unit $\id_\C \to G F$ and the counit $FG \to \id_\D$ are morphisms of lax tensor functors. For example, if $f : X \to Y$ is a morphism of ringed spaces, then we get a lax tensor functor $f_* : \M(X) \to \M(Y)$ from the tensor functor $f^* : \M(Y) \to \M(X)$.
\end{ex}

\section{Categorification} \label{categorif}

We refer to the wonderful papers \cite{BD98a} and \cite{BD01} for the general principles of \emph{categorification}. We have already remarked that cocomplete tensor categories may be seen as categorified semirings: The tensor product $\otimes$ replaces the multiplication, colimits replace sums and we have a distributive law between them. The initial object $0$ replaces the zero element, the unit object $1$ replaces the unit element.
 
So we may just write $A \cdot B = AB$ and $A+B$ for the tensor product resp. direct sum of $A$ and $B$. More generally, write $\sum_i A_i$ for the direct sum $\bigoplus_i A_i$. For $n \in \N$ let us abbreviate $\bigoplus_{i=1}^{n} A$ by $n \cdot A$ and $A^{\otimes n} = A \otimes \dotsc \otimes A$ ($n$ factors) by $A^n$ (whereas the direct sum resp. product of $n$ copies of $A$ is denoted by $n \cdot A = A^{\oplus n}$ resp. $A^{\times n}$).
 
With these operations we may calculate with objects of a cocomplete tensor category as if we were in a semiring, except of course that equalities have to be replaced by (coherent) isomorphisms. We get an honest semiring (which may be large) by considering isomorphism classes of objects. But this process of \emph{decategorification}, which is common in K-theory, loses a lot of information. For example, the cocomplete cartesian category $\Set$ decategorifies to the semiring of cardinal numbers. Many combinatorial identities are actually ``shadows'' of isomorphisms in $\Set$. This is the idea of the so-called bijective-proof in combinatorics. Let us point out that some ``identities'' actually hold in an arbitrary cocomplete tensor category. We use the Binomial Theorem as a sample.

\begin{lemma}[Binomial Theorem] \label{binomi}
Let $A,B$ be objects of a finitely cocomplete tensor category and $n \in \N$. Then there is a natural isomorphism
\[(A + B)^n \cong \sum\limits_{p=0}^{n} \binom{n}{p} \cdot A^p \cdot B^{n-p}.\]
\end{lemma}

\begin{proof}
Because of the distributive law, we may expand $(A+B)^n$ inductively. It is the sum of all possible products $C_1 \cdot \dotsc \cdot C_n$ with $C_i \in \{A,B\}$. Using the symmetry isomorphisms, we may move the $A$'s to the front. If $p$ is the number of $A$'s, we arrive at $A^p \cdot B^{n-p}$. It appears $\binom{n}{p}$ times.
\end{proof}

For example, in the category of $\Z$-graded $R$-modules $\tilde{\gr}_\Z(R)$ with twisted symmetry (\autoref{chain}), the isomorphism $(A+B)^2 \cong A^2 + 2 AB + B^2$ is given by $(a+b)(a'+b') \mapsto (aa',ab',(-1)^{|a'||b|} a'b,bb')$.
 
What we have said so far actually applies to tensor categories with finite direct sums. But cocomplete tensor categories $\C$ also admit infinite direct sums. This allows us to use categorified formal power series: If $X$ is an object of $\C$, then for every sequence of objects $A_0,A_1,\dotsc$ we may look at the power series
\[\sum_{n=0}^{\infty} A_n \cdot X^n.\]
Observe that the usual product formula
\[\sum_{n=0}^{\infty} A_n \cdot X^n ~ \cdot ~ \sum_{n=0}^{\infty} B_n \cdot X^n ~ \cong ~ \sum_{n=0}^{\infty} \left(\sum_{p=0}^{n} A_p \cdot B_{n-p}\right) \cdot X^n\]
holds. But colimits also include quotients by group actions. For example, the symmetric group $\S_n$ acts on $X^n := X^{\otimes n}$ (\autoref{symtrivial}) and the quotient may be denoted by $X^n / n!$. We have the exponential power series
\[\exp(X) := \sum_{n=0}^{\infty} X^n / n!\]
with the usual properties
\[\exp(0) \cong 1,~ \exp(X+Y) \cong \exp(X) \cdot \exp(Y).\]
Actually $\exp(X)$ is nothing else than the symmetric algebra over $X$ (\autoref{symmetric-power}). For $\C=\M(R)$ it is the usual symmetric algebra over some $R$-module $X$ and for $\C=\Set$ it is the set of finite \emph{multisets} whose elements come from $X$. Unfortunately, we cannot define logarithms in this context since their usual power series expressions contain negative numbers.

If $\C$ is an $R$-linear cocomplete tensor category with $2 \in R^*$, we will define exterior powers in $\C$ (\autoref{exterior-power}). Motiviated by the fact that $\dim(\Lambda^n(V))=\binom{\dim(V)}{n}$ for finite-dimensional vector spaces $V$, we define binomial coefficients in $\C$ by $\binom{X}{n} := \Lambda^n(X)$ for $n \in \N$ and $X \in \C$. The Vandermonde identity for binomial coefficients categorifies to the isomorphism (see \autoref{sym-exakt})
\[\binom{X+Y}{n} \cong \sum_{p=0}^{n} \binom{X}{p} \cdot \binom{Y}{n-p}.\]
Motivated by the complex power series expansion $(1+z)^r = \sum_{p=0}^{\infty} \binom{r}{p} \cdot z^p$ for $z \in \mathds{C}$ with $|z|<1$ and $r \in \R$, we define a new power operation on $\C$ by
\[(1+X)^Y := \sum_{p=0}^{\infty} \binom{Y}{p} \cdot X^p\]
for objects $X,Y \in \C$. Of course we do not have to worry about convergence. Clearly $(1+X)^0  \cong 1$ and $(1+X)^1=1+X$. The Vandermonde isomorphism above yields
\[(1+X)^{Y} \cdot (1+X)^{Z} \cong (1+X)^{Y+Z}.\]
In particular, $(1+X)^{n \cdot 1}$  agrees with the usual power $(1+X)^n$ for $n \in \N$. We do not know if this power operation has been already studied for modules. Unfortunately, the isomorphism $(1+X)^{Y \cdot Y'} \cong ((1+X)^Y)^{Y'}$ (the right hand side is well-defined because of $(1+X)^Y = 1+\sum_{p=1}^{\infty} \binom{Y}{p} \cdot X^p$) does \emph{not} hold for modules. Note that
\[2^X = \sum_{p=0}^{\infty} \binom{X}{p} = \Lambda(X)\]
is the usual exterior algebra over $X$. Observe that although $\Set$ is not linear, we can still define $\binom{X}{p}$ to be the set of subsets of $X$ with $p$ elements and the isomorphisms above still hold. We do not know any conceptual explanation which includes $\Set$ and $\M(R)$ simultaneously.

Although we have defined natural numbers inside $\C$ via $n \cdot 1_C$, we cannot define rational numbers: For example, if $1 \cong 2 \cdot X$, then $0 = \binom{1}{2} = \binom{X}{2} + \binom{X}{1} \cdot \binom{X}{1} +\binom{X}{2}$, hence $X^2=0$ and then $X \cong X(X+X)\cong 0$, hence $1 \cong 0$ and $\C = 0$.

We even have an analogue of integrals, namely coends or weighted colimits (\autoref{colimits}). We have the following reformulation of the Yoneda Lemma:

\begin{lemma} \label{yoneda}
Let $I$ be a small category and let $\C$ be a cocomplete category. Then every functor $F : I^{\op} \to \C$ may be expressed as a coend:
\[F \cong \int^{p \in I} \Hom(-,p) \otimes F(p).\]
\end{lemma}

This may be seen as a categorification of the formula
\[\mu(A) = \int_{p \in I} \chi_A(p) \cdot d \mu(p)\]
for a measure $\mu$ on a measurable space $I$ and measurable subsets $A \subseteq I$.
\section{Element notation} \label{Element}

In this work we will often encounter situations where it is too complicated or even confusing to write down large commutative diagrams or equations of morphisms in a tensor category. Instead, we will use some kind of \emph{element notation} which mimics the usual one from the case of sets or modules over a commutative ring. Each equation of ``elements'' written down in element notation is actually an equation of morphisms. A similar notation is already quite common for homomorphisms of sheaves on a space, where ``elements'' are actually sections of the involved sheaves (not just global ones). It also has been widely used in the cartesian case (where it is justified by the Yoneda embedding).

\begin{ex}[Equality]
Let $f,g : A \otimes B \to C$ be two morphisms. When we write $f(a \otimes b) = g(a \otimes b)$ for all ``elements'' $a \in A$, $b \in B$, we actually mean $f = g$.
\end{ex}

\begin{ex}[Symmetrizers]
We know how to turn a bilinear form into a symmetric bilinear form. This can be generalized as follows: Let $\C$ be an $R$-linear tensor category with $2 \in R^*$. Let $f : A \otimes A \to T$ be an arbitrary morphism in $\C$. Define $s : A \otimes A \to T$ by
\[s(a \otimes b) = \frac{f(a \otimes b) + f(b \otimes a)}{2}.\]
in element notation. This just means $s = \frac{1}{2}(f + f \circ S_{A,A})$. In element notation we immediately see that $s$ is symmetric, i.e. $s(a \otimes b) = s(b \otimes a)$. This actually means $f = f \circ S_{A,A}$. We get a rigorous proof directly from $S_{A,A}^2=\id_{A \otimes A}$. If we define $t : A \otimes A \to T$ by
\[t(a \otimes b) = \frac{f(a \otimes b) - f(b \otimes a)}{2},\]
then $t$ is anti-symmetric in the sense that $t(a \otimes b) = - t(b \otimes a)$. This actually means $t = \frac{1}{2}(f - f \circ S_{A,A})$ and $t = - t \circ S_{A,A}$. Note that $f = s + t$.
\end{ex}

\begin{rem}[Caution]
One has to be a bit careful when one wants to translate element notation to the usual notation in specific examples: Consider the tensor category of graded $R$-modules, equipped with the twisted symmetry (\autoref{chain}). In the setting of the previous example, for homogeneous elements $a,b \in A$ -- in the usual sense -- of the graded module $A$ we have
\[s(a \otimes b) = \frac{f(a \otimes b) + (-1)^{|a||b|} f(b \otimes a)}{2}.\]
That $s$ is symmetric means $s(a \otimes b) = (-1)^{|a||b|} s(b \otimes a)$. We would like to emphasize that this does not need any extra reasoning, since the previous example already covered arbitrary tensor categories. This general point of view enables us to treat graded modules and modules simultaneously; of course even a lot more types of modules, for example quasi-coherent sheaves on a scheme. Nevertheless we only need a single element notation which covers all cases.
\end{rem}

\begin{ex}[Epimorphisms]
Let $f,g : A \to B$ be morphisms. Assume that $p : E \to A$ is an epimorphism. If $f(p(e))=g(p(e))$ for all ``elements'' $e \in E$, then we have $f(a) = g(a)$ for all ``elements'' $a \in A$, i.e. $f = g$. This is just a reformulation of the definition of an epimorphism. We would like to point out that -- using element notation -- we may forget about the abstract definition of an epimorphism and pretend in our proofs that this means that every ``element'' of $A$ has the form $p(e)$ for some ``element'' $e$ of $E$.
\end{ex}
 
\begin{rem}[Soundness]
In general, the translation of a calculation in element notation to a rigorous proof can be done by rote. It works in each case: Every equation is actually a commutative diagram. Either commutativity is a special instance of the Coherence Theorem  (\autoref{coherence}), or it follows from the assumptions. A chain of equations pastes together commutative diagrams.

Let us illustrate this:
\end{rem}

\begin{ex}
Let $f : A \otimes A \otimes A \to B$ be a morphism satisfying
\[f(a \otimes b \otimes c) = f(b \otimes a \otimes c) = f(a \otimes c \otimes b)\]
in element notation, i.e. we can interchange the first two and the last two variables. Then it is clear that we can also interchange the first with the last variable:
\[f(a \otimes b \otimes c) = f(b \otimes a \otimes c) = f(b \otimes c \otimes a) = f(c \otimes b \otimes a).\]
Formally this can be derived from an equation of symmetry automorphisms of $A^{\otimes 3}$ and the two assumptions
\[f =f \circ (S_{A,A} \otimes \id_A) = f \circ (\id_A \otimes S_{A,A}).\]
See also \autoref{symtrivial}.
\end{ex}
 
Here is a more complicated example (needed in \autoref{localfunc}):
 
\begin{lemma} \label{non-unital}
Let $\C$ be a tensor category, $s : A \to 1$ be a split epimorphism in $\C$. Assume that $A$ is a non-unital commutative algebra in $\C$ (\autoref{algebradef}), say with multiplication $m : A \otimes A \to A$, such that the following diagram commutes:
\[\xymatrix@C=60pt@R=50pt{A \otimes A \ar[r]^{m} \ar[d]_{A \otimes s} & A \\ A \otimes 1 \ar[ur]^{\cong}_{r_A} & }\]
Then $s$ is an isomorphism.
\end{lemma}

\begin{proof}[Proof in element notation]
We have a product $A \otimes A \to A$, $a \otimes b \mapsto a \cdot b$ and a map $s : A \to 1$ such that $a \cdot b = a \, s(b)$. Besides, there is some map $i : 1 \to A$ such that $s(i(r))=r$. Then $e:=i(1)$ is a right unit for the product, because $a \cdot e = a \, s(i(1))=a \, 1 = a$. Since $\cdot$ is commutative, it follows that $e$ is also a left unit. This implies $a = e \cdot a = e \, s(a) = i(s(a))$. Hence, $i$ is inverse to $s$.
\end{proof}

\begin{proof}[Rigorous proof] Choose some $i : 1 \to A$ with $s \circ i=\id_{1}$. Then $i$ is a right unit for $A$, since
\[m \circ (A \otimes i) = r_A \circ (A \otimes s) \circ (A \otimes i) = r_A \circ (A \otimes si) = r_A.\]
Since $A$ is commutative, it is also a left unit:
\[m \circ (i \otimes A) = m \circ S_{A,A} \circ (i \otimes A) = m \circ (A \otimes i) \circ S_{1,A} = r_A \circ S_{1,A} = l_A\]
Hence, in the following diagram, every part commutes:
\[\xymatrix@C=40pt@R=40pt{A \ar[d]_{s} & 1 \otimes A \ar[l]_{l_A} \ar[d]^{1 \otimes s} \ar@{=}[r] & 1 \otimes A \ar[dd]_{i \otimes s} \ar@{=}[r] & 1 \otimes A \ar[dr]^{l_A} \ar[d]_{i \otimes A} \\
1 \ar[d]_{i} & 1 \otimes 1 \ar[d]^{i \otimes 1} \ar[l]^{l_1=r_1} & & A \otimes A \ar[r]^{m} \ar[d]_{A \otimes s} & A \\
A & A \otimes 1 \ar[l]^{r_A} \ar@{=}[r] & A \otimes 1 \ar@{=}[r] & A \otimes 1 \ar[ur]_{r_A} & }\]
The outer shape implies that $is=\id_{A}$.
\end{proof}
\section{Adjunction between stacks and cocomplete tensor categories} \label{adju}

Every morphism $f : Y \to X$ of schemes (or algebraic stacks) induces a pullback functor $f^* : \Q(X) \to \Q(Y)$, which is a cocontinuous tensor functor. Now we may ask if $f$ can be reconstructed from $f^*$ and if any cocontinuous tensor functor $\Q(Y) \to \Q(X)$ is induced by a morphism. This motivates the definition of \emph{tensorial schemes}, or more generally \emph{tensorial stacks}. We will give a quite conceptual definition of them with the help of an adjunction between stacks and cocomplete tensor categories.

In this section, we assume that all schemes and stacks are defined over some fixed commutative ring $R$ and we work with $R$-linear cocomplete tensor categories. This assumption is not really necessary and in fact one should omit it when one is interested in algebraic geometry over $\F_1$.
 
\begin{defi}[$\SPEC$ and $\Q$] \label{stacksdef}
\noindent
\begin{enumerate}
\item For us a \textit{stack} is a pseudo-functor $\Sch^{\op} \to \Cat$ which satisfies effective descent with respect to the fpqc topology. Together with natural transformations and modifications, we obtain the $2$-category of stacks, denoted by $\Stack$. We do not require that stacks factor through groupoids $\Gpd$.
\item Let $\C$ be a cocomplete tensor category. Its \emph{spectrum} $\SPEC(\C)$ is a stack defined by
\[\SPEC(\C)(X) = \Hom_{c\otimes}\bigl(\C,\Q(X)\bigr)\]
for schemes $X$. If $X \to Y$ is a morphism of schemes, its pullback functor $\Q(Y) \to \Q(X)$ induces a functor $\SPEC(\C)(Y) \to \SPEC(\C)(X)$. Then $\SPEC(\C)$ is a stack basically because $\Q(-)$ is a stack by descent theory for quasi-coherent modules (\cite[4.23]{Vis05}). This construction provides us with a $2$-functor
\[\SPEC : \Cat_{c\otimes}^{\op} \to \Stack.\]
\item Let $F$ be a stack. The cocomplete tensor category of quasi-coherent modules $\Q(F)$ is defined as $\Hom_{\Stack}(F,\Q(-))$ with pointwise defined tensor products and colimits. This means that a quasi-coherent module $M$ on $F$ is given by functors $M_X : F(X) \to \Q(X)$ for every scheme $X$, which are compatible with base change, which means that there are compatible isomorphisms $f^* \circ M_X \cong M_Y \circ F(f)$ for morphisms $f : Y \to X$. These isomorphisms also belong to the data. We obtain a $2$-functor
\[\Q : \Stack \to \Cat_{c\otimes}^{\op}.\]
\end{enumerate}
\end{defi}

\clearpage

\begin{rem} \noindent
\begin{enumerate}
\item If $F$ is an algebraic stack, then our definition of quasi-coherent modules on $F$ coincides with the usual one (\cite[7.18]{Vis89}). 
\item The definition of $\SPEC$ has also appeared in the independent work \cite[Section 2]{Liu12} (including a very detailed proof that $\SPEC(\C)$ is, indeed, a stack). The current proof of \cite[Theorem 3.4]{Liu12} has a serious gap.
\end{enumerate}
\end{rem}
 
I have learned the following adjunction from David Ben-Zvi. It is a categorification of the well-known adjunction between commutative rings and schemes \cite[Proposition 1.6.3]{EGAI}.
  
\begin{prop}[Adjunction between $\SPEC$ and $\Q$] \label{adjunction}
If $F$ is a stack and $\C$ is a cocomplete tensor category, then there is a natural equivalence
\[\Hom_{\Stack}(F,\SPEC(C)) \simeq \Hom_{c\otimes}(\C,\Q(F)).\]
Thus, $\SPEC : \Cat_{c\otimes}^{\op} \to \Stack $ is right adjoint to $\Q : \Stack \to \Cat_{c\otimes}^{\op}$.
\end{prop}
 
\begin{proof}
This is entirely formal. For a stack $F$ and a scheme $X$, the component functor
\[F(X) \to \Hom_{c\otimes}(\Q(F),\Q(X))\]
of the unit $\eta_F : F \to \SPEC(\Q(F))$ is defined to be the obvious evaluation which comes from the definition of $\Q(F)$.

The counit $\varepsilon_\C : \C \to \Q(\SPEC(\C))$ is also given by evaluation, i.e.
\[\varepsilon_\C(T)_X : \Hom_{c\otimes}(\C,\Q(X)) \to \Q(X)\]
evaluates at $T \in \C$.
\end{proof}

\begin{defi}[Tensorial stacks and stacky tensor categories] \label{tensorial-stacky}
A stack $F$ is called \emph{tensorial} if the unit
\[\eta_F : F \to \SPEC(\Q(F))\]
is an equivalence, i.e. for every scheme $Y$ we have an equivalence
\[F(Y) \simeq \Hom_{c\otimes}(\Q(F),\Q(Y)).\]
A cocomplete tensor category $\C$ is called \textit{stacky} if the counit
\[\varepsilon_\C : \C \to \Q(\SPEC(\C))\]
is an equivalence.
\end{defi}
 
For example, a scheme $X$ is tensorial if for every scheme $Y$ the functor
\[\Hom(Y,X) \to \Hom_{c\otimes}\bigl(\Q(X),\Q(Y)\bigr),~ f \mapsto f^*\]
is an equivalence of categories. This forces $\Hom_{c\otimes}(\Q(X),\Q(Y))$ to be essentially discrete. We will say more about tensoriality in \autoref{tensoriality}.
  
Every adjunction restricts to an equivalence between its fixed points. Thus:

\begin{prop}
The $2$-functors $\SPEC$ and $\Q$ yield  an anti-equivalence of $2$-categories between tensorial stacks and stacky cocomplete tensor categories.
\end{prop}

\begin{rem}[Analogies]
This can be seen as a categorification of the anti-equivalence between affine schemes and commutative rings. It connects algebraic geometry, which is commutative algebra locally, with categorified or global commutative algebra. We will explain later what this means in detail. Notice that there is a striking analogy to functional analysis: We have a functor $C(-,\mathds{C})$ from topological spaces to the dual category of (involutive) $\mathds{C}$-algebras which is left adjoint to the Gelfand spectrum $\Hom(-,\mathds{C})$. The Theorem of Gelfand-Naimark states that this adjunction restricts to an equivalence of categories between compact Hausdorff spaces and the dual category of unital commutative $C^*$-algebras. Back to our setting, we may imagine stacks as $2$-geometric objects whose $2$-regular functions are precisely quasi-coherent modules. Tensorial stacks are those stacks which are determined by their $2$-semiring of $2$-regular functions. 
\end{rem}

\begin{rem}[Size issues]
The reader may notice that our definitions of $\Q(F)$ for a general stack $F$ and of $\SPEC(\C)$ for a general cocomplete tensor category $\C$ might cause set-theoretic problems, since they are outside of some Grothendieck universe fixed beforehand. However, the definitions of tensorial stacks and stacky cocomplete tensor categories including their anti-equivalence still make sense.
\end{rem}

\begin{rem}[Stacky tensor categories]
There are lots of examples of cocomplete tensor categories which are not stacky, since there are quite a few properties of quasi-coherent modules which do not hold in all tensor categories. An example is that every epimorphism $\O \to \O$ is already an isomorphism (see \autoref{jim}). Unfortunately we do not know any classification of stacky tensor categories. At least Sch\"appi (\cite{Sch12a}) has proven that the anti-equivalence above restricts to an anti-equivalence between Adams stacks and so-called weakly Tannakian categories. The latter have an intrinsic characterization in characteristic zero (\cite{Sch13}). This is a generalization of the classical Tannaka duality (\cite{Del90}).
\end{rem}

\begin{rem}[Artin's criteria]
The spectrum $\SPEC(\C)$ of a cocomplete tensor category $\C$ should only be considered to be a geometric object when it is an algebraic stack. Artin has found criteria for a stack to be algebraic (\cite[Theorem 5.3]{Art74}). Recently Hall and Rydh have found a refinement of these criteria (\cite[Main Theorem]{Hall13}). Let us write down what they mean for our stack $X=\SPEC(\C)$  defined by $X(Y) = \Hom_{c\otimes/R}(\C,\Q(Y))$. For brevity we will just say algebras when we mean commutative algebras.

Let $R$ be an excellent ring and let $\C$ be a cocomplete $R$-linear tensor category. Then the stack $\SPEC(\C)$ is an algebraic stack, locally of finite presentation over $R$, if and only if the following conditions are satisfied:
\begin{itemize}
\item For every $R$-algebra $A$ the category $\Hom_{c\otimes/R}(\C,\M(A))$ is actually a groupoid.
\item For every directed diagram of $R$-algebras $\{A_i\}$ the canonical functor
\[\colim_i \Hom_{c\otimes/R}(\C,\M(A_i)) \to \Hom_{c\otimes/R}(\C,\M(\colim_i A_i))\]
is an equivalence of groupoids.
\item For homomorphisms $A' \to A \leftarrow B$ of local artinian $R$-algebras of finite type, where $A' \to A$ is surjective and $B \to A$ is an isomorphism on residue fields, the natural functor from $\Hom_{c\otimes}(\C,\M(A' \times_A B))$ to
\[\Hom_{c\otimes}(\C,\M(A')) \times_{\Hom_{c\otimes}(\C,\M(A))} \Hom_{c\otimes}(\C,\M(B))\]
is an equivalence of groupoids. 
\item For every complete local noetherian $R$-algebra $(A,\mathfrak{m})$ the canonical functor
\[\Hom_{c\otimes}(\C,\M(A)) \to \lim_n \Hom_{c\otimes}(\C,\M(A/\mathfrak{m}^n))\]
is an equivalence of groupoids.
\item Automorphisms, deformations and obstructions for $\SPEC(\C)$ are bounded, constructible and Zariski-local.
\end{itemize}
It would take too long to elaborate here what the last condition means in terms of $\C$. Let us explain for example what boundedness of deformations and automorphisms means. Let $A$ be a finitely generated $R$-algebra without zero divisors, let $s : \C \to \M(A)$ be an $R$-linear cocontinuous tensor functor. Define a category $\mathsf{Def}_\C(s)$ whose objects are $R$-linear cocontinuous tensor functors $\overline{s} : \C \to \M(A[\e]/\e^2)$ with an isomorphism $\overline{s} \bmod \e \cong s$. The set of isomorphism classes of this category is actually an $A$-module. This is required to be coherent. Besides, the $A$-module of automorphisms in $\mathsf{Def}_\C(s)$ of the trivial extension $s[\e]$ is required to be coherent.
\end{rem}



\clearpage
\chapter{Commutative algebra in a cocomplete tensor category} \label{Comm}

\section{Algebras and modules} \label{algmod}
 
We can develope commutative algebra internal to an arbitrary (cocomplete) tensor category. Let us start with algebras and modules over them. The following definitions are well-known (\cite[I.6]{SR72}):

\begin{defi}[Algebras] \label{algebradef}
Let $\C$ be a tensor category.
\begin{enumerate}
\item Recall that an \emph{algebra object} in $\C$ is an object $A \in \C$ together with two morphisms $e : 1 \to A$ (unit) and $m : A \otimes A \to A$ (multiplication) such that $m$ is associative (i.e. $m \circ (A \otimes m) = m \circ (m \otimes A)$ as morphisms $A \otimes A \otimes A \to A$) and $e$ is a unit for $m$ (i.e. $m \circ (A \otimes e) = m \circ (e \otimes A) = \id_A$). Formally it is a triple $(A,e,m)$, but we will often abuse notation and just write $A$.
\item We call $A$ \emph{commutative} if $m \circ S_{A,A} = m$.
\item There is an obvious notion of homomorphisms of algebras. We obtain the category of algebras $\Alg(\C)$ and the full subcategory $\CAlg(\C)$ of commutative algebras.
\item If we do not require the existence of a unit $e$, we speak of non-unital (commutative) algebras.
\end{enumerate}
\end{defi}

\begin{rem}
Many texts use the terminology of monoid objects. We have not adopted this usage since we turn our attention mainly to linear tensor categories, where one is more used to the terminology of algebras.
\end{rem}

\begin{ex}
An algebra object in $\Set$ is a monoid in the usual sense. Algebras in $\M(R)$ are $R$-algebras in the usual sense. More generally, if $X$ is a scheme, then commutative algebras in $\Q(X)$ are quasi-coherent algebras on $X$. Algebras in the tensor category of chain complexes (\autoref{chain}) are also known as differential-graded algebras.
\end{ex}

\begin{ex}[Tensor algebra] \label{tensoralgebra}
Let $\C$ be a cocomplete tensor category. Then the forgetful functor $\Alg(\C) \to \C$ has a left adjoint: It maps an object $M \in \C$ to the \emph{tensor algebra} $T(M)$ over $M$ (\cite[VII.3. Theorem 2]{ML98}). The underlying object is
\[T(M) = \bigoplus_{n=0}^{\infty} M^{\otimes n}.\]
The unit is the inclusion of $M^{\otimes 0} = 1$. The multiplication is induced by the isomorphisms $M^{\otimes p} \otimes M^{\otimes q} \to M^{\otimes p+q}$.
For $\C=\M(R)$ this is the usual tensor algebra. For $\C=\Set$ we recover the free monoid of words over an alphabet $M$. In a quantale (see \autoref{quantale-def}) we have $T(M) = \sup_{n \to \infty} M^n$.
\end{ex}

\begin{rem}[Constructions with algebras] \label{algebra-cons} \noindent
\begin{itemize}
\item The initial algebra is given by $1=T(0)$.
\item If $(A,e,m)$ is an algebra in $\C$, then its opposite algebra is defined by $(A,e,m)^{\op} := (A,e,m \circ S_{A,A})$.
\item If $(A,e,m),(B,e',m')$ are (commutative) algebras in $\C$, then their tensor product $(A \otimes B,e \otimes e', (m \otimes m') \circ (A \otimes S_{B,A} \otimes B))$ is also a (commutative) algebra in $\C$.
\item This tensor product gives the coproduct in $\CAlg(\C)$. In fact, $\CAlg(\C)$ is cocomplete (\cite[Proposition 1.2.14]{Mar09}).
\end{itemize}
\end{rem}

\begin{defi}[Center]
Let $(A,e,m)$ be an algebra in a closed tensor category $\C$ with equalizers. Then we have two morphisms $A \to \HOM(A,A)$, one corresponding to the multiplication of $A$ and one corresponding to the multiplication of $A^{\op}$. We define $Z(A)$ to be the equalizer of these two morphisms. This turns out to be a subalgebra of $A$, the \emph{center} of $A$.
\end{defi}
 
\begin{defi}[Modules and bimodules] \label{moduldef}
Let $(A,e,m)$ be an algebra in a tensor category $\C$.
\begin{enumerate}
\item We define a \emph{right $A$-module} to be an object $M \in \C$ equipped with a morphism $r : M \otimes A \to M$ such that $r \circ (M \otimes e)=\id_M$ as well as $r \circ (M \otimes m) = r \circ (r \otimes A)$ as morphisms $M \otimes A \otimes A \to M$. There is an obvious notion of homomorphisms of right $A$-modules. The category of right $A$-modules is denoted by $\M(A)$. For example, $A$ is a right $A$-module with $r=m$.
\item If $\C$ is closed, $r$ may be identified with its dual $\check{r} : M \to \HOM(A,M)$ and we may also define modules in terms of this datum.
\item It is clear how to define the category of \emph{left $A$-modules}. If $\C$ is closed, a left module action of $A$ on $M$ may be also described as a homomorphism of algebras $A \to \HOM(M,M)$ in $\C$. If $A$ is commutative, then right modules can be made into left modules and vice versa, so we just call them \emph{modules}.
\item More generally, if $(A,B)$ is a pair of algebras in $\C$, we define an \emph{$(A,B)$-bimodule} in $\C$ as an object $M$ with a morphism $s : A \otimes M \otimes B \to M$ with the property that $M$ becomes both a left $A$-module via $s \circ (A \otimes M \otimes e_B) : A \otimes M \to M$ and a right $B$-module via $s \circ (e_A \otimes M \otimes B) : M \otimes B \to M$. We obtain the category of $(A,B)$-bimodules. It is isomorphic to the category of right $A^{\op} \otimes B$-modules. If $A$ is commutative, we usually consider every $A$-module also as an $(A,A)$-bimodule in a canonical way.
\end{enumerate}
\end{defi}

\begin{ex}[Free modules]
The forgetful functor $\M(A) \to \C$ has a left adjoint which maps $X \in \C$ to the right $A$-module $(X \otimes A, X \otimes m)$, the \emph{free $A$-module} on $X$ (\cite[Proposition 1.3.3]{Mar09}).
\end{ex}

\begin{defi}[Tensor product of bimodules]
Let $\C$ be a cocomplete tensor category and let $A,B,C$ be algebras in $\C$. If $M$ is an $(A,B)$-bimodule and $N$ is a $(B,C)$-bimodule in $\C$, we define $M \otimes_B N$ to be the following $(A,C)$-bimodule:\\
The underlying object of $M \otimes_B N$ is the coequalizer of the two obvious morphisms $M \otimes B \otimes N \rightrightarrows M \otimes N$. Roughly, $M \otimes_B N$ results from $M \otimes N$ by identifying the two $B$-actions. The actions of $A$ on $M$ and of $C$ on $N$ induce a  morphism $A \otimes M \otimes N \otimes C \to M \otimes N \twoheadrightarrow M \otimes_B N$ which coequalizes the two morphisms from $A \otimes M \otimes B \otimes N \otimes C$, hence it lifts to a morphism $A \otimes (M \otimes_B N) \otimes C \to M \otimes_B N$. This defines the $(A,C)$-bimodule structure on $M \otimes_B N$. In particular, if $A$ is commutative and $M,N$ are two right $A$-modules, we obtain a right $A$-module $M \otimes_A N$. In fact, we have the following well-known result:
\end{defi}

\begin{prop}[The tensor category of modules] \label{modA}
Let $\C$ be a cocomplete tensor category and $A$ be a commutative algebra in $\C$, then $\M(A)$ becomes a cocomplete tensor category with unit $A$ and tensor product $\otimes_A$. Colimits are created by the forgetful functor $\M(A) \to \C$. Moreover, its left adjoint $- \otimes A : \C \to \M(A)$ becomes a cocontinuous tensor functor. 
\end{prop}

\begin{proof}
The proof that $\M(A)$ becomes a tensor category is a straightforward generalization of the known case $\C=\Ab$ (\cite[Proposition 1.2.15]{Mar09}). For example, in order to prove $M \otimes_A A \cong M$, one checks that the right module action $M \otimes A \to M$ lifts to a morphism $M \otimes_A A \to M$, conversely that the unit of $A$ induces a morphism $M \to M \otimes A \twoheadrightarrow M \otimes_A A$ and that these morphisms are inverse homomorphisms of $A$-modules.

If $(M_i)_{i}$ is a diagram of $A$-modules, then the colimit of the underlying objects $\colim_i M_i$ in $\C$ carries the structure of an $A$-module via
\[(\colim_i M_i) \otimes A \cong \colim_i (M_i \otimes A) \to \colim_i M_i.\]
It is the unique one which makes the colimit inclusions $A$-module homomorphisms. It is easy to verify that $\colim_i M_i$ is then also the colimit in $\M(A)$ (\cite[Proposition 1.2.14]{Mar09}).

Let us show that the tensor product commutes with colimits in $\M(A)$. If $M$ is some $A$-module and $(N_i)_{i}$ is a diagram of $A$-modules, then their colimit $\colim_i N_i$ is constructed as above. It follows that $M \otimes_A (\colim_i N_i)$ is the coequalizer of the two evident morphisms
\[M \otimes A \otimes (\colim_i N_i) \rightrightarrows M \otimes \colim_i N_i\]
which (since $\C$ is a cocomplete tensor category) identify with the colimit of the evident morphisms
\[M \otimes A \otimes N_i \rightrightarrows M \otimes N_i.\]
Since colimits commute with colimits, we see that, in fact,
\[M \otimes_A (\colim_i N_i) \cong \colim_i (M \otimes_A N_i).\]
Thus, $\M(A)$ is a cocomplete tensor category. Finally, for $S,T \in \C$ there is a canonical isomorphism
\[(S \otimes A) \otimes_A (T \otimes A) \cong (S \otimes T) \otimes A\]
of $A$-modules, basically because on the left hand side the two factors of $A$ are identified with each other. We can write down two inverse isomorphisms explicitly, given in element notation by $(s \otimes a) \otimes (t \otimes b) \mapsto (s \otimes t) \otimes (a \cdot b)$ resp. $(s \otimes a) \otimes (t \otimes 1) = (s \otimes 1) \otimes (t \otimes a) \mapsfrom (s \otimes t) \otimes a$.
\end{proof}

\begin{rem}[Restriction and extension of scalars]
Given a homomorphism of algebras $A \to B$ we obtain a forgetful functor $\M(B) \to \M(A)$, $M \mapsto M|_A$ (restriction of scalars). It is left adjoint to $\M(A) \to \M(B)$, $N \mapsto N \otimes_A B$ (extension of scalars).
\end{rem}

\begin{rem}[Amitsur complex] \label{Amitsur}
If $A$ is an algebra in a linear tensor category $\C$, then we have the \emph{Amitsur complex}
\[0 \to 1 \to A \to A^{\otimes 2} \to A^{\otimes 3} \to \dotsc\]
where the differentials are alternating sums of insertions of the unit of $A$. That is, $A^{\otimes n} \to A^{\otimes n+1}$ maps $a_1 \otimes \dotsc \otimes a_n$ in element notation to
\[\sum_{k=0}^{n} (-1)^k a_1 \otimes \dotsc \otimes a_k \otimes 1 \otimes a_{k+1} \otimes \dotsc \otimes a_n.\]
More generally, we may tensor the Amitsur complex with any object $M \in \C$:
\[0 \to M \to M \otimes A \to M \otimes A^{\otimes 2} \to M \otimes A^{\otimes 3} \to \dotsc\]
\end{rem}

\begin{prop}[Rank is unique]
Let $\C$ be a cocomplete $R$-linear tensor category. Let $A$ be a commutative algebra in $\C$ with $A \neq 0$ such that $A^{\oplus n} \cong A^{\oplus m}$ as $A$-modules for some $n,m \in \N$. Then $n=m$. If $\C$ is locally finitely presentable, the same holds for all cardinals $n,m$.
\end{prop}

\begin{proof}
By considering $\M(A)$ instead of $\C$ we may assume that $A=\O$, the unit of $\C$. We know that $S:=\End(\O)$ is a commutative $R$-algebra such that $S \neq 0$. Applying $\Hom(\O,-) : \C \to \M(S)$ to the given isomorphism $\O^{\oplus n} \cong \O^{\oplus m}$, we obtain an isomorphism of $S$-modules $S^{\oplus n} \cong S^{\oplus m}$. It is known that this implies $n=m$.
\end{proof}

\begin{rem}
If $\C$ is not linear, the rank doesn't have to be unique. For example, in any quantale we have $1 \cong \sup(1,1) = 1 \oplus 1$. \marginpar{I've added this remark.}
\end{rem}

\begin{prop}[Lagrange's Theorem] \label{lagrange}
Let $\C$ be a tensor category. Let $A \to B$ be a homomorphism of algebras in $\C$. Let $M$ be a free right $B$-module. Assume that $B|_A$ is a free right $A$-module. Then the right $A$-module $M|_A$ is also free. Specifically, if $M \cong T \otimes B$ in $\M(B)$ and $B|_A \cong S \otimes A$ in $\M(A)$, then $M|_A \cong (T \otimes S) \otimes A$.
\end{prop}

\begin{proof}
We have $M|_A \cong (T \otimes B)|_A  \cong T \otimes B|_A \cong T \otimes (S \otimes A) \cong (T \otimes S) \otimes A$.
\end{proof}

\begin{ex}
If we apply \autoref{lagrange} to $\C=\Set$, we obtain (a generalization of) Lagrange's Theorem in group theory. If we apply \autoref{lagrange} to $\C=\M(R)$, we obtain (a generalization of) the multiplicativity formula for degrees of field extensions. The language of tensor categories unifies both results.
\end{ex}

\begin{defi}[$A$-algebras]
If $\C$ is a cocomplete tensor category and $A$ is a commutative algebra in $\C$, then an \emph{$A$-algebra} is an algebra in the cocomplete tensor category $\M(A)$. Equivalently, an $A$-algebra is an algebra $B$ in $\C$ equipped with a morphism $A \to Z(B)$ of algebras (if the center exists), i.e. a morphism of algebras $f : A \to B$ with the property that $m_B \circ (f \otimes B) = m_B \circ S_{B,B} \circ (f \otimes B)$. Thus, a commutative $A$-algebra is equivalently a commutative algebra $B$ in $\C$ equipped with a homomorphism $A \to B$ (\cite[Proposition 1.2.15]{Mar09}).
\end{defi}
  
\begin{rem}[Homomorphism modules]
With the notations above, if $\C$ is closed and has equalizers, then $\M(A)$ is also closed and has equalizers. Specifically, for $A$-modules $M,N$ we have that their internal hom $\HOM_A(M,N)$ is the equalizer of the morphism
\[\HOM(M,N) \to \HOM(M,\HOM(A,N))\]
which is induced by the (dualized) action of $A$ on $N$ and the morphism
\[\HOM(M,N) \to \HOM(M \otimes A,N) \cong \HOM(M,\HOM(A,N))\]
which is induced by the action of $A$ on $M$. Again this is a straightforward generalization of the known case $\C=\Ab$ (\cite[Proposition 1.2.17]{Mar09}).
\end{rem}

Other types of algebras can also be defined internal to a tensor category:

\begin{defi}[Hopf algebras] \label{hopf}
If $\C$ is a tensor category, we may define \emph{coalgebras} in $\C$ as algebras in $\C^{\op}$, \emph{bialgebras} as coalgebras in the tensor category of algebras in $\C$ (\cite{Por08a}) and \emph{Hopf algebras} as bialgebras $A$ equipped with a morphism $S : A \to A$ in $\C$ such that the usual diagram commutes (\cite{Por08b}):
\[\xymatrix{& A \otimes A \ar[rr]^{S \otimes A} && A \otimes A \ar[dr] & \\ A \ar[ur] \ar[rr] \ar[dr] &&  1 \ar[rr] && A \\ & A \otimes A \ar[rr]^{A \otimes S} && A \otimes A \ar[ur] & }\]
\end{defi}

\begin{ex}
Hopf algebras in $\M(R)$ are Hopf algebras over $R$ in the usual sense, whereas Hopf algebras in $\Set$ are just groups. Therefore, from the tensor categorical point of view these two concepts are really the same. An easy but useful observation is that Hopf algebras are preserved by tensor functors. If we apply this to the unique cocontinuous tensor functor $\Set \to \C$ and some group $G$ (in $\Set$), we obtain a Hopf algebra $\O_\C[G]$ in $\C$, a generalization of the usual group algebra. The underlying object is a $G$-indexed direct sum of copies of $\O_\C$. The multiplication is induced by the one of $G$ and likewise the comultiplication extends the one of $G$ itself, i.e. $G \to G \times G$, $g \mapsto (g,g)$.
\end{ex}

\begin{rem}[Lie algebras]
In a linear tensor category $\C$, we may define \emph{Lie algebras} in $\C$ as objects $L$ equipped with a morphism $[-,-] : L \otimes L \to L$ (Lie bracket) such that, in element notation, $[a,b]=-[b,a]$ (antisymmetry) and $[a,[b,c]]+[c,[a,b]]+[b,[c,a]]=0$ (Jacobi identity). For more details we refer to \cite{GV13}. We can also define modules over (or representations of) $L$ as objects $V$ together with a morphism $\omega : L \otimes V \to V$ such that, in element notation, $\omega([x,y],v) = \omega(x,\omega(y,v)) - \omega(y,\omega(x,v))$. More formally, it is required that $\omega \circ ([-,-] \otimes V) = \omega \circ (L \otimes \omega) \circ (\id - S_{\L,\L} \otimes V)$ as morphisms $L \otimes L \otimes V \to V$.
\end{rem}

\section{Ideals and affine schemes} \label{affineschemes}

In this section let $\C$ be a locally presentable tensor category. We develope some basic ideal theory inside of $\C$ and then define (affine) schemes relative to $\C$. See \cite{Mar09} and \cite{ToVa09} for a similar but more abstract account. See \cite{Dei05} for the case $\C=\Set$, which leads to monoid schemes. For $\C=\Ab$ we retrieve schemes (\cite{EGAI}). For technical reasons we assume that $\C$ is \emph{balanced}, i.e. that every epimorphism, which is a monomorphism, is already an isomorphism.
 
\begin{defi}[Ideals]
Let $A$ be a commutative algebra in $\C$. An \emph{ideal} is a subobject of $A$ in $\M(A)$. Equivalently, it is an object $I$ in $\C$ equipped with a monomorphism $I \hookrightarrow A$ such that $I \otimes A \to A \otimes A \to A$ factors through $I \hookrightarrow A$. We obtain the partial order of ideals of $A$. It has a least element $0$ (zero ideal) and a greatest element $A$ (unit ideal).
\end{defi}

This generalizes various specific notions of ideals. 

\begin{ex}
A commutative algebra in the cartesian tensor category $\Set$ is just a commutative monoid $A$ and an ideal $I \subseteq A$ is a subset such that $I \cdot A \subseteq A$. The zero ideal is here $\emptyset$. A commutative algebra in $\Ab$ is just a commutative ring $A$ and an ideal $I \subseteq A$ is an additive subgroup such that $I \cdot A \subseteq A$. More generally, let $X$ be a scheme. Then a commutative algebra in $\Q(X)$ is just a quasi-coherent algebra $A$ on $X$ and an ideal $I \subseteq A$ is a quasi-coherent submodule such that $I \cdot A \subseteq A$, i.e. what is usually called a quasi-coherent ideal.
\end{ex}

\begin{defi}[Ideal sum]
For a family of ideals $(I_i)$ the induced morphism $\bigoplus_i I_i \to A$ factors as $\bigoplus_i I_i \twoheadrightarrow \sum_i I_i \hookrightarrow A$, where the first arrow is a regular epimorphism and the second is a monomorphism. We claim that $\sum_i I_i$ is an ideal of $A$, called the \emph{sum} of the ideals $I_i$.
\end{defi}

In fact, let us choose a pullback square
\[\xymatrix{P \ar[r] \ar[d] & (\sum_i I_i) \otimes A \ar[d]  \\ \sum_i I_i \ar[r] & A.}\]
Since $\sum_i I_i \to A$ is a monomorphism, the same is true for $P \to (\sum_i I_i) \otimes A$. We claim that it is an epimorphism, hence an isomorphism by our assumption, so that $(\sum_i I_i) \otimes A \to A$ factors through $\sum_i I_i$ as desired. The claim follows from the following commutative diagram:
\[\xymatrix{\bigoplus_i (I_i \otimes A) \ar[rr]^{\cong} \ar[dd] \ar@{..>}[dr] && (\bigoplus_i I_i) \otimes A \ar@{->>}[d] \\ &  \ar[r] \ar[d] P & (\sum_i I_i) \otimes A \ar[d]  \\ \bigoplus_i I_i \ar@{->>}[r] & \sum_i I_i \ar[r] & A.}\]
One checks that $\sum_i I_i$ is indeed the smallest ideal of $A$ which contains all $I_i$.

\begin{defi}[Ideal product]
If $I,J$ are ideals of $A$, the induced morphism $I \otimes J \to A \otimes A \to A$ factors as $I \otimes J \twoheadrightarrow I \cdot J \hookrightarrow A$, where $I \otimes J \twoheadrightarrow I \cdot J$ is a regular epimorphism and $I \cdot J \hookrightarrow A$ is a monomorphism. As above one checks that $I \cdot J$ is an ideal of $A$, called the \emph{product} of $I$ and $J$, using the following commutative diagram:
\end{defi}

\[\xymatrix{(I \otimes J) \otimes A \ar@{->>}@/^2pc/[drr] \ar[dd]^{\cong} \ar@{..>}[dr] \ar[dr] &&  \\ &  \ar[r] \ar[d] P & (I \cdot J) \otimes A  \ar[d] \\ I \otimes (J \otimes A) \ar[r] & I \otimes J \ar[r] & A.}\]

With the same trick one checks that $I \cdot J$ is the smallest ideal $K \subseteq A$ with the property that $I \otimes J \to A \otimes A \to A$ factors through $K$.
 
\begin{lemma}
Let $A$ be a commutative algebra in $\C$.
\begin{enumerate}
\item The ideal product is associative and commutative. It has $A$ as a neutral element.
\item The ideal product distributes over arbitrary sums: For a family of ideals $I_i \subseteq A$ and another ideal $J \subseteq A$ the canonical morphism
\[\sum_i (I_i \cdot J) \to \left(\sum_i I_i\right) \cdot J\]
is an isomorphism.
\item Hence, the ideals of $A$ constitute a quantale.
\end{enumerate}
\end{lemma}

\begin{proof}
1. This follows easily from the isomorphisms $I \otimes (J \otimes K) \cong (I \otimes J) \otimes K$, $I \otimes J \cong J \otimes I$ and $I \otimes_A A \cong I$ for ideals $I,J,K \subseteq A$. 2. This follows easily from $\bigoplus_i (I_i \otimes J) \cong (\bigoplus_i I_i) \otimes J$. 3. This follows from 1. and 2.
\end{proof}

\begin{defi}[Prime ideals]
Let $A$ be a commutative algebra in $\C$. An ideal $\mathfrak{p} \subseteq A$ is called a \emph{prime ideal} if $\mathfrak{p} \neq A$ and if for all ideals $I,J \subseteq A$ we have
\[I \cdot J \subseteq \mathfrak{p} ~\Longrightarrow ~ I \subseteq \mathfrak{p} ~ \vee ~ J \subseteq \mathfrak{p}.\]
We obtain the set $\Spec(A)$ of prime ideals of $A$.
\end{defi}

\begin{prop}[Existence of prime ideals]
If $A \neq 0$ and $A$ is finitely presentable, then $\Spec(A) \neq \emptyset$.
\end{prop}

\begin{proof}
We first observe that $A$ has a maximal proper ideal by Zorn's Lemma: In fact, if $\{I_i\}_{i \in I}$ is a non-empty chain of proper ideals of $A$, then $\sum_{i \in I} I_i$ is also a proper ideal of $A$. Otherwise, since $A$ is finitely presentable, we would have $\sum_{i \in E} I_i = A$ for some finite subset $E \subseteq I$ and then $I_i = A$ for some $i \in E$, a contradiction. Hence, there is a maximal proper ideal $\m \subseteq A$. We claim that $\m$ is a prime ideal. The usual calculation works: If $I \not\subseteq \m$ and $J \not\subseteq \m$, but $I \cdot J \subseteq \m$, then $I+\m$ and $J+\m$ are larger than $\m$, hence must be equal to the unit ideal $A$.  But then
\[A=A \cdot A = (I+\m) \cdot (J+\m) \subseteq I \cdot J + I \cdot \m + \m \cdot J + \m \cdot \m \subseteq \m\]
is a contradiction.
\end{proof}

\begin{defi}[Zariski topology]
For an ideal $I \subseteq A$ we define
\[V(I) = \{\mathfrak{p} \in \Spec(A) : I \subseteq \mathfrak{p}\}.\]
Because of the evident properties
\begin{itemize}
\item $V(0)=\Spec(A)$, $V(A)=\emptyset$
\item $V(I \cdot J) = V(I) \cap V(J)$
\item $V(\sum_i I_i) = \cap_i V(I_i)$
\end{itemize}
these sets constitute the closed sets of a topology on $\Spec(A)$, which we call the \emph{Zariski topology}.
\end{defi}

Next, we would like to define the structure sheaf. The functorial approach in \cite{ToVa09} is very abstract. Although \cite{Mar09} is more concrete, there it is assumed that $1 \in \C$ is a projective object (\cite[Definition 1.4.18]{Mar09}), which excludes global examples such as $\C=\Q(X)$ for non-affine $X$. But we would like to include these into our theory and therefore choose the following alternative:

If $X$ is a \emph{noetherian} scheme, then we have Deligne's formula \cite[Proposition 6.9.17]{EGAI}: If $M$ is a quasi-coherent module on $X$ and $U \subseteq X$ is an open subset, then the canonical homomorphism
\[\colim_{V(I) \cap U = \emptyset} \Hom_{\O_X}(I,M)  \to \Gamma(U,M)\]
is an isomorphism. Here, $I$ runs over all quasi-coherent ideals of $\O_X$ such that $V(I) \cap U = \emptyset$, or equivalently $I|_U = \O_U$. In particular, this gives an explicit description of the sections of the sheaf $\widetilde{M}$ associated to an $A$-module, where $A$ is a noetherian commutative ring. This motivates the following definition:

\begin{defi}[Associated sheaves]
If $M \in \M(A)$, we define the sheaf $\widetilde{M}$ on $\Spec(A)$ to be the one associated to the presheaf
\[U \mapsto \colim_{V(I) \cap U = \emptyset} \Hom_A(I,M).\]
In particular, we have $\O_{\Spec(A)} := \widetilde{A}$. This is a sheaf of commutative monoids on $\Spec(A)$ (in fact of commutative rings if $\C$ is linear).
\end{defi}
 
\begin{defi}[Schemes]
The space $\Spec(A)$ togethter with its structure sheaf $\O_{\Spec(A)}$ is called the \emph{spectrum} of $A$. It is called an \emph{affine scheme}. A \emph{scheme} with respect to $\C$ is a space with a sheaf of commutative monoids (resp. commutative rings if $\C$ is linear) which is locally isomorphic to an affine scheme. We obtain the category ${\Sch}/\C$ of \emph{schemes relative to $\C$}.
\end{defi}

\begin{ex}
In case of $\C=\Ab$ we obtain the category of usual schemes, for $\C=\M(R)$ the category of $R$-schemes, for $\C=\Set$ the category of monoid schemes and for $\C=\Set_*$ the category of pointed monoid schemes. \marginpar{I've changed the wording so that there is no line break in the formula.}
\end{ex}

As an application, we can prove the following reconstruction result. There we consider $\O_X$ as a commutative algebra in $\Q(X)$.
 
\begin{thm}[Reconstruction] \label{recon}
Let $X$ be a noetherian scheme. Then there is an isomorphism of ringed spaces
\[X \cong \Spec(\O_X).\]
\end{thm}

\begin{proof}
We define a map $i : X \to \Spec(\O_X)$ as follows: If $x \in X$, then $\mathfrak{p}_x \subseteq \O_X$ is the vanishing ideal of $\overline{\{x\}}$. Explicitly, we have for $U \subseteq X$ open
\[\mathfrak{p}_x(U) = \{s \in \O_X(U) : s_x \in \m_x \text{ if } x \in U\}.\]
If $I \subseteq \O_X$ is a quasi-coherent ideal, we have $I \subseteq \mathfrak{p}_x \Leftrightarrow I_x \subseteq \m_x$. Since $\m_x$ is a prime ideal of $\O_{X,x}$, it follows easily that $\mathfrak{p}_x$ is a prime ideal of $\O_X$. Hence, $i$ is well-defined. Since $x$ is uniquely determined by $\overline{\{x\}} = \supp(\O_X/\mathfrak{p}_x)$, we see that $i$ is injective.

For surjectivity, let $I \subseteq \O_X$ be a prime ideal. By considering the closed subscheme of $X$ which is cut out by $I$, we may assume that $I=0$ is a prime ideal (in particular $X \neq \emptyset$) and we want to show that $0=\mathfrak{p}_x$ for some $x \in X$, which means that $X$ is integral with generic point $x$. Since $X$ is noetherian, there is some $n \in \N$ such that $\rad(\O_X)^n = 0$. Since $0$ is a prime ideal, this implies $\rad(\O_X)=0$, i.e. that $X$ is reduced. In order to show that $X$ is irreducible, let $A,B \subseteq X$ be closed subsets with $A \cup B = X$. If $J_1,J_2$ are the corresponding vanishing ideals, it follows $J_1 \cap J_2 = 0$ and hence also $J_1 \cdot J_2 = 0$. Since $0$ is prime, this implies $J_1=0$ or $J_2=0$ and hence $A=X$ or $B=X$. Thus, $i$ is bijective.

Finally, $i$ is a homeomorphism, since for a quasi-coherent ideal $I \subseteq \O_X$ we have
\[V(I) = \{x \in X : I_x \subseteq \m_x\} = \{x \in X : I \subseteq \mathfrak{p}_x\} = i^{-1}(V(I)).\]
That $i$ extends to an isomorphism of the structure sheaves comes out of our construction and Deligne's formula.
\end{proof}

In particular, we can reconstruct every noetherian scheme $X$ from the tensor category $\Q(X)$. Besides, $X$ is actually an \emph{affine} scheme with respect to the ``base'' $\Q(X)$.

\section{Symtrivial objects} \label{symtrivial}

Let $M$ be an object of a tensor category $\C$ and $n \in \N$. Then the symmetric group $\S_n$ acts as follows on $M^{\otimes n}$. For a transposition $\sigma_i = (i~~ i+1)$ we take the symmetry $S_{M,M} : M \otimes M \to M \otimes M$, tensored with $M^{\otimes i-1}$ on the left and $M^{\otimes n-i-1}$ on the right. The group $\S_n$ is generated by $\sigma_1,\dotsc,\sigma_{n-1}$ modulo the relations $\sigma_i^2 = 1$, $\sigma_i \sigma_j = \sigma_j \sigma_i$ ($j \neq i \pm 1$) and $\sigma_i \sigma_{i+1} \sigma_i = \sigma_{i+1} \sigma_i \sigma_{i+1}$ (\cite{CM80}). The Coherence Theorem (\autoref{coherence}) implies that these relations are satisfied by the corresponding symmetries, so that we get, in fact, a homomorphism
\[\S_n \to \Aut(M^{\otimes n}).\]
In the case $\C=\M(R)$ for some commutative ring $R$ a permutation $\sigma$ acts on $M^{\otimes n}$ by
\[m_{\sigma(1)} \otimes \dotsc \otimes m_{\sigma(n)} \mapsto m_1 \otimes \dotsc \otimes m_n.\]
For general $\C$ we will use the same element notation.

The following terminology might be new.
 
\begin{defi}[Symtrivial objects]
Let $\C$ be a tensor category. An object $M \in \C$ is called \emph{symtrivial} if for all $n \in \N$ the action of $\S_n$ on $M^{\otimes n}$ is trivial. Equivalently, $S_{M,M} : M \otimes M \to M \otimes M$ is required to be the identity, i.e. in element notation $a \otimes b = b \otimes a$ for $a,b \in M$.
\end{defi}

\begin{rem}[Closure properties] \label{symtrivial-properties}
Let us list some easy properties of symtrivial objects:
\begin{enumerate}
\item The unit $1$ and the initial object $0$ are always symtrivial.
\item Any tensor functor preserves symtrivial objects.
\item In a finitely cocomplete tensor category, a direct sum $M \oplus N$ is symtrivial if and only if $M \otimes N = 0$ and $M,N$ are symtrivial. In a cocomplete tensor category, the same holds for infinite direct sums.
\item If $M$ is a symtrivial object and $N$ is any object, then we have in element notation $c \otimes a \otimes b = c \otimes b \otimes a$ for $a,b \in M$ and $c \in N$. After twisting with the symmetry $N \otimes M \cong M \otimes N$ we also get $a \otimes c \otimes b = b \otimes c \otimes a$.
\item If $M,N$ are symtrivial, then their tensor product $M \otimes N$ is also symtrivial: Using element notation we calculate in $(M \otimes N) \otimes (M \otimes N)$ that
\[a \otimes b \otimes a' \otimes b' = a' \otimes b \otimes a \otimes b' = a' \otimes b' \otimes a \otimes b.\]
The rigorous proof uses the commutative diagram
\[\xymatrix@C=40pt{(M \otimes N) \otimes (M \otimes N) \ar[r]^{S_{M \otimes N}} \ar[d]^{h} & (M \otimes N) \otimes (M \otimes N) \ar[d]^{h} \\ (M \otimes M) \otimes (N \otimes N) \ar[r]^{S_M \otimes S_N} & (M \otimes M) \otimes (N \otimes N),}\]
where $h$ is the canonical isomorphism which is induced by the symmetry $S_{N,M} : N \otimes M \cong M \otimes N$ and the associativity isomorphisms.
\item The same diagram shows: If $M \otimes N \cong 1$ and $M$ is symtrivial, then  $N$ is also symtrivial.
\item Let $s : M \to M'$ be an epimorphism in a finitely cocomplete tensor category. Then $M'$ is symtrivial if and only if $s(a) \otimes s(b) = s(b) \otimes s(a)$ for $a,b \in M$ in element notation. If $M$ is symtrivial, then $M'$ is symtrivial.
\item In a cocomplete tensor category symtrivial objects are closed under directed colimits.
\end{enumerate}
\end{rem} 
 
\begin{lemma}
Let $X$ be a locally ringed space and $M$ an $\O_X$-module. If every stalk $M_x$ is a cyclic $\O_{X,x}$-module, then $M$ is symtrivial. The converse holds when $M$ is of finite type. For example, invertible modules are symtrivial.
\end{lemma}

\begin{proof}
Clearly $M$ is symtrivial if and only if each $M_x$ is symtrivial. So it suffices to consider the case that $M$ is a module over a local ring $R$. If $M$ is cyclic, then $M$ is a quotient of the symtrivial module $R$, hence symtrivial. If $M$ is symtrivial and of finite type, we claim that $M$ is cyclic. This is clear when $R$ is a field. In the general case, let $k$ be the residue field. Then $M \otimes_R k$ is symtrivial over $k$, hence cyclic. By Nakayama's Lemma $M$ is cyclic.
\end{proof}

The following result gives more examples of symtrivial modules.

\begin{lemma}
Let $\C$ be a tensor category and let $A$ be a commutative algebra in $\C$. The following are equivalent:
\begin{enumerate}
\item $A$ is a symtrivial object of $\C$.
\item We have $a \otimes 1 = 1 \otimes a$ in $A \otimes A$ for all $a \in A$, i.e. the two coproduct inclusions from $A$ to $A \otimes A$ coincide.
\item The unit $1 \to A$ is an epimorphism in $\CAlg(\C)$.
\end{enumerate}
\end{lemma}

\begin{proof}
2. $\Rightarrow$ 1. follows since $a \otimes b = (a \otimes 1)b = (1 \otimes a)b = 1 \otimes ab$ is symmetric in $a,b$. 1. $\Rightarrow$ 2. is trivial and 2. $\Leftrightarrow$ 3. is formal.
\end{proof}

See \cite{Sam68} for examples and the theory of epimorphisms of commutative rings. Will Sawin has classified all symtrivial $R$-modules when $R$ is a Dedekind domain. See \url{http://mathoverflow.net/questions/119689}.
\section{Symmetric and exterior powers} \label{symmetric-power}

We fix a finitely cocomplete tensor category $\C$.

\begin{defi}[Symmetric powers]
Let $f : M^{\otimes n} \to N$ be a morphism in $\C$.
\begin{enumerate}
\item We call $f$ \emph{symmetric} if $f \circ \sigma = f$ for all $\sigma \in \Sigma_n$, i.e. in element notation $f(a_1 \otimes \dotsc \otimes a_n) = f(a_{\sigma(1)} \otimes \dotsc \otimes a_{\sigma(n)})$ for all $a_i \in M$.
\item Assume that $\C$ is linear. We call $f$ is \emph{antisymmetric} if, instead, we have $f \circ \sigma = \sgn(\sigma) f$ for all $\sigma \in \Sigma_n$. In element notation we may write this as $f(a_1 \otimes \dotsc \otimes a_n) = \sgn(\sigma) f(a_{\sigma(1)} \otimes \dotsc \otimes a_{\sigma(n)})$.
\item We obtain two functors $\C \to \Set$ of symmetric resp. antisymmetric morphisms on $M^{\otimes n}$. A representing object is called an $n$-th \emph{symmetric power} resp. $n$-th \emph{antisymmetric power} of $M$ and is denoted by $\Sym^n(M)$ resp. $\ASym^n(M)$.
\end{enumerate}
\end{defi}

\begin{rem}
More generally we may start with a homomorphism of groups $\chi : \S_n \to \Aut_\C(1)$, define $\chi$-symmetric maps and a classifying object $\Sym_\chi^n(M)$, so that $\Sym_1^n(M)= \Sym(M)$ and $\Sym_{\sgn }^n(M)=\ASym^n(M)$.
\end{rem}

\begin{ex}
There is an epimorphism $M^{\otimes n} \to \Sym^n(M)$ which we may write as
\[x_1 \otimes \dotsc \otimes x_n \mapsto x_1 \cdot \dotsc \cdot x_n\]
in element notation, which is symmetric and has the universal property that every symmetric morphism $M^{\otimes n} \to T$ extends uniquely to a morphism $\Sym^n(M) \to T$. This generalizes the usual notion in the case $\C=\M(A)$ for some commutative ring $A$, or more generally $\C=\Q(X)$ for some scheme or algebraic stack $X$. We may apply this definition also to the cartesian category of sets $\C=\Set$ where $\Sym^n(M)$ is the set of \emph{multisets} of $n$ elements in $M$.
\end{ex}
  
\begin{prop}[Properties of symmetric powers] \label{sym-exakt} \noindent
\begin{enumerate}
\item The symmetric and antisymmetric powers exist and provide functors
\[\Sym^n,\,\ASym^n : \C \to \C.\]
\item There is a unique epimorphism
\[\Sym^p(M) \otimes \Sym^q(M) \to \Sym^{p+q}(M)\]
lying over $M^{\otimes p} \otimes M^{\otimes q} \cong M^{\otimes p+q}$, i.e. given in element notation by
\[a_1 \cdot \dotsc \cdot a_p \otimes b_1 \cdot \dotsc \cdot b_q \mapsto a_1 \cdot \dotsc a_p \cdot b_1 \cdot \dotsc \cdot b_q.\]
\item The induced morphism
\[\bigoplus\limits_{p+q=n} \Sym^p(M) \otimes \Sym^q(N) \to \Sym^n(M \oplus N)\]
is an isomorphism.
\item Let $p : N \to P$ be the coequalizer of two morphisms $f,g : M \to N$. Let $n \geq 1$. There are two morphisms $M \otimes \Sym^{n-1}(N) \rightrightarrows \Sym^n(N)$ defined as the composition
\[M \otimes \Sym^{n-1}(N) \rightrightarrows N \otimes \Sym^{n-1}(N) \twoheadrightarrow \Sym^n(N).\]
Then
\[M \otimes \Sym^{n-1}(N) \rightrightarrows \Sym^n(N) \xrightarrow{\Sym^n(p)} \Sym^n(P)\]
is exact.
\item All these statements hold verbatin for $\ASym$ instead of $\Sym$.
\end{enumerate}
\end{prop}

\begin{proof}
1. Define $\Sym^n(M)$ resp. $\ASym^n(M)$ to be the coequalizer of all the endomorphisms $\sigma$ resp. $\sgn(\sigma) \sigma$ of $M^{\otimes n}$, where $\sigma \in \Sigma_n$. The universal property is easily verified.

2. This is because $M^{\otimes n} \otimes M^{\otimes m} \cong M^{\otimes n+m} \twoheadrightarrow \Sym^{n+m}(M)$ is symmetric both on $M^{\otimes n}$ and on $M^{\otimes m}$, so that \autoref{coeq-tensor} may be applied.

3. Using 2. we get for all $p+q=n$ a morphism
\[\Sym^p(M) \otimes \Sym^q(N) \to \Sym^p(M \oplus N) \otimes \Sym^q(M \oplus N) \to \Sym^n(M \oplus N).\]
These induce a morphism
\[\bigoplus_{p+q=n} \Sym^p(M) \otimes \Sym^q(M) \to \Sym^n(M \oplus N).\]
We claim that it is an isomorphism, i.e. that the morphisms to any test object coincide. A morphism on $\Sym^n(M \oplus N)$ corresponds to a symmetric morphism on $(M \oplus N)^{\otimes n}$. This tensor product decomposes by \autoref{binomi} as $\bigoplus_{p+q=n} \binom{n}{p} M^{\otimes p} \otimes N^{\otimes q}$. Thus the morphism corresponds to a family of morphisms on copies of $M^{\otimes p} \otimes N^{\otimes q}$. The symmetry condition means exactly that these morphisms agree on the same copies and that they are symmetric on $M^{\otimes p}$ and $N^{\otimes q}$ separately. Hence, by \autoref{coeq-tensor} they correspond to a morphism on $\bigoplus_{p+q=n} \Sym^p(M) \otimes \Sym^q(N)$. For $\ASym$ a similar argument works.

4. A morphism on $\Sym^n(P)$ corresponds to a symmetric morphism on $P^{\otimes n}$. By \autoref{coeq-tensor} $N^{\otimes n} \to P^{\otimes m}$ is the coequalizer of the $n$ pairs of morphisms $N^{\otimes i} \otimes M \otimes N^{n-i-1} \rightrightarrows N^{\otimes n}$ ($i=0,\dotsc,n-1$), hence a symmetric morphism on $P^{\otimes n}$ corresponds to a symmetric morphism on $N^{\otimes n}$ which coequalizes all the pairs on $N^{\otimes i} \otimes M \otimes N^{n-i-1}$. But the symmetry allows us to reduce this condition to the case $i=0$. Finally, we use that $N^{\otimes n-1} \to \Sym^{n-1}(N)$ is an epimorphism. This proof also works for $\ASym$.
\end{proof}

\begin{lemma}[Symmetric algebra] \label{symalg}
Let $\C$ be a cocomplete tensor category and $M \in \C$. Then $\Sym^0(M)=1$ and the natural epimorphisms (\autoref{sym-exakt}) $\Sym^p(M) \otimes \Sym^q(M) \twoheadrightarrow \Sym^{p+q}(M)$  endow
\[\Sym(M) := \bigoplus_{p=0}^{\infty} \Sym^p(M)\]
with the structure of a commutative algebra in $\C$, which we call the \emph{symmetric algebra} on $M$. In fact, $\Sym : \C \to \CAlg(\C)$ is left adjoint to the forgetful functor.
\end{lemma}

\begin{proof}
The first part is obvious. For the second part, let $A$ be a commutative algebra in $\C$ and $f : M \to A$ be a morphism in $\C$. For every $p \geq 0$, the multiplication of $A$ induces a morphism $A^{\otimes p} \to A$ in $\C$ (resp. for $p=0$ it is the unit). Since $A$ is commutative, it lifts to a morphism $\Sym^p(A) \to A$. Hence, $f$ induces a morphism $\Sym^p(M) \to \Sym^p(A) \to A$ for every $p$ and hence a morphism $\tilde{f} : \Sym(M) \to A$. One checks that $\tilde{f}$ is the unique homomorphism of commutative algebras such that $\tilde{f}|_M = f$. See \cite[Proposition 1.3.1]{Mar09}.
\end{proof}

\begin{ex}[Polynomial algebras]
Let $A$ be a commutative algebra in $\C$ and $M$ some $A$-module. Then we define the $A$-algebra $\Sym_A(M)$ to be the symmetric algebra of $M$ as an object in $\M(A)$. For example, we can construct the polynomial algebra $A[T_1,\dotsc,T_n] := \Sym_A(A^{\oplus n})$. It satisfies the usual universal property: If $B$ is a commutative $A$-algebra and $(e_i : \O \to B)_{1 \leq i \leq n}$ is an $n$-tuple of ``global sections'' of $B$, then there is a unique homomorphism of $A$-algebras $A[T_1,\dotsc,T_n] \to B$ which maps the global sections $T_i : \O \to A[T_1,\dotsc,T_n]$ to $e_i$. If $\C$ is assumed to be abelian, then Hilbert's Basis Theorem holds: If $A$ is noetherian, then $A[T_1,\dotsc,T_n]$ is also noetherian (\cite{Ye90}).
\end{ex}

\begin{rem}[Exterior powers?]
Let $A$ be a commutative ring and $M,N$ be two $A$-modules. Recall that a homomorphism $f : M^{\otimes n} \to N$ is called \emph{alternating} if $f(m_1 \otimes \dotsc \otimes m_n)=0$ for all $m \in M^n$ such that $m_i = m_j$ for some $i \neq j$. A representing object of the alternating homomorphisms on $M^{\otimes n}$ is called the $n$th \emph{exterior power} of $M$ and is denoted by $\Lambda^n M$. The universal alternating homomorphism $M^{\otimes n} \to \Lambda^n(M)$ is denoted by $m_1 \otimes \dotsc \otimes m_n \mapsto m_1 \wedge \dotsc \wedge m_n$.

This construction commutes with localization. Hence, by gluing, we also have an exterior power $\Lambda^n M$ for quasi-coherent modules $M$ on a scheme. Since any alternating homomorphism is antisymmetric, there is a canonical epimorphism $\ASym^n(M) \to \Lambda^n(M)$ over $M^{\otimes n}$.
 
Now, the exterior powers satisfy some properties which are fundamental for commutative algebra and algebraic geometry, but which fail for antisymmetric powers: If $M$ is locally free of rank $n$, then $\Lambda^d M$ is locally free of rank $\binom{n}{d}$. In particular, $\Lambda^d M=0$ for $d>n$ and $\det(M) := \Lambda^n(M)$ is invertible. However, we have $\ASym^n(\O)=\O/2\O$ for $n \geq 2$.

Therefore, it is tempting to define exterior powers in arbitrary cocomplete tensor categories. Unfortunately, to the best knowledge of the author, this is not possible in general. On the other hand, if $2 \in A^*$, then alternating homomorphisms $M^{\otimes n} \to N$ coincide with antisymmetric homomorphisms, so that $\ASym^n(M) \cong \Lambda^n(M)$. This motivates:
\end{rem}
 
\begin{defi}[Exterior powers] \label{exterior-power}
Let $\C$ be a finitely cocomplete $R$-linear tensor category with $2 \in R^*$. For $M \in \C$ and $n \in \N$ we define the $n$th \emph{exterior power} by $\Lambda^n(M) := \ASym^n(M)$. We denote the universal antisymmetric morphism by $\wedge : M^{\otimes n} \to \Lambda^n(M)$ and write it in element notation as
\[x_1 \otimes \dotsc \otimes x_n \mapsto x_1 \wedge \dotsc \wedge x_n\]
\end{defi}
 
For module categories, there seems to be some extra structure or language which makes it possible to define alternating homomorphisms. The following Proposition gives an idea what that might be.

\begin{prop}[Extra structure]
Let $A$ be a commutative ring such that $a^2-a \in 2A$ for all $a \in A$ (for example $A=\Z$ or $A=\F_2$). Let $M \in \M(A)$. For every $n \in \N$ the map of sets
\[f_n : M^n \to \ASym^{n+1}(M),\, (m_1,\dotsc,m_n) \mapsto m_1 \wedge m_1 \wedge \dotsc \wedge m_n\]
is multilinear, hence extends to a homomorphism $\widetilde{f_n} : M^{\otimes n} \to \ASym^{n+1}(M)$. Besides, there is an exact sequence
\[M^{\otimes n} \stackrel{\widetilde{f_n}}{\longrightarrow} \ASym^{n+1}(M) \longrightarrow  \Lambda^{n+1}(M) \longrightarrow 0.\]
\end{prop}

\begin{proof}
Clearly $f_n$ is linear in $m_2,\dotsc,m_n$. For the linearity in $m_1$, we first observe that
\[m_1 \wedge m_1 \wedge \dotsc \wedge m_n = - m_1 \wedge m_1 \wedge \dotsc \wedge m_n\]
by interchanging the two copies of $m_1$. It follows $2A(m_1 \wedge m_1 \wedge \dotsc \wedge m_n)=0$, hence for all $a \in A$:
\begin{eqnarray*}
a m_1 \wedge a m_1 \wedge \dotsc \wedge m_n &=& a^2 (m_1 \wedge m_1 \wedge \dotsc \wedge m_n) \\
&=& a (m_1 \wedge m_1 \wedge \dotsc \wedge m_n)\\
&& + (a^2-a) (m_1 \wedge m_1 \wedge \dotsc \wedge m_n)\\
&=& a (m_1 \wedge m_1 \wedge \dotsc \wedge m_n)
\end{eqnarray*}
We also have for another $m'_1 \in M$:
\begin{eqnarray*}
(m_1+m'_1) \wedge (m_1+m'_1) \wedge \dotsc &=& ~m_1 \wedge m'_1 \wedge \dotsc + m'_1 \wedge m_1 \wedge \dotsc  \\
&& + m_1 \wedge m_1 \wedge \dotsc + m'_1 \wedge m'_1 \wedge \dotsc \\
&=& m_1 \wedge m_1 \wedge \dotsc + m'_1 \wedge m'_1 \wedge \dotsc 
\end{eqnarray*}
This shows that $f_n$ is multilinear. The exactness of the sequence means that an antisymmetric map $g : M^{n+1} \to N$ is alternating if and only if it satisfies $g(m_1,m_1,m_2,\dotsc,m_n)=0$ for all $(m_1,m_2,\dotsc,m_n) \in M^n$, which is obvious.
\end{proof}

\begin{lemma}[Exterior algebra] \label{exterior-algebra}
Let $\C$ be a cocomplete $R$-linear tensor category with $2 \in R^*$ and $M \in \C$. Then $\Lambda^0(M)=1$ and the natural epimorphisms $\wedge : \Lambda^p(M) \otimes \Lambda^q(M) \twoheadrightarrow \Lambda^{p+q}(M)$ (\autoref{sym-exakt}) endow
\[\Lambda(M) := \bigoplus_{p=0}^{\infty} \Lambda^p(M)\]
with the structure of an algebra in $\C$, which we call the \emph{exterior algebra} of $M$. It enjoys the following universal property: There exists a natural morphism
\[i : M \cong \Lambda^1(M) \hookrightarrow \Lambda(M)\]
in $\C$ with the property that $i(m) \wedge i(n)=-i(n) \wedge i(m)$ in element notation. If conversely $A$ is an algebra in $\C$ and $f : M \to A$ is a morphism in $\C$ satisfying $f(m) f(n) = - f(n) f(m)$, then there is a unique homomorphism  $\tilde{f} : \Lambda(M) \to A$ of algebras such that $\tilde{f} \circ i = f$.
\end{lemma}

\begin{proof}
Again the first part is obvious. If $A$ is an algebra in $\C$ and $f : M \to A$ is a morphism in $\C$ with $f(m)  f(n) = - f(n)  f(m)$, this precisely means that the morphism $M^{\otimes p} \to A^{\otimes p} \to A$ is antisymmetric for $p=2$ and then also for every $p \geq 0$, i.e. it lifts to a morphism $\Lambda^p(M) \to A$ and then to a morphism $\tilde{f} : \Lambda(M) \to A$. It is easy to check that it is a homomorphism of algebras, the unique one extending $f$ in degree $1$.
\end{proof}

\begin{rem}[Clifford algebras]
More generally, with the above notations, if $\beta : M \otimes M \to 1$ is a symmetric morphism, there is an algebra $Cl(M,\beta)$ in $\C$ equipped with a morphism $i : M \to Cl(M,\beta)$ in $\C$ such that
\[i(m) \cdot i(n) + i(n) \cdot  i(m) = 2 \beta(m \otimes n) \cdot 1\]
in element notation, and satisfying the corresponding universal property. We call $Cl(M,\beta)$ the \emph{Clifford algebra} of $M$ with respect to $\beta$. It is a suitable quotient of the tensor algebra $T(M)$ (\autoref{tensoralgebra}).\\
If $\C=\M(R)$ for some commutative ring $R$, this coincides with the usual notion. Observe that $Cl(M,0)=\Lambda(M)$.
\end{rem}

\begin{lemma} \label{antipaar}
Let $\C$ be a cocomplete $R$-linear tensor category and consider two antisymmetric morphisms $\omega : V^{\otimes p} \to W_1$, $\eta : V^{\otimes q} \to W_2$. Denote by $\Sigma_{p,q} \subseteq \Sigma_{p+q}$ be the set of $(p,q)$-shuffles. Then, the morphism
\[\omega \wedge \eta : V^{\otimes p+q} \to W_1 \otimes W_2\]
defined by
\[(\omega \wedge \eta)(v_1 \otimes \dotsc \otimes v_{p+q})\]
\[ := \sum_{\sigma \in \Sigma_{p,q}} \sgn(\sigma) \cdot \omega(v_{\sigma(1)} \otimes \dotsc \otimes v_{\sigma(p)}) \otimes \eta(v_{\sigma(p+1)} \otimes \dotsc \otimes v_{\sigma(p+q)})\]
in element notation, is also antisymmetric. 
\end{lemma}

\begin{proof}
Let $1 \leq i<p+q$. If we interchange $v_i$ and $v_{i+1}$ in the expression $(\omega \wedge \eta)(v_1 \otimes \dotsc \otimes v_{p+q})$, we get
\[\sum_{\substack{\sigma \in \Sigma_{p,q} \\ i,i+1 \notin \{\sigma(1),\dotsc,\sigma(p)\}}} \sgn(\sigma) \cdot \omega(v_{\sigma(1)} \otimes \dotsc \otimes v_{\sigma(p)}) \otimes \eta(\underbrace{v_{\sigma(p+1)} \otimes \dotsc \otimes v_{\sigma(p+q)}}_{\text{ with } v_i,v_{i+1} \text{ interchanged}})\]
\[+\sum_{\substack{\sigma \in \Sigma_{p,q} \\ i,i+1 \in \{\sigma(1),\dotsc,\sigma(p)\}}} \sgn(\sigma) \cdot \omega(\underbrace{v_{\sigma(1)} \otimes \dotsc \otimes v_{\sigma(p)}}_{\text{ with } v_i,v_{i+1} \text{ interchanged}}) \otimes \eta(v_{\sigma(p+1)} \otimes \dotsc \otimes v_{\sigma(p+q)})\]
\[+\sum_{\substack{\sigma \in \Sigma_{p,q} \\ i \in \{\sigma(1),\dotsc,\sigma(p)\} \not\ni i+1}} \sgn(\sigma) \cdot \omega(\underbrace{v_{\sigma(1)} \otimes \dotsc \otimes v_{\sigma(p)}}_{v_i \text{ exchanged by } v_{i+1}}) \otimes \eta(\underbrace{v_{\sigma(p+1)} \otimes \dotsc \otimes v_{\sigma(p+q)}}_{v_{i+1} \text{ exchanged by } v_{i}})\]
\[+\sum_{\substack{\sigma \in \Sigma_{p,q} \\ i+1 \in \{\sigma(1),\dotsc,\sigma(p)\} \not\ni i}} \sgn(\sigma) \cdot \omega(\underbrace{v_{\sigma(1)} \otimes \dotsc \otimes v_{\sigma(p)}}_{v_{i+1} \text{ exchanged by } v_i}) \otimes \eta(\underbrace{v_{\sigma(p+1)} \otimes \dotsc \otimes v_{\sigma(p+q)}}_{v_i \text{ exchanged by } v_{i+1}})\]
By changing back $v_i$ and $v_{i+1}$ in the first and the second sum and by interchanging the third with the fourth sum and thereby composing $\sigma$ with the transposition $(i ~ i+1)$, we arrive at the corresponding decomposition of $-(\omega \wedge \eta)(v_1 \otimes \dotsc \otimes v_{p+q})$.

This shows that $\omega \wedge \eta$ is antisymmetric.
\end{proof}

\begin{cor} \label{hilfs-omega}
Let $\C$ be a cocomplete $R$-linear tensor category such that $2 \in R^*$, $V \in \C$ and $d \in \N$. Then, there is a canonical morphism
\[\omega : \Lambda^{d+1}(V) \to V \otimes \Lambda^d(V)\]
which is given in element notation by
\[\omega(v_0 \wedge \dotsc \wedge v_d) = \sum_{k=0}^{d} (-1)^k \, v_k \otimes (v_1 \wedge \dotsc \wedge \widehat{v_k} \wedge \dotsc \wedge v_d).\]
\end{cor}

\begin{rem}[Hopf algebra structure]
We can put a \emph{graded} Hopf algebra structure (\autoref{hopf}) on $\Lambda(V)$ as follows: Applying \autoref{antipaar} to the universal antisymmetric morphisms $V^{\otimes p} \to \Lambda^p(V)$ and $V^{\otimes q} \to \Lambda^q(V)$, we get a morphism $\Lambda^{p+q}(V) \to \Lambda^p(V) \otimes \Lambda^q(V)$. This induces a morphism
\[\Delta : \Lambda(V) \to \Lambda(V) \otimes \Lambda(V),\]
the comultiplication, which is explicitly given by
\[\Delta(v_1 \wedge \dotsc \wedge v_n) = \sum_{\substack{n=p+q \\ \sigma \in \Sigma_{p,q}}} \sgn(\sigma) \cdot v_{\sigma(1)} \wedge \dotsc \wedge v_{\sigma(p)} \otimes v_{\sigma(p+1)} \wedge \dotsc \wedge v_{\sigma(n)}.\]
If we define the counit $\e : \Lambda(V) \to 1$ by $\e|_{\Lambda^0(V)}=\id$ and $\e|_{\Lambda^p(V)} = 0$ for $p>0$, one easily checks that $\Delta$ is coassociative and that $\e$ is coneutral for $\Delta$. It is clear that $\e$ is an algebra homomorphism. But $\Delta$ is \emph{not} an algebra homomorphism. In fact, we have
\begin{eqnarray*}
\Delta(v_1) \wedge \Delta(v_2) & = & (v_1 \otimes 1 + 1 \otimes v_1) \wedge (v_2 \otimes 1 + 1 \otimes v_2) \\
&=&v_1 \wedge v_2 \otimes 1 + v_1 \otimes v_2 + v_2 \otimes v_1 + 1 \otimes v_1 \wedge v_2 \\
& \neq &v_1 \wedge v_2 \otimes 1 + v_1 \otimes v_2 - v_2 \otimes v_1 + 1 \otimes v_1 \wedge v_2 \\
& =& \Delta(v_1 \wedge v_2).
\end{eqnarray*}
We can remedy this as follows (which I have learned from Mariano Su\'{a}rez-Alvarez): Do not view $\Lambda(V)$ as an algebra in $\C$, but rather as an $\N$-graded-commutative algebra (\autoref{graded-comm}) in $\C$, i.e. a commutative algebra in the cocomplete tensor category $\tilde{\gr}_\N(\C)$ of $\N$-graded objects of $\C$ equipped with the \emph{twisted} symmetry. It follows that the multiplication on the tensor product $\Lambda(V) \otimes \Lambda(V)$ (\autoref{algebra-cons}) is actually given -- in element notation -- by $(a \otimes b) \cdot (c \otimes d) = (-1)^{|b| \cdot |c|} (a \cdot c) \otimes (b \cdot d)$. Then a calculation shows that $\Delta$ is a homomorphism of $\N$-graded-commutative algebras; alternatively one can construct $\Delta$ using the universal property of $\Lambda(V)$ as the universal graded-commutative algebra equipped with a morphism from $V$ in degree $1$. Hence, $\Lambda(V)$ is an $\N$-graded-commutative bialgebra, i.e. a bialgebra in $\tilde{\gr}_\N(\C)$. Actually it is an involutive Hopf algebra with antipode given by $S(x)=(-1)^{|x|} x$.
\end{rem}

\begin{lemma}[Symmetry Lemma] \label{symlem}
Let $\C$ be a finitely cocomplete $R$-linear tensor category. Assume that $p,d \in \mathds{N}$ such that $d! \in R^*$ and $1 \leq p \leq d$. Let $V,W \in \C$ be objects and let
\[f : \Lambda^p(V) \otimes \Lambda^d(V) \to W\]
be a morphism in $\C$. We assume that the following relation holds (in element notation):
\[\sum_{k=0}^{d} (-1)^k f(v_1 \wedge \dotsc \wedge v_{p-1} \wedge w_k \otimes w_0 \wedge \dotsc \wedge \widehat{w_k} \wedge \dotsc \wedge w_d) = 0\]
In other words, we have (via $w_0=v_p$)
\[(\star)~~~~ f(v_1 \wedge \dotsc \wedge v_{p-1} \wedge v_p \otimes w_1 \wedge \dotsc \wedge w_d)\]
\[ = \sum_{k=1}^{d} (-1)^{k+1} f(v_1 \wedge \dotsc \wedge v_{p-1} \wedge w_k \otimes v_p \wedge w_1 \wedge \dotsc \wedge \widehat{w_k} \wedge \dotsc \wedge w_d)\]
Then, we have
\[(\star)_p ~~~~ f(v_1 \wedge \dotsc \wedge v_p \otimes w_1 \wedge \dotsc \wedge w_d)\]
\[ = \sum_{\sigma \in \Sigma_{p,d-p}} \sgn(\sigma) \cdot f(w_{\sigma_1} \wedge \dotsc \wedge w_{\sigma_p} \otimes v_1 \wedge \dotsc \wedge v_p \wedge w_{\sigma_{p+1}} \wedge \dotsc \wedge w_{\sigma_d}).\]

For $p=d$ this means that $f : \Lambda^d(V) \otimes \Lambda^d(V) \to W$ is symmetric, i.e.
\[f(v_1 \wedge \dotsc \wedge v_d \otimes w_1 \wedge \dotsc \wedge w_d) = f(w_1 \wedge \dotsc \wedge w_d \otimes v_1 \wedge \dotsc \wedge v_d).\]
\end{lemma}

\begin{proof}
We make an induction on $p$. The case $p=1$ is trivial. Although given a morphism $f : \Lambda^{p+1}(V) \otimes \Lambda^d(V) \to W$, we can only construct a morphism $\tilde{f} : \Lambda^p(V) \otimes \Lambda^d(V) \to \HOM(V,W)$ and apply the induction hypothesis if $\C$ is closed, in general we may simply strengthen the statement of the lemma by adding an additional tensor factor on the left. In order to simplify the element notation further, we abbreviate $f(v_1 \wedge \dotsc \wedge v_p \otimes w_1 \wedge \dotsc \wedge w_d)$ by $v_1 \dotsc v_p | w_1 \dotsc w_d$. If in a calculation some of these entries are not involved on both sides, we just omit them from the notation. Thus, $(\star)$ is just
\[v_p | w_1 \dotsc w_d = \sum_{k=1}^{d} (-1)^{k+1} w_k | v_p w_1 \dotsc w_{k-1} \widehat{w_k} w_{k+1} \dotsc w_d\]
Following a very helpful suggestion by Stefan Vogel, we can rewrite this without any sign as
\[v_p | w_1 \dotsc w_d = \sum_{k=1}^{d} w_k | w_1 \dotsc w_{k-1} v_p w_{k+1} \dotsc w_d,\]
which we abbreviate further as
\[(\star)_1 ~~~~ v_p | w_1 \dotsc w_d = \sum_{k=1}^{d} w_k | w(w_k=v_p).\]
We may apply this not just for the last $v_p$, but also for example to $v_1$ because of antisymmetry.
Our goal is to prove the relation
\[(\star)_p ~~~~ v_1 \dotsc v_p | w_1 \dotsc w_d = \sum_{1 \leq i_1 < \dotsc < i_p \leq d} w_{i_1} \dotsc w_{i_p} | w(w_{i_k} = v_k)_{1 \leq k \leq p}.\] 
Assuming this is true for some $p<d$, we prove $(\star)_{p+1}$ as follows:
\[v_0 v_1 \dotsc v_p | w \stackrel{(\star)_p}{=} \sum_{i_1<\dotsc<i_p} v_0 w_{i_1} \dotsc w_{i_d} | w(w_{i_k} = v_k)_{1 \leq k \leq p}\]
\[ \stackrel{(\star)_1}{=} \sum_{i_1<\dotsc<i_p} \left(\sum_{j \notin \{i_1,\dotsc,i_p\}} w_j w_{i_1} \dotsc w_{i_p} | w((w_{i_k}=v_k)_{1 \leq k \leq p}, w_{j}=v_0) \right.\]
\[ ~~~~~~~~~~~~~~~~~~~~~~~+\left. \sum_{n=1}^{p} v_n w_{i_1} \dotsc w_{i_p} | w((w_{i_k}=v_k)_{k \neq n},w_{i_n}=v_0)\right) ~~~~~(\dagger)\]
The expression $(\dagger)$ expands into two double sums. Using the convention $i_0:=0$, $i_{p+1}:=d+1$, the first one equals  
\[ \sum_{m=1}^{p+1} \, \sum_{i_1<\dotsc<i_p} \, \sum_{i_{m-1}<j<i_m} w_j w_{i_1} \dotsc w_{i_p} | w((w_{i_k}=v_k)_{1 \leq k \leq p}, w_{j}=v_0).\]
In order to align the indices left from the bar, we have to apply an $m$-cycle. Also on the right we have apply the $m$-cycle $v_0 \mapsto v_{m-1} \mapsto \dotsc \mapsto v_1 \mapsto v_0$ on $v_1,\dotsc,v_{m-1},v_0,v_m,\dotsc,v_p$ in order to obtain $v_0,v_1,\dotsc,v_p$. Therefore the signs resulting from left and right cancel. After the renaming
\[\{j_0<\dotsc<j_p\} = \{i_1<\dotsc<i_{m-1}<j<i_m<\dotsc<i_p\}\]
this yields the sum
\[ \sum_{m=1}^{p+1} \, \sum_{j_0<\dotsc<j_p} w_{j_0} \dotsc w_{j_p} | w((w_{j_k}=v_k)_{0 \leq k \leq p})\]
\[  = (p+1) \cdot  \sum_{j_0<\dotsc<j_p} w_{j_0} \dotsc w_{j_p} | w((w_{j_k}=v_k)_{0 \leq k \leq p}).\]
This is $p+1$ times the right hand side of $(\star)_{p+1}$. Now let us look at the second sum in $(\dagger)$:
\begin{eqnarray*}
&& \sum_{n=1}^{p} \, \sum_{i_1<\dotsc<i_p}  v_n w_{i_1} \dotsc w_{i_p} | w((w_{i_k}=v_k)_{k \neq n},w_{i_n}=v_0) \\
& \stackrel{(\star)_p}{=} & \sum_{n=1}^{p} v_n v_1 \dotsc v_{n-1} v_0 v_{n+1} \dotsc v_p | w \\
&=& - p \cdot v_0 \dotsc v_p | w
\end{eqnarray*}
Altogether we arrive at
\[(p+1) \cdot v_0 \dotsc v_p | w = (p+1) \cdot  \sum_{j_0<\dotsc<j_p} w_{j_0} \dotsc w_{j_p} | w((w_{j_k}=v_k)_{0 \leq k \leq p})\]
Since $p+1 \in R^*$, we may divide by it and obtain $(\star)_{p+1}$.
\end{proof}

\section{Derivations} \label{deriv}

Let $R$ be a commutative ring and let $\C$ be a finitely cocomplete $R$-linear tensor category.

\begin{defi}[Derivations and $\Omega^1$] \label{diffdef}
Let $A$ be a commutative algebra in $\C$ and $M$ some $A$-module. We define a \emph{derivation} from $A$ to $M$ to be a morphism $d : A \to M$ in $\C$ such that the morphism $A \otimes A \to A \xrightarrow{d} M$ is equal to the sum of the two morphisms $A \otimes A \xrightarrow{A \otimes d} A \otimes M \to M$ and $A \otimes A \xrightarrow{d \otimes A} M \otimes A \to M$. In element notation, this is the usual Leibniz rule
\[d(a \cdot b) = a \cdot d(b) + d(a) \cdot b.\]
We obtain a functor
\[\Der(A,-) : \M(A) \to \Set\]
which maps $M \in \M(A)$ to the set of all derivations $A \to M$ (actually it is an abelian group). A representing object of this functor is denoted by $\Omega^1_A$ and is called the \emph{module of differentials} of $A$.
\end{defi}

\begin{defi}[Relative module of differentials]
If $A \to B$ is a homomorphism of commutative algebras in $\C$ (equivalently, $B$ is a commutative $A$-algebra) and $M$ is some $B$-module, an \emph{$A$-derivation} from $B$ to $M$ is a derivation from $B$ to $M$ in $\M(A)$, i.e. a homomorphism of $A$-modules $d : B \to M$ satisfying the Leibniz rule. We obtain a functor
\[\Der_A(B,-) : \M(B) \to \Set.\]
A representing object of this functor is denoted by $\Omega^1_{B/A}$ and is called the \emph{module of differentials} of $B$ over $A$.

Of course $\Omega^1_{B/A}$ is just $\Omega^1_{B}$ when $B$ is considered as a commutative algebra in $\M(A)$. Note that $\Der_A(B,-) = \{d \in \Der(B,-) : d|_A = 0\}$.
\end{defi}

As always these definitions coincide with the usual ones when $\C=\M(R)$ for some commutative ring $R$ (\cite[Chapter III, Paragraph 10]{Bou98}). But even in this special case the following simple construction of $\Omega^1$ does not seem to be well-known.
 
\begin{prop} \label{omega-ex}
If $A \to B$ is a homomorphism of commutative algebras in $\C$, then $\Omega^1_{B/A}$ exists.
\end{prop}

\begin{proof}
The $A$-linear map $h : B \otimes_A B \to B \otimes_A B$ defined by
\[h(x \otimes y) := xy \otimes 1 - y \otimes x - x \otimes y\]
in element notation extends to a $B$-linear map $h_B : (B \otimes_A B) \otimes_A B \to B \otimes_A B$, when $B$ acts on the right. Let $\Omega^1_{B/A}$ be its cokernel and define the $A$-linear map $d : B \to \Omega^1_{B/A}$ to be the composition of $B \to B \otimes_A B,\, b \mapsto b \otimes 1$ with $B \otimes_A B \twoheadrightarrow \Omega^1_{B/A}$. Then $d$ is an $A$-derivation by construction and it satisfies the desired universal property: If $M$ is some $B$-module and $B \to M$ is some $A$-derivation, then it lifts to a $B$-linear map $B \otimes_A B \to M$ which kills $h_B$, hence lifts to a $B$-linear map on the cokernel $\Omega^1_{B/A}$.
\end{proof}

\begin{rem}
Actually we do not need that $\C$ is enriched in $\M(R)$ in order to make sense of the notion of derivations and of modules of differentials. We only need that $\C$ is enriched over the tensor category of commutative monoids $\CMon$. This leads to the notion of \emph{absolute derivations} (\cite{Kur02}).
\end{rem}

\begin{cor} \label{omegaH}
If $H : \C \to \D$ is a finitely cocontinuous tensor functor and $A \to B$ is a homomorphism of commutative algebras in $\C$, then there is an isomorphism of $H(B)$-modules
\[H(\Omega^1_{B/A}) \cong \Omega^1_{H(B)/H(A)}.\]
In particular, if $C$ is another commutative algebra in $\C$, we have an isomorphism of $B \otimes C$-modules
\[\Omega^1_{B/A} \otimes C \cong \Omega^1_{B \otimes C / A \otimes C}.\]
\end{cor}

\begin{proof}
This follows from the construction in \autoref{omega-ex}. For the second part, take $H : \C \to \M(C)$, $M \mapsto M \otimes C$.
\end{proof}

\begin{prop}[Properties of modules of differentials] \label{deriv-prop}
Let $A$ be a fixed commutative algebra in $\C$.
\begin{enumerate}
\item Let $B_1,B_2$ be two commutative $A$-algebras. Then there is an isomorphism of $B_1 \otimes_A B_2$-modules
\[\Omega^1_{B_1 \otimes_A B_2 / A} \cong  \Omega^1_{B_1/A} \otimes_A B_2 ~\oplus ~ B_1 \otimes_A \Omega^1_{B_2/A}.\]
\item Let $B \to C$ be a homomorphism of commutative $A$-algebras. Then there is an exact sequence of $C$-modules
\[\Omega^1_{B/A} \otimes_B C \to \Omega^1_{C/A} \to \Omega^1_{C/B} \to 0.\]
\item Let $B \to C$ be a homomorphism of commutative $A$-algebras. Assume that it is a regular epimorphism of $B$-modules, choose an exact sequence of $B$-modules $I \to B \to C \to 0$. Then there is an exact sequence of $C$-modules
\[I \otimes_B C \to \Omega^1_{B/A} \otimes_B C \to \Omega^1_{C/A} \to 0.\]
\end{enumerate}
\end{prop}

\begin{proof}
1. If $T$ is some $B_1 \otimes_A B_2$-module, we have natural bijections
\[\Hom_{B_1 \otimes_A B_2}(\Omega^1_{B_1/A} \otimes_A B_2 \oplus B_1 \otimes_A \Omega^1_{B_2/A},T)\]
\[\cong \Hom_{B_1}(\Omega^1_{B_1/A},T|_{B_1}) \times \Hom_{B_2}(\Omega^1_{B_2/A},T|_{B_2})\]
\[\cong  \Der_A(B_1,T|_{B_1}) \times \Der_A(B_2,T|_{B_2})\]
Therefore, it suffices to construct a natural bijection
\[\Der_A(B_1,T|_{B_1}) \times \Der_A(B_2,T|_{B_2}) \cong \Der_A(B_1 \otimes_A B_2,T).\]
If $(d_1,d_2)$ is contained in the left hand side, then $d : B_1 \otimes_A B_2 \to T$ defined by
\[d(b_1 \otimes b_2) = b_1 \cdot d_2(b_2) + d_1(b_1) \cdot  b_2\]
in element notation (as usual it is easy to write this down formally) is an $A$-derivation by an easy calculation. If conversely, $d : B_1 \otimes_A B_2 \to T$ is an $A$-derivation, then $d_1(b_1):=d(b_1 \otimes 1)$ is an $A$-derivation $B_1 \to T|_{B_1}$. Similarly we define $d_2$. These constructions are easily seen to be inverse to each other.

2. By the definition of a right exact sequence and of $\Omega^1$ we have to show that for every $C$-module $T$ there is a natural exact sequence
\[0 \to \Der_B(C,T) \to \Der_A(C,T) \to \Der_A(B,T|_B).\]
The first map is the obvious inclusion and the second one is given by composition with $B \to C$. Exactness is obvious.

3. By a similar reasoning as above we have to prove that for every $C$-module $T$ there is a natural exact sequence
\[0 \to \Der_A(C,T) \to \Der_A(B,T|_B) \to \Hom_B(I,T|_B).\]
The first map is given by composition with $B \to C$. The second map is given by composition with $I \to B$. It is well-defined: For an $A$-derivation $d : A \to T$ the restriction $d|_I : I \to T$ is $B$-linear because of the Leibniz rule. The map $\Der_A(C,T) \to \Der_A(B,T)$ is injective because $B \to C$ is an epimorphism of $A$-modules. To show exactness, let $d : B \to T$ be an $A$-derivation which vanishes on $I$. Since $I \to B \to C \to 0$ is also exact as a sequence of $A$-modules, it follows that $d$ lifts to an $A$-linear map $\tilde{d} : C \to T$. It is a derivation, as can be easily deduced from our assumptions that $B \to C$ is an epimorphism of $A$-modules and that $d$ is a derivation.
\end{proof}

\begin{prop} \label{deriv-sym}
Let $A$ be a commutative algebra in $\C$ and $E$ some $A$-module. There is an isomorphism of $\Sym(E)$-modules
\[\Omega^1_{\Sym(E)/A} \cong E \otimes_A \Sym(E).\]
The universal differential is given by $\Sym(E) \to E \otimes_A \Sym(E)$ mapping
\[e_1 \cdot \dotsc \cdot e_n \mapsto \sum_{k=1}^{n} e_k \otimes (e_1 \cdot \dotsc \cdot \widehat{e_k} \cdot \dotsc \cdot e_n)\]
in element notation. In particular, $\Omega^1_{A[T_1,\dotsc,T_n]/A}$ is a free $A[T_1,\dotsc,T_n]$-module with basis $\{d(T_i) : 1 \leq i \leq n\}$.
\end{prop}

\begin{proof}
If $T$ is some $\Sym(E)$-module, we have to find a natural bijection
\[\Der_A(\Sym(E),T) \cong \Hom_A(E,T).\]
If $d : \Sym(E) \to T$ is an $A$-derivation, we may restrict it to $\Sym^1(E) = E$ to get an $A$-linear map $A \to T$. If conversely $f : E \to T$ is some $A$-linear map, then we first define the $A$-linear map
$E^{\otimes n} \to T$ by mapping
\[e_1 \otimes \dotsc \otimes e_n \mapsto \sum_{k=1}^{n} f(e_k) \cdot (e_1 \cdot \dotsc \cdot \widehat{e_k} \cdot \dotsc \cdot e_n)\]
in element notation. It is obviously symmetric, hence extends to an $A$-linear map $\Sym^n(E) \to T$. All these induce an $A$-linear map $\Sym(E) \to T$. One checks that this is an $A$-derivation, namely the unique one extending $f$ in degree $1$. See \cite[Chapter III, Paragraph 10.11, Example]{Bou98} for the special case of modules.
\end{proof}

\begin{lemma} \label{eder}
Let $A \to B$ be a homomorphism of commutative algebras in $\C$ and $C$ be some commutative $B$-algebra. Then there is a bijection between $A$-derivations $B \to C$ and homomorphisms of $A$-algebras $B \to C[\e]/\e^2$ such that $B \to C[\e]/\e^2 \xrightarrow{\e \mapsto 0} C$ is the given map.
\end{lemma}

\begin{proof}
Since $C[\e]/\e^2 = C \oplus C$ as $A$-modules, an $A$-linear map $B \to C[\e]/\e^2$ corresponds to a pair of $A$-linear maps $B \to C$. We specify the first one to be the given one and then the second one is an $A$-linear map $d : B \to C$ such that $B \to C[\e]/\e^2$ defined by $b \mapsto b + d(b) \e$ in element notation is multiplicative, which comes down to the Leibniz rule of $d$.
\end{proof}

\begin{rem}
\autoref{eder} offers an alternative proof of \autoref{deriv-sym} which simply uses the universal property of the symmetric algebra.
\end{rem}

\begin{thm}[De Rham complex] \label{rham}
Assume that $\C$ is an $R$-linear cocomplete tensor category with $2 \in R^*$. Let $A \to B$ be a homomorphism of commutative algebras in $\C$. For $n \in \N$ we define the $B$-module of differentials of order $n$ by
\[\Omega^n_{B/A} := \Lambda^n_B(\Omega^1_{B/A}).\]
There is a unique sequence of homomorphisms of $A$-modules
\[d^n : \Omega^n_{B/A} \to \Omega^{n+1}_{B/A}\]
with the following properties:
\begin{enumerate}
\item $d^0=d : B \to \Omega^1_{B/A}$ is the universal differential.
\item We have $d^{n+1} \circ d^n = 0$.
\item If $n=p+q$, then we have $d^n(f \wedge g) = d^p(f) \wedge g + (-1)^p f \wedge d^q(g)$ in element notation for $f \in \Omega^p_{B/A}$ and $g \in \Omega^q_{B/A}$.
\end{enumerate}
We call $(\Omega^*_{B/A},d^*)$ the (algebraic) de Rham complex associated to $B/A$.
\end{thm}

\begin{proof}
Uniqueness follows by induction from 1. and 3. In order to construct the differentials, we cannot only rely on the defining universal property of $\Omega^1_{B/A}$, because that only deals with homomorphisms of \emph{$B$-modules} on $\Omega^1_{B/A}$. Instead, we use the construction given in the proof of (\autoref{omega-ex}), namely that $\Omega^1_{B/A}$ is the quotient of $B \otimes_A B$ by the relation
\[ab \otimes c - b \otimes ac - a \otimes bc = 0\]
in element notation -- as a module over $B$ as well as over $A$. Here, $a \otimes b$ represents $d(a) \cdot b = b \cdot d(a)$. This approach is similar to \cite[16.6]{EGAIV}.
 
We define $d^0:=d$. In order to construct $d^1$, look at the homomorphism of $A$-modules defined by
\[B \otimes_A B \to \Omega^2_{B/A},~ a \otimes b \mapsto d(b) \wedge d(a).\]
It maps $a  b \otimes c - b \otimes a c - a \otimes bc$ to
\begin{eqnarray*}
&& d(c) \wedge d(a \cdot b) - d(a \cdot c) \wedge d(b) - d(b \cdot c) \wedge d(a) \\
&=& a \cdot d(c) \wedge d(b) + b \cdot d(c) \wedge d(a)  \\
&& - a \cdot d(c) \wedge d(b) - c \cdot d(a) \wedge d(b) \\
&& - b \cdot d(c) \wedge d(a) - c \cdot d(b) \wedge d(a) \\
&=& 0.
\end{eqnarray*}
Hence, it lifts to a homomorphism of $A$-modules $d^1 : \Omega^1_{B/A} \to \Omega^2_{B/A}$ which is characterized by $d^1(b \cdot d(a)) = d(b) \wedge d(a)$ for ``elements'' $a,b \in B$. In particular, $d^1 \circ d^0=0$.
 
More generally, consider $\Omega^n_{B/A}$ as a quotient of $B^{\otimes 2n}$ and consider the homomorphism of $A$-modules $B^{\otimes 2n} \to \Omega^{n+1}_{B/A}$ defined by
\[a_1 \otimes b_1 \otimes \dotsc \otimes a_n \otimes b_n \mapsto d(b_1 \cdot \dotsc \cdot b_n) \wedge d(a_1) \wedge \dotsc \wedge d(a_n).\]
A similar calculation as above, using \autoref{coeq-tensor} and \autoref{sym-exakt}, shows that this extends to a homomorphism of $A$-modules $d^n : \Omega^n_{B/A} \to \Omega^{n+1}_{B/A}$ characterized by
\[d^n(b \cdot d(a_1) \wedge \dotsc \wedge d(a_n)) = d(b) \wedge d(a_1) \wedge \dotsc \wedge d(a_n).\]
In particular, $d^n(d(a_1) \wedge \dotsc \wedge d(a_n))=0$, from which we easily deduce 2.. In order to prove the graded Leibniz rule 3., we write
\[f = b \cdot d(a_1) \wedge \dotsc \wedge d(a_p),~ g = c \cdot d(a_{p+1}) \wedge \dotsc \wedge d(a_n)\]
and compute
\begin{eqnarray*}
&& d^p(f) \wedge g + (-1)^p f \wedge d^q(g) \\
&=& d(b) \wedge d(a_1) \wedge \dotsc \wedge d(a_p) \wedge c \cdot d(a_{p+1}) \wedge \dotsc \wedge d(a_n) \\
&& + (-1)^p \, b \cdot d(a_1) \wedge \dotsc \wedge d(a_p) \wedge d(c) \wedge d(a_{p+1}) \wedge \dotsc \wedge d(a_n)\\
&=& (d(b) \cdot c + b \cdot d(c)) \wedge d(a_1) \wedge \dotsc \wedge d(a_n)\\
&=& d(b \cdot c) \wedge d(a_1) \wedge \dotsc \wedge d(a_n) = d(f \wedge g). 
\end{eqnarray*}
\end{proof}

\begin{rem}
If $\C$ is even abelian, we can define the (algebraic) \emph{de Rham cohomology} by $H^n_{dR}(B/A) := H^n(\Omega^*_{B/A},d^*)$.
\end{rem}

\begin{rem}[Quasi-coherent module of differentials] \label{defo}
If $X \to S$ is a morphism of schemes, we have the associated quasi-coherent module of differentials $\Omega^1_{X/S}$. It is an object of $\Q(X)$ defined to satisfy the universal property $\Hom(\Omega^1_{X/S},M) \cong \Der_{\O_S}(\O_X,M)$, where derivations on the right hand side are certain homomorphisms of abelian sheaves on the underlying space of $X$. Thus we leave our tensor category $\Q(X)$! There does not seem to be any general definition or characterization of $\Omega^1_{X/S}$ which stays inside $\Q(X)$ or uses just the pullback functor $\Q(S) \to \Q(X)$. At least if $X$ is affine over $S$, we have found a description above. The Euler sequence below gives a description when $X$ is projective over $S$. See \autoref{tangentt} for related questions.
\end{rem}

\begin{thm}[Euler sequence] \label{euler}
Let $E$ be a quasi-coherent module  on a scheme $S$. Let $p : \P(E) \to S$ be the associated projective scheme.
Then there is an exact sequence of quasi-coherent modules on $\P(S)$
\[0 \to \Omega^1_{\P(E)/S} \to p^*(E)(-1) \to \O \to 0.\]
The epimorphism is dual to the canonical one $p^*(E) \to \O_{\P(E)}(1)$.
\end{thm}
 
We do not know if this general form of the Euler sequence has already appeared in the literature. Usually $E$ is assumed to be (locally) free of finite rank (\cite[21.4.9]{Vak13}), but this is not necessary to make the proof work. Later we will even replace $\Q(S)$ by an arbitrary locally presentable linear tensor category (\autoref{tangent-proj}).
 
\begin{proof}
First, let us assume that $E$ is generated by global sections. Then the scheme $\P(E)=\Proj(\Sym(E))$ is covered by the open subsets
\[D_+(\lambda)=\Spec(\Sym(E)_{(\lambda)}),\]
where $\lambda$ runs through the global sections of $E$. There is an isomorphism between the homogeneous localization
$\Sym(E)_{(\lambda)}$ and the algebra $\Sym(E)/(\lambda-1)$ given by $\frac{a}{\lambda} \mapsto \overline{a}$. Using \autoref{deriv-sym} and \autoref{deriv-prop}, we obtain
\[\Omega^1_{\Sym(E)_{(\lambda)}/\O_S} \cong E/\lambda \otimes \Sym(E)_{(\lambda)}.\]
The universal differential $d_{\lambda}$ maps $\tfrac{a}{\lambda}$ to $\overline{a} \otimes 1$.
 
The dual of the Serre twist $\O_{\P(E)}(-1)$ is free on $D_+(\lambda)$ with some generator $e_{\lambda}^*$.
On the overlaps $D_+(\lambda) \cap D_+(\mu)$ they satisfy $e_\lambda^* = \frac{\mu}{\lambda} e_{\mu}^*$.
The desired exact sequence when restricted to $D_+(\lambda)$ is
\[\xymatrix{0 \ar[r] &  E/(\lambda) \otimes \Sym(E)_{(\lambda)} \ar[r]^-{\iota} &  E \otimes \Sym(E)_{(\lambda)} \cdot e_{\lambda}^* \ar[r]^-{p} &  \Sym(E)_{(\lambda)} \ar[r] &  0,}\]
where $p(a \otimes e_{\lambda}^*) = \frac{a}{\lambda}$ and $\iota(\overline{a} \otimes 1) := a \otimes e_{\lambda}^* - \lambda \otimes \frac{a}{\lambda} e_{\lambda}^*$.
Observe that $\iota$ is well-defined and satisfies $p \circ \iota = 0$. We claim that the sequence is exact, in fact contractible, i.e. that there are homomorphisms
of $\Sym(E)_{(\lambda)}$-modules
\[\xymatrix{\Sym(E)_{(\lambda)} \ar[r]^-{t} & E \otimes \Sym(E)_{(\lambda)} \cdot e_{\lambda}^* \ar[r]^-{s} & E/(\lambda) \otimes \Sym(E)_{(\lambda)}  }\]
such that $s \circ \iota = \id$, $p \circ t = \id$ and $t \circ p + \iota \circ s = \id$. Define $t(1)=\lambda \otimes e_{\lambda}^*$ and $s(a \otimes e_{\lambda}^*) = \overline{a} \otimes 1$.
Then $p \circ t = \id$ is obvious, $s \circ \iota = \id$ follows from $\overline{\lambda}=0$ and for the third equation we calculate
\[(t \circ p + \iota \circ s)(a \otimes e_{\lambda}^*)=t(\tfrac{a}{\lambda})+\iota(\overline{a} \otimes 1)=\tfrac{a}{\lambda} \cdot \lambda \otimes e_{\lambda}^* + a \otimes e_{\lambda}^* - \lambda \otimes \tfrac{a}{\lambda} e_{\lambda}^* = a \otimes e_{\lambda}^*.\]
Let us check that these homomorphisms $\iota_{\lambda} : \Omega^1_{D_+(\lambda)/S} \to p^*(E)(-1)|_{D_+(\lambda)}$ are compatible on the overlaps
$D_+(\lambda) \cap D_+(\mu)$. The canonical isomorphism
\[\Omega^1_{D_+(\lambda)} |_{D_+(\lambda) \cap D_+(\mu)} \cong \Omega^1_{D_+(\mu)}|_{D_+(\mu) \cap D_+(\lambda)}\]
corresponding to
\[(E/(\lambda) \otimes \Sym(E)_{(\lambda)})_{\frac{\mu}{\lambda}} \cong (E/(\mu) \otimes \Sym(E)_{(\lambda)})_{\frac{\lambda}{\mu}}\]
maps $\overline{a} \otimes 1 = d_{\lambda}(\tfrac{a}{\lambda})$ to
\begin{eqnarray*}
d_{\mu}(\tfrac{a}{\lambda}) & =& \frac{d_{\mu}(\tfrac{a}{\mu}) \cdot \tfrac{\lambda}{\mu} - d_{\mu}(\tfrac{\lambda}{\mu}) \cdot \frac{a}{\mu}}{(\frac{\lambda}{\mu})^2}\\
&=& d_{\mu}(\tfrac{a}{\mu}) \cdot \tfrac{\mu}{\lambda} - d(\tfrac{\lambda}{\mu}) \cdot \tfrac{a \mu}{\lambda^2}\\
&=& \overline{a} \otimes \tfrac{\mu}{\lambda} - \overline{\lambda} \otimes  \tfrac{a \mu}{\lambda^2}.
\end{eqnarray*}
The diagram
\[\xymatrix{(E/(\lambda) \otimes \Sym(E)_{(\lambda)})_{\frac{\mu}{\lambda}} \ar[r]^{\iota_{\lambda}} \ar[d]^{\cong} &   E \otimes (\Sym(E)_{(\lambda)})_{\frac{\mu}{\lambda}} \cdot e_{\lambda}^* \ar[d]^{\cong}  \\
 (E/(\mu) \otimes \Sym(E)_{(\lambda)})_{\frac{\lambda}{\mu}} \ar[r]^{\iota_{\mu}} &  E \otimes (\Sym(E)_{(\mu)})_{\frac{\lambda}{\mu}} \cdot e_{\mu}^* }\]
commutes: The image of $\overline{a} \otimes 1$ under $\iota_{\lambda}$ is $a \otimes e_{\lambda}^* - \lambda \otimes \frac{a}{\lambda} e_{\lambda}^*$, which gets mapped to
$a \otimes \tfrac{\mu}{\lambda} e_{\mu}^* - \lambda \otimes \frac{a}{\lambda} \cdot \frac{\mu}{\lambda} e_{\mu}^*=:x$ by the isomorphism. The other way round we first get $\overline{a} \otimes \tfrac{\mu}{\lambda} - \overline{\lambda} \otimes  \frac{a \mu}{\lambda^2}$, which gets mapped by $\iota_{\mu}$ to
\[\frac{\mu}{\lambda} \cdot  \left(a \otimes e_{\mu}^* - \mu \otimes \frac{a}{\mu} e_{\mu}^*\right) -  \frac{a \mu}{\lambda^2} \cdot \left(\lambda \otimes e_{\mu}^* - \mu \otimes \frac{\lambda}{\mu} e_{\mu}^*\right),\]
in which the second and the fourth summand cancel, whereas the rest simplifies to $x$.

Thus, the $\iota_{\lambda}$ glue to a homomorphism $\iota : \Omega^1_{\P(E)/S} \to p^*(E)(-1)$. The sequence
\[0 \to \Omega^1_{\P(E)/S} \xrightarrow{\iota} p^*(E)(-1) \xrightarrow{p} \O \to 0\]
is exact, because this holds on each $D_+(\lambda)$. Note that the definition of $\iota$ does not depend on a system of global generators of $E$,
since we have just used all global sections of $E$. In particular, $\iota$ restricts to the same $\iota$ when we preform a base change to some open subscheme of $S$.
This makes it possible to finish the proof when $E$ is not necessarily generated by global sections, because $E|_T$ is generated by global sections for every
open affine $T \subseteq S$.
\end{proof}

\section{Flat objects} \label{flat}

\begin{defi}[Flat objects]
Let $\C$ be a tensor category. An object $M \in \C$ is called \emph{flat} if $M \otimes - : \C \to \C$ is left exact, i.e. preserves finite limits. If $\C$ is an abelian tensor category, this functor is right exact anyway, therefore flatness means that $M \otimes -$ preserves monomorphisms.
\end{defi}

Of course this coincides with the usual definition when $\C=\M(R)$ for some commutative ring $R$. However, if $X$ is a scheme, one usually calls a quasi-coherent module $M \in \Q(X)$ flat if for every $x \in X$ the $\O_{X,x}$-module $M_x$ is flat. Equivalently, for every morphism from an affine scheme $\Spec(A) \to X$, the pullback of $M$ is associated to a flat $A$-module. The latter definition also works for algebraic stacks $X$. Following Lurie (\cite[Example 5.8]{Lur04}), we call $M$ \emph{locally flat} if this is satisfied.
 
\begin{lemma}
Every locally flat quasi-coherent module is also flat. The converse holds over quasi-separated schemes as well as over geometric stacks.
\end{lemma}

\begin{proof}
The first part is clear. Now let $X$ be a quasi-separated scheme and let $M \in \Q(X)$ be flat. It suffices to prove that $M|_U \in \Q(U)$ is flat for every open affine subscheme $j : U \hookrightarrow X$. Let $N \to N'$ be a monomorphism in $\Q(U)$. Then $j_* N \to j_* N'$ is a monomorphism in $\Q(X)$. By assumption $M \otimes j_* N \to M \otimes j_* N'$ is a monomorphism in $\Q(X)$. Applying $j^*$, we see that $M|_U \otimes N \to M|_U \otimes N'$ is a monomorphism, as desired. The case of geometric stacks is similar and has been proven by Lurie (\cite[Example 5.8]{Lur04}).
\end{proof}

\begin{lemma}[Closure properties of flat objects] \label{flatclosure}
Let $\C$ be a cocomplete linear tensor category.
\begin{enumerate}
\item The zero object $0$ and the unit $\O_\C$ are flat in $\C$.
\item Flat objects are closed under finite direct sums. In particular, for every $d \in \N$, the free $\O_\C$-module $\O_\C^{\oplus d}$ is flat.
\item Assume that in $\C$ directed colimits are exact. If $\{M_i\}$ is a directed system of flat objects in $\C$, then $\varinjlim_i M_i$ is also a flat object in $\C$.
\end{enumerate}
\end{lemma}

\begin{proof}
$1.$ is trivial. For $2.$, assume that $M_1,M_2$ are flat. Then the functor $(M_1 \oplus M_2) \otimes - $ identifies with $(M_1 \otimes - ) \times (M_2 \otimes -)$, which is a product of two left exact functors, hence it is also left exact. For $3.$ we observe that the functor $\varinjlim_i M_i \otimes - \cong \varinjlim_i (M_i \otimes -)$ is exact, since it is a directed colimit of exact functors.
\end{proof}

These closure properties actually capture all flat modules when $\C=\M(R)$ for some ring $R$ -- by the Theorem of Lazard-Govorov (\cite[Theorem 4.34]{Lam99}).

\begin{defi}[Faithfully flatness]
An object $M \in \C$ is called \emph{faithfully flat} if $M \otimes -$ preserves \emph{and} reflects finite limits. In particular $M$ is flat and a given diagram $A \to B \rightrightarrows C$ is exact as soon as $A \otimes M \to B \otimes M \rightrightarrows C \otimes M$ is exact. If $\C$ is abelian, this comes down to the property $A \otimes M = 0 \Rightarrow A = 0$ for all $A \in \C$.
\end{defi}

\begin{rem}
If $A$ is a faithfully flat algebra in a linear tensor category $\C$ and $M \in \C$ is arbitrary, then the beginning of the Amitsur complex (\autoref{Amitsur})
\[0 \to M \to M \otimes A \to M \otimes A \otimes A\]
is exact because it becomes split exact after tensoring with $A$. This implies that $A$ satisfies descent, as we will see in \autoref{descent-theory}. If $\C$ is abelian, even the whole Amitsur complex $(M \otimes A^{\otimes n})_{n \geq 0}$ is exact.
\end{rem}
\section{Dualizable objects}

\begin{defi}[Dualizable objects]
Let $\C$ be a tensor category with unit $\O_\C$ and let $M \in \C$ be an object. Recall (\cite[Definition 4.2.11]{KS06}) that a \emph{dual} of $M$ is an object $M^*$ together with morphisms
\[e : M^* \otimes M \to \O_\C,~ c : \O_\C \to M \otimes M^*\]
such that the two diagrams
\[\xymatrix@C=40pt@R=40pt{M \otimes \O_\C \ar[dr]_{\cong} \ar[r]^-{c \otimes \id_M} & M \otimes M^* \otimes M \ar[d]^{\id_M \otimes e} \\ & M } ~~~~~~ \xymatrix@C=40pt@R=40pt{M^* \otimes \O_\C \ar[dr]_{\cong} \ar[r]^-{\id_{M^*} \otimes c} & M^* \otimes M \otimes M^* \ar[d]^{e \otimes \id_{M^*}} \\ & M }\]
commute (``triangle identities''). We call $(M,M^*,c,e)$ a \emph{duality} in $\C$. If $M$ has a dual, we call $M$ \emph{dualizable}. We think of $e$ as an evaluation and write it as $\omega \otimes m \mapsto \omega(m)$ in element notation.
\end{defi}

\begin{rem}[Equivalent characterizations]
By the Yoneda Lemma, the following are equivalent: 
\begin{enumerate}
\item $M$ is dualizable.
\item There is an object $M^*$ and a morphism $e : M^* \otimes M \to \O_\C$ such that for every $A,B \in \C$ the induced map $\Hom(A,B \otimes M^*) \to \Hom(A \otimes M,B)$ is an isomorphism.
\item There is an object $M^*$ and a morphism $c : \O_\C \to M \otimes M^*$ such that for every $A,B \in \C$ the induced map $\Hom(A \otimes M,B) \to \Hom(A,B \otimes M^*)$ is an isomorphism.
\item There is an object $M^*$ such that $- \otimes M$ is left adjoint to $- \otimes M^*$.
\end{enumerate}
Since $M$ is dual to $M^*$ if and only if $M^*$ is dual to $M$, we can interchange the roles of $M$ and $M^*$ in each of these characterizations.
\end{rem}

\begin{rem}[Perfect pairings] \label{perfpair}
We may call a morphism $e : A \otimes B \to \L$ a \emph{perfect pairing} if there is a morphism $c : \L \to B \otimes A$ such that
\[\xymatrix{A \otimes \L \ar[r]^-{A \otimes c} \ar[dr]_{S_{A,\L}} & A \otimes B \otimes A \ar[d]^{e \otimes A} \\ & \L \otimes A} ~~~~~~~ \xymatrix{\L \otimes B \ar[r]^-{c \otimes B} \ar[dr]_{S_{\L,B}} & B \otimes A \otimes B \ar[d]^{B \otimes e} \\ & B \otimes \L}\]
commute. If $\L$ is invertible (\autoref{invob}), this means that we have a duality between $\L^{\otimes -1} \otimes A$ and $B$. A typical example of a perfect pairing is Serre-duality $H^k(F) \otimes H^{n-k}(F^* \otimes \omega^{\circ}) \to H^n(\omega)$ for a coherent sheaf $F$ on a smooth projective variety of dimension $n$. Another example is given by the exterior power $\wedge : \Lambda^k(V) \otimes \Lambda^{n-k}(V) \to \Lambda^n(V)$ for an $n$-dimensional vector space $V$. See \autoref{lokdual} for a generalization. We do not know if perfect pairings have already been studied for non-invertible $\L$.
\end{rem}

\begin{lemma} \label{dual-flat}
Let $M$ be a dualizable object in a cocomplete tensor category $\C$. Then $M$ is flat. If $\O_\C$ is finitely presentable, then $M$ is also finitely presentable.
\end{lemma}

\begin{proof}
Since $M \otimes -$ is a right adjoint, it preserves all limits. In particular $M$ is flat. If $\O_\C$ is finitely presentable, then $\Hom(M,-) \cong \Hom(\O_\C,- \otimes M^*)$ preserves directed colimits, i.e. $M$ is finitely presentable.
\end{proof}

\begin{prop}[The case of schemes] \label{dual-char}
Let $X$ be a scheme, or even algebraic stack, and $M \in \Q(X)$. Then the following are equivalent:
\begin{enumerate}
\item $M$ is dualizable in the sense above.
\item $M$ is locally free of finite rank.
\item $M$ is flat and of finite presentation.
\item If $U$ is an affine scheme which maps to $X$, then the pullback $M|_U$ is associated to a finitely generated projective module.
\end{enumerate}
\end{prop}

\begin{proof}
If $X$ is affine, then $2. \Leftrightarrow 3. \Leftrightarrow 4.$ are well-known (\cite[Chapter II, \para 5.2]{Bou72}). By descent this suffices for the general case. $1. \Rightarrow 3.$ follows from \autoref{dual-flat}. For $2. \Rightarrow 1.$, consider $M^* := \underline{\Hom}(M,\O_X)$ (this is quasi-coherent by \autoref{qchom}) with the evaluation homomorphism $e : M^* \otimes M \to \O_X$. For  $A,B \in \Q(X)$ we claim that $\Hom(A,B \otimes M^*) \to \Hom(A \otimes M,B)$ is bijective. Descent allows us to reduce to the case $M=\O_X^{\oplus d}$ for some $d$, where both sides identify with $\Hom(A,B)^d$.
\end{proof}

\begin{cor} \label{lokalfrei-bleibt}
Let $X,Y$ be algebraic stacks. If $F : \Q(X) \to \Q(Y)$ is a tensor functor and $M \in \Q(X)$ is locally free of finite rank, then the same is true for $F(M) \in \Q(Y)$.
\end{cor}

\begin{rem}
This corollary will be useful for showing tensoriality of stacks. If $F = f^*$ is a pullback functor, then it is clear anyway that $F$ preserves locally free modules of finite rank. However, for general $F$ this is not clear at all, since $F$ is only defined \emph{globally}. For this, one replaces the \emph{local} notion of locally freeness by the \emph{global} notion of dualizability which is clearly preserved by any tensor functor. This kind of globalization is one of the main themes in this thesis.
\end{rem}

\begin{defi}[Dual morphisms]
If $V,W$ are dualizable objects, we have an isomorphism $\Hom(V,W) \cong \Hom(W^* \otimes V,\O_\C) \cong \Hom(W^*,V^*)$. For a morphism $f : V \to W$ the induced morphism $f^* : W^* \to V^*$ is usually called the \emph{mate} or \emph{dual} of $f$. It is characterized by the commutativity of each one of the following diagrams.
\[\xymatrix@C=40pt{W^* \otimes V \ar[r]^-{f^* \otimes \id_V} \ar[d]_{\id_{W^*} \otimes f} & V^* \otimes V \ar[d]^{e} \\ W^* \otimes W \ar[r]^-{e} & \O_\C} ~~~~ \xymatrix@C=40pt{\O_\C \ar[r]^-{c} \ar[d]_{c} & W \otimes W^* \ar[d]^{\id_W \otimes f^*} \\ V \otimes V^* \ar[r]^-{f \otimes \id_{V^*}} & W \otimes V^*}\]
In element notation, the diagram on the left reads as $f^*(\omega)(v)=\omega(f(v))$.
\end{defi}

\begin{lemma}[Duals invert] \label{dual-inv}
Let $\eta : F \to G$ be a morphism of tensor functors $F,G : \C \to \D$.  If $V \in \C$ is dualizable, then $\eta(V) : F(V) \to G(V)$ is an isomorphism. In particular, if $\C$ is a cocomplete tensor category in which the dualizable objects are colimit-dense, then $\Hom_{c\otimes}(\C,\D)$ is a groupoid.
\end{lemma}

\begin{proof}
One checks that $G(V) \cong G(V^*)^* \xrightarrow{\eta(V^*)^*} F(V^*)^* \cong F(V)$ is inverse to $\eta(V)$. Details can be found in \cite[Theorem 3.2]{LFSW11}.
\end{proof}

\begin{lemma}[Pure exactness] \label{pure}
Let $\C$ be a linear cocomplete tensor category and let
\[U \to V \to W \to 0\]
be an exact sequence of dualizable objects in $\C$. Assume that the dual sequence
\[W^* \to V^* \to U^* \to 0\]
is also exact. Then, the sequence
\[0 \to U \to V \to W \to 0\]
is pure exact, i.e. if $M \in \C$ is an arbitrary object, then the induced sequence
\[0 \to U \otimes M \to V \otimes M \to W \otimes M \to 0\]
is also exact.
\end{lemma}

\begin{proof}
It is clear that $U \otimes M \to V \otimes M \to W \otimes M \to 0$ is exact. So we are left to prove that $0 \to U \otimes M \to V \otimes M \to W \otimes M$ is exact. Let $T \to V \otimes M$ be a morphism which vanishes when composed with $V \otimes M \to W \otimes M$. By duality this corresponds to a morphism $V^* \otimes T \to M$ which vanishes when precomposed with $W^* \otimes T \to M$. Since by assumption
\[W^* \otimes T \to V^* \otimes T \to U^* \otimes T \to 0\]
is exact, the morphism corresponds to a morphism $U^* \otimes T \to M$, which by duality corresponds to a morphism $T \to U \otimes M$.
\end{proof}

\section{Invertible objects} \label{invob}

We recall some general facts about invertible objects (\cite[Chapitre I, 2.5]{SR72}), introduce line objects (following James Dolan) and as usual see what happens for schemes.
 
\begin{defi}[Invertible objects]
Let $\C$ be a tensor category. Recall that an object $\L \in \C$ is called \emph{invertible} if there is an object $\K \in \C$ such that $\L \otimes \K \cong \O_\C$.
\end{defi}
 
Note that $\O_\C$ is invertible and that invertible objects are closed under tensor products, thus we get a tensor category $\Pic(\C)$ of invertible objects in $\C$. Usually $\Pic(\C)$ denotes the monoid (in fact, group) of isomorphism classes of invertible objects (\cite{May01}), but we would like to retain the morphisms and write $\Pic(\C)/{\cong}$ for the monoid of isomorphism classes.
 
If $\L \otimes \K \cong \O_\C$, then $\K$ is unique up to isomorphism, called the inverse of $\L$ and denoted by $\L^{\otimes -1}$. We can then extend the definition of tensor powers to negative powers. Observe that $\L \otimes - : \C \to \C$ is an equivalence of categories (with inverse $\L^{\otimes -1} \otimes -$). For every isomorphism $c : \O_\C \to \L \otimes \L^{\otimes -1}$ there is a unique isomorphism $e : \L^{\otimes -1} \otimes \L \to \O_\C$ such that the triangle identities are satisfied, i.e. $(\L,\L^{\otimes -1},c,e)$ is a duality in $\C$. We refer to \cite{Du13} for coherence results about invertible objects.
 
\begin{ex}[The case of schemes]
Let $X$ be a locally ringed space and $M \in \M(X)$. Then $M$ is invertible (as defined above) if and only if $M$ is locally free of rank $1$. In this case, $M^{\otimes -1} = \underline{\Hom}(M,\O_X)$. This is proven in \cite[Chapitre 0, Proposition 5.5.9]{EGAI} under the assumption that $M$ is of finite type. But this is automatic (I learned this from Tom Goodwillie):

If $M \otimes N \cong \O_X$, then $X$ is covered by open subsets $U$ such that $1 \in \Gamma(U,\O_X)$ has a preimage in $\Gamma(U,M) \otimes \Gamma(U,N)$. Only finitely many sections of $M$ appear. Let $T \subseteq M|_U$ be the submodule generated by these sections. By construction $T \otimes N|_U \to M|_U \otimes N|_U \cong \O_U$ is an epimorphism. Since $N|_U$ is invertible, it follows that $T \to M|_U$ is an epimorphism, i.e. $M|_U = T$. Hence, $M$ is of finite type.
 
The same equivalence follows for quasi-coherent modules on schemes. Using descent it is easy to generalize this to the case of algebraic stacks.
\end{ex}

\begin{defi}[Signature]
If $\L \in \C$ is invertible, then  the canonical morphism $\End(\O_\C) \to \End(\L)$ is an isomorphism (since $\L \otimes -$ is an equivalence). This makes it possible to define the \emph{signature} of an invertible object $\L$ (\cite[Chapitre I, 2.5.3.1]{SR72}), which is the endomorphism $\epsilon(\L) : \O_\C \to \O_\C$ corresponding to the symmetry $S_{\L,\L} : \L \otimes \L \cong \L \otimes \L$ under the canonical isomorphism $\End(\O_\C) \cong \End(\L \otimes \L)$. Since the symmetry is an involution, the same is true for $\epsilon(\L)$. It is easy to check that $\epsilon(L_1 \otimes L_2) \cong \epsilon(L_1) \circ \epsilon(\L_2)$. If $(\L,\L^{\otimes -1},c,e)$ is a duality as above, then the signature equals the composition
\[\O_\C \xrightarrow{c} \L \otimes \L^{\otimes -1} \xrightarrow{S_{\L,\L^{\otimes -1}}} \L^{\otimes -1} \otimes \L \xrightarrow{e} \O_\C.\]
The typical examples deserve a special name:
\end{defi} 
 
\begin{defi}[Line objects]
A \emph{line object} in $\C$ is an invertible object with signature $\id$, i.e. an invertible object which is symtrivial. If $\C$ is linear, then an \emph{anti-line object} in $\C$ is an invertible object with signature $-\id$. It follows from the multiplicativity of the signature above that line (resp. anti-line) objects are closed under tensor products and inverses (which can also be checked directly, see \autoref{symtrivial-properties}), thus they constitute a tensor category $\Pic_+(\C)$ (resp. $\Pic_-(\C)$).
\end{defi}

\begin{ex}
For example, if $X$ is a scheme (or algebraic stack), then we have seen above that in fact every invertible object in $\Q(X)$ is a line object. Outside of algebraic geometry there are natural examples of anti-line objects. In fact, we will see in \autoref{grZ-twist} that the universal example over $R$ is $R$ itself considered as a $\Z$-graded $R$-module concentrated in degree $+1$, where $\Z$-graded modules are endowed with a twisted symmetry.
\end{ex}
 
\begin{rem}
Let $\C$ be a cocomplete $R$-linear tensor category. Assume $2 \in R^*$. If $\L \in \C$ is invertible, then there is a decomposition $\C \cong \C_1 \times \C_2$ under which $\L$ gets mapped to $(\L_+,\L_-)$, where $\L_+$ is a line object and $\L_-$ is an anti-line object. We will not prove this here and only indicate that it follows from \autoref{zerleg} applied to the idempotent $\frac{1+\epsilon(\L)}{2}$.
\end{rem}

The rest of this section is devoted to ``presentations'' of line objects. This will be needed later for projective tensor categories.

\begin{lemma}[Epis cancel] \label{epi-cancel}
Let $s : E \to \L$ be an epimorphism, where $\L$ is invertibe. If $h,h' : A \to B$ are two morphisms with $h \otimes s = h' \otimes s$, then $h=h'$.
\end{lemma}

\begin{proof}
By construction the epimorphism
\[s \otimes \L^{\otimes -1} \otimes A : E \otimes \L^{\otimes -1} \otimes A \to A\]
equalizes $h$ and $h'$. This implies $h=h'$.
\end{proof}

\begin{rem}
In the following, this Lemma will be often used with element notation: If $h(a) \otimes s(e) = h'(a) \otimes s(e)$ for all $e \in E$, then we already have $h(a)=h'(a)$. Intuitively, equality of morphisms may be checked after tensoring with generators of an invertible object. This replaces some of the usual local calculations for invertible sheaves.
\end{rem}
 
Now let $\L \in \C$ be a line object. Assume that $s : E \to \L$ is an epimorphism. Since $\L$ is symtrivial, we have $s(a) \otimes s(b) = s(b) \otimes s(a)$ for $a,b \in E$ in element notation. We may see this as a ``relation between the generators'' $s$ of $\L$. More precisely, consider the two morphisms
\[E \otimes s,\, (E \otimes s) \circ S_{E,E} : E^{\otimes 2} \to E \otimes \L.\]
Since $\L$ is invertible, they correspond to two morphisms 
\[E^{\otimes 2} \otimes \L^{\otimes -1} \rightrightarrows E.\]
That $\L$ is symtrivial means exactly that $s$ coequalizes these morphisms. Notice that if $\C$ is $R$-linear with $2 \in R^*$, then the difference of the two morphisms $E^{\otimes 2} \to E \otimes \L$ mapping $a \otimes b \mapsto a \otimes s(b) - b \otimes s(a)$ in element notation is alternating, hence gives rise to a morphism on the exterior power $\Lambda^2 E$.
 
\begin{lemma}[Good epimorphisms] \label{goodepi}
With the above notations, the following are equivalent:
\begin{enumerate}
\item The sequence $E^{\otimes 2} \otimes \L^{\otimes -1} \rightrightarrows E \xrightarrow{s} \L$ from above is exact.
\item The epimorphism $s$ is regular, i.e. the coequalizer of \emph{some} pair of morphisms.
\end{enumerate}
If $\C$ is $R$-linear with $2 \in R^*$, then 1. is also equivalent to the exactness of $\Lambda^2 E \otimes \L^{\otimes -1} \to E \xrightarrow{s} \L \to 0$.
\end{lemma}
 
Intuitively, in that case, $s(a) \otimes s(b) = s(b) \otimes s(a)$ is the ``only relation'' between the generators $s$ of $\L$.

\begin{proof}
Let $\alpha,\beta : K \to E$ be two morphisms such that $s$ is the coequalizer of $\alpha$ and $\beta$. Let $h : E \to T$ be a morphism which coequalizes $E^{\otimes 2} \otimes \L^{\otimes -1} \rightrightarrows E$, i.e. $h \otimes \L$ coequalizes $E \otimes s$ and $(E \otimes s) \circ S_{E,E} : E^{\otimes 2} \to E \otimes \L$. It follows
\[h \alpha \otimes s = (h \otimes \L) \circ (E \otimes s) \circ (\alpha \otimes E) = (h \otimes \L) \circ (E \otimes s) \circ S_{E,E} \circ (\alpha \otimes E)\]
\[= (h \otimes \L) \circ (E \otimes s) \circ (E \otimes \alpha) \circ S_{K,E} = (h \otimes s \alpha) \circ S_{K,E} = (h \otimes s \beta) \circ S_{K,E}.\]
By the same calulcation backwards this equals $h \beta \otimes s$. By \autoref{epi-cancel} we get $h \alpha = h \beta$. Since $s$ is the coequalizer of $\alpha$ and $\beta$, this means that $h$ factors uniquely through $s$, as desired.

In element notation, the calculation above simplifies as follows: For $a,b \in E$ we have $h(a) \otimes s(b)=h(b) \otimes s(a)$, hence
\[h(\alpha(c)) \otimes s(b) = h(b) \otimes s(\alpha(c))=h(b) \otimes s(\beta(c)) = h(\beta(c)) \otimes s(b).\]
\end{proof}

\begin{cor} \label{abelian-good}
Let $\C$ be an abelian tensor category. If $\L \in \C$ is a line object and $s : E \to \L$ is an epimorphism, then the sequence
\[E^{\otimes 2} \otimes \L^{\otimes -1} \to E \xrightarrow{s} \L \to 0\]
is exact.
\end{cor}
 
\begin{cor} \label{epi-ist-iso}
If $\L$ is a line object in a tensor category, then every regular epimorphism $\O \to \L$ is an isomorphism.
\end{cor}

This property does not hold for all epimorphisms, I have learned this from James Dolan:
 
\begin{ex} \label{jim}
Let $A$ be a commutative ring and $a \in A$. Consider the category $\C$ of $A$-modules $M$ on which $a$ is regular, i.e. $a : M \to M$ is injective. It is a full reflective subcategory of $\M(A)$ with reflector $R : \M(A) \to \C$ given by $R(M) = M/\cup_{n \in \N} \ker(a^n : M \to M)$. One can show that $\C$ becomes a (cocomplete) $A$-linear tensor category with unit $\O := R(A)$ and tensor product $(M,N) \mapsto R(M \otimes_A N)$, see \autoref{M0}. Besides, multiplication with $a$ defines an \emph{epimorphism} $a : \O \to \O$. When $a$ is regular on $A$, we have $\O=A$ and $a : \O \to \O$ is an isomorphism if and only if $a$ is a unit. In particular, $\C$ is not stacky when $a$ is not a unit.
\end{ex}

\begin{lemma}[Discreteness] \label{inv-diskret}
Let $\C$ be a tensor category, let $\L,\L' \in \C$ be two line objects objects and assume that $s : E \to \L$ and $s' : E \to \L'$ are regular epimorphisms. If there is a morphism $\phi : \L \to \L'$ with $\phi \circ s = s'$, this is unique and $\phi$ is an isomorphism. In other words, the category of regular invertible quotients of $E$ is essentially discrete.
\end{lemma}

\begin{proof}
Uniqueness of $\phi$ is clear ($s$ is an epi). According to \autoref{goodepi}, $\phi$ exists if and only if $E^{\otimes 2} \otimes \L^{\otimes -1} \rightrightarrows E \xrightarrow{s'} \L'$ commutes, which in turn means that $s \otimes s' = (s \otimes s') \circ S_{E,E}$ as morphisms $E \otimes E \to \L \otimes \L'$. Composing with the symmetry $S_{\L,\L'}$, we get $(s' \otimes s) \otimes S_{E,E} = (s' \otimes s)$, which is the same equation with $s,s'$ interchanged. Thus, there is a morphism $\L \to \L'$ over $E$ if and only if there is a morphism $\L' \to \L$ over $E$. By uniqueness, they have to be inverse to each other.
\end{proof}

Actually we can improve \autoref{abelian-good} for schemes as follows:
 
\begin{lemma}[Koszul resolution] \label{Koszul}
Let $X$ be a scheme, $\L$ be an invertible sheaf on $X$ and let $s : E \to \L$ be an epimorphism in $\Q(X)$. Then there is a long exact sequence
\[\xymatrix{\dotsc \ar[r] & \Lambda^3 E \otimes \L^{\otimes -3} \ar[r] & \Lambda^2 E \otimes \L^{\otimes -2} \ar[r] & E \otimes \L^{\otimes -1} \ar[r] & \O_X \ar[r] & 0.}\]
The differential $d_n : \Lambda^n E \otimes \L^{\otimes -n} \to \Lambda^{n-1} E \otimes \L^{\otimes -n+1}$ is dual to the morphism $\Lambda^n E \to \Lambda^{n-1} E \otimes \L$ defined by
\[e_1 \wedge \dotsc \wedge e_n \mapsto \sum_{k=1}^{n} (-1)^k (e_1 \wedge \dotsc \wedge \widehat{e_k} \wedge \dotsc \wedge e_n) \otimes s(e_k).\]
\end{lemma}

Remark that if $E$ is locally free of finite rank, this is locally just a Koszul resolution of a regular sequence (\cite[4.5.5]{Wei94}).

\begin{proof}
We omit the calculation that $d_n \circ d_{n+1}=0$, because it is quite standard: It only consists of rearranging two double sums and uses that $\L$ is symtrivial. In order to show exactness, we may look at the stalks and thereby assume that $X$ is the spectrum of a local ring $A$. Then $E$ is just an $A$-module and we may assume $\L=A$. There is some $e \in E$ such that $s(e) \in A^*$. By replacing $s$ with $s(e)^{-1} s$, we may even assume $s(e)=1$. Now define the linear map $t_n : \Lambda^n E \to \Lambda^{n+1} E$ by $a_1 \wedge \dotsc \wedge a_n \mapsto e \wedge a_1 \wedge \dotsc \wedge a_n$ and compute:
\begin{eqnarray*}
&& d_{n+1}(t_n(a_1 \wedge \dotsc \wedge a_n)) \\
&=& (a_1 \wedge \dotsc \wedge a_n) s(e) + \sum_{k=1}^{n} (-1)^k (e \wedge a_1 \wedge \dotsc \wedge \widehat{a_k} \wedge \dotsc \wedge a_n) s(e_k)\\
&=& \id(a_1 \wedge \dotsc \wedge a_n) + t_{n-1}(d_n(a_1 \wedge \dotsc \wedge a_n))
\end{eqnarray*}
Thus, the $t_n$ provide a contraction of the complex. In particular, it is exact.
\end{proof}

\section{Locally free objects} \label{localfree}

In order to globalize locally free sheaves of a given rank, we first have to find an alternative \emph{global} description:

\begin{prop}[Locally free modules] \label{locally-free-sch}
Let $X$ be a scheme (or more generally an algebraic stack) and $V \in \Q(X)$. Let $d \geq 1$. Then the following are equivalent:
\begin{enumerate}
 \item $V$ is locally free of rank $d$.
 \item $\Lambda^d V$ is invertible and the morphism $\omega : \Lambda^{d+1}(V) \to V \otimes \Lambda^d(V)$ constructed in \autoref{hilfs-omega} vanishes.
 \item $\Lambda^d V$ is invertible.
\end{enumerate}
In that case, $V$ is dualizable with $V^* \cong \Lambda^{d-1}(V) \otimes (\Lambda^d(V))^{\otimes -1}$.
\end{prop}

\begin{proof}

$1. \Rightarrow 2.$ This is well-known. We may assume that $V$ is free of rank $d$. But then $\Lambda^d V$ is free of rank $1$ and $\Lambda^{d+1}(V)=0$. $2. \Rightarrow 3.$ is trivial. So we only have to prove $3. \Rightarrow 1.$ I have learnt the following argument from Neil Strickland.

We may work locally on $X$ and thereby assume that $\Lambda^d V$ is free of rank $1$. Let $\sum_i w_i$ be a generator, where $w_i$ are pure wedges. After localizing at the images of $w_i$ in $\Lambda^d V \cong \O_X$, we may as well assume that $\Lambda^d V$ is generated by a pure wedge $v_1 \wedge \dotsc \wedge v_d$. Consider the linear map
\[\phi : V \to (\Lambda^d V)^d,\, v \mapsto (v_1 \wedge \dotsc \wedge v_{i-1} \wedge v \wedge v_{i+1} \wedge \dotsc \wedge v_d)_{i=1,\dotsc,d}.\]
It satisfies $\phi(v_i)=e_i (v_1 \wedge \dotsc \wedge v_d)$. It follows that $v_1,\dotsc,v_d$ generate a free submodule $U:=\langle v_1,\dotsc,v_d \rangle$ of $V$ with $V = U \oplus \ker(\phi)$.
By construction $U \hookrightarrow V$ induces an isomorphism $\Lambda^d U \cong \Lambda^d V$. Using \autoref{sym-exakt} it follows that $\Lambda^{d-1}(U) \otimes \ker(\phi)=0$. Since $\Lambda^{d-1}(U)$ is free of rank $\binom{d}{d-1}=d > 0$ and therefore faithfully flat, this shows $\ker(\phi)=0$, i.e. $V=U$ is free. 

Now let $V$ be locally free of rank $d$. By \autoref{dual-char}, $V$ is dualizable with $V^* = \HOM(V,\O_X)$. The homomorphism from \autoref{hilfs-omega} (for schemes we do not need that $2$ is invertible here) 
\[\Lambda^d V \to V \otimes \Lambda^{d-1} V\]
dualizes to the homomorphism
\[\delta : V^* \otimes \Lambda^d V \to \Lambda^{d-1} V\]
which maps $\phi \otimes (v_1 \wedge \dotsc \wedge v_d)$ to $\sum_k \pm \phi(v_k) \cdot (v_1 \wedge \dotsc \wedge \widehat{v_k} \wedge \dotsc \wedge v_d)$.

In order to show that it is an isomorphism, we may work locally on $X$. So let us assume that $V$ is free with basis $e_1,\dotsc,e_d$. Let $e_1^*,\dotsc,e_d^*$ be the dual basis of $V^*$. Then $\delta$ maps the basis element $e_i^* \otimes (e_1 \wedge \dotsc \wedge e_d)$ to $\pm e_1 \wedge \dotsc \wedge \widehat{e_i} \wedge \dotsc \wedge e_d$. But the latter elements form a basis of $\Lambda^{d-1} V$.
\end{proof}

This motivates the following definition:

\begin{defi}[Locally free objects] \label{lokfreedef}
Let $R$ be a commutative ring with $2 \in R^*$. Let $\C$ be a finitely cocomplete $R$-linear tensor category, $V \in \C$ and $d \geq 1$. We call $V$ \emph{locally free of rank $d$} if $\Lambda^d(V)$ is invertible and the morphism $\omega : \Lambda^{d+1}(V) \to V \otimes \Lambda^d(V)$ vanishes.
\end{defi}

\begin{rem} \label{rel-frei}
In element notation, the vanishing of $\omega$ means that for all $v_0,\dotsc,v_d \in V$ we have \marginpar{I've replaced v_1 by v_0 here.}
\[(\star) ~ ~ \sum_{k=0}^{d} (-1)^k v_k \otimes (v_0 \wedge \dotsc \wedge \widehat{v_k} \wedge \dotsc \wedge v_d)=0.\] 
For $d=1$ this means $v_0 \otimes v_1 = v_1 \otimes v_0$, i.e. that $V$ is symtrivial. It follows that locally free objects of rank $1$ are precisely the line objects.

Notice that for $\C=\M(R)$ and $V=R^{\oplus d}$, by Laplace expansion of the first row the relation $(\star)$ becomes the vanishing of the determinant
\[0 = \det \begin{pmatrix} v_{i0} & \dotsc & v_{id} \\ v_{10} & \dotsc & v_{1d} \\ \vdots & & \vdots \\ v_{d0} & \cdots & v_{dd} \end{pmatrix}\]
for all $1 \leq i \leq d$. Thus, in some sense, our abstract definition of locally free objects includes that this determinant equation still holds.

Also notice that $(\star)$ implies $(d+1) \cdot (v_0 \wedge \dotsc \wedge v_d)=0$. \marginpar{I've added this paragraph.} This implies $\Lambda^{d+1} V = 0$ at least if $d+1 \in R^*$. Therefore, $V$ is locally free of rank $d$ if and only if $\Lambda^d(V)$ is invertible and $\Lambda^{d+1}(V)=0$. This could serve as an alternative definition.
\end{rem}

\begin{ex}
From \autoref{sym-exakt} we conclude $\Lambda^d(\O^d) \cong \O$. The isomorphism maps $a_1 \wedge \dotsc \wedge a_d$ to the determinant $\sum_{\sigma \in \Sigma_d} \sgn(\sigma) \prod_{i=1}^{d} a_{i,\sigma(i)}$. Hence, $\O^d$ is (locally) free of rank $d$.
\end{ex}
 
In the following, let us fix $\C$ and $d \geq 1$ as above such that $d! \in R^*$ (for example when $R$ is a $\QQ$-algebra).

\begin{prop} \label{det}
Let $V \in \C$ be locally free of rank $d$. Then $\Lambda^d(V)$ is a line object. We call it the \emph{determinant} of $V$.
\end{prop}

\begin{proof}
We only have to prove that $\Lambda^d(V)$ is symtrivial. This follows from \autoref{symlem} with $p=d$. The assumption of this Lemma is precisely $(\star)$.
\end{proof}

\begin{prop}[Locally free $\Rightarrow$ dualizable] \label{lokdual}
Let $V \in \C$ be locally free of rank $d$. If $1 \leq p \leq d$, then $\Lambda^p(V)$ is dualizable with
\[\Lambda^p(V)^* \cong \Lambda^{d-p}(V) \otimes \Lambda^d(V)^{\otimes -1}.\]
In particular, $V$ itself is dualizable with $V^* \cong \Lambda^{d-1}(V) \otimes \Lambda^d(V)^{\otimes -1}$.
\end{prop}

\begin{proof}
Since $\Lambda^d(V)$ is invertible, we only have to construct morphisms
\begin{eqnarray*}
e :& \Lambda^{d-p}(V) \otimes \Lambda^p(V) \to \Lambda^d(V)\\
c :& \Lambda^d(V) \to \Lambda^p(V) \otimes \Lambda^{d-p}(V) 
\end{eqnarray*}
such that the two diagrams
\[(\dagger) ~~~ \xymatrix@R=50pt@C=35pt{\Lambda^d(V) \otimes \Lambda^p(V) \ar[r]^-{c \otimes V} \ar[dr]_{\cong} & \Lambda^p(V) \otimes \Lambda^{d-p}(V) \otimes \Lambda^p(V) \ar[d]^{\Lambda^p(V) \otimes e} \\ & \Lambda^p(V) \otimes \Lambda^d(V) }\]
\[(\ddagger) ~~~ \xymatrix@R=50pt@C=35pt{\Lambda^{d-p}(V) \otimes \Lambda^d(V) \ar[r]^-{\Lambda^{d-p}(V) \otimes c} \ar[dr]_{\cong} & \Lambda^{d-p}(V) \otimes \Lambda^{p}(V) \otimes \Lambda^{d-p}(V) \ar[d]^{e \otimes \Lambda^{d-p}(V)} \\ & \Lambda^d(V) \otimes \Lambda^{d-p}(V)} \]
commute (see \autoref{perfpair}). We choose $e$ to be the multiplication from the exterior algebra (\autoref{exterior-algebra}). The morphism $c$ comes from the comultiplication, see \autoref{antipaar}. Then it is readily checked that $(\dagger)$ is precisely the statement of \autoref{symlem} and $(\ddagger)$ is simply $(\dagger)$ with $p$ replaced by $d-p$.
\end{proof}

\begin{prop}[Cramer's rule] \label{cramer}
Let $f : V \to W$ be a morphism between locally free objects of rank $d$. If $\Lambda^d f : \Lambda^d V \to \Lambda^d W$ is an isomorphism, then $f$ is an isomorphism.
\end{prop}

\begin{proof}[Sketch of proof.]
We may invert and then dualize $\Lambda^d f$ to obtain an isomorphism $(\Lambda^d V)^* \cong (\Lambda^d W)^*$. The composition
\[V^* \cong \Lambda^{d-1} V \otimes  (\Lambda^d V) ^* \cong \Lambda^{d-1} V \otimes (\Lambda^d W)^* \xrightarrow{\Lambda^{d-1} f} \Lambda^{d-1} W \otimes (\Lambda^d W)^* \cong W^*\]
is inverse to $f^*$. Hence, $f^*$ is an isomorphism, but then also $f=f^{**}$.
\end{proof}

Next, we give a more ``geometric'' characterization of locally free objects, which is motivated by the fact that a quasi-coherent module is locally free of rank $d$ if and only if it becomes free of rank $d$ after tensoring with a faithfully flat algebra. This has been a collaboration with Daniel Sch\"appi. We begin with the case $d=1$.

\begin{defi} \label{SymZ}
Let $\L \in \C$ be a line object. Define $\Sym^{\Z}(\L)$ to be the $\Z$-graded commutative algebra $\bigoplus_{n \in \Z} \L^{\otimes n}$ whose unit is the inclusion from $\L^{\otimes 0} \cong \O$ and whose multiplication is  induced by the isomorphisms $\L^{\otimes n } \otimes \L^{\otimes m} \cong \L^{\otimes n+m}$ for $n,m \in \Z$. Then clearly $\Sym^{\Z}(\L) \otimes \L \cong \Sym^\Z(\L)$ as modules over $\Sym^{\Z}(\L)$. In fact, this algebra is universal with this property:
\end{defi}

\begin{prop} \label{lintriv}
Let $\L \in \C$ be a line object. If $A$ is a commutative algebra in $\C$, then homomorphisms of algebras $\Sym^\Z(\L) \to A$ correspond to isomorphisms of $A$-modules $\L \otimes A \cong A$.
\end{prop}

\begin{proof}[Sketch of proof.]
A homomorphism of algebras $\Sym^\Z(\L) \to A$ is determined by two morphisms $\L \to A$ and $\L^{\otimes -1} \to A$ in $\C$ such that tensoring them (in both orders) results in the unit of $A \otimes A$. These correspond to morphisms of $A$-modules $\L \otimes A \to A$ and $A \to \L \otimes A$ which are inverse to each other.
\end{proof}

\begin{defi}
Let $V \in \C$ be locally free of rank $d$ with determinant $\L$. Define $\delta : \L \to \Sym^d(V^d)$ (``determinant'') by
\[\delta(v_1 \wedge \dotsc \wedge v_d) = \sum_{\sigma \in \Sigma_d} \sgn(\sigma)  \prod_{i=1}^{d} \iota_i(v_{\sigma(i)})\]
in element notation. Here, $\iota_i : V \hookrightarrow V^d$ denotes the inclusion of the $i$th summand. Then $\delta$ induces a homomorphism $\Sym(\L) \to \Sym(V^d)$ of commutative algebras. We define the commutative algebra
\[H(V) := \Sym(V^d) \otimes_{\Sym(\L)} \Sym^\Z(\L).\]
\end{defi}

\begin{prop}[Universal trivializing algebra]
Let $V \in \C$ be locally free of rank $d$ and let $A$ be a commutative algebra in $\C$. Then, homomorphisms of commutative algebras $H(V^*) \to A$ correspond to isomorphisms of $A$-modules $V \otimes A \cong A^d$. Thus, $H(V^*)$ is the universal commutative algebra which makes $V$ free of rank $d$.
\end{prop}

\begin{proof}[Sketch of proof.]
A homomorphism $H(V^*) \to A$ of algebras corresponds to homomorphisms of algebras $\Sym((V^*)^d) \to A$ and $\Sym^\Z(\L^*) \to A$ which agree on $\Sym(\L^*)$. Using the universal property of symmetric algebras (\autoref{symalg}) as well as the case of line objects (\autoref{lintriv}), this corresponds to a morphism $(V^*)^d \to A$ and an isomorphism $A \otimes \L \cong A$ of $A$-modules such that a certain diagram commutes. This corresponds to a homomorphism of $A$-modules $A^d \to V \otimes A$ which induces an isomorphism on $\Lambda_A^d$, i.e. which is an isomorphism by \autoref{cramer}.
\end{proof}

\begin{rem}[Open questions]
We do not know if $H(V)$ is faithfully flat. At least for $d=1$ this is easy to check as soon as countable direct sums are exact in $\C$. It would be interesting to investigate if further well-known properties of locally free sheaves generalize to locally free objects of tensor categories. A very basic question is if they are closed under finite direct sums and tensor products. It would be also interesting to find a connection to Deligne's notion of rank in Tannakian categories (\cite{Del90}). Finally, it would be desirable to get rid of the assumption $d! \in R^*$.
\end{rem}

\section{Descent theory} \label{descent-theory}

Let $\C$ be a cocomplete tensor category (not assumed to be linear). We develope descent theory internal to $\C$, generalizing the well-known case of $R$-modules for a commutative ring $R$ (\cite{Vis05}). This will be needed in \autoref{tensoriality}, in particular the notion of a (special) descent algebra.

\begin{defi}[Descent data]
Let $A$ be a commutative algebra in $\C$.
\begin{enumerate}
\item 
The category of \emph{descent-data} $\Desc(A)$ has as objects pairs $(N,\phi)$, where $N \in \M(A)$ and $\phi : A \otimes N \cong N \otimes A$ is an isomorphism of $A \otimes A$-modules, such that the cocycle condition is satisfied: Define the three isomorphisms
\[\begin{array}{lll}
\phi_1 & := &  A \otimes N \otimes A \xrightarrow{\phi \otimes A} N \otimes A \otimes A,\bigskip \\
\phi_2 & := & A \otimes A \otimes N \hspace{10mm} N \otimes A \otimes A,\\
&& {\scriptstyle A \otimes S_{A,N}} \downarrow  \hspace{25mm} \uparrow {\scriptstyle N \otimes S_{A,A}}\\
&& A \otimes N \otimes A \xrightarrow{\phi \otimes A} N \otimes A \otimes A  \bigskip \\
\phi_3 & := & A \otimes A \otimes N \xrightarrow{A \otimes \phi} A \otimes N \otimes A.
\end{array}\]
Then it is required that the diagram
\[\xymatrix@R=40pt@C=20pt{A \otimes A \otimes N \ar[rr]^{\phi_3} \ar[dr]_{\phi_2} && A \otimes N \otimes A \ar[dl]^{\phi_1} \\ & N \otimes A \otimes A & }\]
commutes. A morphism $(M,\phi) \to (N,\psi)$ of descent-data is defined to be a morphism $f : M \to N$ in $\M(A)$ such that the diagram
\[\xymatrix{A \otimes M \ar[r]^{\phi} \ar[d]_{A \otimes f} & M \otimes A \ar[d]^{f \otimes A} \\ A \otimes N \ar[r]^{\psi} & N \otimes A}\]
commutes.
\item There is a functor $F : \C \to \Desc(A)$ mapping $M \mapsto (M \otimes A, \psi)$, where $\psi$ is the composition
\[A \otimes (M \otimes A) \cong (A \otimes M) \otimes A \cong (M \otimes A) \otimes A,\]
given by $\psi(a \otimes m \otimes b) = m \otimes a \otimes b$ in element notation. We call $A$ a \emph{descent algebra} if this functor is an equivalence of categories
\[\C \simeq \Desc(A).\]
\end{enumerate}
\end{defi}
 
In element notation one easily checks that $\psi : A \otimes M \otimes A \cong M \otimes A \otimes A$ in 2. above satisfies the cocycle condition. In fact we have
\begin{eqnarray*}
\psi_1(a \otimes m \otimes b \otimes c) &= &m \otimes a \otimes b \otimes c,\\
\psi_2(a \otimes b \otimes m \otimes c) &=& m \otimes a \otimes b \otimes c,\\
\psi_3(a \otimes b \otimes m \otimes c) &=& a \otimes m \otimes b \otimes c.
\end{eqnarray*}

\begin{rem} \label{descfaith}
In general, the forgetful functor $\Desc(A) \to \M(A)$ is faithful and conservative (i.e. reflects isomorphisms). Hence, if $A$ is a descent algebra, then $\C \to \M(A)$, $M \mapsto M \otimes A$ is faithful and conservative.
\end{rem}

\begin{prop}[Descent of properties]
Let $A$ be a descent algebra in $\C$. If $M \in \C$ is an object such that $M \otimes A \in \M(A)$ is either
\begin{enumerate}
\item symtrivial,
\item dualizable,
\item invertible,
\item a line object,
\item or locally free of rank $d$ (if $\C$ is $R$-linear with $2 \in R^*$)
\end{enumerate}
then the same is true for $M \in \C$.
\end{prop}

\begin{proof}
1. The symmetry $S_{M \otimes A,M \otimes A}$ of the $A$-module $M \otimes A$ equals the symmetry $S_{M,M}$ tensored with $A$. Since it is the identity, it follows that $S_{M,M}$ is also the identity by \autoref{descfaith}.

2. Assume that $M \otimes A$ is dualizable with dual $N \in \M(A)$. We define an isomorphism of $A \otimes A$-modules $N \otimes A \cong A \otimes N$ to be dual to the canonical isomorphism of $A \otimes A$-modules $M \otimes A \otimes A \cong A \otimes M \otimes A$. The cocycle condition is satisfied, so that there is some $M^* \in \C$ with an isomorphism of $A$-modules $M^* \otimes A \cong N = (M \otimes A)^*$ which is compatible with the descent data. The evaluation $(M^* \otimes M) \otimes A \cong (M^* \otimes A) \otimes_A (M \otimes A) \to A$ is compatible with the descent data, so that it is induced by a morphism $M^* \otimes M \to \O$. Likewise one constructs a morphism $\O \to M \otimes M^*$. The triangle identities hold: By \autoref{descfaith} it suffices to check this after tensoring with $A$, but then the identities hold by construction.

3. If $M \otimes A$ is an invertible object of $\M(A)$, this means that it is dualizable such that the evaluation $(M \otimes A)^* \otimes_A (M \otimes A) \to A$ is an isomorphism. By 2. and \autoref{descfaith} this means that $M$ is dualizable such that the evaluation $M^* \otimes M \to \O$ is an isomorphism, i.e. that $M$ is invertible.

4. This follows from 1. and 3.

5. Keeping in mind \autoref{lokfreedef}, this follows from 3. and \autoref{descfaith}.
\end{proof}

\begin{prop}[Adjunction] \label{desc-adjoint}
Assume that $\C$ has equalizers. Let $A$ be a commutative algebra in $\C$ with unit $u : \O \to A$. Then the canonical functor $F : \C \to \Desc(A)$ has a right adjoint $G$ which maps $(N,\phi)$ to the equalizer of the two morphisms
\[N \xrightarrow{u \otimes N} A \otimes N \xrightarrow{\phi} N \otimes A,~N \xrightarrow{u \otimes N} A \otimes N \xrightarrow{S_{A,N}} N \otimes A.\]
\end{prop}

\begin{proof}
Clearly $G : \Desc(A) \to \C$ defines a functor. In element notation, we may write $G(N,\phi) = \{n \in N : \phi(1 \otimes n) = n \otimes 1\}$.

If $M \in \C$, then $G(F(M)) = G(M \otimes A,\phi)$ is the equalizer of the two canonical morphisms $M \otimes u \otimes A, \, M \otimes A \otimes u : M \otimes A \rightrightarrows M \otimes A \otimes A$ given by (in element notation) $m \otimes a \mapsto m \otimes 1 \otimes a$ resp. $m \otimes a \otimes 1$. Since $M \otimes u : M \to M \otimes A$ equalizes both morphisms (the composite being $M \otimes u \otimes u$ in each case), it induces a canonical morphism $M \to G(F(M))$.
 
Now let $(N,\phi) \in \Desc(A)$ and $M = G(N,\phi)$. This equalizer comes equipped with a monomorphism $i : M \to N$ in $\C$. It induces a morphism of $A$-modules $f : M \otimes A \to N$ described by $f(m \otimes a)=i(m) \cdot a$ in element notation. We claim that $F(G(N,\phi))=(M \otimes A,\psi) \to (N,\phi)$ is a morphism in $\Desc(A)$, i.e. that
\[\xymatrix@C=40pt{A \otimes M \otimes A \ar[r]^-{A \otimes f} \ar[d]_{S_{A,M} \otimes A} & A \otimes N \ar[d]^{\phi} \\ M \otimes A \otimes A \ar[r]^-{f \otimes A} & N \otimes A}\]
commutes. Let us use element notation. Then $a \otimes m \otimes b \in A \otimes M \otimes A$ goes to $a \otimes (i(m) \cdot b) \in A \otimes N$ and then to $\phi(a \otimes (i(m) \cdot b)) \in N \otimes A$ and on the other way it goes to $m \otimes a \otimes b \in M \otimes A \otimes A$ and then to $(i(m) \cdot a) \otimes b \in N \otimes A$. But these agree since
\[\phi(a \otimes (i(m) \cdot b)) = a \cdot \phi(1 \otimes i(m)) \cdot b = a \cdot (i(m) \otimes 1) \cdot b = (i(m) \cdot a) \otimes b.\]
As always this element calculation may be also replaced by a suitable large commutative diagram. It is straightforward to check the triangular identities, so that in fact $F$ is left adjoint to $G$ with unit $M \to G(F(M))$ and counit $F(G(N,\phi)) \to (N,\psi)$ constructed as above.
\end{proof}

\begin{prop}[Faithfully flat descent] \label{desc-krit}
Assume that $\C$ has equalizers. Let $A$ be a commutative algebra in $\C$ such that
\begin{enumerate}
\item $A$ is a flat object of $\C$.
\item For every $M \in \C$ the canonical sequence
\[M \to M \otimes A \rightrightarrows M \otimes A \otimes A\]
is exact.
\end{enumerate}
Then $A$ is a descent algebra.

In particular, every faithfully flat commutative algebra is a descent algebra.
\end{prop}

We recall that the morphisms in the canonical sequence are defined by $m \mapsto m \otimes 1$ and $m \otimes a \mapsto m \otimes 1 \otimes a$ resp. $m \otimes a \otimes 1$ in element notation. The result follows from Beck's monadicity Theorem (\cite[VI.7]{ML98}), but we will give a more direct proof below.

\begin{proof}
We use the proof of \autoref{desc-adjoint} as well as its notation. The second condition says exactly that the unit $M \to G(F(M))$ is an isomorphism for every $M \in \C$. We claim that the counit $f : F(M) \to (N,\phi)$ is also an isomorphism for all $(N,\phi) \in \Desc(A)$, where $M:=G(N,\psi)$. Consider the following diagram $(\star)$:

\[\xymatrix@C=60pt@R=60pt{M \otimes A  \ar[r]^{i \otimes A} \ar[d]^{f} & N \otimes A \ar@/_1pc/[r]_{(S_{A,N} \circ (u \otimes N)) \otimes A} \ar@/^1pc/[r]^{(\phi \circ (u \otimes N)) \otimes A} \ar[d]^{\phi^{-1}} & N \otimes A \otimes A \ar[d]_{\phi_2^{-1}} \\
N \ar[r]^{u \otimes N} & A \otimes N \ar@/_1pc/[r]_{A \otimes u \otimes N} \ar@/^1pc/[r]^{u \otimes A \otimes N} & A \otimes A \otimes N}\]

Assume for the moment that it is commutative. The row below is exact by assumption, the row above is exact by definition of the equalizer $i : M \to N$ and since $A$ is flat by assumption. The two vertical morphisms $\phi^{-1}$ and $\phi_2^{-1}$ are isomorphisms. Thus, $f$ is an isomorphism and we are done.

Now let us check the commutativity of the three squares involved in $(\star)$. For $m \otimes a \in M \otimes A$ in element notation have
\[\phi^{-1}(i(m) \otimes a) = \phi^{-1}((i(m) \otimes 1)) \cdot a=(1 \otimes m) \cdot a = 1 \otimes m \cdot a = (u \otimes N)(f(m \otimes a)).\]
This proves that the square on the left commutes. The square on the right in the front commutes because its inverse can be decomposed as
\[\xymatrix@C=50pt@C=50pt{A \otimes N \ar[r]^-{A \otimes M \otimes u} \ar[d]^{\phi} & A \otimes N \otimes A \ar[d]^{\phi \otimes A} & A \otimes A \otimes N \ar[l]_-{A \otimes S_{A,N}} \ar[d]^{\phi_2} \\ N \otimes A \ar[r]^-{M \otimes A \otimes u} & N \otimes A \otimes A \ar[r]^-{N \otimes S_{A,A}} & N \otimes A \otimes A.}\]
The inverse of the square on the right in the back of $(\star)$ can be decomposed as follows:
\[\xymatrix@C=60pt@R=50pt{A \otimes N \ar[rr]^{u \otimes A \otimes N}  \ar[d]^{\phi}  & & A \otimes A \otimes N \ar[d]^{\phi_2} \ar[dl]_{A \otimes \phi} \\
N \otimes A \ar[r]^{u \otimes N \otimes A} & A \otimes N \otimes  A \ar[r]^{\phi \otimes A} & N \otimes A \otimes A}\]
The commutativity of the region on the right is precisely the cocycle condition, whereas the region on the left commutes because for trivial reasons.
\end{proof}

One drawback of descent algebras and faithfully flat algebras is that it remains unclear if they are \emph{universal} -- in the sense that they are preserved by arbitrary cocontinuous tensor functors. The following notion of a \emph{special descent algebra} is universal by construction and is motivated by the proof of \cite[Theorem 1.3.2]{Sch12a}. 

\begin{defi}[Special descent algebras]
Let $\C$ be a cocomplete linear tensor category. A commutative algebra $A$ in $\C$ is called a \emph{special descent algebra} if the underlying object of $A$ may be written as a directed colimit of dualizable objects $A_i$ such that for all $i$
\begin{itemize}
\item the unit $\O \to A$ factors naturally as $\O \to A_i \to A$,
\item the object $A_i / \O := \coker(\O \to A_i)$ is dualizable
\item and the sequence $(A_i/\O)^* \to A_i^* \to \O^* \to 0$ dual to the exact sequence $\O \to A_i \to A_i/\O \to 0$ is also exact.
\end{itemize}
\end{defi}

\begin{rem} \label{specstab}
Notice that if $F : \C \to \D$ is any cocontinuous linear tensor functor and $A$ is a special descent algebra in $\C$, then $F(A)$ is also a special descent algebra in $\D$.
\end{rem}
 
\begin{prop} \label{specdes}
Assume that $\C$ is a cocomplete linear tensor category with equalizers such that directed colimits are exact in $\C$ (for example, if $\C$ is locally finitely presentable). Then every special descent algebra in $\C$ is a descent algebra in $\C$.
\end{prop}
 
\begin{proof}
Let $A$ be a special descent algebra in $\C$. We use \autoref{desc-krit} to show that $A$ is a descent algebra. By \autoref{dual-flat} and \autoref{flatclosure}, $A$ is a flat object of $\C$. Let $M \in \C$, we want to prove that
\[0 \to M \to M \otimes A \to M \otimes A \otimes A\]
is exact. By \autoref{pure} we know that for every $i$  the sequence
\[0 \to M \to M \otimes A_i \to M \otimes A_i / \O \to 0\]
is exact (in particular $\O \to A_i$ is a monomorphism, which justifies the notation $A_i/\O$). In the colimit we get an exact sequence
\[0 \to M \to M \otimes A \to M \otimes A/\O \to 0.\]
In particular, we see that $f \otimes A = 0 \Rightarrow f = 0$ for every morphism $f$ in $\C$.

Clearly $0 \to M \to M \otimes A \to M \otimes A \otimes A$ is exact after tensoring with $A$ (in fact, it even becomes split exact). Thus, it suffices to prove the following more general claim: If $0 \to M_1 \to M_2 \to M_3$ is a complex in $\C$, which becomes exact after tensoring with $A$ and such that $M_1 \to M_2$ is a regular monomorphism, then $0 \to M_1 \to M_2 \to M_3$ is also exact. For the proof, take a morphism $T \to M_2$ such that $T \to M_2 \to M_3$ vanishes. Choose an exact sequence $0 \to M_1 \to M_2 \to M_4$. We want to prove that $T \to M_2 \to M_4$ vanishes. It suffices to prove that $T \otimes A \to M_2 \otimes A \to M_4 \otimes A$ vanishes. In fact $T \otimes A \to M_2 \otimes A$ factors through $M_1 \otimes A \to M_2 \otimes A$. This is because $T \otimes A \to M_2 \otimes A \to M_3 \otimes A$ vanishes and $0 \to M_1 \otimes A \to M_2 \otimes A \to M_3 \otimes A$ is exact.
\end{proof}

The following example may be extracted from the proof of \cite[Theorem 1.3.2]{Sch12a}.

\begin{prop} \label{adamsspec}
Let $X$ be an Adams stack (see \autoref{density}). Choose a presentation $\pi : P \to X$ such that $\pi$ and $P$ are affine. Then $A=\pi_* \O_P$ is a special descent algebra in $\Q(X)$.
\end{prop}

\begin{proof}
By \autoref{strong-resolution} we may choose some directed system of locally free quasi-coherent modules $A_i \in \Q(X)$ with colimit $A$. By \autoref{dual-char} each $A_i$ is dualizable. Since $\O_X$ is finitely presentable, the unit $\O_X \to A$ factors through some $A_j$. For $i \geq j$ it factors as $\O_X \to A_j \to A_i \to A$. In the following we may restrict to the cofinal subset of all $i \geq j$. Notice that $\pi^* \O_X \to \pi^* A$ is a split monomorphism (with splitting given by $\pi^* A = \pi^* \pi_* \O_P \to \O_P$), the same is true for $\pi^* \O_X \to \pi^* A_i$. It follows by descent that $\O_X \to A_i$ is a monomorphism. The cokernel $A_i / \O_X$ is locally free, because $\pi^* (A_i / \O_X)$ is a direct summand of $\pi^* A_i$ and therefore locally free. We have seen that $0 \to \O_X \to A_i \to A_i / \O_X \to 0$ is an exact sequence of locally free modules which splits locally, so that its dual sequence $0 \to (A_i/\O_X)^* \to A_i^* \to \O_X^* \to 0$ is also exact. This shows that $A$ is a special descent algebra.
\end{proof}





\section{Cohomology} \label{cohomo}

\begin{defi}[Cohomology]
Let $\C$ be a cocomplete $R$-linear tensor category with unit $\O_\C$. Then we have the \emph{global section} functor
\[\Gamma : \C \to \M(R),~ A \mapsto \Gamma(A) := \Hom(\O_\C,A).\]
It is left exact. Assume that the underlying category of $\C$ is abelian and has enough injectives (for example, when $\C$ is a Grothendieck category). Then we may consider the derived functors
\[H^p := R^p \Gamma : \C \to \M(R).\]
The abelian group $H^p(M)$ is the \emph{$p$-th cohomology} of $M$ (with respect to $\C$).
\end{defi}

\begin{ex}
When $\C=\M(X)$ for some ringed space $X$, this is the usual definition of sheaf cohomology. If $\C=\Q(X)$ for some scheme which is either noetherian or quasi-compact and semi-separated, then the cohomology with respect to $\Q(X)$ agrees with the one with respect to $\M(X)$ (\cite[Proposition B.8]{TT90}).
\end{ex}

\begin{defi}[Higher direct images]
Let $F : \D \to \C$ be a cocontinuous $R$-linear tensor functor between abelian cocomplete $R$-linear tensor categories with enough injectives. If $F$ has a right adjoint $F_* : \C \to \D$, then we may define the \emph{higher direct image} of $F$ as the derived functor $R^p F_* : \C \to \D$.
\end{defi}

\begin{ex}
If $\D=\M(R)$, then by (\autoref{initial}) $F$ is essentially unique and $F_* = \Gamma$, so that $R^p F_* = H^p$. If $f : X \to Y$ is a morphism of ringed spaces and $F := f^* : \M(Y) \to \M(X)$, then $F_* = f_*$ and we obtain the usual higher direct image functor $R^p f_*$.
\end{ex}

\begin{rem}
One can define cup products in the general setting of abelian tensor categories (\cite[Section 2]{Loe09}).
\end{rem}

\begin{rem}
In \cite{Sch13} $F : \D \to \C$ is called \emph{cohomological affine} if $F_*$ is faithful and right exact, i.e. $R^p F_* = 0$ for $p>0$. Every affine $F$ (\autoref{affinefunk}) is cohomological affine and the converse is true if $\C$ is locally presentable and is generated by dualizable objects (\cite[Proposition 3.8]{Sch13}). This can be seen as a tensor categorical version of Serre's criterion for affineness.
\end{rem}

\begin{rem}
Artin and Zhang (\cite[Theorem 4.5]{AZ94}) have found a cohomological characterization of those abelian categories which arise from projective schemes (subject to a certain technical condition called $\chi_1$). It would be interesting to find an analogous characterization for abelian tensor categories. The special object $\mathcal{A}$ should be the tensor unit $\mathcal{O}$ so that $H^0$ becomes $\Gamma$ and the shift operator $s$ should be tensoring with an ample line object.
\end{rem}

\clearpage
\chapter{Constructions with cocomplete tensor categories} \label{Cons}
\section{Basic free constructions} \label{free}

\subsection{Free tensor categories}
 
\begin{rem}[Free monoidal categories]
We construct a left adjoint of the  forgetful $2$-functor from monoidal categories (with strong monoidal functors) to categories as follows. Given a category $\C$, we construct the free monoidal category $M(\C)$ on $\C$ as follows: Objects are tuples $(X_1,\dotsc,X_n)$ of objects in $\C$. We allow $n=0$ which yields the unit object. The tensor product is simply
\[(X_1,\dotsc,X_n) \otimes (Y_1,\dotsc,Y_m) = (X_1,\dotsc,X_n,Y_1,\dotsc,Y_m).\]
Morphisms only exist between tuples of the same length and then they are just ``diagonal'', i.e.
\[\Hom((X_1,\dotsc,X_n),(Y_1,\dotsc,Y_n)) := \Hom(X_1,Y_1) \times \dotsc \times \Hom(X_n,Y_n).\]
Given a functor $F : \C \to \D$ into (the underlying category of) a monoidal category $\D$, we may extend it to a strong monoidal functor $\overline{F} : M(\C) \to \D$ via $(X_1,\dotsc,X_n) \mapsto F(X_1) \otimes \dotsc \otimes F(X_n)$ on objects (and it is clear what happens on morphisms). The monoidal structure of $\overline{F}$ is induced by the canonical isomorphisms
\[(F(X_1) \otimes \dotsc \otimes F(X_n)) \otimes (F(Y_1) \otimes \dotsc \otimes F(Y_m)) \cong F(X_1) \otimes \dotsc \otimes F(Y_m).\]
This establishes an equivalence of categories between strong monoidal functors $M(\C) \to \D$ and all functors $\C \to \D$. This is just another version of Mac Lane's Coherence Theorem for \emph{monoidal} categories.
\end{rem}

\begin{rem}[Free symmetric monoidal categories] \label{freetensor}
The forgetful $2$-functor from symmetric monoidal categories (i.e. tensor categories) to categories has a left adjoint, constructed as follows: Given a category $\C$, the free symmetric monoidal category $S(\C)$ on $\C$ has objects tuples $(X_1,\dotsc,X_n)$ as above, but morphisms $(X_1,\dotsc,X_n) \to (Y_1,\dotsc,Y_n)$  are now given by pairs $(\sigma,f)$, where $\sigma \in \S_n$ is a permutation and $f=(f_i : X_i \to Y_{\sigma(i)})_{1 \leq i \leq n}$ is a tuple of morphisms in $\C$. These morphisms may be visualized via string diagrams. For example, the following depicts a typical morphism with $\sigma = (1 \,2 \, 3)$.
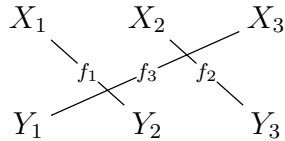
\[\xymatrix{X_1 \ar@{-}[dr]|{f_1} & X_2 \ar@{-}[dr]|{f_2} & X_3  \ar@{-}[dll]|{f_3} \\ Y_1 & Y_2 & Y_3}\]
The tensor product does not change on objects and the action on morphisms is defined using the canonical embeddings $\Sigma_n \times \Sigma_m \hookrightarrow \Sigma_{n+m}$. If $\D$ is a symmetric monoidal category and $F : \C \to \D$ is a functor into the underlying category, we extend it to a symmetric monoidal functor $\overline{F} : S(\C) \to \D$ by $\overline{F}(X_1,\dotsc,X_n) := F(X_1) \otimes \dotsc \otimes F(X_n)$ on objects and on morphisms using the Coxeter presentation of the symmetric group (\cite{CM80}). This works out because of Mac Lane's Coherence Theorem for symmetric monoidal categories.
\end{rem}

\begin{rem}
The forgetful $2$-functor from symmetric monoidal categories to monoidal categories also has a left adjoint.
\end{rem}

\begin{ex}[Free tensor category on one object] \label{oneobject}
The free category on one object $X$ is of course the one with exactly one object, also denoted by $X$, and $\End(X)=\{\id\}$. It follows that the free monoidal category on $X$ is discrete with objects $\{X^{\otimes n} : n \in \N\}$. The free symmetric monoidal category on $X$ has the same objects, but it is not discrete. Instead, we have $\Hom(X^{\otimes n},X^{\otimes m})=\emptyset$ for $n \neq m$ and $\Hom(X^{\otimes n},X^{\otimes n})=\S_n$ otherwise. This tensor category is also known as the \emph{permutation groupoid} $\P$, because it simplifies to $\coprod_{n \geq 0} \S_n$ with tensor product induced by $\Sigma_n \times \Sigma_m \hookrightarrow \Sigma_{n+m}$.

Notice that, however, the free tensor category on a \emph{symtrivial} object is discrete and may be realized as $\N \cong \{X^{\otimes n} : n \in \N\}$.

Similarly, the free tensor category on a line object $\L$ is discrete and may be realized as $\Z \cong \{\L^{\otimes z} : z \in \Z\}$.
\end{ex}

\subsection{Cocompletions} \label{cocompletions}

Let us discuss free cocompletions of tensor categories ``relative'' to a given cocomplete tensor category -- we refer to \cite{ImK86} for the corresponding enriched version. We will frequently use the coend interpretation of the Yoneda Lemma (\autoref{yoneda}).
 
\begin{defi}[Day convolution]
Let $I$ be a small tensor category and let $\C$ be a cocomplete tensor category. For functors $F,G : I^{\op} \to \C$ between the underlying categories we define a functor $F \otimes G : I^{\op} \to \C$ by
\begin{eqnarray*}
(F \otimes G)(n) & := & \int^{p,q \in I} F(p) \otimes G(q) \otimes \Hom(n, p \otimes q) \\
& =&  \colim_{n \to p \otimes q} F(p) \otimes G(q).
\end{eqnarray*}
on objects. It is clear how to define the action on morphisms.
\end{defi}

\begin{prop} 
The category of functors $I^{\op} \to \C$ endowed with the Day convolution becomes a cocomplete tensor category $\widehat{I}_\C$. Besides, there is a tensor functor
\[Y : I \to \widehat{I}_\C,~ n \mapsto \Hom(-,n) \otimes 1_\C\]
and a cocontinuous tensor functor
\[\iota : \C \to \widehat{I}_\C,~ M \mapsto \Hom(-,1_I) \otimes M.\]
\end{prop}

\begin{proof}
It is clear that  Day convolution provides a functor which is cocontinuous in each variable. To construct the associator, we compute for $F,G,H : I^{\op} \to \C$:
\begin{eqnarray*}
&& F \otimes (G \otimes H) \\
& \cong & \int^{p,q,r,s} F(p) \otimes G(r) \otimes H(s) \otimes \Hom(-,p \otimes q) \otimes \Hom(q,r \otimes s) \\
& \cong & \int^{p,r,s} F(p) \otimes G(r) \otimes H(s) \otimes \Hom(-, p \otimes r \otimes s).
\end{eqnarray*}
Now it is clear that this is also isomorphic to $(F \otimes G) \otimes H$. The symmetry $F \otimes G \cong G \otimes F$ is obvious.

We define the unit by $1 := \Hom(-,1_I) \otimes 1_\C$. We have
\[F \otimes 1 = \int^{p,q} F(p) \otimes \Hom(q,1_I) \otimes \Hom(-, p \otimes q) \cong \int^p F(p) \otimes \Hom(-,p) \cong F.\]
The axioms of a tensor category are straightforward to check.

Both $Y$ and $\iota$ preserve the unit by definition. For $n,m \in I$ we compute
\begin{eqnarray*}
Y(n) \otimes Y(m) &\cong&  \int^{p,q} \Hom(p,n) \otimes \Hom(q,m) \otimes \Hom(-, p \otimes q) \otimes 1_\C \\
&\cong & \Hom(-,n \otimes m) \otimes 1_\C = Y(n \otimes m).
\end{eqnarray*}
For $M,N \in \C$ we compute
\begin{eqnarray*}
&& \iota(M) \otimes \iota(N) \\
 & \cong & \int^{p,q} \Hom(p,1_I) \otimes \Hom(q,1_I) \otimes \Hom(-,p \otimes q) \otimes (M \otimes N) \\
& \cong & \Hom(-,1_I) \otimes (M \otimes N) = \iota(M \otimes N).
\end{eqnarray*}
It is straightforward to check that $Y$ and $\iota$ become tensor functors that way. Finally, it is clear that $\iota$ is cocontinuous.
\end{proof}

\begin{ex}[Tensor category of sequences]
We consider the partially ordered monoid $(\N,\geq,+)$ as a tensor category $I$. Then $\widehat{I}_\C$ consists of sequences
\[M_0 \to M_1 \to M_2 \to \dotsc \]
in $\C$. The unit object is $0 \to 1 \to 1 \to \dotsc$. The tensor product is given by
\[(M \otimes N)_n = \colim_{p+q \leq n} M_p \otimes N_q.\]
\end{ex}

\begin{prop}[Universal property of $\widehat{I}_\C$] \label{hat}
Let $I$ be a small tensor category and let $\C$ be a cocomplete tensor category. Then $Y$ and $\iota$ induce, for every cocomplete tensor category $\D$, an equivalence of categories
\[\Hom_{c\otimes}(\widehat{I}_\C,\D) \simeq \Hom_{c\otimes}(\C,\D) \times \Hom_{\otimes}(I,\D)\]
which maps $L : \widehat{I}_\C \to \D$ to $(L \circ \iota,L \circ Y)$. The inverse functor maps $G : \C \to \D$ and $H : I \to \D$ to $L : \widehat{I}_\C \to \D$ defined by
\[L(F) := \int^{p \in I} G(F(p)) \otimes H(p).\]
\end{prop}

\begin{proof}
It is clear that $L$ defines a cocontinuous functor. We endow it with the structure of a tensor functor as follows:
\begin{eqnarray*}
L(1) &= & \int^{p \in I} G(\Hom(p,1_I) \otimes 1_\C) \otimes H(p) \\
& \cong & \int^{p \in I} \Hom(p,1_I) \otimes G(1_\C) \otimes H(p) \\
&\cong & G(1_\C) \otimes H(1_I) \cong 1_\D.
\end{eqnarray*}
For $F,F' \in \widehat{I}_\C$ we have
\begin{eqnarray*}
L(F \otimes F') & = & \int^{p,q,n \in I} G(F(p) \otimes F'(q)) \otimes H(n) \otimes \Hom(n,p \otimes q) \\
& \cong & \int^{p,q \in I} G(F(p)) \otimes G(F'(q)) \otimes H(p) \otimes H(q) \\
& \cong & L(F) \otimes L(F').
\end{eqnarray*}
For $M \in \C$, $n \in I$ we have
\[(L \iota)(M)  =  \int^{p \in I} \Hom(p,1_I)  \otimes G(M) \otimes H(p)  \cong   G(M) \otimes H(1_I) \cong G(M). \]
\[(L Y)(n)  =  \int^{p \in I} \Hom(p,n) \otimes G(1_\C) \otimes H(p)  \cong G(1_\C) \otimes H(n) \cong H(n).\]
These are isomorphisms of tensor functors $L \iota \cong G$, $L Y \cong H$. It is straightforward to check that $(G,H) \mapsto L$ is inverse to $L \mapsto (L \iota,L Y)$.
\end{proof}

\begin{rem}
The same holds in the $R$-linear case: If $I$ is a small $R$-linear tensor category and $\C$ is a cocomplete tensor category, then we may endow the category of $R$-linear functors $I^{\op} \to \C$ between the underlying $R$-linear categories with the Day convolution (notice that in the $R$-linear case coends cannot be reduced to colimits) and obtain a cocomplete $R$-linear tensor category $\widehat{I}_\C$ satisfying the universal property
\[\Hom_{c\otimes/R}(\widehat{I}_\C,\D) \simeq \Hom_{c\otimes/R}(\C,\D) \times \Hom_{\otimes/R}(I,\D)\]
for cocomplete $R$-linear tensor categories $\D$.
\end{rem}

\begin{cor}[Free cocompletion] \label{freecoco}
Given a small tensor category $I$, the category $\widehat{I}$ of presheaves on $I$ equipped with Day convolution is a cocomplete tensor category. The Yoneda embedding $Y : I \to \widehat{I}$ becomes a cocontinuous tensor functor, inducing for every cocomplete tensor category $\C$ an equivalence of categories
\[\Hom_{c\otimes}(\widehat{I},\C) \cong \Hom_{\otimes}(I,C).\]
A corresponding construction works for $R$-linear tensor categories.
\end{cor}

\begin{proof}
Apply \autoref{hat} to $\C=\Set$ (resp. $\C=\M(R)$).
\end{proof}

\begin{rem}
If $I$ is a tensor category which is not small, we define $\widehat{I}$ as the tensor category of \emph{small} presheaves on $I$, i.e. of those sheaves which can be written as a small colimit of representable presheaves on $I$. Then the universal property still holds. Therefore, we have found a left adjoint to the forgetful $2$-functor $\Cat_{c \otimes} \to \Cat_{\otimes}$.
\end{rem}
 
\begin{ex}[Simplicial sets and joins]
Consider the augmented simplex category $\Delta$ (objects are $[n] := \{0,1,\dotsc,n\}$, $n \in \N$ -- morphisms are monotonic maps). It becomes a monoidal category with respect to the addition of natural numbers (but it is not symmetric monoidal!). It is well-known (\cite[VII.5]{ML98}) that $\Delta$ is the universal monoidal category equipped with a monoid object. The free cocompletion $\widehat{\Delta}$ identifies with the category of simplicial sets $\sSet$, whose monoidal structure is given by the \emph{join} operation (not the cartesian product). It follows by \autoref{freecoco} that this is the universal cocomplete monoidal category equipped with a monoid object. For a symmetric analogue, we consider instead of $\Delta$ the cocartesian monoidal category (i.e. $\otimes:=\coprod$) of finite sets $\FinSet$. This is the universal tensor category equipped with a \emph{commutative} monoid object (\cite{Gra01}) and hence its free cocompletion $\widehat{\FinSet}$ is the universal cocomplete tensor category equipped with a commutative monoid object. It identifies with the category of so-called symmetric simplicial sets.
\end{ex}

\begin{rem}[Free cocomplete tensor categories]
By \autoref{freecoco} and \autoref{freetensor} we see that the forgetful $2$-functor $\Cat_{c\otimes} \to \Cat$ has a left adjoint, i.e. every category $\C$ has a free cocomplete tensor category on $\C$. In general it is quite hard to describe concretely. However:
\end{rem}
 
\begin{rem}[Free cocomplete tensor category on one object]
Let $\C$ be a cocomplete tensor category. We already know the free tensor category $\P$ on one object $X$ (\autoref{oneobject}). It follows by \autoref{hat} that $\C[X] := \widehat{\P}_\C$ is the free cocomplete tensor category equipped with an object $X$ and a cocontinuous tensor functor $\C \to \C[X]$ (``over $\C$''). Its objects are families $M=(M_n)_{n \in \N}$ of objects $M_n \in \C$ with a right $\Sigma_n$-action. The universal object $X$ is here $1_\C$ concentrated in degree $1$. Observe that
\[M = \bigoplus_{n \in \N} M_n \otimes_{\S_n} X^{\otimes n}.\]
The tensor product is given by the following formula:
\[M \otimes N = \bigoplus_{n \in \N} \left(\bigoplus_{p+q=n} (M_p \times M_q) \otimes_{\S_p \times \S_q} \S_n\right) \otimes_{\S_n} X^{\otimes n}.\]
We have also the free cocomplete tensor category over $\C$ containing a symtrivial object $\widehat{\N}_\C$, whose objects are just $\N$-graded objects of $\C$, as well as the free cocomplete tensor category over $\C$ containing a line object $\widehat{\Z}_\C$, whose objects are $\Z$-graded objects of $\C$. See \autoref{gradings} for more on gradings.
\end{rem}

\begin{rem}[Free cocomplete tensor category on a morphism]
Let $\C$ be a cocomplete tensor category and $A,B \in \C$ be two objects. We would like to construct $\C[A \to B]$, the free cocomplete tensor category over $\C$ on a morphism $A \to B$. We will encounter some special cases later, for example  \autoref{mod-UE} implies $\C[A \to 1]=\M(\Sym A)$. However, the general case seems to be quite complicated. Let us assume that $\C$ is locally presentable. In private communication James Dolan has suggested to realize $\C[A \to B]$ as the category of pairs $(X,\alpha)$, where $X \in \C$ and $\alpha : \HOM(B,X) \to \HOM(A,X)$ is a morphism in $\C$ such that the diagram
\[\xymatrix{\HOM(B \otimes B,X) \ar[r] \ar[d] & \HOM(B \otimes A,X) \ar[d] \\ \HOM(A \otimes B,X) \ar[r] & \HOM(A \otimes A,X)}\]
commutes. The morphisms in that diagram are defined as follows: The vertical morphism above is defined by
\[\xymatrix@C=60pt{\HOM(B \otimes B,X)  \ar[d]_{\cong}  & \HOM(B \otimes A,X). \\  \HOM(B,\HOM(B,X)) \ar[r]^-{\HOM(B,\alpha)} & \HOM(B,\HOM(A,X)) \ar[u]_{\cong}}\]
The definition of the vertical morphism below is analoguous. The horizontal morphisms are defined similarly but they are also twisted with the symmetry.

The idea behind this definition is that ``$X$ believes that $\alpha$ pulls back a morphism $A \to B$''. However, it is not clear a priori how to define colimits and tensor products in this category. One probably has to construct a free coherent monoidal monad $T$ on $\C$ (\autoref{monoidalmonads}) on a Kleisli morphism $A \to B$ (via transfinite compositions  and alike) and define $\C[A \to B] := \M(T)$. The universal property should follow from \autoref{monadUE2}. This has not been worked out yet.
\end{rem}

\subsection{Indization}
 
We refer to \cite[Chapter 6]{KS06} for the general theory of indization of categories. Given a category $\C$ its indization $\Ind(\C)$ is the full subcategory of the free cocompletion $\widehat{\C}$ which consists of \emph{directed} colimits of representable presheaves. Then $\Ind(\C)$ has directed colimits, created by the inclusion to $\widehat{\C}$. In fact for every category $\D$ with directed colimits we have an equivalence between the category of functors $\C \to \D$ and the category of finitary functors $\Ind(\C) \to \D$. This means that $\Ind(-)$ is left adjoint to the forgetful $2$-functor from categories with directed colimits and finitary functors to the ``category'' of all categories and functors.

If $\C$ is finitely cocomplete, then $\Ind(\C)$ is actually cocomplete (however, colimits are \emph{not} preserved by the inclusion to $\widehat{\C}$) in such a way that the Yoneda embedding $Y : \C \to \Ind(\C)$ becomes finitely cocontinuous. Moreover, using \cite[Theorem 6.4.3]{KS06} one can show that if $\D$ is a cocomplete category, then $Y$ induces an equivalence of categories $\Hom_c(\Ind(\C),\D) \cong \Hom_{fc}(\C,\D)$. Hence, $\Ind(-)$ also serves as a left adjoint of the forgetful $2$-functor from cocomplete categories (and cocontinuous functors) to finitely cocomplete categories (and finitely cocontinuous functors). We have the same result for tensor categories:

\begin{prop}[Indization of tensor categories]
Let $\C$ be a finitely cocomplete tensor category. Then $\Ind(\C)$ is actually a cocomplete tensor category, $Y : \C \to \Ind(\C)$ is a finitely cocontinuous tensor functor. If $\D$ is a cocomplete tensor category, then $Y$ induces an equivalence of categories
\[\Hom_{c\otimes}(\Ind(\C),\D) \simeq \Hom_{fc\otimes}(\C,\D).\]
Hence, $\Ind(-)$ is left adjoint to the forgetful $2$-functor $\Cat_{c\otimes} \to \Cat_{fc\otimes}$.
\end{prop}

\begin{proof}[Sketch of proof.]
This is very similar to the proof of \autoref{freecoco}. The tensor product is defined by $\Ind(\C) \times \Ind(\C) \simeq \Ind(\C \times \C) \to \Ind(\C)$, which is just the restriction of the tensor product of $\widehat{\C}$. For $M \in \C$ we have a cocontinuous functor $Y(M) \otimes - : \Ind(\C) \to \Ind(\C)$, namely the unique cocontinuous functor extending the finitely cocontinuous functor $M \otimes - : \C \to \C$. The tensor product is finitary in each variable (since directed colimits are preserved by the inclusion to $\widehat{C}$), so that we easily deduce that $N \otimes - : \Ind(\C) \to \Ind(\C)$ is cocontinuous for every $N \in \Ind(\C)$. Finally, if $F : \C \to \D$ is a finitely cocontinuous tensor functor, we may extend it to a cocontinuous functor $\overline{F} : \Ind(\C) \to \D$ and only have to show that it carries a canonical structure of a tensor functor, for which we may use the same trick as above.
\end{proof}

\begin{prop}
Let $\C$ be a locally finitely presentable (tensor) category and let $\C_{\fp}$ be the full (tensor) subcategory of finitely presentable objects. Then there is an equivalence of (tensor) categories $\Ind(\C_{\fp}) \simeq \C$.
\end{prop}

\begin{proof}[Sketch of proof.]
For the case of categories we refer to \cite[Corollary 6.3.5]{KS06}. The case of tensor categories now follows easily.
\end{proof}

\begin{cor} \label{qcohind}
Let $X$ be a scheme resp. algebraic stack as in \autoref{Qcohfin2}. Then for every cocomplete tensor category $\C$ we have an equivalence of categories
\[\Hom_{c\otimes}(\Q(X),\C) \cong \Hom_{fc\otimes}(\Q_{\fp}(X),\C).\]
\end{cor}


\subsection{Limits and colimits}

\begin{rem}[Limits]
The $2$-category $\Cat_{c\otimes}$ of cocomplete tensor categories is $2$-complete; limits are simply created by the forgetful $2$-functor to $\Cat$. For example, the product $\prod_{i \in I} \C_i$ of a family $(\C_i)_{i \in I}$ of cocomplete tensor categories has  tuples $(X_i)_{i \in I}$ with $X_i \in \C_i$ as objects and
\[\Hom((X_i),(Y_i)) = \prod_i \Hom_{\C_i}(X_i,Y_i)\]
as morphisms. The tensor product is $(X_i)_i \otimes (Y_i)_i := (X_i \otimes Y_i)_i$. Likewise, colimits are computed in each entry. Then $\prod_{i \in I} \C_i$ is a cocomplete tensor category with the universal property
\[\Hom_{c\otimes}(\D,\prod_{i \in I} \C_i) \simeq \prod_{i \in I} \Hom_{c\otimes}(\D,\C_i).\]
Another example are $2$-pullbacks. Given cocontinuous tensor functors
\[\xymatrix{\C_1 \ar[r]^{F_1} & \E & \ar[l]_{F_2} \C_2,}\]
the $2$-pullback $\C_1 \times_{\E} \C_2$ (or more precisely $\C_1 \times_{F_1,F_2} \C_2$) has  triples $(X_1,X_2,\sigma)$ as objects, where $X_i \in \C_i$ and $\sigma : F_1(X_1) \to F_2(X_2)$ is an isomorphism in $\E$. Morphisms are defined in an evident manner and tensor products are given by $(X_1,X_2,\sigma) \otimes (Y_1,Y_2,\tau) = (X_1 \otimes Y_1,X_2 \otimes Y_2,\sigma \otimes \tau)$. Colimits are created by the forgetful functor to the product $\C_1 \times \C_2$. It follows easily that $\C_1 \times_{\D} \C_2$ is a cocomplete tensor category. It satisfies the universal property
\[\Hom_{c\otimes}(\D,\C_1 \times_{\E} \C_2) \simeq \Hom_{c\otimes}(\D,\C_1) \times_{\Hom_{c\otimes}(\D,\E)} \Hom_{c\otimes}(\D,\C_2).\]
\end{rem}

\begin{ex} \label{qprod}
Let us compare these constructions to corresponding constructions from algebraic geometry. If $(X_i)_{i \in I}$ is a family of schemes (or even algebraic stacks), then we have an equivalence of cocomplete tensor categories
\[\Q\bigl(\coprod_{i \in I} X_i\bigr) \simeq \prod_{i \in I} \Q(X_i).\]
This can either be checked directly or deduced from the adjunction in \autoref{adjunction}, which implies that, in fact, $\Q : \Stack \to \Cat_{c\otimes}^{\op}$ preserves \emph{all} colimits. In particular, if $X = X_1 \cup X_2$ with two open subschemes $X_1,X_2$ and intersection $U$, we have
\[\Q(X) \simeq \Q(X_1) \times_{\Q(U)} \Q(X_2).\]
Gluing of schemes corresponds to pullbacks of tensor categories -- a fact which was already heavily used in Gabriel's thesis \cite{Gab62}.
\end{ex}

\begin{rem}[Colimits] 
We already know that $\Set$ (resp. $\M(R)$) is the $2$-initial cocomplete ($R$-linear) tensor category. It is not clear to the author if $\Cat_{c\otimes}$ is $2$-cocomplete. Colimits of cocomplete tensor categories, or just coproducts, seem to be quite hard to ``write down'' explicitly (in terms of objects, morphisms and tensor products). By categorification of the well-known construction of coproducts of commutative rings using tensor products of abelian groups, we expect that $\Cat_c$ is a symmetric monoidal $2$-category, which allows us to construct the coproduct of two cocomplete tensor categories by endowing the tensor product of the underlying cocomplete categories with a tensor structure. At least, such a procedure is possible in the context of \emph{locally presentable} categories (\cite[Corollary 2.2.5, Remark 2.3.8]{CJF13}). It is also possible for categories with $\kappa$-small colimits for a fixed cardinal number $\kappa$ -- this is a consequence of $2$-dimensional monad theory (\cite{Bla89}). In any case, an explicit description seems to be out of reach. Still, it is a natural question how colimits of tensor categories compare to limits of schemes and stacks. We will attack this question without using any explicit constructions in \autoref{prodschemes}. Besides, we will see some natural examples of $2$-pushouts of cocomplete tensor categories which globalize pullbacks in algebraic geometry (e.g. \autoref{mod-BW} and \autoref{proj-BW}). Notice that for example $\widehat{I}_\C$ is a $2$-coproduct of $\widehat{I}$ and $\C$ by \autoref{hat}.
\end{rem}

\section{Global schemes and stacks} \label{globalization}

Before we give the definition of global schemes, let us motivate it by some specific detailed example. Some of the following remarks are paradigms for other global schemes which will be studied later.

\begin{ex}[Globalization of affine schemes] \label{aff-glob}
Let $R$ be a fixed commutative base ring. Let $A$ be a commutative $R$-algebra. Then $\Spec(A)$ is an $R$-scheme with the following universal property: If $X$ is an $R$-scheme, then $R$-morphisms $X \to \Spec(R)$ correspond naturally to homomorphisms of $R$-algebras $R \to \Gamma(X,\O_X)$. The bijection is given by $f \mapsto \Gamma(f)$. Actually $\Gamma(X,\O_X)$ may be reconstructed from the cocomplete $R$-linear tensor category $\Q(X)$: It is isomorphic to $\End(\O_X)$, where $\O_X$ is the unit of $\Q(X)$. Also, the morphism $X \to \Spec(R)$ may be identified with its pullback functor $F : \M(R) \to \Q(X)$. The action $\End(R) \to \End(F(R))$ induced by $F$ is precisely the homomorphism $R \to \Gamma(X,\O_X)$.

This leads to the following more general observation: If $\C$ is an arbitrary cocomplete $R$-linear tensor category (with unit $\O_\C$), then we have a functor
\[\alpha : \Hom_{c\otimes/R}(\M(A),\C) \to \Hom_{\Alg(R)}(A,\End_\C(\O_\C)).\]
It maps $F : \M(A) \to \C$ to $A \cong \End_{\M(A)}(A) \to \End_{\C}(\O_\C)$. If $F \to G$ is a morphism of tensor functors, then the two resulting homomorphisms $A \to \End_\C(\O_\C)$ are equal. Since $\Hom_{\Alg(R)}(A,\End_\C(\O_\C))$ is a set, regarded as a discrete category, this describes the action on morphisms.
 
If $\C=\Q(X)$ for some $R$-scheme $X$, then we can write down a functor $\beta$ in the other direction: It maps a homomorphism of $R$-algebras $\sigma : A \to \Gamma(X)$ first to the correponding morphism $f : X \to \Spec(A)$ and then takes its associated pullback functor $f^* : \M(A) \to \Q(X)$. Again we can write this down in terms of tensor categories: We have $f^*(M) = M \otimes_A \O_X$, which is the sheaf of modules associated to the presheaf $U \mapsto M \otimes_A \Gamma(U,\O_X)$. But it may also be seen globally as an external tensor product, defined by the adjunction
\[\Hom(M \otimes_A \O_X,N) \cong \Hom_{\M(A)}(M,\Hom(\O_X,N)) \cong \Hom_{\M(A)}(M,\Gamma(N)).\]
Now we are able to write down $\beta$ for arbitrary $\C$. Given a homomorphism of $R$-algebras $\sigma : A \to \End_\C(\O_\C)$, we can endow $\Gamma(N) := \Hom(\O_\C,N)$, for every $N \in \C$, with the structure of a right $A$-module. Then $\Gamma : \C \to \M(A)$ is a continuous functor, which in fact has a left adjoint, which will be denoted by $M \mapsto M \otimes_A \O_\C$. One can construct it explicitly as follows: If $M$ is free, say $M = A^{\oplus I}$, then $M \otimes_A \O_\C = \O_\C^{\oplus I}$. In general, write $M$ as the cokernel of some map $A^{\oplus I} \to A^{\oplus J}$. Regard this as a matrix over $A$ and transport it via $\sigma$ to a matrix over $\End_\C(\O_\C)$. We optain a morphism $\O_\C^{\oplus I} \to \O_\C^{\oplus J}$, whose cokernel is precisely $M \otimes_A \O_\C$. Since the functor $? \otimes_A \O_\C$ is a left adjoint, it is cocontinuous.
 
Since $\Gamma$ is a lax tensor functor, i.e. there are canonical homomorphisms
\[\sigma : A \to \Gamma(\O_\C),~\Gamma(M) \otimes \Gamma(N) \to \Gamma(M \otimes N),\]
where $M,N \in \C$, the adjunction tells us that $? \otimes_A \O_\C$ is an oplax tensor functor, i.e. there are canonical morphisms
\[A \otimes_A \O_\C \to \O_\C,~ (M \otimes_A N) \otimes_A \O_\C \to (M \otimes_A \O_\C) \otimes_\C (N \otimes_A \O_\C),\]
where $M,N$ are $A$-modules. The first map is an isomorphism by construction and the second one is so in the special case $M=N=A$. The general case then follows from a colimit argument. Thus, $? \otimes_A \O_\C$ is a cocontinuous $R$-linear tensor functor. This defines a functor
\[\beta : \Hom_{\Alg(R)}(A,\End_\C(\O_\C)) \to \Hom_{c\otimes/R}(\M(A),\C).\]
It follows easily from the constructions that $\alpha \beta = \id$. In the other direction, we first spot a natural transformation $\eta : \beta \alpha \to \id$ defined as follows:

Let $F : \M(A) \to \C$ be a cocontinuous $R$-linear tensor functor. We define $\sigma = \alpha(F) : A \to \End_\C(\O_\C)$. Then $\eta_F(M) : M \otimes_A \O_\C \to F(M)$ corresponds to a homomorphism of $A$-modules $M \to \Gamma(F(M))$, defined as follows: For every $m \in M$ the homomorphism $A \to M$ mapping $1 \mapsto m$ induces a morphism $\O_\C = F(A) \to F(M)$, i.e. an element of $\Gamma(F(M))$. It is easy to check that $\eta_F$ is, indeed, a morphism of tensor functors $\M(A) \to \C$, i.e. a morphism in $\Hom_{c\otimes/R}(\M(A),\C)$ and that $\eta : \beta \alpha \to \id$ is, indeed, a morphism of endofunctors of $\Hom_{c\otimes/R}(\M(A),\C)$.

Clearly $\eta_F(M)$ is an isomorphism if $M=A$. By a colimit argument, this follows for arbitrary $M$. Thus, $\eta_F$ is an isomorphism for every $F$, i.e. $\eta$ is an isomorphism. We have shown that $\alpha$ and $\beta$ are quasi-inverse to each other, thus implementing an equivalence of categories (where the right hand side is discrete)
\[\xymatrix{\Hom_{c\otimes/R}(\M(A),\C) \ar@/^2pc/[r]^{\alpha} & \ar@/^2pc/[l]^{\beta} \Hom_{\Alg(R)}(A,\End_C(\O_\C)).}\]
This means that the tensor category $\M(A)=\Q(\Spec(A))$ has \textit{the same universal property} as the affine scheme $\Spec(A)$. Besides, this universal property of $\Q(\Spec(A))$ is induced via pullback functors from the universal property of $\Spec(A)$ if we apply it to $\C=\Q(X)$ for some scheme $X$. In other words, the well-known universal property of affine schemes is not restricted to schemes, where one typically uses local arguments and constructions, but may be extended to all cocomplete tensor categories, where we only have global arguments and constructions. This example, as well as many others which we will study later, motivates the notion of a \textit{global scheme}. We have just seen that affine schemes are global. Let us summarize:
\end{ex}
 
\begin{prop} \label{affin-global}
If $\C$ is a cocomplete $R$-linear tensor category and $A$ is a commutative $R$-algebra, then there is an equivalence of categories
\[\Hom_{c\otimes/R}(\M(A),\C) \simeq \Hom_{\Alg(R)}(A,\End_\C(\O_\C)).\]
Hence, affine schemes are global.
\end{prop}

In particular we see that $\M(R)$ is a $2$-initial object of  $\Cat_{c\otimes/R}$, which is a globalization of the fact that $\Spec(R)$ is a terminal object of $\Sch/R$.
 
\begin{cor} \label{affin-tensorial}
Affine schemes are tensorial: If $X$ is an affine scheme and $Y$ is an arbitrary scheme, then $f \mapsto f^*$ implements an equivalence of categories
\[\Hom(Y,X) \simeq \Hom_{c\otimes}\bigl(\Q(X),\Q(Y)\bigr).\]
\end{cor}

\begin{proof}
We apply \autoref{affin-global} to $R=\Z$ and $\C=\Q(Y)$ and $A$ with $X=\Spec(A)$.
\end{proof}

Let us sketch the general notion of global schemes and stacks.
 
Many constructions for schemes (or algebraic stacks) can be translated to the realm of cocomplete tensor categories. We call this process \emph{globalization}, basically because local constructions and proofs have to be replaced by global ones (see also \autoref{Comm}, in particular \autoref{localfree}).

\begin{defi}[Globalization]
Let $F : \Sch \to \Sch$ be a functor, which thus constructs a scheme from a given scheme. A \emph{globalization} of $F$ is a functor $F^{\otimes} : \Cat_{c\otimes} \to \Cat_{c\otimes}$ with natural equivalences
\[F^{\otimes}(\Q(X)) \simeq \Q(F(X)),\]
as in the following diagram:

\[\xymatrix@=40pt{\Sch \ar[d]_{\Q} \ar@{~>}[r]^{F} & \Sch  \ar[d]^{\Q} \\ \Cat_{c\otimes}^{\op} \ar@{~>}[r]^{F^{\otimes}} & \Cat_{c\otimes}^{\op}}\]
 
Besides, if $F$ has some universal property which is definable in the language of cocomplete tensor categories (consisting of objects, morphisms, colimits, tensor products), we require that $F^{\otimes}$ satisfies the corresponding universal property, but within all cocomplete tensor categories. This translation should be compatible with $F^{\otimes}(\Q(X)) \simeq \Q(F(X))$.
 
We can also look at globalizations for more general constructions by allowing additional inputs such as
 
\begin{itemize}
\item quasi-coherent modules or algebras on the given scheme,
\item more than one scheme (i.e. functors $\Sch^d \to \Sch$ for some $d \in \N$),
\item morphisms of schemes,
\item algebraic stacks instead of schemes.
\end{itemize}
\end{defi}

The above definition is a bit vague, but we will make the meaning of globalization very precise in each specific example. In \autoref{aff-glob} we have seen that an affine scheme $\Spec(A)$ (seen as a functor without any input scheme, but with a parameter $A$) globalizes to $\M(A)=\Q(\Spec(A))$ since the universal property
\[\Hom_R(X,\Spec(A)) \cong \Hom_{\Alg(R)}(A,\End(\O_X))\]
generalizes to the universal property
\[\Hom_{c\otimes/R}(\M(A),\C) \simeq \Hom_{\Alg(R)}(A,\End(\O_\C)).\]
We may therefore write $\Spec^{\otimes}(A) = \M(A)$.
 
\begin{rem}
As a general rule, globalized constructions of schemes preserve tensorial schemes (likewise for stacks). For the author this was the original motivation for globalization, because it made possible to prove that projective schemes are tensorial (\cite{Bra11}).
\end{rem} 
\section{Module categories over algebras} \label{module-categories}

Let $S$ be a scheme (or even an algebraic stack) and $A$ be a quasi-coherent algebra on $S$. Consider the relative spectrum $\Spec_S(A)$, it is an $S$-scheme with the following universal property (\cite[Definition 9.1.8]{EGAI}): If $f : T \to S$ is another $S$-scheme, then there is a bijection (equivalence of categories) between $\Hom_S(T,\Spec_S(A))$ and the set of homomorphisms of quasi-coherent algebras $f^*(A) \to \O_T$. In other words, $\Spec_S(A)$ is the universal $S$-scheme equipped with a section of (the pullback of) $A$.

We can also drop the base $S$ from the structure and obtain the following universal property of $\Spec_S(A)$ as an absolute scheme: If $T$ is an arbitrary scheme, then $\Hom(T,\Spec_S(A))$ can be identified with the set (or discrete category) of pairs $(f,h)$, where $f : T \to S$ is a morphism and $h$ is as above. We will globalize this to cocomplete tensor categories as follows (see \autoref{aff-tensor} for the original motivation):
 
\begin{prop}[Universal property of module categories] \label{mod-UE}
Let $\C$ be a cocomplete tensor category and $A$ be a commutative algebra in $\C$. Then the cocomplete tensor category $\M(A)$ enjoys the following universal property: For every cocomplete tensor category $\D$ there is a natural equivalence of categories
\[\Hom_{c\otimes}(\M(A),\D) \simeq \{(F,h) : F \in \Hom_{c\otimes}(\C,\D),\, h \in \Hom_{\CAlg(\D)}(F(A),\O_\D)\}.\]
\end{prop}

Here, a morphism $(F,h) \to (F',h')$ is a morphism $\alpha : F \to F'$ with the property $h' \circ \alpha(A)=h$.

\begin{proof}
The functor maps $G : \M(A) \to \D$ to $F : \C \to \D$ which is defined by $F(N) = G(N \otimes A)$ and $h = G(m)$, where $m : A \otimes A \to A$ is the multiplication. The inverse functor maps $(F,h)$ to $G : \M(A) \to \D$ defined by
\[G(M)=F(M) \otimes_{F(A)} \O_\D.\]
It is a routine verification to check that these functors are well-defined.

That they are inverse to each other follows from the following observation: For $M \in \M(A)$ we have an isomorphism of $A$-modules
\[M \cong (M \otimes A) \otimes_{A \otimes A} A. \qedhere\]
\end{proof}

\begin{rem}
In other words, $\M(A)$ is the universal cocomplete tensor category over $\C$ equipped with a section of (the image of) the algebra $A$, and we may write
\[\M(A) = \C[A \to \O \text{ algebra homomorphism}].\]
I would like to thank Jacob Lurie for suggesting the definition of the inverse functor in \autoref{mod-UE}.
\end{rem}

\begin{cor}[Globalization of affine morphisms] \label{aff-global}
Let $S$ be a scheme (or even algebraic stack) and $A$ be a quasi-coherent algebra on $S$, i.e. a commutative algebra in $\Q(S)$. Then there is an equivalence of tensor categories
\[\Q(\Spec_S(A)) \simeq \M(A).\]
Hence, if $\C$ is a cocomplete tensor category, there is an equivalence of categories
\[\begin{array}{c}
\Hom_{c\otimes}(\Q(\Spec_S(A)),\C) \medskip \\ 
|\wr \medskip \\ 
\{(F,h) : F \in \Hom_{c\otimes}(\Q(S),\C),~ h \in \Hom_{\CAlg(\C)}(F(A),\O_\C)\}.
\end{array}\]
\end{cor}

\begin{proof}
The first part is essentially \cite[9.2]{EGAI}, which is just a straightforward generalization of the classification of quasi-coherent modules on affine schemes to relative affine schemes. Then, the second part follows from \autoref{mod-UE}.
\end{proof}

\begin{cor} \label{aff-tensor}
If $S$ is a tensorial scheme (or even algebraic stack) and if $X$ is affine over $S$, then $X$ is also tensorial.
\end{cor}

\begin{proof}
Write $X=\Spec_S(A)$ for some quasi-coherent algebra $A$ on $S$. If $Y$ is another scheme, by \autoref{aff-global} $\Hom_{c\otimes}(\Q(X),\Q(Y))$ identifies with the category of pairs $(F,h)$, where $F : \Q(S) \to \Q(Y)$ is a cocontinuous tensor functor  and $h : F(A) \to \O_Y$ is a homomorphism of quasi-coherent algebras on $Y$. Since $S$ is tensorial, $F$ corresponds to a morphism $f : Y \to S$ and $h$ corresponds to a homomorphism $f^*(A) \to \O_Y$. Then $(f,h)$ corresponds to a morphism $Y \to \Spec_S(A) = X$.
\end{proof}

The following results have some overlap with the independent work \cite[Section 3]{Sch13}.

\begin{defi}[Affine tensor functors] \label{affinefunk}
Cocontinuous tensor functors $\C \to \D$ which are isomorphic to $\C \to \M(A) =: \Spec^{\otimes}_\C(A)$ for some commutative algebra $A$ in $\C$ are called \emph{affine}. With this notation, \autoref{aff-global} becomes
\[\Q(\Spec_S(A)) \simeq \Spec^{\otimes}_{\Q(S)}(A).\]
\end{defi}

\begin{cor}
If $\C$ is a cocomplete tensor category and $A,B$ are commutative algebras in $\C$, then there is an equivalence of categories
\[\Hom_{c\otimes/\C}(\Spec^\otimes_\C(A),\Spec^\otimes_\C(B)) \simeq \Hom_{\CAlg(\C)}(A,B).\]
\end{cor}

\begin{proof}
We apply \autoref{mod-UE} to $\D = \Spec^{\otimes}_\C(B)$. It follows that the left hand side is equivalent to the (discrete) category of homomorphism of commutative $B$-algebras $A \otimes B \to B$, which reduces to a homomorphism of commutative algebras $A \to B$.
\end{proof}

\begin{cor}[Base change for affine tensor functors] \label{mod-BW}
Let $F : \C \to \D$ be a cocontinuous tensor functor. Let $A$ be a commutative algebra in $\C$. Then $B:=F(A)$ is a commutative algebra in $\D$ and $F$ induces a cocontinuous tensor functor $F' : \M(A) \to \M(B)$. In fact, the following is a $2$-pushout diagram in $\Cat_{c\otimes}$:
\[\xymatrix{\C \ar[r]^{F} \ar[d] & \D \ar[d] \\ \M(A) \ar[r]^{F'} & \M(B)}\]
\end{cor}

\begin{proof}
Let $\E$ be a cocomplete tensor category. Using \autoref{mod-UE} twice, we have equivalences of categories
\begin{eqnarray*}
&& \Hom_{c\otimes}(\M(A),\E) \times_{\Hom_{c\otimes}(\C,\E)} \Hom_{c\otimes}(\D,\E) \\
& \simeq & \{G \in \Hom_{c\otimes}(\C,\E),~ h \in \Hom_{\Alg(\E)}(G(A),\O_\E), \\
 && ~ H \in \Hom_{c\otimes}(\D,\E),~ \sigma : H F \cong G\} \\
& \simeq & \{H \in \Hom_{c\otimes}(\D,\E),~ h \in \Hom_{\Alg(\E)}(H(B),\O_\E)\} \\
& \simeq & \Hom_{c\otimes}(\M(B),\E).
\end{eqnarray*}
This establishes $\M(B)$ as a $2$-pushout of $\C \to \M(A)$ and $F : \C \to \D$. It also shows the existence of $F'$. Explicitly, if $M \in \C$ is equipped with a module action $M \otimes A \to M$, then $F'(M):=F(M)$ is equipped with the module action $F(M) \otimes B \cong F(M \otimes A) \to F(M)$.
\end{proof}

\begin{ex}[Affine spaces] \label{affsp}
Let us define the  \emph{$n$-dimensional tensorial affine space} over $\C$ by $\mathds{A}^n_\C := \Spec_\C^\otimes(\O_\C[T_1,\dotsc,T_n])$. The objects of $\mathds{A}^n_\C$ are objects of $\C$ equipped with $n$ commuting endomorphisms. By \autoref{mod-UE} we have the universal property
\[\Hom_{c\otimes/\C}(\mathds{A}^n_\C,\D) \simeq \Gamma(\O_\D)^n\]
for all $\C \to \D$. Therefore, like in algebraic geometry, it classifies $n$ global sections. The universal property (or \autoref{mod-BW}) implies
\[\mathds{A}^n_{\D} = \mathds{A}^n_{\C} \sqcup_\C \D.\]
for every $\C \to \D$.
\end{ex}

Next, we globalize closed immersions. If $S$ is a scheme, $I \subseteq \O_S$ is a quasi-coherent ideal, then the corresponding closed subscheme $V(I)$ of $S$ has the following universal property: A morphism $T \to V(I)$ corresponds to a morphism $f: T \to S$ such that the ``sheaf part'' $f^\# : \O_S \to f_* \O_T$ vanishes on $I$, equivalently that $f^* I \to f^* \O_S = \O_T$ is the trivial homomorphism. This can be globalized as follows:
 
\begin{cor}[Closed immersions] \label{closed-UE}
Let $\C$ be a cocomplete $R$-linear tensor category and $I \subseteq \O_\C$ be an ideal (or more generally any morphism $I \to \O_\C$). Then the cocomplete $R$-linear tensor category
\[V^{\otimes}(I):=\M(\O_\C/I) \simeq \{M \in \C : (I \otimes M \to M)= 0 \}\]
enjoys the following universal property: If $\D$ is a cocomplete $R$-linear tensor category, then there is an equivalence between $\Hom_{c\otimes/R}(V^{\otimes}(I),\D)$ and the category of those $F \in \Hom_{c\otimes/R}(\C,\D)$ such that $F(I) \to F(\O_\C)$ vanishes.
\end{cor}

\begin{proof}
This follows easily from \autoref{mod-UE} (which holds also in the $R$-linear case) applied to $A=\O_\C/I$.
\end{proof}

Recall from \cite[\href{http://stacks.math.columbia.edu/tag/01R5}{Tag 01R5}]{stacks-project} that the \emph{scheme-theoretic image} of a morphism $f : Y \to X$ of schemes is the smallest closed subscheme $Z \hookrightarrow X$ through which $f$ factors. \marginpar{I've added this remark.} If $f$ is qc qs, the ideal $I:= \ker(f^\# : \O_X \to f_* \O_Y)$ is quasi-coherent and the scheme-theoretic image is given by $V(I) \hookrightarrow X$. If $f$ is arbitrary, we replace $I$ by the largest quasi-coherent ideal contained in $\ker(f^\#)$, i.e. such that $f^* I \to f^* \O_X = \O_Y$ vanishes. We globalize this construction as follows:

\begin{lemma}[Closed Image] \label{closedimage}
Let $F : \C \to \D$ be a cocontinuous linear tensor functor between cocomplete linear tensor categories\marginpar{Added 'cocomplete'.}. We assume that $\C$ is locally presentable and abelian. There is a largest ideal $I_F \subseteq \O_\C$ such that $F(I_F \to \O_\C)$ vanishes. Moreover, there is a factorization of $F$ as  \marginpar{I've added this Lemma.}
\[\C \to V^{\otimes}(I_F) \xrightarrow{G} \D,\]
where $I_G = 0$.
\end{lemma}

\begin{proof}
Consider the set of all ideals $I \subseteq \O_\C$ such that $F(I \to \O_\C)$ vanishes. Let $I_F \subseteq \O_\C$ be the sum of all these ideals. Then $F(\bigoplus_I I) \twoheadrightarrow F(I_F) \to F(\O_\C)$ vanishes, so that also $F(I_F) \to F(\O_C)$ vanishes. By construction $I_F$ is the largest ideal with this property. By \autoref{closed-UE} $F$ factors as $\C \to V^{\otimes}(I_F)$ followed by $G : V^{\otimes}(I_F) \to \D$, where the underlying functor of $G$ is just the restriction of $F$ to the full subcategory $V^{\otimes}(I_F) \subseteq \C$. If $J \subseteq \O_{V^{\otimes}(I_F)} = \O_\C/I_F$ is an ideal such that $G(J) \to G(\O_\C/I_F)$ vanishes, we have $J=I/I_F$ for some ideal $I_F \subseteq I \subseteq \O_\C$ such that $F(I_F \to \O_\C)$ vanishes. Hence, $I \subseteq I_F$ and therefore $J=0$.
\end{proof}

\begin{defi} \label{ringchange}
Let $\C$ be a cocomplete $R$-linear tensor category. If $A$ is a commutative $R$-algebra, then $A \otimes_R \O_\C$ is a commutative algebra in $\C$. The corresponding module category is denoted by $\C_A$ or $\C \otimes_R A$ and is isomorphic to the category of $A$-modules in $\C$ (\cite[II.1.5.2]{SR72}), i.e. pairs $(M,\alpha)$, where $M \in \C$ and $\alpha : A \to \End_\C(M)$ is a homomorphism of $R$-algebras. This category is actually $A$-linear and enjoys the following universal property:
\end{defi}

\begin{cor}[Change of rings] \label{changering}
If $A$ is a commutative $R$-algebra, then the construction $\C \mapsto \C_A$ is a $2$-functor $\Cat_{c\otimes/R} \to \Cat_{c\otimes/A}$ which is left adjoint to the forgetful $2$-functor.
\end{cor}

\begin{proof}
If $\D$ is an $R$-linear cocomplete tensor category, then \autoref{mod-UE} tells us that $\Hom_{c\otimes/R}(\C_A,\D)$ is equivalent to the category of pairs $(F,h)$, where $F : \C \to \D$ is an $R$-linear cocontinuous tensor functor and $h$ is a  homomorphism of algebras $F(A \otimes_R \O_\C) \to \O_\D$. Since $F(A \otimes_R \O_\C) = A \otimes_R \O_\D$, we see that $h$ corresponds to a homomorphism of $R$-algebras $A \to \End(\O_\D)$, i.e. to an enrichment of $\D$ to an $A$-linear cocomplete tensor category.
If $\D$ is $A$-linear, it is clear that $\Hom_{c\otimes/A}(\C_S,\D) \subseteq \Hom_{c\otimes/R}(\C_S,\D)$ corresponds to the subcategory of those $(F,h)$ for which $h$ is the given $A$-enrichment of $\D$, i.e. we end up with the claim
\[\Hom_{c\otimes/A}(\C_S,\D) \simeq \Hom_{c\otimes/R}(\C,\D). \qedhere\]
\end{proof}

\begin{rem}
The connection to algebraic geometry is given as follows: If $\C=\Q(S)$ for some $R$-scheme $S$ and $A$ is a commutative $R$-algebra, then $\C_A \simeq \Q(S \otimes_R A)$, where $S \otimes_R A := S \times_{\Spec(R)} \Spec(A)$ is the scalar extension of $S$. The functor $S \mapsto S \otimes_R A$ from $R$-schemes to $A$-schemes is right adjoint to the forgetful functor.
\end{rem}

\section{Gradings} \label{gradings}


\begin{defi}
Let $\C$ be a cocomplete tensor category and $\Gamma$ be a commutative monoid (written additively). Then we can form the cocomplete tensor category of $\Gamma$-graded objects $\gr_{\Gamma}(\C)$ of $\C$. As a category, this is just $\prod_{n \in \Gamma} \C$. The tensor product is given by
\[(M \otimes N)_n = \bigoplus_{p+q=n} M_p \otimes N_q.\]
We define the symmetry using the symmetries of $\C$. Notice that $\gr_\Gamma(\C)$ is nothing else than $\widehat{\Gamma}_\C$ (see \autoref{cocompletions}) when we consider $\Gamma$ as a discrete tensor category.

We say that $M \in \gr_{\Gamma}(\C)$ is concentrated in degree $d \in \Gamma$ if $M_n = 0$ (the initial object) for all $n \neq d$. There is a canonical embedding $\iota : \C \to \gr_{\Gamma}(\C)$ which regards an object as a graded object concentrated in degree $0$. Notice that $\iota$ is left adjoint to $M \mapsto M_0$.
\end{defi}

Observe that if $X^n$ denotes the $\Gamma$-graded object which is concentrated in degree $n$ and is given by $\O_\C$ there, we can write every $M \in \gr_{\Gamma}(\C)$ as a kind of power series
\[M \cong \bigoplus_{n \in \Gamma} \iota(M_n) \otimes X^n.\]
The tensor product then looks like the usual multiplication of power series. In particular, we have $X^n \otimes X^m \cong X^{n+m}$. This suggests to think of $\gr_\Gamma(\C)$ as a sort of categorified group semiring and maybe even write
\[\gr_\Gamma(\C) = \C[\Gamma].\]
This is made precise by the following universal property:

\begin{prop}[Universal property of $\gr_\Gamma(\C)$] \label{grG}
Let $\Gamma$ be a commutative monoid and let $\C,\D$ be cocomplete tensor categories. Then there is an equivalence of categories
\[\Hom_{c\otimes}(\gr_\Gamma(\C),\D) \simeq \{(F,\L) : F \in \Hom_{c\otimes}(\C,\D),~ \L \in \Hom_{\otimes}(\Gamma,\D)\}\]
Explicitly, $\L$ associates to each object $n \in \Gamma$ a symtrivial object $\L_n \in \D$ and comes equipped with compatible isomorphisms $\L_0 \cong \O_\D$ and $\L_{p+q} \cong \L_p \otimes \L_q$. If $\Gamma$ is an abelian group, then all these $\L_i$ are line objects.
\end{prop}

\begin{proof}
This is a special case of from \autoref{hat}. Explicitly, to $(F,\L)$ one associates
\[G : \gr_\Gamma(\C) \to \D,~ M \mapsto \bigoplus_{n \in \Gamma} F(M_n) \otimes \L_n. \qedhere\]
\end{proof}
 
The special case $\Gamma=\Z$ will be important for us. Here $\gr_\Z(\C)$ has a line object $X$ (concentrated in degree $1$) and every object $M \in \gr_\Z(\C)$ may be written as a ``Laurent series''
\[M \cong \bigoplus_{n \in \Z} \iota(M_n) \otimes X^{\otimes n}.\]
 
\begin{cor}[Universal property of $\gr_\Z(\C)$] \label{grZ}
If $\C,\D$ are cocomplete tensor categories, then there is an equivalence of categories
\[\Hom_{c\otimes}(\gr_{\Z}(\C),\D) \simeq \{(F,\L) : F \in \Hom_{c\otimes}(\C,\D),~ \L \in \D \text{ line object}\}.\]
Here, a morphism $(F,\L) \to (F',\L')$ consists of a morphism $F \to F'$ and an isomorphism $\L \cong \L'$.
\end{cor}

Note that $\L \to \L'$ has to be an isomorphism because of \autoref{dual-inv}.
 
Recall the construction of the commutative and cocommutative bialgebra $\O_\C[\Gamma]$ (\autoref{hopf}), whose underlying object is written as $\bigoplus_{n \in \Gamma} \O_\C \cdot n$.

\begin{prop}[Alternative description] \label{gralt}
Let $\C$ be a cocomplete linear tensor category satisfying the following two conditions:
\begin{itemize}
\item $\C$ has equalizers, which are preserved by coproducts.
\item The canonical map $\Hom(T,\bigoplus_i M_i) \to \prod_i \Hom(T,M_i)$ is injective for all families of objects $T,M_i \in \C$.
\end{itemize}
For example, this holds when $\C$ is locally finitely presentable. Then, for every commutative monoid $\Gamma$, there is an equivalence of tensor categories
\[\gr_\Gamma(\C) \simeq \CoM(\O_\C[\Gamma]).\]
\end{prop}

\begin{proof}[Sketch of proof.]
If $\C=\M(R)$ for some commutative ring $R$, this result is well-known (\cite[II, \para 2, 2.5]{DG70}). We can use the same argument for those $\C$ satisfying the mentioned conditions -- as usual in categorical algebra we just have to avoid elements. Given a $\Gamma$-graded object $(M_n)_{n \in \Gamma}$, we endow $M:=\bigoplus_{n \in \Gamma} M_n$ with a comodule structure in such a way that on $M_n$ the coaction is given by $T^n$. This means that we define $\delta : M \to M[\Gamma]$ by $\delta|_{M_n} : M_n \cong M_n \cdot n \hookrightarrow M[\Gamma]$ for $n \in \Gamma$. The comodule axioms are readily checked. This defines a functor
\[\gr_\Gamma(\C) \to \CoM(\O_\C[\Gamma])\]
which is easily seen to be a cocontinuous tensor functor. It is fully faithful: For two graded objects $(M_n),(N_p)$, let $f : \bigoplus_n M_n \to \bigoplus_p N_p$ be a morphism which is compatible with the comodule actions. By the second condition on $\C$, it is determined by morphisms $f_{np} : M_n \to N_p$. The compatibility implies $f_{np}=0$ for $n \neq p$, so that $f$ is ``diagonal'' and therefore is induced by a unique morphism in $\gr_\Gamma(\C)$.

It remains to check essential surjectivity. Let $\delta : M \to M[\Gamma]$ be a coaction. For $n \in \Gamma$ we define a morphism $M_n \to M$ as the equalizer of $\delta$ and the morphism $M \cong M \cdot n \hookrightarrow M[\Gamma]$. We get a morphism $\bigoplus_n M_n \to M$. Using coassociativity of $\delta$, one checks that $M \to M[\Gamma] \xrightarrow{\pr_n} M$ factors through $M_n \to M$. This readily implies that $M \to M[\Gamma] = \bigoplus_n M$ factors through $\bigoplus_n M_n \to \bigoplus_n M$, using the first condition on $\C$. One can easily check now that the constructed morphisms between $M$ and $\bigoplus_n M_n$ are inverse to each other. We get the identity on $\bigoplus_n M_n$ by construction of $M_n$ and on $M$ because of counitality of $\delta$. Furthermore, $M \cong \bigoplus_n M_n$ is compatible with the comodule actions by construction.
\end{proof}

If $S$ is some base scheme and $\G_{m} := \G_{m,S}$ is the multiplicative group over $S$, we may consider its classifying stack $B \G_m$. For every $S$-scheme $T$ the groupoid $\Hom_S(T,B \G_m)$ is equivalent to the groupoid of invertible sheaves on $T$.
 
\begin{cor}[Globalization of $B \G_m$] \label{Gm-global}
If $S$ is some base scheme, then there is an equivalence of tensor categories
\[\Q(B \G_m) \simeq \gr_{\Z}(\Q(S)).\]
Hence, If $\C$ is a cocomplete tensor category and $F : \Q(S) \to \C$ is a cocontinuous tensor functor, then $\Hom_{c\otimes/\Q(S)}(\Q(B \G_m),\C)$ is equivalent to the groupoid of line objects of $\C$.
\end{cor}

\begin{proof}
Using the well-known equivalences of tensor categories
\[\Q(B \G_m) \simeq \Rep_S(\G_m) \simeq \CoM(\O_S[\Z]),\]
the claim follows from \autoref{gralt} and \autoref{grG}.
\end{proof}

As a general rule, global operations preserve tensorial stacks. In our case, we obtain the following result:

\begin{cor} \label{Gm-tens}
If $S$ is a tensorial scheme and $\G_m$ is the multiplicative group over $S$, then $B \G_m$ is a tensorial stack.
\end{cor}

\begin{proof}
If $Y$ is a scheme, then $\Hom_{c\otimes}(\Q(B \G_m),\Q(Y))$ is equivalent to
\[\{(F,\L) : F \in \Hom_{c\otimes}(\Q(S),\Q(Y)), \L \in \Q(Y) \text{ invertible}\}.\]
Since $S$ is tensorial, this simplifies to $\Hom(Y,B \G_m)$.
\end{proof}

\begin{defi}
If $\C$ is $R$-linear, we can endow the category of $\Z$-graded objects with another symmetry, induced by
\[M_p \otimes N_q \xrightarrow{(-1)^{pq}} M_p \otimes N_q \xrightarrow{S_{M_p,M_q}} N_q \otimes M_p.\]
This tensor category is called $\tilde{\gr}_{\Z}(\C)$. If we define $X$ as above, notice that $X$ is an anti-line object. In fact we have:
\end{defi}
 
\begin{prop}[Universal property of $\tilde{\gr}_\Z(\C)$] \label{grZ-twist}
Let $\C$ be a cocomplete $R$-linear tensor category. Then $\tilde{\gr}_{\Z}(\C)$ is a cocomplete $R$-linear tensor category with the following universal property: If $\D$ is a cocomplete $R$-linear tensor category, then there is an equivalence of categories
\[\Hom_{c\otimes/R}(\tilde{\gr}_{\Z}(\C),\D) \simeq \{(F,\L) : F \in \Hom_{c\otimes/R}(\C,\D),\, \L \in \D \text{ anti-line object}\}.\]
Here, a morphism $(F,\L) \to (F',\L')$ is a pair $(\alpha,f)$ consisting of a morphism $F \to F'$ and an isomorphism $f : \L \to \L'$.
\end{prop}

\begin{proof}
Notice that the free $R$-linear tensor category on an anti-line object is given by $\{X^{\otimes p} : p \in \Z\}$ with symmetry $(-1)^{pq} : X^p \otimes X^q \cong X^q \otimes X^p$. Now use \autoref{hat}.
\end{proof}

\begin{rem}
We may  similarly define $\tilde{\gr}_{\Z/2\Z}(\C)$. If $\C$ is the tensor category of vector spaces over a field, the objects of $\tilde{\gr}_{\Z/2\Z}(\C)$ are also known as \emph{super vector spaces}.
\end{rem}

\begin{defi}
If $M \in \gr_\Z(\C)$ and $d \in \Z$, the shift $M[d] \in \gr_\Z(\C)$ is defined by $M[d]_n = M_{n+d}$. Note that $M[d] = M \otimes X^{\otimes -d}$.
\end{defi}
  
\begin{defi}[Graded-commutative algebras and modules] \label{graded-comm} \noindent
\begin{enumerate}
 \item A (commutative) graded algebra in $\C$ is a (commutative) algebra in $\gr_{\Z}(\C)$. Such an algebra $A$ consists of a family of objects $A_n$ as well as morphisms $\O_\C \to A_0$ and $A_p \otimes A_q \to A_{p+q}$ in $\C$ subject to the usual conditions. Notice that we get a morphism $A \otimes \iota(A_q) \to A[q]$ in $\gr_{\Z}(\C)$.
 \item A graded-commutative algebra in $\C$ is a commutative algebra in $\tilde{\gr}_{\Z}(\C)$.
 \item If $A$ is one of these types of algebras, then $A$-modules (in the general sense defined in \autoref{module-categories} applied to $\gr_{\Z}(\C)$ resp. $\tilde{\gr}_{\Z}(\C)$) are also called graded $A$-modules (in order to stress that we do not mean modules over the underlying algebra in $\C$). If $A$ is commutative, we will also write $\grM(A)$ instead of $\M(A)$.
\end{enumerate}
\end{defi}

\begin{rem}  \noindent
\begin{enumerate}
\item This coincides with the usual notions when $\C= \M(R)$ for some commutative ring $R$. Using the theory of tensor categories, we see that actually the two notions of commutative graded algebras and graded-commutative algebras can be unified and both types of algebras are commutative in a unified sense. Quite a few ad hoc sign conventions for basics concerning graded-commutative algebras are actually part of natural generalizations to commutative algebras in arbitrary cocomplete linear tensor categories.
\item Strictly speaking, these extra notations are not necessary. Unfortunately, it is a quite common practice to apply forgetful functors secretly, so that $A$-modules, where $A$ is a graded algebra, usually refer to modules over the underlying algebra. When we do not want to throw away the graded structure, we say graded $A$-modules. In my opinion it would be much better just to say which tensor category one works in and to explicitly state each time which forgetful functors are applied.
\end{enumerate}
\end{rem}
 
\begin{prop}[Universal property of $\grM(A)$] \label{grMod}
Let $\C$ be a cocomplete tensor category and let $A$ be a commutative graded algebra in $\C$. Then $\grM(A)$ enjoys the following universal property: For every cocomplete tensor category $\D$, we have an equivalence of categories
\begin{eqnarray*}
 \Hom_{c\otimes}(\grM(A),\D) &  \simeq  &  \{(F,\L,\tau) : F \in \Hom_{c\otimes}(\C,\D),\, \L \in \D \text{ line object},\\
 && ~~\tau : (F(A_n))_n \to (\L^{\otimes n})_n \text{ graded algebra hom.}\}
\end{eqnarray*}
It maps $G : \grM(A) \to \D$ to  the triple $(F,\L,\tau)$, where $F := G(A \otimes \iota(-))$, $\L := G(A[1])$ and $\tau_n := G(A \otimes \iota(A_n) \to A[n])$.
\end{prop}

\begin{proof}
A combination of \autoref{mod-UE} and \autoref{grZ} yields that the category $\Hom_{c\otimes}(\grM(A),\D)$ is equivalent to the category of triples $(F,\L,\sigma)$, where $F \in \Hom_{c\otimes}(\C,\D)$, $\L \in \D$ is a line object and $\sigma : \bigoplus_n F(A_n) \otimes \L^{\otimes n} \to \O_\D$ is a morphism of algebras in $\O_\D$. If $\K$ denotes the inverse of $\L$, this corresponds to a family of morphisms $F(A_n) \to \K^{\otimes n}$ in $\C$, inducing a morphism of graded algebras. The analoguous statement for graded-commutative algebras follows similarly from a combination of \autoref{mod-UE} and \autoref{grZ-twist}.
\end{proof}

\begin{rem} \noindent
\begin{enumerate}
\item  If $\C$ is linear and $A$ is a graded-commutative algebra, the analogous universal property for $\grM(A)$ holds, where $\L$ becomes an anti-line object.
\item If $\C$ is locally presentable, then the same is true for $\grM(A)$.
\end{enumerate}
\end{rem}

\begin{cor} \label{grModSym}
Let $\C$ be a cocomplete tensor category and $E \in \C$. Then $\grM(\Sym(E))$ enjoys the following universal property: For every cocomplete tensor category $\D$ we have an equivalence of categories
\begin{eqnarray*}
\Hom_{c\otimes}(\grM(\Sym(E)),\D)&   \simeq &  \{(F,\L,\tau) : F \in \Hom_{c\otimes}(\C,\D),\\
 & & ~ \L \in \D \text{ line object},\, \tau \in \Hom_\D(F(E),\L)\}.
\end{eqnarray*}
In particular, for $d \in \N$, we have
\begin{eqnarray*}
\Hom_{c\otimes}(\grM(\O_\C[T_1,\dotsc,T_d])),\D)&   \simeq &  \{(F,\L,\tau) : F \in \Hom_{c\otimes}(\C,\D),\\
 & & ~~ \L \in \D \text{ line object},\, \tau : \O_\D^{\oplus d} \to \L\}.
\end{eqnarray*}
\end{cor}

Thus, $\grM(\O_\C[T_1,\dotsc,T_d])$ is the universal cocomplete tensor category over $\C$ containing a line object $\L$ with $d$ global sections of $\L$.

\begin{ex}
Let us look more closely at $\grM(\O_\C[T])$. Objects are sequences
\[\dotsc \xrightarrow{T} M_{-2} \xrightarrow{T} M_{-1} \xrightarrow{T} M_0 \xrightarrow{T} M_1 \xrightarrow{T} \dotsc\]
in $\C$. The unit is the sequence
\[\dotsc \to 0 \to 0 \to \O_\C \to \O_\C \to \dotsc\]
concentrated in degrees $\geq 0$. The tensor product of two sequences $M,N$ is the sequence $M \otimes N$ with $(M \otimes N)_n = \bigoplus_{p+q=n} M_p \otimes N_q / (T m \otimes n \sim m \otimes T n)$ in element notation. The universal invertible object $\L$ is given by the sequence
\[\dotsc \to 0 \to \O_\C \to \O_\C \to \O_\C \to \dotsc\]
concentrated in degrees $\geq -1$, its inverse being the concentrated in degrees $\geq +1$. The universal global section is given by
\[\xymatrix{\dotsc \ar[r] & 0 \ar[r] \ar[d] & 0 \ar[r] \ar[d]  & \O_\C \ar[r] \ar[d]  & \O_\C \ar[r] \ar[d]  & \dotsc \\ \dotsc \ar[r] & 0 \ar[r] & \O_\C \ar[r] & \O_\C \ar[r] & \O_\C \ar[r] & \dotsc }\]
\end{ex}
\section{Representations} \label{rep}


Let $I$ be a small category and let $\C$ be a cocomplete tensor category. Recall the cocomplete tensor category $\C^I$ of functors $I^{\op} \to \C$ from \autoref{funcat} whose tensor product and colimits are computed ``pointwise'' (as opposed to the Day convolution which we used in case of a tensor category $I$ in \autoref{cocompletions}). Similarly, ${}^I \C$ consists of functors $I \to \C$.

The following notion of a torsor generalizes \cite[Section 3.2]{CJF13}.

\begin{defi}[Torsors]
An $I$-torsor in $\C$ is a cocommutative coalgebra object $(T,\Delta,\e)$ in ${}^I \C$ such that
\begin{enumerate}
\item The counits $(\e_i : T(i) \to 1_\C)_{i \in I}$ form a colimit cocone, so that
\[\colim_{i \in I} T(i) \cong 1_\C.\]
\item For all $i,j \in I$, the canonical morphism
\[\colim_{k \to i, k \to j} T(k) \to T(i) \otimes T(j),\]
induced by $T(k) \xrightarrow{\Delta_k} T(k) \otimes T(k) \xrightarrow{T(g) \otimes T(h)} T(i) \otimes T(j)$ for all morphisms $g : k \to i$ and $h : k \to j$, is an isomorphism.
\end{enumerate}
We obtain a category $\Tors(I,\C)$, even a $2$-functor $\Tors(I,-) : \Cat_{c\otimes} \to \Cat$.
\end{defi}

\begin{rem}
Actually we have defined a left $I$-torsor. Working with $\C^I$, we arrive at the notion of a right $I$-torsor (which was also chosen in \cite{CJF13}). Of course it really makes no substantial difference since we may just  replace $I$ by $I^{\op}$. Notice that the two torsor axioms may be written as coend formulas:
\[\int^{i \in I} T(i) \cong 1_\C\]
\[\int^{k \in I} \Hom(k,i) \otimes \Hom(k,j) \otimes T(k) \cong T(i) \otimes T(j).\]
\end{rem}

\begin{ex}[Torsors over groups]
Let $G$ be a group, considered as a category with one object. Then a $G$-torsor in $\C$ is an object $T \in \C$ with a left $G$-action $G \otimes T \to T$ and a $G$-equivariant coalgebra structure $\Delta : T \to T \otimes T$, $\e : T \to 1$ such that $G \backslash T \cong 1$ and $G \otimes T \to T \otimes T$ defined by $(g,t) \mapsto (gt,t)$ in element notation is an isomorphism. More precisely, we define $G \otimes T  \to T \otimes T$ on the summand indexed by $g \in G$ as the composite $T \xrightarrow{\Delta} T \otimes T \xrightarrow{g \otimes T} T \otimes T$. If $\C$ is cartesian, then coalgebra structures exist and are unique. We therefore arrive at the usual generalized notion of a torsor which is common for example in topos theory (\cite[VIII.2, Definition 6]{Mac92}). In particular, for $\C=\Set$ we get the usual notion of a non-empty, free and transitive left $G$-set. An example is $G$ with its left regular action. Every other $G$-torsor in $\Set$ is isomorphic to $G$, but the isomorphism is not canonical. Notice $G$ can be even considered as a $G$-torsor in $\Set^G$, the category of right $G$-sets. This will turn out to be the universal $G$-torsor.
\end{ex}

\begin{ex}[Group algebras]
Since $\Set$ is the initial cocomplete tensor category, the $G$-torsor $G$ itself in $\Set$ produces a $G$-torsor in any cocomplete tensor category $\C$, namely the \emph{group algebra} $\O_\C[G]$ given by $\bigoplus_{g \in G} \O_\C$ with the obvious $G$-action, counit $\e \circ \iota_g = \id_{\O_\C}$ and comultiplication $\Delta \circ \iota_g = \iota_g \otimes \iota_g$, where $\iota_g$ is the inclusion of the summand indexed by $g \in G$.
\end{ex}

\begin{ex}[The universal torsor] \label{univtor}
Let $I$ be a small category. We define an $I$-torsor in the cartesian tensor category $\C:=\Set^I$ by
\[T(i) := \Hom(-,i) \in \Set^I\]
i.e. for short $T := \Hom(-,-) \in {}^I (\Set^I)$. The coalgebra structure exists and is unique because we are in the cartesian case. The two torsor axioms state:
\begin{enumerate}
\item The unique morphism $\colim_{i \in I} \Hom(-,i) \to 1$ is an isomorphism.
\item The canonical morphism
\[\colim_{k \to i, k \to j} \Hom(-,k) \to \Hom(-,i) \times \Hom(-,j)\]
is an isomorphism.
\end{enumerate}
Both follow immediately from the Yoneda Lemma.
\end{ex}

The following universal property is a slight generalization of \cite[Proposition 3.2.5, Lemma 2.3.12]{CJF13}, since we do not need any presentability assumptions and accordingly our proof is less abstract. See \cite[VIII.2, Theorem 7]{Mac92} for the case of topoi.

\begin{prop}[Universal property of $\C^I$] \label{torsUE}
Let $I$ be a small category and let $\C$ be a cocomplete tensor category. Then $\C^I$ is a cocomplete tensor category with the following universal property: If $\D$ is a cocomplete tensor category, there is an equivalence of categories
\[\Hom_{c\otimes}(\C^I,\D) \simeq \Hom_{c\otimes}(\C,\D) \times \Tors(I,\D).\]
A corresponding statement holds in the $R$-linear case.
\end{prop}

\begin{proof}
The diagonal $\Delta : \C \to \C^I$ is a cocontinuous tensor functor. The unique cocontinuous tensor functor $\Set \to \C$ induces $\Set^I \to \C^I$ and therefore maps the universal torsor $\Hom(-,-)$ of $\Set^I$ (\autoref{univtor}) to a torsor in $\C^I$, namely $\Hom(-,-) \otimes 1_\C$. Therefore, we get a functor
\[\Hom_{c\otimes}(\C^I,\D) \to \Hom_{c\otimes}(\C,\D) \times \Tors(I,\D)\]
which maps $H : \C^I \to \D$ to $(H \circ \Delta,\Hom(-,-) \otimes 1_\D)$. We construct an inverse functor: Given a torsor $T$ in $\D$ and a cocontinuous tensor functor $A : \C \to \D$, we define
\[H : \C^I \to \D, ~ F \mapsto \int^{i \in I} A(F(i)) \otimes T(i).\]
Then $H$ is a cocontinuous functor. The cocommutative coalgebra structure on $T$ allows us to make $H$ an oplax tensor functor and the torsor conditions imply that it is actually a strong tensor functor -- let us explain this in more details: We have
\[H(1) \cong \int^{i \in I} A(1) \otimes T(i) \cong \int^{i \in I} T(i) \cong 1_\D.\]
For $F,G \in \C^I$ we have
\begin{eqnarray*}
&& H(F) \otimes H(G) \\
& \cong & \int^{i,j \in I} A(F(i)) \otimes A(G(j)) \otimes T(i) \otimes T(j) \\
& \cong & \int^{i,j,k \in I} \Hom(k,i) \otimes \Hom(k,j) \otimes A(F(i)) \otimes A(G(j)) \otimes T(k) \\
& \stackrel{\text{Yoneda}}{\cong} & \int^{k \in I} A(F(k)) \otimes A(G(k)) \otimes T(k) \\
& \cong & H(F \otimes G).
\end{eqnarray*}
For $M \in \C$ we have
\[H(\Delta(M))=\int^{i \in I} A(M) \otimes T(i) \cong A(M) \otimes \int^{i \in I} T(i) \cong A(M).\]
This provides an isomorphism $H \circ \Delta \cong A$. Besides, we have (by Yoneda)
\[H(\Hom(-,-) \otimes 1_\C) = \int^{i \in I} \Hom(-,i) \otimes T(i) \cong T\]
If $H : \C^I \to \D$ is any cocontinuous tensor functor with $H \circ \Delta \cong A$ and $H(\Hom(-,-) \otimes 1_\C) \cong T$, we have for every $F \in \C^I$
\begin{eqnarray*}
H(F) & \cong & H\left(\int^{i \in I} \Delta(F(i)) \otimes \Hom(-,i)\right) \\
& \cong & \int^{i \in I}  H(\Delta(F(i))) \otimes  H(\Hom(-,i) \otimes 1_\C)  \\
& \cong & \int^{i \in I} A(F(i)) \otimes T(i).
\end{eqnarray*}
This finishes the proof.
\end{proof}

\begin{cor}
With the notation above we have that $\C^I$ is a $2$-coproduct of $\C$ and $\Set^I$ in $\Cat_{c\otimes}$.
\end{cor}

\begin{rem}
\autoref{torsUE} implies
\[\Hom_{c\otimes}(\Set^I,\D) \simeq \Tors(I,\D),\]
so that $\Hom(-,-) \in \Set^I$ is in fact the universal $I$-torsor. If $G$ is a group, the universal $G$-torsor is therefore the right $G$-set $G$ (together with its left $G$-action). In the setting of $R$-linear cocomplete tensor categories, it follows that the group algebra $R[G] \in \M(R)^G \simeq \Rep_R(G)$ is the universal $G$-torsor.
\end{rem}

\begin{cor} \label{torsdual}
Let $\C$ be $R$-linear. Let $G$ be a finite group. Then any $G$-torsor $T$ in $\C$ is a dualizable object of $\C$. In fact, we have $T^* = T$.
\end{cor}

\begin{proof}
We use \autoref{torsUE} (although it should be possible to give a direct proof). It suffices to prove that $R[G] \in \M(R)^G$ is self-dual. We define the coevaluation by
\[c : R \to R[G] \otimes R[G], ~ 1 \mapsto \sum_{g \in G} g \otimes g.\]
We define the evaluation by
\[e : R[G] \otimes R[G] \to R, ~ g \otimes h \mapsto \left\{\begin{array}{ll} 1 & g = h \\ 0 & \text{else.}\end{array}\right.\]
It is straightforward to check that both $c$ and $e$ are $G$-equivariant and satisfy the two triangle identities.
\end{proof}

Now let us connect all this to algebraic geometry.

\begin{thm}[Torsors in geometry]
Let $G$ be a finite group and $X$ be a scheme. Then there is an equivalence between the category $\Tors(G,\Q(X))$ and the category $\Tors(G,X)$ of $G$-torsors over $X$ in the sense of algebraic geometry, i.e. $X$-schemes $Y$ with a right action of $G$, such that \'{e}tale-locally on $X$ we have $Y \cong X \times G$.
\end{thm}

\begin{proof}
I have learned the following argument from Alexandru Chirvasitu.\\
The category $\Tors(G,X)$ is anti-equivalent (via $\Spec$) to the category of commutative quasi-coherent algebras $A$ on $X$ equipped with a right action of $G$, such that \'{e}tale-locally on $X$ we have $A \cong \O_X^G$ (where $\O_X^G$ carries the obvious $G$-action). In particular, $A$ is a locally free $\O_X$-module. Given such an algebra, we can associate a cocommutative coalgebra $A^*$ in $\Q(X)$ given by $\HOM_{\Q(X)}(A,\O_X)$ (which is quasi-coherent by \autoref{qchom}) with counit induced by the unit of $A$ and comultiplication induced by the multiplication of $A$. This works because we have $(A \otimes A)^* \cong A^* \otimes A^*$ since $A$ is of finite presentation. The right action of $G$ on $A$ induces a left action on $A^*$.

Let us check that $A^*$ is a $G$-torsor in $\Q(X)$. Denoting by $\Fix_G(A)$ the subalgebra of $G$-fixed points, we have $G \backslash A^* = (\Fix_G(A))^* = \O_X^* = \O_X$. The fact that $G \otimes A^* \to A^* \otimes A^*$ is an isomorphism may be reduced to the case that $A=\O_X^G$, where it identifies with the isomorphism $\O_X^{G \times G} \to \O_X^{G \times G}$ mapping $(g,h) \mapsto (gh,h)$ for $g,h \in G$. Thus, $A^*$ is a $G$-torsor in $\Q(X)$. This describes a functor (covariant, since it is a composition of two contravariant ones) from $\Tors(G,X)$ to $\Tors(G,\Q(X))$. It is easily seen to be fully faithful by restricting to the case of trivial torsors.

Conversely, given $T \in \Tors(G,\Q(X))$, then $T$ is dualizable as a quasi-coherent module (\autoref{torsdual}), hence locally free of finite rank (\autoref{dual-char}), in particular of finite presentation. This implies that $\HOM$ interchanges nicely with tensor products as above, so that the cocommutative coalgebra structure on $T$ yields a commutative algebra structure on $A := \HOM_{\Q(X)}(T,\O_X)$ and the left action of $G$ on $T$ yields a right action on $A$. The torsor conditions translate to $\O_X = \Fix_G(A)$ and $A \otimes A \cong A^G$, $a \otimes b \mapsto (a \otimes bg)_{g \in G}$. Since the underlying module of $A$ is locally free of finite rank and $\O_X$ is a subalgebra of $A$, we deduce that $A$ is a faithfully flat algebra. Since $ A \otimes A \cong A^G$ is \'{e}tale over $A$, it follows that $A$ is \'{e}tale (over $\O_X$). Thus, $\Spec(A) \to X$ is an \'{e}tale cover which trivializes $\Spec(A) \to X$. Hence, $\Spec(A) \to X$ is a $G$-torsor in the sense of algebraic geometry. The corresponding $G$-torsor in $\Q(X)$ is isomorphic to $T$, basically because the underlying module of $T$ is reflexive.
\end{proof}

\begin{cor}[Globalization of classifying stacks] \label{BG-global}
If $G$ is a finite group, then the classifying stack $BG$ globalizes to $(BG)_{\otimes} := \Rep(G)$. This means:
\begin{itemize}
\item $\Rep(G) \simeq \Q(BG)$ enjoys the universal property
\[\Hom_{c\otimes}(\Rep(G),\C) \simeq \Tors(G,\C).\]
\item For schemes $X$ we have the universal property
\[\Hom(X,BG) = \Tors(G,X) \simeq \Tors(G,\Q(X)).\]
\end{itemize}
In particular, $BG$ is tensorial.
\end{cor}

\begin{rem}[Further research]
It is probably possible to globalize $BG$ also for every (nice enough) algebraic group $G$, by finding a universal property of the category of comodules $\CoM(H)$ for (nice enough) Hopf algebra objects $H$. We have already seen the case of the multiplicative group $\G_m$ (\autoref{Gm-global}). If $R$ is a $\QQ$-algebra we hope that $\Rep_R(\GL_n)$ is the universal $R$-linear cocomplete tensor category equipped with a locally free object of rank $n$ (see \autoref{localfree}). See \cite[Theorem 1.6]{Iwa} for the setting of stable presentable $\infty$-categories. \marginpar{I've added Iwanari's paper.}
\end{rem}
\section{Local functors} \label{localfunc}


Most of the results of this section have already been appeared in \cite{BC14}. Our treatment is a little bit more general.

\begin{defi}[Local functors] \label{localdef}
Let $\C,\D,\E$ be tensor categories. Consider a tensor functor $i^* : \C \to \D$. Assume that its underlying functor has a right adjoint $i_* : \D \to \C$, so that $i_*$ becomes a lax tensor functor (\autoref{adj-lax}) and the unit $\eta : \id_\C \to i_* i^*$ as well as the counit $\e : i^* i_* \to \id_\D$ are morphisms of lax tensor functors. Assume that $i_*$ is fully faithful, i.e. the counit $\e : i^* i_* \to \id_\D$ is an isomorphism. Then a functor $F : \C \to \E$ is called \emph{$i$-local} if $F \eta : F \to F i_* i^*$ is an isomorphism.
\end{defi}

\begin{ex}
If $i : Y \to X$ is a quasi-compact immersion of schemes, then the pushforward $i_* : \Q(Y) \to \Q(X)$ is well-defined (\cite[Proposition 6.7.1]{EGAI}), fully faithful and is right adjoint to $i^* : \Q(X) \to \Q(Y)$, which is a tensor functor. In that case we simply say $Y$-local instead of $i$-local.
\end{ex}
  
\begin{prop} \label{local}
In the situation of \autoref{localdef} we have: 
\begin{enumerate}
 \item $F$ is $i$-local if and only if $F$ maps every morphism $\phi$, such that $i^* \phi$ is an isomorphism, to an isomorphism.  
 \item For every given lax tensor functor $G : \D \to \E$, the lax tensor functor $G i^* : \C \to \E$ is $i$-local. Conversely, to every $i$-local lax tensor functor $F : \C \to \E$ we may associate the lax tensor functor $F i_* : \D \to \E$. This establishes an equivalence of categories
\[\Hom_{\lax}(\D,\E) \simeq \{F \in \Hom_{\lax}(\C,\E)  \text{ $i$-local}\}.\]
\item This restricts to an equivalence of categories
\[\Hom_{\otimes}(\D,\E) \simeq \{F \in \Hom_{\otimes}(\C,\E)  \text{ $i$-local}\}.\]
\item If $\C,\D,\E$ are cocomplete tensor categories, then it even restricts to an equivalence of categories
\[\Hom_{c\otimes}(\D,\E) \simeq \{F \in \Hom_{c\otimes}(\C,\E)  \text{ $i$-local}\}.\]
\end{enumerate}
\end{prop}

\begin{proof}
$1.$ First note that $i^* \eta : i^* \to i^* i_* i^*$ is an isomorphism, since it is right inverse to the isomorphism $\varepsilon i^*$. This already shows one direction. Now assume that $F$ is $i$-local and consider a morphism $\phi : M \to N$ in $\C$ such that $i^* \phi$ is an isomorphism. The naturality of $\eta$ with respect to $\phi$ yields the commutative diagram
\[\xymatrix@C=48pt{F(M) \ar[r]^{F(\phi)} \ar[d]_{F(\eta_M)} & F(N) \ar[d]^{F(\eta_N)} \\ F(i_* i^* M) \ar[r]_{F (i_* i^* \phi)} & F(i_* i^* N).}\]
Since $F$ is $i$-local, the vertical arrows are isomorphisms. The bottom arrow is an isomorphism since $i^* \phi$ is one. Thus also the top arrow $F(\phi)$ is an isomorphism.

$2.$ Let $G : \D \to \E$ be a lax tensor functor. Since $i^* \eta$ is an isomorphism, we see that $G i^* : \C \to \E$ is $i$-local. Conversely, if $F : \C \to \E$ is an $i$-local lax tensor functor, then $F i_* : \D \to \E$ is a lax tensor functor. Of course, the same works for morphisms of lax tensor functors. We obtain functors
\[\xymatrix@C=40pt{\Hom_{\lax}(\D,\E) \ar@<1.2ex>[r]^-{- \circ i^*}  &  \ar@<1.2ex>[l]^-{- \circ i_*} \{F \in \Hom_{\lax}(\C,\E)  \text{ $i$-local}\}.}\]
Let us show that they are inverse to each other. Given a lax tensor functor $G : \D \to \E$ we have a natural isomorphism of functors $G \varepsilon : G i^* i_* \to G$, which is even a morphism of tensor functors by  \autoref{adj-lax}. Similarly, for an $i$-local lax tensor functor $F : \C \to \E$ the natural isomorphism $F \eta : F \to F i_* i^*$ is actually an isomorphism of tensor functors.
  
$3.$ Since $i^*$ is a tensor functor, we see that $G i^*$ is a tensor functor provided that $G$ is a tensor functor. Now assume that $F$ is an $i$-local tensor functor. We have to show that $F i_*$ is a tensor functor. The morphism $1_\E \to (F i_*)(1_\D)$ is defined as the composition $1_\E \to F(1_\C) \to F(i_* 1_\D)$. In that composition the first morphism is an isomorphism since $F$ is a tensor functor. The second one is $F$ applied to $\eta_{1_\C} : 1_\C \to i_* 1_\D$, hence also an isomorphism. Now let $M,N  \in \D$. The morphism $(F i_*)(M) \otimes (F i_*)(N) \to (F i_*)(M \otimes N)$ is defined as the composition $F(i_*(M)) \otimes F(i_*(N)) \to F(i_* M \otimes i_* N) \to F(i_* (M \otimes N))$, where the first morphism is an isomorphism since $F$ is a tensor functor and the second one is $c : i_* M \otimes i_* N \to i_* (M \otimes N)$ mapped by $F$. But $i^* c$ is an isomorphism. By $1.$ above, $F$ maps $c$ to an isomorphism.

$4.$ Since $i^*$ has a right adjoint, it is cocontinuous. We see that $G i^*$ is a cocontinuous tensor functor provided that $G$ is a cocontinuous tensor functor. Now assume that $F$ is an $i$-local cocontinuous tensor functor. We have to show that $F i_*$ is cocontinuous. This works as before: For a diagram $\{M_j\}$ in $\D$, the canonical morphism $\colim_j (F i_*)(M_j) \to (F i_*)(\colim_j M_j)$ factors as the isomorphism $\colim_j F(i_* M_j) \cong F(\colim_j i_* M_j)$ followed by $F$ applied to the canonical morphism $\colim_i i_* M_j \to i_* (\colim_i M_j)$, which is clearly an isomorphism after appyling $i^*$. Thus $F$ maps it to an isomorphism.
\end{proof}
 
\begin{cor}[Globalization of immersions] \label{open-immersion}
If $i : Y \to X$ is a quasi-compact immersion of schemes and $\C$ is a cocomplete tensor category, then the functors $i_*$ and $i^*$ induce an equivalence of categories
\[\Hom_{c\otimes}\bigl(\Q(Y),\C\bigr) \simeq \{F \in \Hom_{c\otimes}\bigl(\Q(X),\C\bigr) \text{ $Y$-local}\}.\]
\end{cor}

This is compatible with \autoref{imm}, so that tensorial schemes are closed under quasi-compact subschemes. For closed immersions we gain back \autoref{closed-UE}.

\begin{prop} \label{reduction}
Let $F : \Q(X) \to \Q(Y)$ be a cocontinuous tensor functor, where $X,Y$ are schemes. Let $i : U \to X$ be a quasi-compact open immersion. If $F$ is $U$-local and $U$ is tensorial, then $F \cong f^*$ for some morphism $f : Y \to X$ (which factors through $U$).
\end{prop}

\begin{proof}
Since $F$ is $U$-local, we have $F \cong G i^*$ for some cocontinuous tensor functor $G : \Q(U) \to \Q(Y)$ (\autoref{open-immersion}). Since $U$ is tensorial, we have $G \cong g^*$ for some morphism $g : Y \to U$. Then $F \cong (if)^*$.
\end{proof}

Motivated by this, we would like to characterize the condition of being $U$-local.

\begin{nota}
Let us fix a scheme $X$ and a cocontinuous tensor functor
\[F : \Q(X) \to \C\]
into a cocomplete tensor category $\C$. We also fix a quasi-compact open immersion $i : U \hookrightarrow X$.
\end{nota}

\begin{lemma} \label{mono-reicht}
The following are equivalent:
\begin{itemize}
\item $F$ is $U$-local.
\item For every submodule $M \subseteq N$ in $\Q(X)$ with $M|_U = N|_U$ we have that $F$ maps $M \hookrightarrow N$ to an isomorphism.
\end{itemize}
\end{lemma}
 
\begin{proof}
``$\Rightarrow$" is trivial. ``$\Leftarrow$": Let $\phi : M \to N$ be a homomorphism in $\Q(X)$ such that $\phi|_U$ is an isomorphism. Since $\phi$ factors as an epimorphism followed by a monomorphism, both of which are isomorphisms on $U$, it suffices to treat the case that $\phi$ is an epimorphism. Let
\[K = M \times_N M = \{(m,m') \in M \times M' : \phi(m)=\phi(m')\}\]
be the difference kernel of $\phi$. Since $\phi$ is an epimorphism, it is the coequalizer of the two projections $p_1,p_2 : K \rightrightarrows M$. Then, $F(\phi)$ is the coequalizer of $F(p_1)$ and $F(p_2)$, so that it suffices to show $F(p_1)=F(p_2)$. If $i : M \to K$ is the diagonal homomorphism defined by $m \mapsto (m,m)$, we have $p_1 i = p_2 i = \id_M$. Since $\phi$ is an isomorphism on $U$, the same must be true for $i$. Now, $i$ is a split monomorphism, so by assumption $F(i)$ is an isomorphism. But then we have $F(p_1)= F(i)^{-1} = F(p_2)$.
\end{proof}
 
\begin{lemma} \label{ideal-lemma}
Let $M \hookrightarrow N$ be a submodule in $\Q(X)$ and $I \subseteq \O_X$ be a quasi-coherent ideal such that $IN \subseteq M$ and $F$ maps $I \hookrightarrow \O_X$ to an isomorphism. Then $F$ maps $M \hookrightarrow N$ to an isomorphism.
\end{lemma}

\begin{proof}
Consider the following commutative diagram
\[\xymatrix{I \otimes M \ar[r] \ar[d] & I \otimes N \ar[r] \ar[d] & \O_X \otimes N \ar[d]^{\cong} \\ \O_X \otimes M \ar[r]^-{\cong} & M  \ar[r] & N}\]
and apply $F$. The left part of the diagram shows that $F$ maps $I \otimes M \to I \otimes N$ to a split monomorphism. Then the same must be true for $M \to N$. The right part shows that $F$ maps $M \to N$ to a split epimorphism. Hence, it gets mapped to an isomorphism.
\end{proof}

\begin{defi}
Let us call a quasi-coherent ideal $I \subseteq \O_X$ \emph{nice} if $F$ maps $I \hookrightarrow \O_X$ to an isomorphism.
\end{defi}

\begin{ex}
If $f : Y \to X$ is a morphism of schemes and we consider the associated pullback functor $F = f^* : \Q(X) \to \Q(Y)$, then $I$ is nice for $F$ if and only if $f$ factors through the open subscheme $X \setminus V(I)$. This follows easily from \autoref{pullback-closed}.
\end{ex}

\begin{prop} \label{weak}
Let us assume that $X$ is qc qs. Then $F$ is $U$-local if and only if every quasi-coherent ideal $I \subseteq \O_X$ with $I|_U = \O_U$ is nice.
\end{prop}
 
The proof reminds us of other proofs in homological algebra where some condition is reduced to inclusion of ideals, for example injectivity or flatness of modules.
  
\begin{proof}
``$\Rightarrow$" is trivial. ``$\Leftarrow$": We apply \autoref{mono-reicht}. So let $M \subseteq N$ be a submodule with $M|_U = N|_U$. Write $N$ as a directed colimit of quasi-coherent submodules $\{N_i\}$ of finite type (\cite[Corollaire 6.9.9]{EGAI}). Applying \autoref{pairing} to the pairing $\O_X \otimes M \to N$, we see that the ideal
\[I_i = \{s \in \O_X : s \cdot N_i \subseteq M\}\]
is quasi-coherent. We may replace $N_i$ by $M+N_i$ without changing $I_i$ and therefore assume that $M \subseteq N_i$. The ideal satisfies $I_i N_i \subseteq M$ and $I_i|_U = \O_U$. But then our assumption in conjunction with \autoref{ideal-lemma} implies that $F$ maps $N_i \hookrightarrow M$ to an isomorphism. In the colimit, we see that $F$ maps $N \hookrightarrow M$ to an isomorphism. 
\end{proof}
 
\begin{lemma} \label{split-epi}
If $I \subseteq \O_X$ is a quasi-coherent ideal and the inclusion gets mapped to a split epimorphism by $F$, then $I$ is nice.
\end{lemma}

\begin{proof}
This follows from \autoref{non-unital} applied to $F(I) \in \C$.
\end{proof}

\begin{lemma} \label{nice-closure}
Nice ideals constitute a ``filter":
\begin{enumerate}
\item $\O_X$ is nice.
\item If $I,J$ are ideals such that $I \subseteq J$ and $I$ is nice, then $J$ is nice.
\item If $I_1,I_2$ are nice, then the same is true for $I_1 \cdot I_2$.
\item If $X$ is noetherian and $\sqrt{I}$ is nice, then $I$ is nice.
\end{enumerate}

\end{lemma}
 
\begin{proof}
1. is trivial. 2. $F$ maps $I \to J \to \O_X$ to an isomorphism. Hence, $F$ maps $J \to \O_X$ to a split epimorphism. Now use \autoref{split-epi}. 3. Since $I_1 \to \O$ and $I_2 \to \O$ are mapped to an isomorphism, the same is true for their tensor product $I_1 \otimes I_2 \to \O$. But this factors through $I_1 I_2 \to \O$. Now use \autoref{split-epi} again. 4. There is some $d \in \N$ such that $\sqrt{I}^d \subseteq I$. Now apply 2. and 3..
\end{proof}

\begin{cor} \label{radical}
Let us assume that $X$ is noetherian. Then $F$ is $U$-local if and only if every radical ideal $I \subseteq \O_X$ with $I|_U = \O_U$ is nice.
\end{cor}

\begin{proof}
Combine \autoref{weak} and \autoref{nice-closure}.
\end{proof}
\section{Tensoriality} \label{tensoriality}

Recall (\autoref{tensorial-stacky}) that an algebraic stack $X$ is called \emph{tensorial} if for every scheme $Y$ (and then in fact for every algebraic stack) the pullback construction implements an equivalence of categories
\[X(Y) \simeq \Hom_{c\otimes}\bigl(\Q(X),\Q(Y)\bigr).\]
We already know that affine schemes are tensorial (\autoref{affin-tensorial}), more generally that tensorial stacks are closed under affine morphisms (\autoref{aff-tensor}). We will see later that they are closed under projective morphisms (\autoref{projcat}) and coproducts (\autoref{coprod-tens}).

In this section we prove that qc qs schemes and Adams stacks are tensorial.

\subsection{The case of qc qs schemes}

After having developed the theory of local functors in \autoref{localfunc} it is easy to prove that qc qs schemes are tensorial. This is the main result of a joint work with Alexandru Chirvasitu (\cite{BC14}).
 
\begin{thm}[Tensoriality of qc qs schemes] \label{qcqs-tensorial}
Let $X$ be a qc qs scheme and $Y$ be an arbitrary locally ringed space. Then $f \mapsto f^*$ implements an equivalence of categories
\[\Hom(Y,X) \simeq \Hom_{c\otimes}\bigl(\Q(X),\M(Y)\bigr)\]
In particular, $X$ is tensorial.
\end{thm}

\begin{rem}
Surprisingly, it follows that any cocontinuous tensor functor $\Q(X) \to \M(Y)$  automatically takes its values in $\Q(Y)$, and that the category of cocontinuous tensor functors $\Q(X) \to \Q(Y)$ is actually equivalent to a discrete category. For the case of affine $X$, see \autoref{affin-tensorial}.
\end{rem}
 
First, we treat fully faithfullness.
  
\begin{prop}[Fully faithfullness] \label{fullfaith}
Let $Y$ be a locally ringed space, $X$ a quasi-separated scheme and let $f,g : Y \to X$ be two morphisms. Let $\alpha:f^*\Rightarrow g^*$ be a morphism of tensor functors. Then, $f=g$ and $\alpha=\id$. 
\end{prop}
 
\begin{proof}
Let $y \in Y$. We claim that $f(y)$ is a specialization of $g(y)$. Indeed, let $I \subseteq \O_X$ be the vanishing ideal sheaf of the closed subset $Z:=\overline{\{g(y)\}}$ of $X$. Since $\alpha$ induces a map of $\O_Y$-algebras $f^*(\O_X/I) \to g^*(\O_X/I)$, it follows that $\supp g^*(\O_X/I) \subseteq \supp f^*(\O_X/I)$. By \cite[Chapitre 0, 5.2.4.1]{EGAI} this means that $g^{-1}(Z) \subseteq f^{-1}(Z)$. Now $y \in g^{-1}(Z)$ implies $f(y) \in Z = \overline{\{g(y)\}}$, which proves our claim.

If $i : U \hookrightarrow X$ is an affine open neighborhood of $f(y)$, then it must contain $g(y)$, too. Therefore we have $y \in f^{-1}(U) \cap g^{-1}(U)$. Since we may work locally on $Y$ in order to prove $f=g$ and $\alpha=\id$, we may even assume that $f$ and $g$ both factor through $U$, say $f=if'$ and $g=ig'$ for morphisms $f',g' : Y \to U$. Then $f^* \cong f'^* i^*$ is $U$-local, similarly $g^*$. Now \autoref{local} gives us a morphism of tensor functors $\alpha' : f'^* \Rightarrow g'^*$  with $i^* \alpha'=\alpha$. But now we are in the affine case and \autoref{affin-tensorial} yields $f'=g'$ and $\alpha'=\id$. This implies $f=g$ and $\alpha=\id$.
\end{proof}
  
\begin{rem} \label{work-locally}
In order to prove that the functor
\[\Hom(Y,X) \to \Hom_{c\otimes}\bigl(\Q(X),\M(Y)\bigr),\]
which we have just seen to be fully faithful, is an equivalence, we may work locally on $Y$. Namely, if $F : \Q(X) \to \M(Y)$ is a cocontinuous tensor functor and $Y$ is covered by open subspaces $Y_i$ such that $F|_{Y_i}$ is induced by a morphism $f_i : Y_i \to X$, then \autoref{fullfaith} guarantees the compatibility $f_i |_{Y_i \cap Y_j} = f_j |_{Y_i \cap Y_j}$, and hence that the $f_i$ glue to a morphism $f : Y \to X$ which induces $F$.
\end{rem}
 
\begin{prop} \label{local-ring}
\autoref{qcqs-tensorial} holds when $Y$ is a local ring (considered as a locally ringed space with one element).
\end{prop}
 
\begin{proof}
Let $F : \Q(X) \to \M(A)$ be a cocontinuous tensor functor, where $A$ is a commutative local ring with maximal ideal $\mathfrak{m}$. Choose some affine open covering $X = U_1 \cup \cdots \cup U_n$. If $F$ is not $U_k$-local for every $k$, then by \autoref{weak} there are quasi-coherent ideals $I_k \subseteq \O_X$ with $I_k |_{U_k} = \O_{U_k}$ such that $F(I_k \hookrightarrow \O_X)$ is not a split epimorphism, even not an epimorphism by \autoref{split-epi}. thus factors through $\mathfrak{m} \subseteq A = F(\O_X)$. Then the same is true for $F(\bigoplus_k I_k \to \O_X)$ and hence for $F(\sum_k I_k \hookrightarrow \O_X)$. This is a contradiction since $\sum_k I_k = \O_X$.
\end{proof}
  
\begin{prop} \label{map}
Let $X$ be a qc qs scheme and $Y$ an arbitrary locally ringed space. Let $F : \Q(X) \to \M(Y)$ be a cocontinuous tensor functor. Then, for every $y \in Y$, there is a local homomorphism $\O_{X,x} \to \O_{Y,y}$ for some $x\in X$, together with an isomorphism
\[F(M)_y \cong M_{x} \otimes_{\O_{X,x}} \O_{Y,y}\]
of tensor functors $\Q(X) \to \M(\O_{Y,y})$. Moreover, if we define $f(y) := x$, then the map $f : Y \to X$ is continuous.
\end{prop}

\begin{proof}
By \autoref{local-ring} the composition
\[F_y : \Q(X) \to \M(Y) \to \M(\O_{Y,y})\]
is induced by a morphism $\Spec(\O_{Y,y}) \to X$. By \cite[Proposition 2.5.3]{EGAI} this factors as $\Spec(\O_{Y,y}) \to \Spec(\O_{X,x}) \to X$ for some $x \in X$ and some local homomorphism $\O_{X,x} \to \O_{Y,y}$. This proves the first part of the statement. 

For the second part, let $Z \subseteq X$ be closed. Then there is a quasi-coherent ideal $I \subseteq \O_X$ with $Z=V(I)$ as sets. Then $\O_X \twoheadrightarrow \O_X/I$ gets mapped by $F$ to an epimorphism, say $\O_Y \twoheadrightarrow \O_Y/J$ for some ideal $J \subseteq \O_Y$. For $y \in Y$ and $x:=f(y)$, we have an isomorphism
\[(\O_Y/J)_y \cong (\O_X/I)_{x} \otimes_{\O_{X,x}} \O_{Y,y}.\]
Since $\O_{X,x} \to \O_{Y,y}$ is local, this shows that $\O_Y/J$ vanishes at $y$ if and only if $\O_X/I$ vanishes at $x$. We arrive at
\[f^{-1}(Z)=f^{-1}(\supp(\O_X/I))=\supp(\O_Y/J),\]
which is closed (\cite[Proposition 5.2.2]{EGAI}). Hence, $f$ is continuous.
\end{proof}

\begin{proof}[Proof of \autoref{qcqs-tensorial}]
Let $F : \Q(X) \to \M(Y)$ be a cocontinuous tensor functor. By \autoref{map} we may associate to it a continuous map $f : Y \to X$ of the underlying spaces. Cover $Y$ with open subsets which get mapped into affine open subsets of $X$. Since we may work locally on $Y$ (\autoref{work-locally}), we may therefore assume that $f$ factors as $Y \to U \to X$ for some affine open subscheme $i : U \hookrightarrow X$. We claim that $F$ is $U$-local in the sense of \autoref{localfunc}: Since isomorphisms can be checked stalkwise, it is enough to prove the claim for $F_y : \Q(X) \to \M(\O_{Y,y})$, where $y \in Y$. By \autoref{local-ring}, $F_y$ is induced by some morphism $g : \Spec(\O_{Y,y}) \to X$. Using \autoref{fullfaith} we see that as a continuous map $g$ is just a restriction of $f$ and therefore factors through $U$. Thus $F_y$ is $U$-local.
Since $F$ is $U$-local, we conclude that $F$ is induced by a morphism by \autoref{reduction}.
\end{proof}

\begin{rem} \label{constructive}
The proof of \autoref{qcqs-tensorial} is quite constructive. Given a cocontinuous tensor functor $F : \Q(X) \to \M(Y)$, the corresponding morphism $f : Y \to X$ is constructed as follows: Let $X = \cup_i U_i$ be a finite affine open covering. Let $u_i : U_i \to X$ denote the open immersions. If $y \in Y$, then there is some $i$ such that $F$ maps $\O_X \to (u_i)_* \O_{U_i}$ to an isomorphism (this is where the theory of local functors comes into play). By considering the endomorphism rings, we get a homomorphism of rings $\Gamma(U_i,\O_X) \to \O_{Y,y}$. The maximal ideal $\mathfrak{m}_y \subseteq \O_{Y,y}$ pulls back to a prime ideal of $\Gamma(U_i,\O_X)$, i.e. to a point $f(y) \in U_i$, together with a local homomorphism $\O_{X,f(y)} \to \O_{Y,y}$. Then $Y_i := \{y \in Y : f(y) \in U_i\}$ is an open subset and $F(\O_X \to (u_i)_* \O_{U_i})|_{Y_i}$ is an isomorphism. This induces a homomorphism $f^\# : \Gamma(X_i,\O_X) \to \Gamma(Y_i,\O_Y)$.
\end{rem}

\begin{ex}
Every morphism of the underlying ringed spaces $g : Y \to X$ induces a morphism of locally ringed spaces $f : Y \to X$ such that $g^* \cong f^*$. We just apply the above construction to $F = g^* : \Q(X) \to \M(Y)$. Explicitly, $f(y) \in X$ corresponds to the pullback of the maximal ideal $\mathfrak{m}_y$ of $\O_{Y,y}$ under $g^\#_y : \O_{X,g(y)} \to \O_{Y,y}$, so that $f(y)$ is a generization of $g(y)$.
\end{ex}

\begin{rem}
After completion of this thesis Hall and Rydh have shown in \cite{Hall14} that the theory of local functors in \autoref{localfunc} and therefore the proof of \autoref{qcqs-tensorial} may be generalized to the setting of algebraic stacks. It follows (\cite[Theorem 5.9]{Hall14}):   \marginpar{I've added this remark.} If $X$ is a qc qs algebraic stack with the property that every quasi-coherent module is a directed colimit of its quasi-coherent submodules of finite type (which turns out to be automatic by \cite{Ryd14}), and $X$ is covered by finitely many quasi-compact open tensorial substacks, then $X$ is tensorial, too.
\end{rem}

\subsection{The case of algebraic stacks}

In this section we unify and generalize results by Lurie (\cite{Lur04}) and Sch\"appi  (\cite{Sch12a}).
 
If $\C,\D$ are cocomplete tensor categories, we will denote by $\Hom_{c\otimes \cong}(\C,\D)$ the \emph{core} of $\Hom_{c\otimes}(\C,\D)$, i.e. the category of cocontinuous tensor functors but only with isomorphisms of tensor functors.

\begin{thm}[\cite{Lur04}] \label{fullfaith2}
Let $Y$ be a scheme and let $X$ be a geometric stack. Then, the pullback functor \marginpar{I've added a reference to Lurie's paper.}
\[X(Y) \to \Hom_{c\otimes \cong}\bigl(\Q(X),\Q(Y)\bigr)\]  
is fully faithful.
\end{thm}

\begin{proof}
Let $\pi : P \to X$ be a presentation with $P$ and $\pi$ affine. It follows that $P \times_X P \to P$ is affine, and therefore $P \times_X P$ is an affine scheme.
 
Let $f,g : Y \to X$ be two morphisms. Our goal is to prove that
\[\Isom(f,g) \to \Isom(f^*,g^*)\]
is a bijection. Since by descent theory for morphisms and quasi-coherent modules both sides are fpqc-sheaves on the category of $Y$-schemes evaluated at $Y$, and $h : Y \times_{(X \times X)} {(P \times P)} \to Y$ is an fpqc-cover such that $fh$ and $gh$ factor through $\pi$, we may assume that $f,g$ factor both through $\pi$, say $f = \pi \tilde{f}$ and $g = \pi \tilde{g}$. Then $\Isom(f,g)$ identifies with the set of all $h : Y \to P \times_X P$ such that $\pr_1 h = \tilde{f}$ and $\pr_2 h = \tilde{g}$. Consider it as a discrete category. Then it is equivalent to the category of all $h : Y \to P \times_X P$ equipped with isomorphisms $\pr_1 h \cong \tilde{f}$ and $\pr_2 h \cong \tilde{g}$. Since $P$ and $P \times_X P$ are affine, by \autoref{affin-tensorial} this is equivalent to the category of cocontinuous tensor functors $F : \Q(P \times_X P) \to \Q(Y)$ together with isomorphisms of tensor functors $F \pr_1^* \cong \tilde{f}^*$ and $F \pr_2^* \cong \tilde{g}^*$, which in turn (\autoref{mod-BW}) is equivalent to the category of pairs of cocontinuous tensor functors $F_1 : \Q(P) \to \Q(Y)$ and $F_2 : \Q(P) \to \Q(Y)$ together with isomorphisms of tensor functors $F_1 \pi^* \cong F_2 \pi^*$ and $F_1 \cong \tilde{f}^*$ and $F_2 \cong \tilde{g}^*$. This is equivalent to the discrete category of isomorphisms $\tilde{f}^* \pi^* \cong \tilde{g}^* \pi^*$, i.e. $f^* \cong g^*$. It is clear that the resulting bijection $\Isom(f,g) \cong \Isom(f^*,g^*)$ is the canonical one.
\end{proof}

The following is our main result.

\begin{thm} \label{tensdesc}
Let $X$ be an algebraic stack. Let $\pi : P \to X$ be an affine morphism such that $P$ is tensorial. Let $Y$ be a scheme and consider a cocontinuous tensor functor $F : \Q(X) \to \Q(Y)$ with the property that the algebra $F(\pi_* \O_P)$ is a descent algebra. Then $F \cong f^*$ for some morphism $f : Y \to X$.
\end{thm}

\begin{proof}
Let $A := \pi_* \O_P$ and let $F : \Q(X) \to \Q(Y)$ be a cocontinuous tensor functor such that $B := F(A)$ is a descent algebra. Then we have an affine descent morphism $\sigma : Q = \Spec(B) \to Y$, and $F$ induces a cocontinuous tensor functor $\M(A) \to \M(B)$, i.e. $F' : \Q(P) \to \Q(Q)$, in such a way that
\[\xymatrix{\Q(X) \ar[r]^{F} \ar[d]_{\pi^*} & \Q(Y) \ar[d]^{\sigma^*} \\ \Q(P) \ar[r]_{F'} & \Q(Q)}\]
commutes up to isomorphism (in fact this is a pushout by \autoref{mod-BW}). Since $P$ is tensorial by assumption, $F'$ is induced by a morphism $f' : Q \to P$. Now, the idea is to find a morphism $f : Y \to X$ such that
\[\xymatrix{Y \ar[r]^{f} & X \\ Q \ar[u]^{\sigma} \ar[r]_{f'} & P \ar[u]_{\pi}}\]
commutes up to isomorphism (in fact it will be a pullback) and lifts the above diagram. But this is possible by descent:

Consider the morphism $h : Q \xrightarrow{f'} P \xrightarrow{\pi} X$. We claim that $Q \times_Y Q \rightrightarrows Q \xrightarrow{h} X$ commutes (up to isomorphism). By \autoref{fullfaith2} it suffices to do this for the induced tensor functors
\[\Q(X) \xrightarrow{h^*} \Q(Q) \rightrightarrows \Q(Q \times_Y Q),\]
i.e. to show that $F(\pi_* \pi^* M) \otimes B \cong B \otimes F(\pi_* \pi^* M)$ as $B \otimes B$-modules. But this is just $F$ applied to the evident isomorphism $M \otimes A \otimes A \cong A \otimes M \otimes A$ of $A \otimes A$-modules. The cocycle condition is satisfied (again using \autoref{fullfaith2}), hence there is a unique morphism $f : Y \to X$ such that the above diagram commutes. We have an isomorphism of $B$-modules
\[f^*(M) \otimes B = \sigma^*(f^*(M)) \cong f'^*(\pi^*(M)) = F(\pi_* \pi^* M) = F(M) \otimes B.\]
But $B$ is a descent algebra and the cocycle condition is satisfied, hence the isomorphism comes from an isomorphism $f^*(M) \cong F(M)$. One checks that this is, in fact, an isomorphism of tensor functors $f^* \cong F$.
\end{proof}

\begin{cor}[\cite{Lur04}] \label{preserveflat} 
Let $Y$ be a scheme and let $X$ be a geometric stack. Then, the essential image of the fully faithful functor (see \autoref{fullfaith2}) \marginpar{I've added a reference to Lurie's paper.}
\[X(Y) \to \Hom_{c\otimes \cong}\bigl(\Q(X),\Q(Y)\bigr)\]
consists of those cocontinuous tensor functors which preserve faithfully flatness.
\end{cor}

\begin{proof}
If $f : Y \to X$ is a morphism, then clearly $f^* : \Q(Y) \to \Q(X)$ preserves faithfully flatness. Conversely, if $F : \Q(X) \to \Q(Y)$ is a cocontinuous tensor functor such that $F(\pi_* \O_P)$ is faithfully flat, where $\pi : P \to X$ is a presentation such that $\pi$ and $P$ are affine, then $F(\pi_* \O_P)$ is a descent algebra by \autoref{desc-krit}, so that the claim follows from \autoref{tensdesc}.
\end{proof}

\begin{cor}[\cite{Sch12a}] \label{adams-tensorial} 
Adams stacks are tensorial. \marginpar{I've added a reference to Sch\"appi's paper.}
\end{cor}

\begin{proof}
Let $X$ be an Adams stack and $\pi : P \to X$ be a presentation such that $\pi$ and $P$ are affine. The functor $X(Y) \to \Hom_{c\otimes \cong}\bigl(\Q(X),\Q(Y)\bigr)$ is fully faithful by \autoref{fullfaith2}. Since $X$ has the strong resolution property, \autoref{dual-inv} tells us that $\Hom_{c\otimes \cong}\bigl(\Q(X),-\bigr) =  \Hom_{c\otimes}\bigl(\Q(X),-\bigr)$. Now let $F : \Q(X) \to \Q(Y)$ be a cocontinuous tensor functor. Then $\pi_* \O_P$ is a special descent algebra by \autoref{adamsspec}. But then $F(\pi_* \O_P)$ is also a special descent algebra (\autoref{specstab}), in particular a descent algebra by \autoref{specdes}. Hence, we may apply \autoref{tensdesc} and obtain a morphism $f : Y \to X$ with $F \cong f^*$. This concludes the proof of $X(Y) \simeq  \Hom_{c\otimes}\bigl(\Q(X),\Q(Y)\bigr)$.
\end{proof}

\begin{cor} \label{BG}
If $G$ is a linear algebraic group defined over $\mathds{Z}$, then the classifying stack $BG$ is tensorial.
\end{cor}

\begin{proof}[Sketch of proof]
This follows from the observation that the Hopf algebra $\O(G)$ is a directed colimit of free $\Z$-modules.
\end{proof}

\begin{rem}
The main theorem of \cite{Hall14} (which appeared after completion of this thesis) states the following: If $X$ is a quasi-compact algebraic stack with quasi-affine diagonal and $T$ is a locally noetherian algebraic stack, then \marginpar{I've added this remark.}
\[\Hom(T,X) \to \Hom_{c\otimes}\bigl(\Q(X),\Q(T)\bigr),~ f \mapsto f^*\]
is an equivalence of categories.
\end{rem}

\begin{rem}
Bhatt has proven the following ``derived tensoriality'' of algebraic spaces (\cite[Theorem 1.5]{Bha14}): If $X,Y$ are algebraic spaces with $X$ qc qs, \marginpar{I've added this remark.} then $f \mapsto L f^*$ induces an equivalence of $\infty$-categories
\[\Hom(Y,X) \simeq \Hom_{c\otimes}\bigl(D(X),D(Y)\bigr).\]
Here, $D(X)$ denotes the symmetric monoidal stable derived $\infty$-category of quasi-coherent modules on $X$.
\end{rem}

\begin{prop}[Reduction]
Let $X$ be a geometric stack. Assume that for every geometric stack $X'$ which is affine over $X$ we can show the following: If $X'$ is defined over a commutative ring $A$ and $H : \Q(X') \to \M(A)$ is an $A$-linear cocontinuous tensor functor such that $I_H = 0$, \marginpar{I've added this Proposition.} then $H$ is induced by a morphism. Then, for every scheme $Y$ we have $\Hom(Y,X) \simeq \Hom_{c\otimes\cong}\bigl(\Q(X),\Q(Y)\bigr)$.
\end{prop}

\begin{proof}
By \autoref{fullfaith2} we only have to show that every cocontinuous tensor functor $F : \Q(X) \to \Q(Y)$ is induced by a morphism. We may assume that $Y$ is an affine scheme, say $Y=\Spec(A)$. By \autoref{changering} $F$ factors as $p^* : \Q(X) \to \Q(X_A)$ followed by an $A$-linear cocontinuous tensor functor $G : \Q(X_A) \to \Q(Y)$. Consider $I_G \subseteq \O_{X_A}$, i.e. the largest quasi-coherent ideal $I \subseteq \O_{X_A}$  such that $G(I \to \O_{X_A})$ vanishes. By \autoref{closedimage} $G$ factors as $\Q(X_A) \to \Q(V(I_G))$ followed by $H : \Q(V(I_G)) \to \Q(Y)$ such that $I_H = 0$. Since $V(I_G) \to X_A \to X$ is affine, by assumption $H$ is induced by a morphism. Hence, the same holds for $F$.
\end{proof}
\section{Localization} \label{localization-section}


\subsection{General theory}

In order to motiviate the following constructions, let us look at two basic examples.

\begin{ex}[Torsion-free modules] \label{torsion-free}
Let $R$ be an integral domain. Consider the full subcategory of torsion-free $R$-modules $\tfM(R) \hookrightarrow \M(R)$. We would like to endow it with the structure of a cocomplete tensor category. If the usual tensor product (i.e. that of $\M(R)$) of two torsion-free $R$-modules is again torsion-free, we can just take the tensor product of the underlying $R$-modules. For example, this happens when $R$ is a Dedekind domain, since then torsion-free $\Leftrightarrow$ flat. But in general, we can just mod out the torsion submodule and \emph{enforce} the tensor product to be torsion-free:
\[M \otimes_{\tfM(R)} N := (M \otimes_R N)/\Tor(M \otimes_R N).\]
The usual direct sum of a family of torsion-free $R$-modules is again torsion-free. However, this does not apply to cokernels. Again we can remedy this simply by modding out the torsion submodule. It follows that, for every $a \in R \setminus \{0\}$ and every torsion-free $R$-module $M$, the map $a : M \to M$ is an \emph{epimorphism} in $\tfM(R)$ (but of course not in $\M(R)$). Now one can check that $\tfM(R)$ becomes a cocomplete $R$-linear tensor category $\C$. In fact it is the universal one with the property that  $a : \O_{\C} \to \O_{\C}$ is an epimorphism for every $a \in R \setminus \{0\}$. Note that the inclusion $\tfM(R) \hookrightarrow \M(R)$ is far from being a cocontinuous tensor functor. But the reflection $\M(R) \to \tfM(R)$, $M \mapsto M/\Tor(M)$ is a cocontinuous tensor functor by construction.
\end{ex}

\begin{ex}[Sheaves] \label{sheaves}
Let $X$ be a ringed space. Then it is easy to define a cocomplete tensor category of presheaves of modules over $X$. Colimits and tensor products are constructed sectionwise. But this does not work for sheaves. Instead, we construct the tensor product of sheaves of modules to be the associated sheaf of the tensor product of the underlying presheaves of modules. Similarly, colimits of sheaves are constructed by taking associated sheaves of underlying colimits. Then we obtain a cocomplete tensor category $\M(X)$ of sheaves of modules on $X$.
\end{ex}

\begin{defi}[Orthogonality classes]
Let $S = (M_i \to N_i)_{i \in I}$ be a family of morphisms in a category $\C$. We always assume that $I$ is a set.
\begin{enumerate}
\item Then $S^{\cong}$ is the full subcategory of $\C$ consisting of those objects $A \in \C$ which \emph{believe} that each $M_i \to N_i$ is a isomorphism. More precisely, it is required that $\hom(N_i,A) \to \hom(M_i,A)$ is bijective for every $i \in I$. This is also known as the \emph{orthogonality class} associated to $S$.
\item Similarly, $S^{\twoheadrightarrow} \subseteq \C$ consists of those $A \in \C$ which believe that each $M_i \to N_i$ is an epimorphism, which means that for every $i \in I$ the induced map $\hom(N_i,A) \to \hom(M_i,A)$ is injective.
\item Finally, $S^{\hookrightarrow} \subseteq \C$ consists of those $A \in \C$ such that each of the maps $\hom(N_i,A) \to \hom(M_i,A)$ is surjective.
\end{enumerate}
\end{defi}

Observe $S^{\cong} = S^{\hookrightarrow} \cap S^{\twoheadrightarrow}$. In particular, when each $M_i \to N_i$ is an epimorphism, then $S^{\cong} = S^{\hookrightarrow}$. Each of $S^{\cong},S^{\twoheadrightarrow},S^{\hookrightarrow}$ is closed under limits taken in $\C$. It is an interesting question whether these subcategories are \emph{reflective}, i.e. if the inclusion functor has a left adjoint. It is well-known (\cite[Section 1.C]{AR94}) that this is the case for locally presentable categories; see also \cite{Kel80} for generalizations. We will include the proofs here and rearrange them in a certain way in order to get the corresponding results for locally presentable tensor categories. We also prove universal properties which should be well-known but for which we could not find a suitable reference.
 
\begin{defi}[Transfinite composition] \label{transfinite}
Let $\C$ be a cocomplete category, $R_1 : \C \to \C$ be an endofunctor and $\eta : \id_\C \to R_1$ be a natural transformation. Then we define endofunctors $R_{\alpha}$ of $\C$ for ordinal numbers $\alpha$ together with compatible natural transformations $R_{\alpha} \to R_{\beta}$ for $\alpha \leq \beta$ by transfinite recursion as follows:
\begin{itemize}
\item We let $R_0 = \id_\C$.
\item If $R_{\alpha}$ is already constructed, then we define $R_{\alpha+1} := R_1 \circ R_{\alpha}$ together with $\eta \circ R_{\alpha} : R_{\alpha} \to R_{\alpha+1}$.
\item If $\lambda$ is a limit ordinal and $R_{\alpha}$ for $\alpha < \lambda$ have been constructed, then choose a colimit $(R_{\alpha} \to R_{\gamma})_{\alpha < \lambda}$ in the category of endofunctors.
\end{itemize}
We call the sequence $(R_\alpha)$ the \emph{transfinite composition} of $R_1$. For example, we have $R_n = R_1 \circ \dotsc \circ R_1$ ($n$ times) for $n < \omega$ and $R_\omega$ is the colimit of the sequence $R_0 \to R_1 \to R_2 \to R_3 \to \dotsc$.

An object $M \in \C$ is called a \emph{fixed object} if $\eta(M) : M \to R_1(M)$ is an isomorphism. Fixed objects constitute a full subcategory $\Fix(R_1) \subseteq \C$.
\end{defi}

If $M \to N$ is a morphism in $\C$ and $N \in \Fix(R_1)$, then it extends uniquely to compatible morphisms $R_{\alpha}(M) \to N$, by induction on $\alpha$: The limit step is clear. If $R_{\alpha}(M) \to N$ is already constructed, then we may extend it via $R_{\alpha+1}(M) = R_1(R_{\alpha}(M)) \to R_1(N) \cong N$. This observation implies:

\begin{lemma}[Reflection] \label{fix}
With the notations of \autoref{transfinite}, assume that for every $M \in \C$ there is some ordinal $\alpha$ such that $R_{\alpha}(M)$ is fixed. Then the assignment $M \mapsto R_{\alpha}(M)$ extends to a functor $R_{\infty} : \C \to \Fix(R_1)$ which is left adjoint to the inclusion functor. Hence, $\Fix(R_1) \subseteq \C$ is reflective. In particular, $\Fix(R_1)$ is cocomplete, with colimits computed by applying $R_{\infty}$ to the underlying colimit in $\C$.
\end{lemma} 

This allows us to localize cocomplete categories:

\begin{prop} \label{epi-refl}
Let $S=(s_i : M_i \to N_i)_{i \in I}$ be a family of morphisms in a cocomplete category $\C$. Assume that the objects $M_i$ and $N_i$ are presentable for every $i \in I$. Then $S^{\twoheadrightarrow} \hookrightarrow \C$ is reflective, in particular cocomplete. If $R_{\lambda}$ is the reflector, then each $R_{\lambda}(s_i)$ is an epimorphism in $S^{\twoheadrightarrow}$. In fact, we have the following universal property: For every cocomplete category $\D$ we have an equivalence of categories
\[\Hom_c(S^{\twoheadrightarrow},\D) \simeq \{F \in \Hom_c(\C,\D) : F(s_i) \emph{ epimorphism for all } i \in I\}.\]
given by $G \mapsto G \circ R_\lambda$ and $F|_{S^{\twoheadrightarrow}} \mapsfrom F$.
\end{prop}
 
\begin{proof}
1. Choose some large enough regular cardinal $\lambda$ such that $M_i$ and $N_i$ are $\lambda$-presentable for every $i \in I$. For $A \in \C$, let $A \to R_1(A)$ be the coequalizer of all morphisms $u,v : N_i \to A$ which are equalized by $s_i$. In other words, $\Hom(R_1(A),B)$ identifies with
\[\{f \in \Hom(A,B) : \forall i \, \forall u,v \in \Hom(N_i,A) : u \circ s_i = v \circ s_i \Rightarrow f \circ u = f \circ v\}.\]
Then $R_1$ is an endofunctor of $\C$ with $\Fix(R_1) = S^{\twoheadrightarrow}$. We claim that $R_{\lambda}(A)$ is fixed and therefore provides the reflection by \autoref{fix}.

Let $u,v : N_i \to R_{\lambda}(A)$ be two morphisms which coequalize $M_i \to N_i$. Since $N_i$ is $\lambda$-presentable, there is some $\alpha<\lambda$ (resp. $\beta<\lambda$) such that $u$ (resp. $v$) factors as $N_i \to R_{\alpha}(A) \to R_{\lambda}(A)$. Let $\gamma < \lambda$ be the maximum of $\alpha,\beta$. Then $u,v$ factor both as $N_i \rightrightarrows R_{\gamma}(A) \to R_{\lambda}(A)$. Although the two morphisms $N_i \rightrightarrows R_{\gamma}(A)$ might not agree on $M_i$, they do when we increase $\gamma$, since $M_i$ is $\lambda$-presentable. Then by construction $R_{\gamma}(A) \to R_1(R_\gamma(A)) = R_{\gamma+1}(A)$ coequalizes both morphisms $N_i \rightrightarrows R_{\gamma}(A)$. It follows that $u=v$.

2. If $A \in S^{\twoheadrightarrow}$, then the map $\hom(N_i,A) \hookrightarrow \hom(M_i,A)$ identifies with the map $\hom(R_\lambda(N_i),A) \hookrightarrow \hom(R_\lambda(M_i),A)$, which shows that $R_\lambda(s_i)$ is an epimorphism in $S^{\twoheadrightarrow}$.

3. For a cocontinuous functor $F : \C \to \D$ such that each $F(s_i)$ is an epimorphism, it is clear that $F \to F R_1$ is an isomorphism. It follows by induction on ordinal numbers $\alpha$ that $F \to F R_{\alpha}$ is an isomorphism, too. In particular, $F \to F R_{\lambda}$ is an isomorphism. This implies that the restriction $F|_{ S^{\twoheadrightarrow}} : S^{\twoheadrightarrow} \to \D$ is cocontinuous, since colimits in $S^{\twoheadrightarrow}$ are computed via reflections of colimits in $\C$. Thus, both functors in the claim are well-defined and inverse to each other.
\end{proof}

\begin{rem} \label{pres-rem}
It follows from the construction above that $\C \xrightarrow{R_\lambda} S^{\twoheadrightarrow} \hookrightarrow \C$ preserves $\lambda$-directed colimits. This implies that $R_\lambda$ preserves $\lambda$-presentable objects. In particular, $R_\lambda(M_i)$ and $R_\lambda(N_i)$ are $\lambda$-presentable in $S^{\twoheadrightarrow}$.
\end{rem}
 
\begin{prop} \label{mono-refl}
Let $S=(s_i : M_i \to N_i)_{i \in I}$ be a family of epimorphisms in a cocomplete category $\C$ such that each $M_i$ is presentable. Then $S^{\cong}\hookrightarrow \C$ is reflective, in particular cocomplete. If $R_{\lambda}$ is the reflector, then each $R_{\lambda}(s_i)$ is an isomorphism in $S^{\cong}$. In fact, we have the following universal property: If $\D$ is a cocomplete category, there is an equivalence of categories
\[\Hom_c(S^{\cong},\D) \simeq \{F \in \Hom_c(\C,\D) : F(s_i) \emph{ isomorphism for all } i \in I\}.\]
\end{prop}
 
\begin{proof} 
1. Choose some large enough regular cardinal $\lambda$ such that each $M_i$ is $\lambda$-presentable. For $A \in \C$, let $R_1(A)$ be the colimit of the diagram
\[\xymatrix{M_i \ar[r] \ar[d] & N_i \\ A, & }\]
where $i \in I$ and $M_i \to A$ runs through all possible morphisms. In other words, $\Hom(R_1(A),B)$ identifies with
\[ \{(f,(g_i)) : f \in \Hom(A,B),~ g_i : \Hom(M_i,A) \to \Hom(N_i,B)\]
\[\text{ such that } \forall u \in \Hom(M_i,A) : f \circ u = g_i(u) \circ s_i\}\]
which simplifies to
\[\{f \in \Hom(A,B) : \forall u \in \Hom(M_i,A) : f \circ u \text{ factors through } s_i\}\]
because each $s_i$ is an epimorphism. The definition of $R_1(A)$ ensures that we have commutative diagrams
\[\xymatrix{M_i \ar[r] \ar[d] & N_i \ar[d] \\ A \ar[r] & R_1(A).}\]
Then $R_1$ is an endofunctor of $\C$ satisfying $\Fix(R_1) = S^{\hookrightarrow} = S^{\cong}$. We claim that $R_{\lambda}(A)$ is fixed and therefore provides the desired reflection by \autoref{fix}. Let $i \in I$ and $M_i \to R_{\lambda}(A)$ be a morphism. Using that $M_i$ is $\lambda$-presentable, the morphism factors as $M_i \to R_{\alpha}(A) \to R_{\lambda}(A)$ for some $\alpha < \lambda$. It extends to $N_i$ along $s_i$ because by construction $M_i \to R_{\alpha}(A) \to R_{\alpha+1}(A)$ extends to $N_i$.

2. For $A \in S^{\cong}$, the map $\Hom(R_{\lambda}(N_i),A) \to \Hom(R_{\lambda}(M_i),A)$ identifies with the map $\Hom(N_i,A) \to \Hom(M_i,A)$, which is therefore bijective. This proves that $R_{\lambda}(s_i)$ is an isomorphism.

3. If $F : \C \to \D$ is a cocontinuous functor such that each $F(s_i)$ is an isomorphism, then $F \to F R_1$ is an isomorphism. The rest of the proof works as in \autoref{epi-refl}.
\end{proof}

Now we combine the two constructions:

\begin{prop}[Localization] \label{iso-refl}
Let $S=(s_i : M_i \to N_i)_{i \in I}$ be a family of morphisms in a cocomplete category $\C$. Assume that $M_i$ and $N_i$ are presentable for every $i \in I$. Then $S^{\cong} \hookrightarrow \C$ is reflective, in particular cocomplete. If $R$ is the reflector, then $R(s_i)$ is an isomorphism in $S^{\cong}$. In fact, we have the following universal property: If $\D$ is a cocomplete category, there is an equivalence of categories
\[\Hom_c(S^{\cong},\D) \simeq \{F \in \Hom_c(\C,\D) : F(s_i) \emph{ isomorphism for all } i \in I\}.\]
\end{prop}

\begin{proof}
By \autoref{epi-refl} $S^{\epi}$ is reflective in $\C$, say with reflector $R_\lambda$. Because of \autoref{pres-rem} we may apply \autoref{mono-refl} to the epimorphisms $R_\lambda s_i$ in $S^{\epi}$. It follows that $(S^{\epi})^{\cong} = S^{\cong}$ is reflective in $S^{\epi}$. Hence, $S^{\cong}$ is reflective in $\C$, the reflector being the composition of the reflectors $\C \to S^{\epi} \to S^{\cong}$. If $\D$ is a cocomplete category, then $\Hom_c(S^{\cong},\D)$ identifies with the category of those functors $G \in \Hom_c(S^{\epi},\D)$ such that $G(R_\lambda s_i)$ is an isomorphism for all $i$. But $\Hom_c(S^{\epi},\D)$ identifies with the category of all $F \in \Hom_c(\C,\D)$ such that each $F(s_i)$ is an epimorphism, via $F = G R_\lambda$. This proves the universal property.
\end{proof}

\begin{rem} \label{presdis}
We are going to extend these results to cocomplete tensor categories. See \cite{Day73} for a similar account. If $s_i : M_i \to N_i$ is an isomorphism in a tensor category $\C$, then $s_i \otimes T : M_i \otimes T \to N_i \otimes T$ is an isomorphism for every $T \in \C$. By the Yoneda Lemma this is equivalent to the condition that $\HOM(N_i,A) \to \HOM(M_i,A)$ is an isomorphism for all $A \in \C$, provided that these internal homs exist. This motivates the definition of the \emph{tensor closure} of $S = (s_i)_{i \in I}$ by
\[S_{\otimes} = (s_i \otimes T)_{i \in I,\, T \in \C}\]
and looking at the corresponding orthogonality class $S_{\otimes}^{\cong}$ instead of $S^{\cong}$, which consists of those objects $A \in \C$ such that $\HOM(N_i,A) \to \HOM(M_i,A)$ is an isomorphism for all $i$. Our previous results can be applied if we replace $T \in \C$ by $T \in \C'$ for a colimit-dense (essentially) small subcategory $\C' \subseteq \C$ so that $S_{\otimes}$ is still a \emph{set} of morphisms. This clearly does not change the orthogonality class. If $\C$ is $\lambda$-presentable, we can take $\C'$ to be the category of $\lambda$-presentable objects.
\end{rem}
 
\begin{thm}[Tensor localization] \label{localization}
Let $S = (s_i : M_i \to N_i)_{i \in I}$ be a family of morphisms in a locally presentable tensor category $\C$ with unit $\O_\C$. Then
\[S_{\otimes}^{\cong} = \{A \in \C : \HOM(N_i,A) \to \HOM(M_i,A) \text{ isomorphism for all  } i \in I\}\]
is reflective in $\C$, say with reflector $R : \C \to S_{\otimes}^{\cong}$. Then $S_{\otimes}^{\cong}$ becomes a locally presentable tensor category with unit $R(\O_\C)$ and tensor product
\[M \otimes_{S_{\otimes}^{\cong}} N := R(M \otimes N).\]
Besides, $R : \C \to S_{\otimes}^{\cong}$ becomes a cocontinuous tensor functor with the property that $R(s_i)$ is an isomorphism for all $i \in I$, in fact the universal one: If $\D$ is a cocomplete tensor category, then $R$ induces an equivalence of categories
\[\Hom_{c\otimes}(S_{\otimes}^{\cong},\D) \simeq \{F \in \Hom_{c\otimes}(\C,\D) : F(s_i) \text{ isomorphism for all } i \in I\}.\]
\end{thm}

Therefore, we may write $S_{\otimes}^{\cong} = \C[S^{-1}]$.

\begin{proof}
By \autoref{presdis} and \autoref{iso-refl} there is a reflector
\[R : \C \to S_{\otimes}^{\cong}.\]
We claim first:
\[ ~ ~ ~ \forall M\in \C,~ \forall N \in S_{\otimes}^{\cong} :~ \HOM(M,N) \in S_{\otimes}^{\cong} ~ ~ ~ (\star) \]
This follows from the commutative diagram
\[\xymatrix{ \HOM(N_i,\HOM(M,N)) \ar[r] \ar[d]^{\cong} &  \HOM(M_i,\HOM(M,N)) \ar[d]^{\cong}  \\ \HOM(M,\HOM(N_i,N)) \ar[r]_{\cong} &  \HOM(M,\HOM(M_i,N)).}\]
Next, we claim the canonical morphism (induced by $A \to R(A)$)
\[ ~ ~ ~  R(A \otimes B) \to R(R(A) \otimes B) ~ ~ ~ (\star \star)\]
is an isomorphism for all $A,B \in \C$. In fact, if $T \in S_{\otimes}^{\cong}$, then
\[\Hom(R(R(A) \otimes B),T) \cong \Hom(R(A) \otimes B,T) \cong \Hom(R(A),\HOM(B,T)) \]
\[\stackrel{(\star)}{\cong} \Hom(A,\HOM(B,T)) \cong \Hom(A \otimes B,T) \cong \Hom(R(A \otimes B),T).\]
This can be used to show that $R(\O_\C)$ is a unit and that the tensor product on $S_{\otimes}^{\cong}$ is associative:

For $A,B,C \in S_{\otimes}^{\cong}$ we have
\[R(\O_\C) \otimes_{S_{\otimes}^{\cong}} A = R(R(\O_\C) \otimes A) \stackrel{(\star\star)}{\cong} R(\O_\C \otimes A) \cong R(A) \cong A,\]
as well as
\[A \otimes_{S_{\otimes}^{\cong}} (B \otimes_{S_{\otimes}^{\cong}} C) = R(A \otimes R(B \otimes C)) \stackrel{(\star\star)}{\cong} R(A \otimes (B \otimes C))\]
and similarly $(A \otimes_{S_{\otimes}^{\cong}} B) \otimes_{S_{\otimes}^{\cong}} C \cong R((A \otimes B) \otimes C)$. Hence, there are unique isomorphisms $R(\O_\C) \otimes_{S_{\otimes}^{\cong}} A \cong A$ and $A \otimes_{S_{\otimes}^{\cong}} (B \otimes_{S_{\otimes}^{\cong}} C) \cong (A \otimes_{S_{\otimes}^{\cong}} B) \otimes_{S_{\otimes}^{\cong}} C$ such that the diagrams
\[\xymatrix{R(\O_\C) \otimes_{S_{\otimes}^{\cong}} A \ar[r]^-{\cong} &  A \\ \O_\C \otimes A \ar[u] \ar[r]^-{\cong} & A \ar@{=}[u] } \]
\[\xymatrix{A \otimes_{S_{\otimes}^{\cong}} (B \otimes_{S_{\otimes}^{\cong}} C) \ar[r]^-{\cong} & (A \otimes_{S_{\otimes}^{\cong}} B) \otimes_{S_{\otimes}^{\cong}} C \\ A \otimes (B \otimes C) \ar[r]^-{\cong} \ar[u] & (A \otimes B) \otimes C \ar[u]}\]
commute. It is clear that $\otimes_{S_{\otimes}^{\cong}}$ is commutative. The axioms of a tensor category follow immediately from the ones for $\C$. For $A,B,C \in S_{\otimes}^{\cong}$ we have
\[\Hom(A \otimes_{S_{\otimes}^{\cong}} B,C) = \Hom(R(A \otimes B),C)\]
\[ \cong \Hom(A \otimes B,C) \cong \Hom(A,\HOM(B,C))\]
with $\HOM(B,C) \in S_{\otimes}^{\cong}$ by $(\star)$. This proves that $S_{\otimes}^{\cong}$ is closed and therefore a cocomplete tensor category. It is locally presentable because it is a reflective subcategory of $\C$ closed under $\lambda$-directed colimits (here $\lambda$ is as in the previous Propositions). By construction $R : \C \to S_{\otimes}^{\cong}$ is a cocontinuous tensor functor which maps the morphisms in $S$ to isomorphisms.

If $\D$ is a cocomplete tensor category, then $G \mapsto G \circ R$ defines a functor
\[\Hom_{c\otimes}(S_\otimes^{\cong},\D) \to \{F \in \Hom_{c\otimes}(\C,\D) : F(s_i) \text{ isomorphism for all } i \in I\}.\]
We construct an inverse functor as follows: Given $F : \C \to \D$, by \autoref{iso-refl} we find a corresponding cocontinuous functor $G : S_{\otimes}^{\cong} \to \D$ with $G R \cong F$. The isomorphisms $G(R(\O_\C)) \cong F(\O_\C) \cong \O_\D$ and
\[G(A \otimes_{S_{\otimes}^{\cong}} B)=G(R(A \otimes B)) \cong F(A \otimes B) \cong F(A) \otimes F(B) \cong G(A) \otimes G(B)\]
for $A,B \in S_{\otimes}^{\cong}$ endow $G$ with the structure of a tensor functor. It is easy to check that this provides a functor $F \mapsto G$ which is inverse to the given one.
\end{proof}

\begin{thm}[Solution of universal problems] \label{universal-problem}
Let $\C$ be a locally presentable tensor category.
\begin{enumerate}
\item If $f : A \to B$ is a morphism in $\C$, then there is a locally presentable tensor category $\C/(f \text{ isom.})$ with the universal property
\[\Hom_{c\otimes}(\C/(f \text{ isom.}),\D) \simeq \{F \in \Hom_{c\otimes}(\C,\D) : F(f) \text{ isom.}\}\]
for all cocomplete tensor categories $\D$.
\item If $f,g : A \to B$ are parallel morphisms in $\C$, then there is a locally presentable tensor category $\C/(f=g)$ with the universal property
\[\Hom_{c\otimes}(\C/(f=g),\D) \simeq \{F \in \Hom_{c\otimes}(\C,\D) : F(f)=F(g)\}\]
for all cocomplete tensor categories $\D$.
\item If $A \rightrightarrows B \to P$ is a commutative diagram in $\C$, then there is a locally presentable tensor category $\C/(A \rightrightarrows B \to P \text{ exact})$ with the universal property
\[\Hom_{c\otimes}(\C/(A \rightrightarrows B \to P \text{ exact}),\D)\]
\[ \simeq \{F \in \Hom_{c\otimes}(\C,\D) : F(A) \rightrightarrows F(B) \to F(P) \text{ exact}\}\]
for all cocomplete tensor categories $\D$.
\item If $f : A \to B$ is a morphism in $\C$, then there is a locally presentable tensor category $\C/(f \text{ epi.})$ with the universal property
\[\Hom_{c\otimes}(\C/(f \text{ epi.}),\D) \simeq \{F \in \Hom_{c\otimes}(\C,\D) : F(f) \text{ epi.}\}\]
for all cocomplete tensor categories $\D$.
\item If $M_0 \in \C$, then there is a locally presentable tensor category $\C/(M_0=0)$ with the universal property
\[\Hom_{c\otimes}(\C/(M_0=0),\D) \simeq \{F \in \Hom_{c\otimes}(\C,\D) : F(M_0)=0\}\]
for all cocomplete tensor categories $\D$.
\end{enumerate}
\end{thm}

\begin{proof}
1. is a special case of \autoref{localization}. The rest will follow from 1. using the following reformulations: 2. We have $f=g : A \to B$ if and only if their coequalizer $B \to \coeq(f,g)$ is an isomorphism. 3. For morphisms $f,g : A \to B$ a commutative diagram $A \rightrightarrows B \to P$ is exact if and only if the induced morphism $\coeq(f,g) \to P$ is an isomorphism. 4. A morphism $f : A \to B$ is an epimorphism if and only if the canonical morphism from the pushout $B \cup_A B \to B$ is an isomorphism. 5. We have $M_0=0$ if and only if the unique morphism $0 \to M_0$ is an isomorphism.
\end{proof}

\begin{rem}[Explicit description]
In each case of \autoref{universal-problem}, the underlying category of $\C/(\dotsc)$ is a reflective subcategory of $\C$:
\begin{eqnarray*}
\C/(f \text{ isom.}) & =&  \{M \in \C : \HOM(f,M) : \HOM(B,M) \to \HOM(A,M) \text{ isom.}\} \\
\C/(f=g) & =&  \{M \in \C : \HOM(f,M) = \HOM(g,M)\} \\
\C/(f \text{ epi.}) & =& \{M \in \C : \HOM(f,M) : \HOM(B,M) \to \HOM(A,M) \text{ mono.}\} \\
\C/(A=0) & =&  \{M \in \C : \HOM(A,M)=\text{ terminal object}\}
\end{eqnarray*}
Besides, $\C/(A \rightrightarrows B \to P \text{ exact})$ is given by
\[\{M \in \C : \HOM(P,M) \to \HOM(B,M) \rightrightarrows \HOM(A,M) \text{ exact}\}.\]
\end{rem}


\subsection{Examples} \label{examplesloc}

\begin{ex}
We have already seen an example in \autoref{local}: If $i^* : \C \to \D$ is a cocontinuous tensor functor with a fully faithful right adjoint $i_* : \D \to \C$, then $i^*$ inverts the unit morphisms $\id_\C \to i_* i^*$. In fact we have
\[\D \simeq \C / (M \to i_* i^* M \text{ isom.})_{M \in \C}.\]
\end{ex}
 
\begin{ex}[Killing objects] \label{M0}
Let us describe $\C/(M_0=0)$. Assume for simplicity that $\C$ is linear (or more generally that $0$ is terminal). Then we have an epimorphism $M_0 \to 0$ and $M_0=0$ means that it is an isomorphism. Following the construction in the proof of \autoref{mono-refl}, we may define  $R_1 : \C \to \C$ as follows: $R_1(M)$ is the cokernel of all morphisms $T \otimes M_0 \to M$, where $T \in \C$. We may assume $T$ to be $\lambda$-presentable when $\C$ is $\lambda$-presentable, so that $R_1(M)$ is well-defined. We have $\Fix(R_1)=\{M \in \C : \HOM(M_0,M)=0\} = \C/(M_0=0)$. The reflector is the transfinite composition of $R_1$. In general, one cannot say exactly when it stops.

Let us look at $\C=\M(A)$ for some commutative ring $A$ and $M_0 = A/(a)$ for some $a \in A$, so that $\C/(M_0=0)=\{M \in \M(A) : \HOM(A/(a),M)=0\}$ consists of those $A$-modules without $a$-torsion. We have
\[R_1(M)=M/\ker(a : M \to M).\]
By induction it follows that
\[R_n(M) = M/\ker(a^n : M \to M)\]
for $n<\omega$, hence
\[R_{\omega}(M)=M/\bigcup_{n < \omega} \ker(a^n : M \to M).\]
This is already fixed, therefore $R_{\omega}$ is the reflector. The category of $A$-modules without $a$-torsion is the universal $A$-linear cocomplete tensor category for which $a : \O \to \O$ is an epimorphism. Instead of $\{a\}$ we can take any subset of $A$, for example the set of regular elements, which leads to the category of torsion-free $A$-modules. See also \autoref{torsion-free}.

Now let us look at the case that $\C$ is a quantale. Then
\[\C/(M_0=0) = \{M \in \C : \forall T \in \C ~ (M_0 \cdot T \leq M)\}.\]
The reflector starts with $R_1(M) = \sup(M,\sup_{T \in C} M_0 \cdot T)$. If we have $T \leq 1$ for every $T \in \C$, this simplifies to $\C/(M_0=0) = \{M \in \C : M_0 \leq M\}$ and  the reflector is just $R_1(M) = \sup(M,M_0)$. For a specific example take
\[[0,1]/(\tfrac{1}{2}=0) = [\tfrac{1}{2},1].\]
\end{ex}

\begin{ex}
We explain how \autoref{sheaves} embeds into our theory. Let $X$ be a ringed space and $\M'(X)$ be the cocomplete tensor category of presheaves of modules over $X$ -- with colimits and tensor products simply computed sectionwise, e.g. $(F \otimes_{\O_X} G)(U) := F(U) \otimes_{\O_X(U)} G(U)$. For an open subset $U \subseteq X$ let us denote by $\O_U \in \Q(X)$ the structure sheaf of $U$ extended by zero on $X$. Hence, we have $\Hom(\O_U,F) \cong F(U)$. Let $S$ be the set of canonical diagrams
\[\bigoplus_{i,j \in I} \O_{U_i \cap U_j} \rightrightarrows \bigoplus_{i \in I} \O_{U_i} \to \O_U\]
where $U$ runs through the open subsets of $X$ and $(U_i)_{i \in I}$ is an open covering of $U$ (in order to get a set, we should assume that the $U_i$ do not repeat).
Then $F \in \M'(X)$ believes that $S$ consists of coequalizer diagrams if and only if $F$ is a sheaf i.e. $F \in \M(X)$. It follows that
\[\M(X) \simeq \M'(X) / (\text{the diagrams in } S \text{ are coequalizers}).\]
The reflector starts with the $0$th \v{C}ech cohomology presheaf $\check{H}^0(-)$ (``plus construction''). It is well-known that this is separated and that it maps separated presheaves to sheaves. Therefore, the transfinite composition already stops at the second step $\check{H}^0(-) \circ \check{H}^0(-)$. Curiously, for $n$-stacks one needs $n+1$ steps and for $\infty$-stacks one needs $\omega$ steps (\cite[Section 6.5.3]{HTT}), 
\end{ex}

\begin{ex} \label{EL-lok}
Let $s : E \to \L$ be a morphism in $\C$, where $\L$ is an invertible object. Then $\C/(s \text{ isom.})$ consists of those $M \in \C$ such that the morphism $\HOM(\L,M) \to \HOM(E,M)$ induced by $s$ is an isomorphism. Because of $\HOM(\L,M) = M \otimes \L^{\otimes -1}$, this is the case if and only if the corresponding morphism $M \to \HOM(E,M) \otimes \L$ is an isomorphism. Hence, if we define $R_1 : \C \to \C$ by $M \mapsto \HOM(E,M) \otimes \L$, then $\Fix(R_1)=\C/(s \text{ isom.})$ and the transfinite composition $R_{\infty}$ provides a reflector, with which we can -- at least theoretically -- describe colimits and tensor products in $\C/(s \text{ isom.})$.

After replacing $s$ by $s \otimes \L^{\otimes -1}$, it suffices to treat the case that $\L=\O_\C$ and $s : E \to \O_\C$. In that case we simply have $R_1(M) = \HOM(E,M)$ and then $R_2(M) = \HOM(E,\HOM(E,M)) \cong \HOM(E^{\otimes 2},M)$ etc. Inductively we get
\[R_n(M) = \HOM(E^{\otimes n},M)\]
for $n < \omega$. The transition morphisms are induced by $E^{\otimes n} \otimes s : E^{\otimes n+1} \to E^{\otimes n}$. We have
\[R_{\omega}(M) = \varinjlim_n \, \HOM(E^{\otimes n},M).\]
If $E$ is $\omega$-presentable, then we see that $R_{\omega}(M)$ is fixed by $R_1$. Hence, $R_{\omega}$ is the reflector.
\end{ex}

\begin{ex} \label{EL-lok2}
Let $s : E \to \L$ be as above and $f,g : A \to E$ morphisms such that $sf=sg$. Then $\C/(A \rightrightarrows E \to \L \text{ exact}) = \C/(\coeq(f,g) \to \L \text{ isom.})$ consists of those $M \in \C$ such that $M \to \HOM(E,M) \otimes \L \rightrightarrows \HOM(A,M) \otimes \L$ is exact. If we define $R_1(M) = \eq(\HOM(E,M) \to \HOM(A,M)) \otimes \L$ and $E$ is $\omega$-presentable, then $R_{\omega}$ is the desired reflection.
\end{ex}

\begin{cor}[Making objects invertible] \label{makeinv}
Let $\C$ be a locally presentable tensor category and $E \in \C$. Then there is a locally presentable tensor category $\C[E^{-1}]$ with the universal property
\[\Hom_{c\otimes}(\C[E^{-1}],\D) \cong \{F \in \Hom_{c\otimes}(\C,\D) : F(E) \text{ is a line object}\}\]
for all cocomplete tensor categories $\D$.
\end{cor}

\begin{proof}
We start with $\grM(\Sym(E))$, which according to \autoref{grModSym} is the universal cocomplete tensor category over $\C$ with a morphism from (the image of) $E$ to a line object. Using \autoref{universal-problem} we can make it an isomorphism.
\end{proof}

\begin{rem}[Explicit description]
The category $\grM(\Sym(E))$ consists of sequences of objects $(M_n)_{n \in \Z}$ in $\C$ together with morphisms
\[\alpha_n : E \otimes M_n \to M_{n+1}\]
such that the diagram
\[\xymatrix@C=35pt@R=25pt{E \otimes E \otimes M_n \ar[r]^-{E \otimes \alpha_n} \ar[dd]_{S_{E,E} \otimes M_n}  & E \otimes M_{n+1}  \ar[dr]^{\alpha_{n+1}} & \\&& M_{n+2} \\ E \otimes E \otimes M_n \ar[r]_-{E \otimes \alpha_n} & E \otimes M_{n+1} \ar[ur]_{\alpha_{n+1}} & }\]
commutes. Notice that this is automatic if $E$ is symtrivial. We will call the objects of $\grM(\Sym(E))$ also \emph{$E$-modules}.

The universal morphism $s : F(E) \to \L$, where $F : \C \to \grM(\Sym(E))$ is given by $M \mapsto (\Sym^n(E) \otimes M)_n$ and $\L$ is the twist $(\Sym^{n+1}(E))_n$, is just the multiplication $\Sym^n(E) \otimes E \to \Sym^{n+1}(E)$ in degree $n$. In order to invert it, we apply \autoref{EL-lok}. Thus, $\C[E^{-1}]$ consists of those $E$-modules $M$ such that $M \to \HOM(F(E),M) \otimes \L$ is an isomorphism, which means that each morphism $M_n \to \HOM(E,M_{n+1})$ (corresponding to $E \otimes M_n \to M_{n+1}$) is an isomorphism. Thus, $M_n$ may be reconstructed from $M_{n+1}$.

The reflector starts with $R_1 : \grM(\Sym(E)) \to \grM(\Sym(E))$ defined by $R_1(M)_n = \HOM(E,M_{n+1})$. It follows by induction that
\[R_p(M)_n = \HOM(E^{\otimes p},M_{n+p})\]
for $p<\omega$, hence that
\[R_{\omega}(M)_n = \varinjlim_{p < \omega} \, \HOM(E^{\otimes p},M_{n+p}),\]
which provides the reflection at least if $E$ is $\omega$-presentable, i.e. finitely presentable. \marginpar{Added finitely presentable.}
\end{rem}

\begin{ex}[Quantales]
We can make this even more explicit when $\C$ is a quantale. Here, $E$-modules consists of sequences of elements $M_n$ in $\C$ satisfying $E M_n \leq M_{n+1}$. Such a sequence lies in $\C[E^{-1}]$ (which is the universal quantale over $\C$ which makes $E$ invertible) if and only if we have $M_n = [M_{n+1}:E]$, i.e. $M_n$ is actually the largest element of $\C$ such that $E M_n \leq M_{n+1}$. For the reflectors we have $R_p(M)_n = [M_{n+p}:E^p]$ and $R_{\omega} = \sup_p R_p$ is the reflection when $E$ is $\omega$-presentable.

For example we can apply this to the quantale $\C$ of ideals of a commutative ring $A$. Thus, we can universally make a given (finitely generated) ideal $E$ of $A$ invertible. This construction might be of independent interest. One might investigate if $\C[E^{-1}]$ is again the quantale of ideals of another commutative ring induced by $A$.

For a completely different example, look at the quantale $[0,1]$. No object in $]0,1[$ is $\omega$-presentable, but we can still invert it. For example, a $\frac{1}{2}$-sequence is a sequence of elements $(t_n)_{n \in \Z}$ in $[0,1]$ with $t_n \leq 2 t_{n+1}$ and it lies in the localization if $t_n = \min(2 \cdot t_{n+1},1)$. One can check that $(t_n) \mapsto \lim_{n \to \infty} 2^n \cdot t_n$ provides an isomorphism of quantales $[0,1][(\frac{1}{2})^{-1}] \cong [0,\infty]$. One can also check directly that $[0,\infty]$ is the localization of $[0,1]$ at $\frac{1}{2}$ (or any other element in $]0,1[$), using that $[0,2^n] = (\frac{1}{2})^{-1} \cdot [0,1]$.
\end{ex}


\subsection{Localization at sections}  \label{locsec}
We have the following variant of \autoref{EL-lok}. Let $\C$ be a cocomplete tensor category (not assumed to be locally presentable) and $\L \in \C$ be a symtrivial object (e.g. a line object). Let $s : \O_\C \to \L$ be a morphism. We would like to construct the localization $\C \to \C/(s \text{ isom.}) =: \C_s$. For this we define $R(M) = \L \otimes M$ together with $s \otimes M : M \to R(M)$. The transfinite composition will provide a reflector of $\C_s = \mathrm{Fix}(R)$. We can make this more explicit as follows. The special case $\L=\O$ has already been studied by Florian Marty (\cite[Section 2.2]{Mar09}).

\begin{defi}
We denote by $\O_s$ an initial commutative algebra $A$ in $\C$ with the property that $s \otimes A : A \to \L \otimes A$ is an isomorphism. We call $\O_s$ the \emph{localization} of $\O$ at $s$.
\end{defi}

\begin{prop}
The localization $\O_s$ exists.
\end{prop}

\begin{proof}
We construct $\O_s$ as the colimit of the sequence
\[\O \xrightarrow{s} \L \xrightarrow{\id \otimes s} \L^{\otimes 2} \xrightarrow{\id \otimes \id \otimes s} \L^{\otimes 3} \longrightarrow \dotsc\]
Since $\L$ is assumed to be symtrivial, we could also use $s \otimes \id \otimes \id$ or $\id \otimes s \otimes \id$ etc. as transition morphisms -- this will be useful. We will denote them just by $s$. For each $n \in \N$ we have a canonical ``inclusion" $i_n : \L^{\otimes n} \to \O_s$ such that $i_{n+1} \circ s = i_n$. The morphisms
\[\L^{\otimes n} \otimes \L^{\otimes m} \cong \L^{\otimes n+m} \xrightarrow{i_{n+m}} \O_s\]
are compatible in each variable, hence glue to a morphism $m : \O_s \otimes \O_s \to \O_s$. One can check that $\O_s$ is a commutative algebra with multiplication $m$ and unit $i_0$ (the commutativity uses again that $\L$ is symtrivial).

The morphisms $\L \otimes \L^{\otimes n} \cong \L^{\otimes n+1} \xrightarrow{i_{n+1}} \O_s$ are compatible in $n$, hence glue to a morphism $t : \L \otimes \O_s \to \O_s$. We claim that, in fact, $t$ is an inverse to $s \otimes \O_s : \O_s \to \O_s \otimes \L$. Look at the following commutative diagram
\[\xymatrix@C=40pt{\L \otimes \O_s \ar[r]^{t} & \O_s \ar[r]^-{s} & \L \otimes \O_s \\
\L \otimes \L^{\otimes n} \ar[u]^{\L \otimes i_n} \ar[r]^-{\cong} & \L^{\otimes n+1}  \ar[u]_{i_{n+1}} \ar[r]^-{s} &  \L \otimes \L^{\otimes n+1} \ar[u]_{\L \otimes i_{n+1}}}\]
and use $i_{n+1} \circ s = i_n$. This implies that $st \circ (\L \otimes i_n) = \L \otimes i_n$ for all $n$ and hence $st=\id$. A similar argument shows $ts=\id$. Hence, $s \otimes \O_c$ is an isomorphism.

Now let $A$ be a commutative algebra in $\C$ such that $s \otimes A : A \to \L \otimes A$ is an isomorphism. Then, for every $n \in \N$, we have a commutative diagram
\[\xymatrix@C=40pt{A \ar@{=}[r] \ar[d]_{s^{\otimes n} \otimes \L}^{\cong}  & A \ar[d]^{s^{\otimes n+1} \otimes \L}_{\cong}  \\ \L^{\otimes n} \otimes A \ar[r]^-{s} & \L^{\otimes n+1} \otimes A,}\]
which in the colimit gives $A \cong \O_s \otimes A$. In particular, there is a morphism of algebras $\O_s \to A$. The case $A=\O_s$ shows that $\O_s \cong \O_s \otimes \O_s$, which means that $\O \to \O_s$ is an epimorphism in the category of commutative algebras and therefore $\O_s \to A$ is unique.
\end{proof}

\begin{rem}
Notice that for every cocontinuous tensor functor $F : \C \to \D$ we have an isomorphism $F(\O_s) \cong \O_{F(s)}$, this follows from the construction above.
\end{rem}

\begin{defi}
We call $\C_s := \M(\O_s)$ together with $\C \to \C_s$, $M \mapsto M \otimes \O_s$ the \emph{localization} of $\C$ at $s$. It is equivalent to the full subcategory of $\C$ consisting of those $M \in \C$ such that $s \otimes M : M \to \L \otimes M$ is an isomorphism. Under this identification, $\C \to \C_s$ maps $M \in \C$ to $M_s := \colim_n (\L^{\otimes n} \otimes M)$ with transition morphisms induced by $s$. The tensor product is the same as in $\C$, however the unit is $\O_s$.
\end{defi}

\begin{prop}
We have $\C_s = \C/(s \text{ isom.})$, i.e. for every cocomplete tensor category $\D$ we have an equivalence of categories
\[\Hom_{c\otimes}(\C_s,\D) \simeq \{F \in \Hom_{c\otimes}(\C,\D) : F(s) : \O_\D \to F(\L) \text{ is an isomorphism}\}\]
\end{prop}

\begin{proof}
This follows easily from \autoref{mod-UE}.
\end{proof}

\begin{rem}
If we want to invert a family of sections $s:=(s_i : \O \to \L)_{i \in I}$, we can consider the (possibly infinite) tensor product of algebras $\O_s := \bigotimes_{i \in I} \O_{s_i}$ and then define $\C_s := \M(\O_s)$.
\end{rem}

\begin{ex}
In the special case $\C=\Q(X)$ for some scheme $X$ and $\L$ is invertible, then $s : \O_X \to \L$ is literally a global section of $\L$ and the localization $\C \to \C_s$ identifies with the restriction functor $\Q(X) \to \Q(X_s)$, where $X_s \subseteq X$ is the usual maximal open subscheme on which $s$ generates $\L$. Thus, we have globalized this basic construction from algebraic geometry.
\end{ex}

\begin{ex}
Let $A$ be a commutative graded algebra in $\C$ and let $M$ be a graded $A$-module. Consider a global section $f : \O \to A_d$ for some $d \in \Z$. It induces a homomorphism of graded $A$-modules $f : A \to A[d]$. Since $A[d]$ is a line object in $\grM(A)$, we may localize $M$ at $f$. The degree $0$ part of the resulting graded $A$-module $M_f$ is an $A_0$-module which might be called the homogeneous localization $M_{(f)}$ of $M$ at $f$ since for $\C=\M(R)$ this coincides with the usual homogeneous localization.
\end{ex}

\begin{rem}
Let $\L \in \C$ be a line object and let $s=(s_i) : \O^n \to \L$ be an epimorphism. Then for every $M \in \C$ we have the usual commutative diagram
\[M \to \bigoplus_{i=1}^{n} M_{s_i} \rightrightarrows \bigoplus_{i,j=1}^{n} M_{s_i,s_j}.\]
We claim that it is an equalizer when $\C$ is an abelian locally finitely presentable tensor category -- this is the main ingredient in the construction of the classical associated sheaf $\tilde{M}$. It suffices to prove that $\oplus_{i=1}^{n} \O_{s_i}$ is faithfully flat. Since $\O_{s_i}$ is a directed colimit of flat objects, it is flat (\autoref{flatclosure}). In order to show faithfully flatness, assume $M_{s_i} = 0$ for all $1 \leq i \leq n$. We may assume that $M$ is finitely presentable. Then it follows that $M \otimes s_i^k : M \to M \otimes \L^{\otimes k}$ is zero for some $k$ which can be chosen independently from $i$. But then \autoref{epi-cancel} implies $M=0$.
\end{rem}


\subsection{Ideals} \label{cideals}

If $R$ is a commutative ring, then ideals of $R$ have two equivalent descriptions, namely
\begin{enumerate}
\item \emph{abstractly} as kernels of ring homomorphisms $R \to S$, where $S$ is a suitable commutative ring,
\item \emph{concretely} as subsets of $R$ closed under addition, scalar multiplication and containing $0$.
\end{enumerate}
 
Here we would like to categorify this, replacing $R$ by a linear cocomplete tensor category $\C$. In this section all tensor categories and functors are understood to be linear over a fixed commutative ring.

\begin{defi}[Abstract ideals] \label{ideal-abstract}
We call a full subcategory $\I \subseteq \C$ an \emph{abstract ideal} of $\C$ if there is some cocontinuous tensor functor $F : \C \to \D$ into some cocomplete tensor category $\D$ such that
\[\I = \ker(F) := \{M \in \C : F(M) = 0\}.\]
\end{defi}

\begin{defi}[Concrete ideals] \label{ideal-concrete}
We call a full subcategory $\I \subseteq \C$ a \emph{concrete ideal} of $\C$ if it enjoys the following closure properties:
\begin{enumerate}
\item $\I$ is closed under taking colimits in $\C$. In particular, $0 \in \I$.
\item If $M \to N$ is an epimorphism in $\C$ and $M \in \I$, then $N \in \I$.
\item If $M,N \in \C$ and $M \in \I$, then $M \otimes N \in \I$.
\item If $K_0 \to M_0 \to M_1 \to 0$ is an exact sequence and $K_0,M_1 \in \I$, then $M_0 \in \I$.
\item More generally, if $\lambda$ is a regular cardinal and $[0,\lambda] \to \C,~ \alpha \mapsto M_{\alpha}$ is a cocontinuous functor such that there are exact sequences
\[K_{\alpha} \to M_{\alpha} \to M_{\alpha+1} \to 0\]
for $\alpha < \lambda$ such that $K_{\alpha} \in \I$ and $M_\lambda \in \I$, then $M_0 \in \I$.
\end{enumerate}
\end{defi}

Notice that for $\lambda<\aleph_0$ the condition 5. follows by induction from 4., but for $\lambda=\aleph_0$ 5. is independent from the other axioms.

\begin{lemma} \label{abstract-is-concrete}
In a linear cocomplete tensor category $\C$ every abstract ideal is a concrete ideal.
\end{lemma}

\begin{proof}
Let $\D$ be a cocomplete tensor category and let $F : \C \to \D$ be a cocontinuous tensor functor. We have to show that
\[\I := \{M \in \C : F(M)=0\}\]
satisfies the closure properties listed in \autoref{ideal-concrete}. For 1., 2. and 3. this is is clear. In 4. $F(K_0) \to F(M_0) \to F(M_1) \to 0$ is exact in $\D$. Since $K_0,M_1 \in \I$, this simplifies to $0 \to F(M_0) \to 0 \to 0$, which means $F(M_0)=0$, i.e. $M_0 \in \I$. In 5. we observe that $F(M_\alpha) \to F(M_{\alpha+1})$ is an isomorphism for all $\alpha<\lambda$. It follows by induction on $\alpha<\lambda$ that $F(M_0) \to F(M_\alpha)$ is an isomorphism, the limit case being a colimit argument. It follows that $F(M_0) \to F(M_\lambda)=0$ is an isomorphism, so that $M_0 \in \I$.
\end{proof}

\begin{thm} \label{concrete-is-abstract}
Let $\C$ be a locally presentable linear tensor category. Then every concrete ideal is an abstract ideal. More specifically, if $\I \subseteq \C$ is a concrete ideal, then there is a cocomplete tensor category $\C/\I$ and a cocontinuous tensor functor $F : \C \to \C/\I$ such that $\ker(F)=\I$. It satisfies the universal property
\[\Hom_{c\otimes}(\C/\I,\D) \simeq \{G \in \Hom_{c\otimes}(\C,\D) : G|_\I = 0\}.\]
\end{thm}


\begin{proof}
We consider the full subcategory
\[\C/\I := \{M \in \C : \forall N \in \I : \HOM(N,M)=0\}.\]
Since $I$ is closed under tensoring with objects from $\C$, we may replace $\HOM$ by $\Hom$ here. We would like to apply \autoref{localization} and then \autoref{M0}, but this only works when $\I$ is a set, or more generally, if $\I$ has a colimit-dense subset. 
Instead, we construct a reflection for $\C/\I$ explicitly as follows: If $M \in \C$ and if $f : N \to M$ is a morphism in $\C$ with $N \in \I$, we may factor $f$ as
\[N \xrightarrow{p} N' \xrightarrow{i} M,\]
where $p$ is an epimorphism, so that $N' \in \I$ and $i$ is a monomorphism (\autoref{saft}). It follows that $M \in \C$ belongs to $\C/I$ if and only if $0$ is the only subobject of $M$ which belongs to $\I$. But $M$ has only a set of subobjects at all (\autoref{saft}), so that we may define a functor
\[R_1 : \C \to \C, M \mapsto M/T_\I(M)\]
where $T_\I(M)$ is the sum of all subobjects $N \leq M$ such that $N \in \I$. Using that $\I$ is a concrete ideal, we observe $T_\I(M) \in \I$. Then
\[\Fix(R_1)=\{M \in \C : T_\I(M)=0\} = \C/\I,\]
so that the transfinite composition $R_{\infty} : \C \to \C/I$ provides a reflection by \autoref{fix} and the proof of \autoref{localization} goes through. Clearly, we have $\I \subseteq \ker(R_{\infty})$. Conversely, assume that $M \in \C$ such that $0 = R_{\infty}(M)$. Choose some regular cardinal $\lambda$ such that $0 = R_{\lambda}(M)$. If $\alpha < \lambda$, then there is an exact sequence $T_\I(R_{\alpha}(M)) \to R_\alpha(M) \to R_{\alpha+1}(M) \to 0$ with $T_\I(R_{\alpha}) \in \I$. It follows from the closure property 5. in \autoref{ideal-concrete} that $M=R_0(M) \in \I$.
\end{proof}

\begin{rem}
Several (still open) problems within this thesis would be solvable (in fact, reducible to the affine case) if every (reasonable) algebraic stack $X$ has a descent algebra $A$ such that $\Spec_X(A)$ is affine (over $\Z$) and that $A$ generates the unit ideal $\langle \O_X \rangle = \Q(X)$. The latter means (using \autoref{concrete-is-abstract}): If $\Hom(N \otimes A,M)=0$ for all $N$, then $M=0$.
\end{rem}

\section{Idempotents} \label{idemp}

We already know that products of cocomplete tensor categories exist and how to describe cocontinuous tensor functors \emph{into} them. But actually it is also possible to describe cocontinuous tensor functors \emph{on} them. This is a categorification of the following well-known universal property of the product of a \emph{finite} family of commutative rings $(R_i)_{i \in I}$: Homomorphisms $\prod_{i \in I} R_i \to S$ correspond to decompositions $1=\sum_{i \in I} e_i$ in $S$ into orthogonal idempotent elements $e_i \in S$ and homomorphisms $R_i \to S_{e_i}$, where $S_{e_i}$ denotes the localization of $S$ at the element $e_i$. Therefore, let us first categorify orthogonal idempotent decompositions.

In the following, we assume that $\C$ is a linear cocomplete tensor category. We also fix an index set $I$.

\begin{defi}[o.i.d.]
An \emph{orthogonal idempotent decomposition}, abbreviated o.i.d., of $\C$ consists of a family of morphisms $(\sigma_i : E_i \to \O)_{i \in I}$ such that
\begin{enumerate}
\item $E_i \otimes E_j = 0$ for $i \neq j$,
\item $(\sigma_i : E_i \to \O)_{i \in I}$ is a coproduct diagram, hence
\[\bigoplus_{i \in I} E_i \cong \O.\]
\end{enumerate}
\end{defi} 

\begin{rem} \label{oid-prop}
Let $(\sigma_i : E_i \to \O)_{i \in I}$ be an o.i.d. of $\C$.
\begin{enumerate}
\item $E_i \otimes \sigma_i : E_i \otimes E_i \to E_i$ is actually an \emph{isomorphism} $E_i \otimes E_i \cong E_i$, since it identifies with the isomorphism $\bigoplus_{i \in I} E_i \cong \O$ tensored with $E_i$ on the left.
\item Since $E_i$ is symtrivial (as a quotient of $\O$), we have an equality of isomorphisms $E_i \otimes \sigma_i = \sigma_i \otimes E_i : E_i \otimes E_i \cong E_i$. This implies that $\sigma_i : E_i \to \O$ is an \emph{open idempotent} in the sense of Drinfeld-Boyarchenko (\cite{DB}). It is also a split monomorphism (as a coproduct inclusion).
\item If $F : \C \to \D$ is a cocontinuous tensor functor, then $F$ induces an o.i.d. $(F(\sigma_i) : F(E_i) \to F(\O) \cong \O)_{i \in I}$ of $\D$.
\item A morphism $\alpha : (\sigma_i : E_i \to \O) \to (\tau_i : F_i \to \O)$ between two o.i.d. is by definition a family of morphisms $\alpha_i : E_i \to F_i$ such that $\tau_i \alpha_i = \sigma_i$. Since $\tau_i$ is a monomorphism, it is clear that $\alpha_i$ is unique. Besides, it is an isomorphism. In fact, the commutative diagram
\[\xymatrix{\bigoplus_i E_i \ar[rr]^{\oplus_i \alpha_i} \ar[dr]_{\sigma} && \bigoplus_{i \in I} F_i \ar[dl]^{\tau} \\ & \O & }\]
shows that $\oplus_i \alpha_i$ is an isomorphism. Thus, o.i.d. constitute a category which is essentially discrete.
\end{enumerate}
\end{rem}

Let us give alternative descriptions of o.i.d.

\begin{rem}
Recall that an idempotent morphism $e : \O \to \O$ has an image factorization $\O \xrightarrow{p} \im(e) \xrightarrow{\iota} \O$ with $e=\iota p$ and $p \iota=\id$. In fact, we can take $\im(e):=\coker(e^\perp)$ with $e^\perp := 1-e$ together with the cokernel projection $p : \O \to \im(e)$. Since $e e^\perp = 0$, there is some $\iota$ such that $\iota p=e$. We have $p \iota p=pe=p$ and hence $p \iota=\id$.
\end{rem}
 
\begin{defi}
A \emph{discrete orthogonal idempotent decomposition} of $\C$ consists of a family of idempotent endomorphisms $e_i \in \End(\O)$ such that $e_i e_j = 0$ for $i \neq j$ and the canonical morphism
\[\bigoplus_{i \in I} \im(e_i) \to \O\]
is an isomorphism.
We consider the set of discrete o.i.d. as a discrete category.
\end{defi}

\begin{rem}
If $I$ is finite, the last condition in the definition means that $\sum_{i \in I} e_i = 1$.
\end{rem}

\begin{ex}
Let $X$ be a scheme (or even an algebraic stack) and let $(e_i)_{i \in I}$ be a discrete o.i.d. of $\Q(X)$. Then $X_i := D(e_i)=V(e_i^\perp)$ is open and closed in $X$ and we have $X=\coprod_{i \in I} X_i$. Conversely, if $X = \coprod_{i \in I} X_i$ then there is a unique idempotent $e_i \in \Gamma(X,\O_X) \cong \End(\O_X)$ with $e_i|_{X_j}=\delta_{ij}$ for all $j \in I$ and $(e_i)_{i \in I}$ is a discrete o.i.d. of $\Q(X)$. Hence, coproduct decompositions of $X$ correspond to the discrete o.i.d. of $\Q(X)$.
\end{ex}

\begin{prop} \label{oid-alt}
The category of o.i.d. of $\C$ is equivalent to the category of discrete o.i.d. of $\C$.
\end{prop}

\begin{proof}
Let $(\sigma_i : E_i \to \O)$ be an o.i.d. For each $i \in I$, we have an idempotent endomorphism $e_i : \O \to \O$ defined via the corresponding idempotent endomorphism of $\bigoplus_j E_j$ which projects onto $E_i$. Thus, we have $e_i \sigma_i = \sigma_i$ and $e_i \sigma_j = 0$ for $i \neq j$. This defines $e_i$. For $i \neq j$ we have $e_i e_j = 0$: It suffices to prove $e_i e_j \sigma_k = 0$ for all $k$. For $k \neq j$ this is clear and for $k=j$ we have $e_i e_j \sigma_j = e_i \sigma_j = 0$. It is clear that $\O_C \cong \bigoplus_j E_j \twoheadrightarrow E_i \xrightarrow{\sigma_i} \O$ is the image factorization of $e_i$. From this it follows that $(e_i)$ is a discrete o.i.d.

Let $\alpha : (\sigma_i : E_i \to \O) \to (\tau_i : F_i \to \O)$ be a morphism between o.i.d. We already know that it is a unique isomorphism. We claim that the corresponding idempotents $e_i,f_i$ are equal: We have $e_i \tau_i \alpha_i = e_i \sigma_i = \sigma_i = \tau_i \alpha_i$, hence $e_i \tau_i = \tau_i$. For $i \neq j$ we have $e_i \tau_j \alpha_j = e_i \sigma_j = 0$, hence $e_i \tau_j = 0$. This means that $e_i$ satisfies the defining properties of $f_i$, hence $e_i=f_i$.

This defines a functor from o.i.d. to discrete o.i.d., which is clearly faithful. It is essentially surjective: Given a discrete o.i.d. $(e_i)_{i \in I}$ we may consider the o.i.d. $(\im(e_i) \hookrightarrow \O)_{i \in I}$ whose corresponding discrete o.i.d. is clearly the given one.

In order to show fullness, let $(\sigma_i : E_i \to \O)$ and $(\tau_i : F_i \to \O)$ be two o.i.d. whose discrete o.i.d. are equal, $e_i = f_i$. Then we have $f_i \sigma_i = \sigma_i$. Define $\alpha_i : E_i \to F_i$ to be the composition
\[E_i \xrightarrow{\sigma_i} \O \xrightarrow{\cong} \bigoplus_j F_j \twoheadrightarrow F_i.\]
We have to prove $\tau_i \alpha_i = \sigma_i$. This is a consequence of the following commutative diagram:
\[\xymatrix{E_i \ar[r]^{\sigma_i} \ar[dr]_{\sigma_i} \ar[dd]_{\sigma_i} & \O \ar[r]^{\cong} \ar[d]^{f_i} & \bigoplus_j F_j \ar[dd] \\ & \O_C \ar[dl]^{f_i} & \\
\O_C && F_i \ar[ul]^{\tau_i} \ar[ll]^{\tau_i} }\]
\end{proof}

In the following, we will not distinguish between o.i.d. and discrete o.i.d.

\begin{ex}
Let $\C=\prod_{i \in I} \C_i$ be a product. Then $(\O_i \to \O)_{i \in I}$ is an o.i.d. of $\C$, where $\O_i$ is defined by $(\O_i)_i = \O_{\C_i}$ and $(\O_i)_j = 0$ for $j \neq i$. Conversely:
\end{ex}

\begin{prop} \label{zerleg}
Let $\C$ be a cocomplete tensor category and let $(e_i)_{i \in I}$ be an o.i.d. of $\C$. Then the localizations $\C \to \C_{e_i}$ induce an equivalence $\C \simeq \prod_{i \in I} \C_{e_i}$. In fact, the category of o.i.d. of $\C$ is equivalent to the category of product decompositions of $\C$.
\end{prop}

\begin{proof}
Recall from \autoref{locsec} that the localization $\C_{e_i}$ consists of those $M \in \C$ such that $e_i \otimes M :  M \to M$ or equivalently $\sigma_i \otimes M : E_i \otimes M \to M$ is an isomorphism. This happens if and only if $E_j \otimes X = 0$ for all $j \neq i$. The localization functor $\C \to \C_{e_i}$ maps $M$ to $E_i \otimes M$. The isomorphism $\bigoplus_i E_i \cong \O$ gives an isomorphism $\bigoplus_i (E_i \otimes M) \cong M$ and we have $E_i \otimes M \in \C_{e_i}$. Now the Proposition easily follows.
\end{proof}

Since o.i.d. are preserved by cocontinuous tensor functors, this implies:
 
\begin{prop} \label{prodUE}
Let $(\C_i)_{i \in I}$ be a family of cocomplete tensor categories and $\C=\prod_{i \in I} \C_i$. For every cocomplete tensor category $\D$ there is an equivalence of categories
\[\begin{array}{c}
\Hom_{c\otimes}(\C,\D) \simeq \{(e_i)_{i \in I} \text{ o.i.d. of } \D, ~ F \in \prod\limits_{i \in I} \Hom_{c\otimes}(\C_i,\D_{e_i})\}
\end{array}\]
\end{prop}

\begin{cor} \label{coprod-tens}
Tensorial stacks are closed under arbitrary coproducts.
\end{cor}

\begin{proof}
Let $X = \coprod_i X_i$ be a coproduct of tensorial stacks and $Y$ be any scheme. Since $\Q(X) \simeq \prod_{i \in I} \Q(X_i)$, $\Hom_{c\otimes}(\Q(X),\Q(Y))$ identifies with the category of o.i.d. $(e_i)_{i \in I}$ of $\Q(Y)$, i.e. coproduct decompositions $Y = \coprod_i Y_i$, together with objects of
\[\Hom_{c\otimes}(\Q(X_i),\Q(Y)_{e_i}) \simeq \Hom_{c\otimes}(\Q(X_i),\Q(Y_i)) \simeq \Hom(Y_i,X).\]
This identifies with the set of morphisms $Y \to X$.
\end{proof}
\section{Projective tensor categories} \label{projcat}

\subsection{Definition and comparison to schemes}

\begin{rem} \label{projUE}
Let $S$ be a scheme and $A$ be some $\N$-graded quasi-coherent algebra on $S$ such that $A$ is generated by $A_1$. Then $\Proj_S(A)$ is a scheme with the following well-known universal property (see \cite[\href{http://stacks.math.columbia.edu/tag/01O4}{Tag 01O4}]{stacks-project}): The set of morphisms $T \to \Proj_S(A)$ is equivalent to the (essentially discrete) category of triples $(f,\L,s)$, where $f : T \to S$ is a morphism, $\L$ is a line bundle on $T$ and $s : f^*(A) \to \bigoplus_n \L^{\otimes n}$ is a homomorphism of $\N$-graded algebras such that $s_1 : f^*(A_1) \to \L$ is an epimorphism. Notice that it follows automatically that $s_n : f^*(A_n) \to \L^{\otimes n}$ is an epimorphism for all $n \in \N^+$, since $A_1^{\otimes n} \to A_n$ is an epimorphism.

We globalize this as follows:
\end{rem}

\begin{defi}[Projective tensor categories] \label{proj-def}
Let $\C$ be a cocomplete tensor category, $A$ a commutative $\N$-graded algebra in $\C$ which is generated by $A_1$. Then we define $\Proj^{\otimes}_\C(A)$ to be a cocomplete tensor category with the following universal property (if it exists): If $\D$ is a cocomplete tensor category, then $\Hom_{c\otimes}(\Proj^{\otimes}_\C(A),\D)$ is naturally equivalent to the category of triples $(F,\L,s)$, where $F: \C \to \D$ is a cocontinuous tensor functor, $\L \in \D$ is a line object and $s : F(A) \to \bigoplus_n \L^{\otimes n}$ is a homomorphism of graded algebras in $\C$ such that $s_1 : F(A_1) \to \L$ is a regular epimorphism in $\D$. We call $\Proj^\otimes_\C(A)$ the \emph{projective tensor category} over $\C$ associated to $A$.
\end{defi}

\begin{rem}
It follows automatically that $s_n : F(A_n) \to \L^{\otimes n}$ is a regular epimorphism for all $n \in \N^+$. In fact, if $s_1 : F(A_1) \to \L$ is a regular epimorphism, by \autoref{coeq-tensor} the same is true for $F(A_1)^{\otimes n} \to \L^{\otimes n}$. Since it factors as $F(A_1)^{\otimes n} = F(A_1^{\otimes n}) \twoheadrightarrow F(A_n) \xrightarrow{s_n} \L^{\otimes n}$, we infer that $s_n$ is a regular epimorphism.
\end{rem}

\begin{rem}[Universal triple] \label{uni-triple}
The definition of $\Proj^\otimes_\C(A)$ includes a (universal) triple $(P,\O(1),x)$, where $P : \C \to \Proj^\otimes_\C(A)$ is a cocontinuous tensor functor, $\O(1)$  is a line object in $\Proj^\otimes_\C(A)$ and $x : P(A) \to \bigoplus_n \O(1)^{\otimes n}$ a homomorphism of graded algebras such that $x_1 : P(A_1) \to \O(1)$ is a regular epimorphism. We may write $\O(n) := \O(1)^{\otimes n}$.
\end{rem}
 
\begin{defi}[Projective bundle tensor categories] \label{projbundle}
If $\C$ is a cocomplete tensor category and $E \in \C$, we define $\P^\otimes_\C(E):=\Proj^\otimes_\C(\Sym(E))$. Thus, by definition $\Hom_{c\otimes}(\P^\otimes_\C(E),\D)$ is equivalent to the category of triples $(F,\L,s)$, where $F : \C \to \D$ is a cocontinuous tensor functor, $\L \in \D$ is a line object and $s : F(E) \to \L$ is a regular epimorphism. If $n \in \N$, we define $\P^n_\C := \P^\otimes_\C(\O_\C^{\oplus (n+1)})$. For example, we have $\P^{\otimes}_\C(\O)=\P^0_\C=\C$ by \autoref{epi-ist-iso}.
\end{defi}

\begin{thm}[Existence of projective tensor categories] \label{proj-ex}
Let $\C$ be a locally presentable tensor category and $A$ a commutative $\N$-graded algebra in $\C$ which is generated by $A_1$. Then $\Proj^{\otimes}_\C(A)$ exists and it is again a locally presentable tensor category.\\
Explicitly, the underlying category of $\Proj^{\otimes}_\C(A)$ is the full subcategory of $\grM(A)$ consisting of those graded $A$-modules $M$ such that the canonical diagram of graded $A$-modules
\[M \to \HOM(A_1,M[1]) \rightrightarrows \HOM(A_1 \otimes A_1,M[2])\]
is exact. It is reflective, the reflector being the transfinite composition of the functor $R_1 : \grM(A) \to \grM(A)$ defined by
\[R_1(M) = \eq\bigl(\HOM(A_1,M[1]) \rightrightarrows \HOM(A_1 \otimes A_1,M[2])\bigr).\]
\end{thm}

\begin{proof}
We know the universal property of $\grM(A)$ by \autoref{grMod}: It classifies triples $(F,\L,s)$ where $s : F(A) \to \bigoplus_n \L^{\otimes n}$ is any homomorphism of  graded algebras. According to \autoref{goodepi}, $s_1 : F(A_1) \to \L$ is a regular epimorphism if and only if the sequence $F(A_1) \otimes F(A_1) \otimes \L^{\otimes -1} \rightrightarrows F(A_n) \to \L$ is exact. Now we may apply \autoref{universal-problem} to construct
\[\Proj^{\otimes}_\C(A) := \grM(A) / (A_1 \otimes A_1 \otimes A[-1] \rightrightarrows A_1 \otimes A \to A[1] \text{ exact})\]
with the desired universal property.
\end{proof}

\begin{cor} \label{proj-ex2}
If $\C$ is a locally presentable tensor category and $E \in \C$, then $\P^\otimes_\C(E)$ exists. It is given by the full subcategory of $\grM(\Sym(E))$ consisting of those graded $E$-modules such that the canonical diagram
\[M \to \HOM(E,M[1]) \rightrightarrows \HOM(E^{\otimes 2},M[2])\]
of graded $E$-modules is exact.
\end{cor}

\begin{rem}
If $\C$ is linear, we could also define  ``anti-projective'' tensor categories for graded-commutative algebras in $\C$ which classify \emph{anti}-line quotients.
\end{rem}

\begin{rem}
If we omit ``regular'' from the definition of $\P^{\otimes}_\C(E)$, we obtain another cocomplete tensor category $\P'^{\otimes}_\C(E)$ which can be realized as the category of graded $E$-modules for which $M \to \HOM(E,M[1])$ is just a monomorphism. For example, $\P'^{\otimes}_\C(1)$ consists of sequences
\[\dotsc \hookrightarrow M_n \hookrightarrow M_{n+1} \hookrightarrow \dotsc\]
of monomorphisms in $\C$, whereas $\P^{\otimes}_\C(1)=\C$.
\end{rem}

\begin{ex}
Let $\C$ be a locally presentable tensor category. Then $\P^n_\C$ is the universal example of a cocomplete tensor category ``over $\C$'' with a line object $\L$ and $n+1$ global sections $x_i : \O \to \L$ for $0 \leq i \leq n$ such that $x : \O^{\oplus d+1} \to \L$ is a regular epimorphism. Using \autoref{universal-problem} we can enforce further relations between these global sections. For example, we may construct the ``Fermat tensor category'' $\P^2_\C / (x_0^{\otimes d} + x_1^{\otimes d} = x_2^{\otimes d})$ for $d \in \Z$. It is equivalent to $\Proj^\otimes_\C(\O_\C[x_0,x_1,x_2]/(x_0^d + x_1^d = x_2^d))$. We may also consider the localization $(\P^n_\C)_{x_i}$ at the section $x_i$ (\autoref{locsec}), which is clearly isomorphic to the affine space $\mathds{A}^n_\C$ (\autoref{affsp}) with variables $\{x_i^{-1} x_j : j \neq i\}$.
\end{ex}

\begin{thm}[Comparison to schemes] \label{proj-compare}
Let $S$ be an arbitrary scheme and $A$ be an $\N$-graded quasi-coherent algebra on $S$ which is generated by $A_1$ and such that $A_1$ is a quasi-coherent module of finite presentation over $S$. There is an equivalence of cocomplete tensor categories
\[\Proj^\otimes_{\Q(S)}(A) \simeq \Q(\Proj_S(A)).\]
\end{thm}

\begin{proof}
The universal property of $\Proj_S(A)$ recalled in \autoref{projUE} yields a universal triple $(p,\O(1),s)$, where $p : \Proj_S(A) \to S$ is the structure morphism, inducing a cocontinuous tensor functor
\[p^* : \Q(S) \to \Q(\Proj_S(A)),\]
$\O(1)$ is the Serre twist and $s : p^* A \to \bigoplus_n \O(n)$ is a homomorphism of graded algebras such that $s_1$ is an epimorphism. Since the category of quasi-coherent modules is abelian, $s_1$ is automatically a regular epimorphism. By the defining universal property of $\Proj^\otimes_{\Q(S)}(A)$ we get a cocontinuous tensor functor
\[\Proj^\otimes_{\Q(S)}(A) \to \Q(\Proj_S(A)).\]
We have to show that it is an equivalence of categories. Recall (\cite[3.4.4]{EGAII}) that there is a cocontinuous tensor functor
\[\grM(A) \to \Q(\Proj_S(A)),\,M \mapsto \widetilde{M}\]
with a fully faithful right adjoint
\[\Gamma_* : \Q(\Proj_S(A)) \to \grM(A).\]
Thus, $\Q(\Proj_S(A))$ identifies with the full subcategory of $\grM(A)$ consisting of those graded $A$-modules $M$ such that the unit morphism
\[M \to \Gamma_*(\widetilde{M})\]
is an isomorphism. We would like to compare this with the description of $\Proj^\otimes_{\Q(S)}(A)$ in \autoref{proj-ex}. Notice that since $A_1$ is of finite presentation, $\HOM(A_1,-)$ is the usual homomorphism sheaf (\autoref{qchom}).
This shows that we may work locally on $S$. Let's assume that $S=\Spec(k)$ is affine. Then $A$ is just an $\N$-graded $k$-algebra. Choose a finite generating set $B$ of $A_1$ as a module over $k$. If $M$ is a graded $A$-module, we compute
\[\Gamma_*(\widetilde{M}) = \bigoplus_{n \in \Z} \Gamma(\Proj(A),\widetilde{M[n]}) = \bigoplus_{n \in \Z} \eq\left(\bigoplus_{f \in B} M[n]_{(f)} \rightrightarrows \bigoplus_{f,g \in B} M[n]_{(fg)}\right).\]
Thus, $M$ belongs to the full subcategory $\simeq \Q(\Proj_S(A))$ if and only if
\[(1) ~~ M \to \bigoplus_{f \in B} M_f \rightrightarrows \bigoplus_{f,g \in B} M_{fg}\]
is an exact sequence of graded $A$-modules (``$M$ believes to be a sheaf on $\Proj$''). We have to show that this is equivalent to the condition that
\[(2) ~~ M \to \HOM(A_1,M[1]) \twoheadrightarrow \HOM(A_1 \otimes A_1,M[2])\]
is an exact sequence of graded $A$-modules (``$M$ believes the universal property of $\Q(\Proj)$'').
 
\textbf{Proof of $(1) \Rightarrow (2)$}: Let $\phi : A_1 \to M_{p+1}$ be a $k$-linear map with the property $f \cdot \phi(g) = g \cdot \phi(f)$ for all $f,g \in A_1$. Then $(\frac{\phi(f)}{f})_{f \in B} \in \bigoplus_{f \in B} M_f$ lies in the equalizer of (1) in degree $p$. Hence, there is a unique $m \in M_p$ such that $\frac{\phi(f)}{f} = \frac{m}{1}$ in each $M_f$. Choose some $k \in \N$ with $f^k \phi(f)=f^{k+1} m$ in $M$. For an arbitrary $g \in B$ we have
\[f^{k+1} \phi(g) = f^k f \phi(g) = f^k g \phi(f) = g f^{k+1} m\]
and hence $\phi(g)=gm$ in $M_f$. By injectivity in (1) it follows that this holds in $M$. This shows that $m \in M_p$ is a preimage of $\phi$.

\textbf{Proof of $(2) \Rightarrow (1)$}: We first show injectivity of $M \to \bigoplus_{f \in B} M_f$. Let $m \in M_p$ be in the kernel. Since $B$ is finite, there is some $k \in \N$ such that $f^k \cdot m = 0$ for all $f \in B$. Let $n \geq k \cdot \# B$. Then we have $f_1 \cdot \dotsc \cdot f_n \cdot m = 0$ for all $f_1,\dotsc,f_n \in A_1$. Applying injectivity of (2), we get $f_2 \cdot \dotsc \cdot f_n \cdot m = 0$. It follows by induction on $n$ that $m = 0$.

Now let $(\frac{m_f}{f^k})_{f \in B}$ be contained in the equalizer of $\bigoplus_{f \in B} M_f \rightrightarrows \bigoplus_{f,g \in B} M_{fg}$ and of degree $p$. Since $B$ is finite, we may choose $k$ to be independent from $f$. Let $f,g \in B$. Since $\frac{m_f}{f^k} = \frac{m_g}{g^k}$ in $M_{fg}$, we have $(fg)^l f^k m_g = (fg)^l g^k m_f$ in $M$ for  some $l \in \N$. Then $n_f := f^l m_f$ satisfies $\frac{m_f}{f^k} = \frac{n_f}{f^{k+l}}$ and $f^{k+l} n_g = g^{k+l} n_f$. Thus, we may assume that $l=0$, i.e. $g^k m_f = f^k m_g$.

Let $n \geq k \cdot \# B$. We claim that for all $l \leq n$ and for all $a \in A_k$ there is a unique element $\phi(a) \in M_{l+p}$ with the property that for all $f \in B$ we have $f^k \cdot \phi(a) = a \cdot m_f$ in $M_{l+p+k}$. We have already seen uniqueness above. Note that for $l=0$ and $a=1$ this is exactly what we want, because then $\phi(1)$ is the desired preimage in $M_p$.

We proceed by descending induction on $l$. We start with $l=n$. Then we may assume that $a = f_1 \cdot \dotsc \cdot f_n$ with $f_1,\dotsc,f_n \in B$. The choice of $n$ implies that there are $k$ indices, say $1,\dotsc,k$, with $g := f_1=\dotsc=f_k$. Let $\phi(a) := f_{k+1} \cdot \dotsc \cdot f_n \cdot m_g$. This is contained in $A_{n-k} M_{p+k} \subseteq M_{n+p}$. Then, for every $f \in B$, we have
\[f^k \phi(a) = f_1 \dotsc f_{n-k} f^k m_g = f_1 \dotsc f_{n-k} g^k m_f = a m_f.\]
Now let us assume that the claim is true for $l+1$. We may assume again that $a=f_1 \cdot \dotsc \cdot f_l$ with $f_i \in B$. By induction hypothesis we can consider the map $A_1 \to M_{l+p+1}$, $x \mapsto \phi(x f_1 \dotsc f_l)$. It is $k$-linear (this follows easily from the uniqueness assumption). For $f,g,h \in B$ we have
\[h^k f \phi(g f_1 \dotsc f_l) = f g f_1 \dotsc f_l m_h = h^k g \phi(f f_1 \dotsc f_l).\]
By the injectivity part, this implies $f \phi(g f_1 \dotsc f_l) = g \phi(f f_1 \dotsc f_l)$. By the assumption (2) there is a unique $\phi(a) \in M_{l+p}$ such that $\phi(g f_1 \dotsc f_l) = g \phi(a)$ for all $g \in B$. We claim that $\phi(a)$ has the desired property $f^k \phi(a) = a m_f$ for $f \in B$. This follows from the computation
\[g^{k+1} f^k \phi(a) = g^k f^k g \phi(a) = (fg)^k \phi(g f_1 \dotsc f_l) = g^k g f_1 \dotsc f_l m_f = g^{k+1} a m_f.\]
\end{proof}

\begin{rem}
A similar proof shows directly that $\Q(\Proj_S(A))$ satisfies the universal property of $\Proj^\otimes_{\Q(S)}(A)$ without using the existence of the latter.
\end{rem}

\begin{ex}
We do not really have to restrict to tensor categories or schemes over $\Z$. We have constructed for example $\P^n_{\M(\F_1)}=\Q(\P^n_{\F_1})$, where $\F_1$ is the field with one element and $\M(\F_1):=\Set_*$ is the cocomplete tensor category of pointed sets. We refer to \cite{BHS} for the classification of dualizable objects of $\Q(\P^n_{\F_1})$.
\end{ex}

\begin{lemma} \label{TN}
Let $\C$ be a locally presentable tensor category and let $A$ be a commutative $\N$-graded algebra in $\C$ which is generated by $A_1$. If we denote by $R : \grM(A) \to \Proj^{\otimes}(A)$ the reflector in \autoref{proj-ex}, then $R$ maps the inclusion $A[1]_{\geq 0} \hookrightarrow A[1]$ to an isomorphism.
\end{lemma}

\begin{proof}
Let $(P,\O(1),x)$ be the universal triple (\autoref{uni-triple}). Notice that $A[1]$ is concentrated in degrees $\geq -1$, so that $A[1]_{\geq 0} \neq A[1]$ in $\grM(A)$. The diagram
\[A_1^{\otimes 2} \otimes A[-1] \rightrightarrows A_1 \otimes A \to A[1]_{\geq 0}\]
commutes and $A_1 \otimes A \to A[1]_{\geq 0}$ is an epimorphism. By applying $R$, we get that
\[P(A_1)^{\otimes 2} \otimes \O(-1) \rightrightarrows P(A_1) \to R(A[1]_{\geq 0})\]
also commutes and $P(A_1) \to R(A[1]_{\geq 0})$ is an epimorphism. Since
\[P(A_1)^{\otimes 2} \otimes \O(-1) \rightrightarrows P(A_1) \to \O(1)\]
is exact in $\Proj^{\otimes}(A)$, there is a morphism $\O(1) \to R(A[1]_{\geq 0})$ lying over $P(A_1)$. Conversely, we have a morphism $R(A[1]_{\geq 0}) \to R(A[1]) = \O(1)$ lying over $P(A_1)$. Since both sides are quotients of $P(A_1)$, the morphisms are inverse to each other.
\end{proof}

\begin{prop}[Base change] \label{proj-BW}
Let $\C$ be a cocomplete tensor category and $A$ be a commutative $\N$-graded algebra in $\C$ which is generated by $A_1$. Let $H : \C \to \D$ be a cocontinuous tensor functor. Then $B:=H(A)$ is a commutative $\N$-graded algebra in $\D$ which is generated by $B_1$ and there is a $2$-pushout square
\[\xymatrix{\C \ar[r] \ar[d] & \Proj_\C^\otimes(A) \ar[d] \\ \D \ar[r] & \Proj_\D^\otimes(B)}\]
provided that the projective tensor categories exist.
\end{prop}

\begin{proof}
Let $\E$ be a cocomplete tensor category. Then we have natural equivalences of categories
\begin{eqnarray*}
&& \Hom_{c\otimes}(\D,\E) \times_{\Hom_{c\otimes}(\C,\E)} \Hom_{c\otimes}(\Proj_\C^\otimes(A),\E) \\
&\simeq& \{(F,(G,\L,s),\sigma): F \in \Hom_{c\otimes}(\D,\E),\, G \in \Hom_{c\otimes}(\C,\E),\, \L \in \E \\
&&\text{ line object},\,s : G(A) \twoheadrightarrow \oplus_n \L^{\otimes n},\, s_1 \text{ regular},\, \sigma : F H \cong G\} \\
 &\simeq &\{(F,\L,s) : F \in \Hom_{c\otimes}(\D,\E),\, \L \in \E \text{ line object},\, s : F(B) \twoheadrightarrow \oplus_n \L^{\otimes n}, \\
 && ~ s_1 \text{ regular}\} \\
& \simeq & \Hom_{c\otimes}(\Proj_\D^\otimes(B),\E).
\end{eqnarray*} \qedhere
\end{proof}


\subsection{Blow-ups} \label{blow}

Let $X$ be a scheme and consider a closed subscheme $V(I) \subseteq X$. Recall that the blow-up
\[\Bl_I(X)=\Proj_X\bigl(\bigoplus_{n \geq 0} I^n\bigr)\]
of $X$ at $V(I)$ has a morphism $p : \Bl_I(X) \to X$ which is universal with respect to the property that $p$ pulls back the closed subscheme $V(I) \subseteq X$ to an effective Cartier-Divisor on $\Bl_I(X)$ (\cite[\href{http://stacks.math.columbia.edu/tag/0806}{Tag 0806}]{stacks-project}). The latter property means that the ideal $p^* I \cdot \O_{\Bl_I(X)}$ becomes invertible. One has to be careful here as for the meaning of the word ``universal'': The blow-up does \emph{not} represent the ``functor''
\[\Sch^{\op} \to \Set, ~ Y \mapsto \{f \in \Hom(Y,X) : f^* I \cdot \O_Y \subseteq \O_Y \text{ invertible}\},\]
as one might expect, for the trivial reason that this is not a functor at all! If $f^* I \cdot \O_Y$ is invertible and $g : Z \to Y$ is a morphism which is not flat, there is no reason why $g^* f^* I \cdot \O_Z$ is invertible. It is just an ideal of $\O_Z$ which is a quotient of the invertible quasi-coherent module $g^* (f^* I \cdot  O_Y)$.
 
However, $p : \Bl_I(X) \to X$ is a terminal object in the category of $X$-schemes $f : Y \to X$ such that $f^* I \cdot \O_Y$ is invertible (\emph{loc.cit.}). The latter is by definition a full subcategory of $\Sch/X$.

Now we are ready to globalize blow-ups. In order to apply ideal theory (\autoref{affineschemes}) we assume that all tensor categories here are locally presentable and balanced.

\begin{defi}[Image of ideals]
Let $F : \C \to \D$ be a cocontinuous tensor functor between locally presentable tensor categories. Let $I \subseteq \O_\C$ be an ideal (i.e. simply a subobject). The induced morphism $F(I) \to F(\O_\C) \cong \O_\D$ has a factorization $F(I) \to J \to \O_\D$ for some regular epimorphism $F(I) \to J$ and a monomorphism $J \to \O_\D$. We write $J=F(I) \cdot \O_\D$.
\end{defi}

\begin{rem}
One readily checks the compatibilities $\id_\C(I) \cdot \O_\C = I$ and $(GF)(I) \cdot \O_\E = G(F(I) \cdot \O_\D) \cdot \O_\E$ for another $G : \D \to \E$. The construction also commutes with ideal products: If $I,J \subseteq \O_C$ are ideals, then
\[(F(I) \cdot \O_D) \cdot (F(J) \cdot \O_\D) = F(I \cdot J) \cdot \O_\D.\]
In fact, we observe that each side is the smallest ideal $K \subseteq \O_\D$ with the property that $F(I \otimes J) \cong F(I) \otimes F(J) \to \O_\D$ factors through $K$.
\end{rem}

\begin{ex}
Let $f : Y \to X$ be a morphism of schemes and $I \subseteq \O_X$ be a quasi-coherent ideal. Then $f^*(I) \cdot \O_Y=\Im(f^* I \to f^* \O_X \cong \O_Y)$ is an ideal of $\O_Y$. If $X=\Spec(R)$ and $Y=\Spec(S)$ are affine, then $I$ corresponds to an ideal $I \subseteq R$ and the induced ideal of $S$ is usually denoted by $I \cdot S$.
\end{ex}

\begin{defi}[Blow-ups of tensor categories]
Let $\C$ be a locally presentable tensor category and let $I \subseteq \O_\C$ be an ideal. The \emph{blow-up} of $\C$ at $I$ is an $2$-initial locally presentable tensor category $\Bl_I(\C)$ equipped with a cocontinuous tensor functor $P : \C \to \Bl_I(\C)$ such that $P(I) \cdot \O_{\Bl_I(\C)}$ is a line object.

In other words, if $F : \C \to \D$ is a cocontinuous tensor functor to a locally presentable tensor category $\D$ such that the ideal $F(I) \cdot \O_\D$ is a line object, then there is essentially a unique cocontinuous tensor functor $H : \Bl_I(\C) \to \D$ with an isomorphism $HP \cong F$.
\end{defi}

Actually it suffices to demand that $F(I) \cdot \O_\D$ is invertible (since $\O_\D$ is symtrivial and this property descends to flat ideals of $\O_\D$).

\begin{thm} \label{blow-ex}
Let $\C$ be a locally presentable tensor category and let $I \subseteq \O_\C$ be an ideal. Then $\Bl_I(\C)$ exists. It is given by $\Proj^\otimes_\C(A)$, where
\[A=\bigoplus_{n \geq 0} I^n\]
and $I^n$ is the ideal product of $n$ copies of $I$ and the multiplication of $A$ is induced by $I^n I^m = I^{n+m}$.
\end{thm}

\begin{proof}
We use \autoref{proj-ex}. Let $P : \C \to \B:=\Proj^{\otimes}_\C(A)$, $\O(1) \in \B$ and $x : P(A) \to \bigoplus_{n \geq 0} O(n)$ be the universal triple (\autoref{uni-triple}). We claim $P(I) \cdot \O_\B = \O(1)$. In fact, we have
\[(I \otimes A) \cdot A = \bigoplus_{n \geq 0} I^{n+1} = A[1]_{\geq 0}\]
in $\grM(A)$ and by \autoref{TN} the inclusion $A[1]_{\geq 0} \hookrightarrow A[1]$ becomes an isomorphism in $\B$. In particular, $P(I) \cdot \O_\B$ is a line object. Now let $F : \C \to \D$ be as above such that $\L := F(I) \cdot \O_\D$ is a line object. Since $\L$ is flat, for every $n \geq 0$ the ideal product $\L^n$ coincides with the tensor power $\L^{\otimes n}$. It follows $F(I^n) \cdot  \O_\D = \L^{\otimes n}$. In particular, there is a regular epimorphism $F(I^n) \to \L^{\otimes n}$. These induce a homomorphism of graded algebras
\[s : \bigoplus_{n \geq 0} F(I^n) \to \bigoplus_{n \geq 0} \L^{\otimes n}.\]
Therefore, by \autoref{proj-def} the triple $(F,\L,s)$ corresponds to a cocontinuous tensor functor $\B \to \D$ extending $F$. Conversely, assume that $G : \B \to \D$ is a cocontinuous tensor functor with an isomorphism $GP \cong F$. Then we have an epimorphism
\[G(\O(1)) \twoheadrightarrow G(\O(1)) \cdot \O_\D = G(P(I) \cdot \O_\B) \cdot \O_\D \cong F(I) \cdot \O_\D = \L\]
between line objects lying over $F(I)$. By \autoref{inv-diskret} it has to be an isomorphism. This shows that $G$ is unique.
\end{proof}

\begin{prop}[Base change]
Let $F : \C \to \D$ be a cocontinuous tensor functor between locally presentable tensor categories and let $I \subseteq \O_\C$ be an ideal. Assume that $F$ preserves monomorphisms. Then $J:=F(I) \subseteq \O_\D$ is an ideal and we have a $2$-pushout of cocomplete tensor categories
\[\xymatrix{\C \ar[r] \ar[d] & \D \ar[d] \\ \Bl_I(\C) \ar[r] & \Bl_J(\D).}\]
\end{prop}

\begin{proof}
We use \autoref{blow-ex}. By \autoref{proj-BW} the $2$-pushout exists and is given by $\Proj_\D^{\otimes}(B)$, where $B=\bigoplus_{n \geq 0} F(I^n)$. Since $F$ preserves monomorphisms, we have $F(I^n) = F(I)^n=J^n$ as ideals of $\O_\D$.
\end{proof}

\begin{thm}[Comparison to schemes] \label{blow-compare}
Let $X$ be a scheme and let $I \subseteq \O_X$ be a quasi-coherent ideal of finite presentation. Then there is an equivalence of cocomplete tensor categories
\[\Q\bigl(\Bl_I(X)\bigr) \simeq \Bl_I \bigl(\Q(X)\bigr).\]
\end{thm}

\begin{proof}
This follows from \autoref{blow-ex} and \autoref{proj-compare}.
\end{proof}

\begin{ex}[Blow-up of the affine plane]
Let us look at the standard example. Let $X = \mathbb{A}^2_R=\Spec(R[s,t])$ and $I=(s,t)$. Then $A=\bigoplus_{n \geq 0} I^n$ is generated by the elements $s,t$ in degree $1$ subject to the single relation $s \cdot t = t \cdot s$, where the factors on the left come from degree $0$. In other words, there is an isomorphism $\mathcal{O}_X[U,V]/(sV-tU) \cong A$ with $U \mapsto s$ and $V \mapsto t$. It follows that $\Bl_I(X)$ is  the closed subscheme of $\mathbb{P}^1_X = \mathrm{Proj}_X(\mathcal{O}_X[U,V])$ which is cut out by the equation $sV=tU$. Now \autoref{proj-compare} implies that $\Q(\Bl_I(X))$ is the initial $R[s,t]$-linear cocomplete tensor category equipped with a line object $\mathcal{L}$ and two global sections $U,V : \O \to \L$  such that $s V = t U$ and which generate $\L$ in the sense that $(U,V) : \O^2 \to \L$ is a regular epimorphism. On the other hand, by \autoref{blow-compare} this is also the initial $R[s,t]$-linear locally presentable tensor category  $\C$ with the property that the ideal $(s,t) \cdot \O_\C \subseteq \O_\C$ is invertible. The correspondence is given by $\L=(s,t) \cdot \O_\C$, $U=s$, $V=t$.
\end{ex}

\subsection{The Segre embedding} \label{segre}

Let $S$ be a  scheme and $E_1,E_2$ quasi-coherent modules on $S$. The usual Segre embedding (\cite[9.8.6]{EGAI}) is a closed immersion
\[\phi: \P(E_1) \times_S \P(E_2) \hookrightarrow \P(E_1 \otimes E_2),\]
which on $T$-valued points maps a pair of invertible quotients of $E_1|_T$ and $E_2|_T$ to their tensor product, which is an invertible quotient of $(E_1 \otimes E_2)|_T$. The usual proof that $\phi$ is a closed immersion works locally. But it is also possible to find global equations in $\Sym(E_1 \otimes E_2)$ which cut out the embedding, namely 
\[(a \otimes b) \cdot (c \otimes d) = (a \otimes d) \cdot (c \otimes b)\]
for local sections $a,c$ of $E_1$ and $b,d$ of $E_2$. In other words, we have a correspondence between invertible quotients $s_1 : E_1 \to \L_1$, $s_2 : E_2 \to \L_2$ and invertible quotients $s : E_1 \otimes E_2 \to \L$ such that the \emph{Segre relations} hold in $\L \otimes \L$:
\[s(a \otimes b) \otimes s(c \otimes d) = s(a \otimes d) \otimes s(c \otimes b)\]
We would like to prove this correspondence in arbitrary cocomplete tensor categories. It is clear how to get $s$ from $s_1$ and $s_2$, namely via $s = s_1 \otimes s_2$. But for the converse one has to globalize the usual local constructions. Roughly, the idea is as follows: If $s = s_1 \otimes s_2$, then we have the relation
\[s_1(a) \otimes s(c \otimes b) = s_1(c) \otimes s(a \otimes b).\]
Conversely, given $s$, we define $s_1$ to be the universal solution for these relations. We will work this out in detail.
 
\begin{thm}[Segre relations]
Let $\C$ be a cocomplete tensor category and $E_1,E_2 \in \C$. Then there is an equivalence between the category of line objects $\L_1,\L_2$ with regular epimorphisms $s_1 : E_1 \to \L_1$, $s_2 : E_2 \to \L_2$ and the category of line objects $\L$ with a regular epimorphism $s : E_1 \otimes E_2 \to \L$ satisfying the Segre relations: The diagram
\[\xymatrix@R=10pt@C=30pt{ E_1 \otimes E_2 \otimes E_1\otimes E_2 \ar@/^0.5pc/[dr]^-{s \otimes s} \ar[dd]_{\cong} \\ & \L \otimes \L \\ E_1 \otimes E_2 \otimes E_1 \otimes E_2 \ar@/_0.5pc/[ur]_-{s \otimes s} &  }\]
commmutes, where the vertical isomorphism is induced by the symmetry $S_{E_2,E_2}$. In element notation, this relation can be written as
\[s(a \otimes b) \otimes s(c \otimes d) = s(a \otimes d) \otimes s(c \otimes b) \in \L \otimes \L\]
for $a,c \in E_1$ and $b,d \in E_2$.
\end{thm}

\begin{proof}
Let $A$ be the category of all $(\L_1,s_1,\L_2,s_2)$ and let $B$ be the category of all $(\L,s)$ satisfying the Segre relations. According to \autoref{inv-diskret} both categories are essentially discrete. 
There is an obvious functor $A \to B$ given by $\L = \L_1 \otimes \L_2$ and $s = s_1 \otimes s_2$. Note that $s : E_1 \otimes E_2 \to \L_1 \otimes \L_2$ is a regular epimorphism (\autoref{coeq-tensor}) satisfying the Segre relations: Since $\L_2$ is symtrivial, we have
\[s(a \otimes b) \otimes s(c \otimes d) = s_1(a) \otimes s_2(b) \otimes s_1(c) \otimes s_2(d)\]
\[ = s_1(a) \otimes s_2(d) \otimes s_1(c) \otimes s_2(b) = s(a \otimes d) \otimes s(c \otimes b).\]
The action on morphisms is clear. The functor $B \to A$ is more complicated. Let $\L$ be a line object and let $s : E_1 \otimes E_2 \to \L$ be a regular epimorphism satisfying the Segre relations $s(a \otimes b) \otimes s(c \otimes d) = s(a \otimes d) \otimes s(c \otimes b)$, i.e. we can interchange the two ``elements'' of $E_2$. But the same is true for $E_1$ as well, using that $\L$ is symtrivial:
\[s(a \otimes b) \otimes s(c \otimes d) = s(c \otimes d) \otimes s(a \otimes b) = s(c \otimes b) \otimes s(a \otimes d)\]
Thus, actually there are two Segre relations, but they imply each other. Consider the two morphisms
\[E_1^{\otimes 2} \otimes E_2 \rightrightarrows E_1 \otimes \L\]
defined by $a \otimes c \otimes b \mapsto a \otimes s(c \otimes b)$ (i.e. $E_1 \otimes s$) resp. $c \otimes s(a \otimes b)$ (i.e. $(E_1 \otimes s) \circ (S_{E_1,E_1} \otimes E_2)$).
These correspond to morphisms
\[E_1^{\otimes 2} \otimes E_2 \otimes \L^{\otimes -1} \rightrightarrows E_1.\]
We define $s_1 : E_1 \to \L_1$ to be an coequalizer of these morphisms. It follows that $s_1 \otimes \L : E_1 \otimes \L \to \L_1 \otimes \L$ is the coequalizer of the two morphisms defined first. In particular, this means $s_1(a) \otimes s(c \otimes b) = s_1(c) \otimes s(a \otimes b)$. Define $s_2 : E_2 \to \L_2$ analogously. Hence, we have $s_2(b) \otimes s(a \otimes d) = s_2(d) \otimes s(a \otimes b)$. We claim that $(\L_1,s_1,\L_2,s_2)$ is an object in $A$. For this we claim:
\begin{enumerate}
\item $\L_1$ and $\L_2$ are symtrivial.
\item There is an isomorphism $\L_1 \otimes \L_2 \cong \L$ such that the diagram
\[\xymatrix{ & E_1 \otimes E_2 \ar[dr]^{s} \ar[dl]_{s_1 \otimes s_2 ~ } & \\ \L_1 \otimes \L_2 \ar[rr]_{\cong} && \L}\]
commutes.
\item $\L_1$ and $\L_2$ are line objects.
\end{enumerate}

\emph{Proof of 1}. For $\L_1$ it suffices to prove that $s_1(a) \otimes s_1(a') = s_1(a') \otimes s_1(a)$ for $a,a' \in E$. Let us introduce additional parameters $c \in E_1,\, d \in E_2$ and compute
\[s_1(a) \otimes s_1(a') \otimes s(c \otimes d) = s_1(a) \otimes s_1(c) \otimes s(a' \otimes d)\]
\[ = s_1(a') \otimes s_1(c) \otimes s(a \otimes d) = s_1(a') \otimes s_1(a) \otimes s(c \otimes d).\]
Since $s$ is an epimorphism, by \autoref{epi-cancel} this implies the desired equation. (This kind of tensoring with an epimorphism replaces the usual local calculations in the case of schemes.)

\emph{Proof of 2}. By \autoref{coeq-tensor} and \autoref{goodepi} we have two coequalizer diagrams
\[(E_1 \otimes E_2)^{\otimes 2} \rightrightarrows E_1 \otimes E_2 \otimes \L \xrightarrow{s \otimes \L} \L^{\otimes 2}\]
\[\bigl(E_1^{\otimes 2} \otimes E_2  \bigr) \otimes E_2 \oplus E_1 \otimes \bigl(E_2^{\otimes 2} \otimes E_1\bigr) \rightrightarrows E_1 \otimes E_2 \otimes \L \xrightarrow{s_1 \otimes s_2 \otimes \L} \L_1 \otimes \L_2 \otimes \L.\]
Thus, in order to construct a morphism $\L_1 \otimes \L_2 \otimes \L \to \L^{\otimes 2}$ over the epimorphisms from $E_1 \otimes E_2 \otimes \L$, it suffices to prove that the two morphisms
\[\bigl(E_1^{\otimes 2} \otimes E_2 \bigr) \otimes E_2 \oplus E_1 \otimes \bigl(E_2^{\otimes 2} \otimes E_1\bigr) \rightrightarrows E_1 \otimes E_2 \otimes \L \xrightarrow{s \otimes \L} \L^{\otimes 2}\]
agree. On the first summand, this means that
\[s(a \otimes d) \otimes s(c \otimes b) = s(c \otimes d) \otimes s(a \otimes b).\]
On the second summand, this means that
\[s(a \otimes b) \otimes s(c \otimes d) = s(a \otimes d) \otimes s(c \otimes b).\]
But these are just the Segre relations. Next, in order to construct a morphism $\L^{\otimes 2} \to \L_1 \otimes \L_2 \otimes \L$ over $E_1 \otimes E_2 \otimes \L$, it suffices to prove that the two morphisms
\[(E_1 \otimes E_2)^{\otimes 2} \rightrightarrows E_1 \otimes E_2 \otimes \L \xrightarrow{s_1 \otimes S_2 \otimes \L} \L_1 \otimes \L_2 \otimes \L\]
agree, which means that
\[s_1(c) \otimes s_2(d) \otimes s(a \otimes b) = s_1(a) \otimes s_2(b) \otimes s(c \otimes d).\]
But this  follows easily from the definitions of $s_1$ and $s_2$ above. First we interchange $a$ and $c$, and then $d$ and $b$.

We have constructed an isomorphism $\L^{\otimes 2} \to \L_1 \otimes \L_2 \otimes \L$ over $E_1 \otimes E_2 \otimes \L$, i.e. an isomorphism $\L \cong \L_1 \otimes \L_2$ over $E_1 \otimes E_2$.

\emph{Proof of 3}. From 2. we know that $\L_1 \otimes \L_2 \cong \L$ is invertible. It follows that $\L_1$ (and likewise $\L_2$) is invertible with $\L_1^{\otimes -1} = \L_2 \otimes \L^{\otimes -1}$. By 1., $\L_1$ and $\L_2$ are symtrivial. Hence, $\L_1$ and $\L_2$ are line objects.

This finishes the proof that $(\L_1,s_1,\L_2,s_2)$ is in $A$. The action on morphisms is as follows: Given an isomorphism $\alpha : \L \to \L'$ from $s : E_1 \otimes E_2 \to \L$ to $s' : E_1 \otimes E_2 \to \L'$, there is a unique isomorphism $\alpha_1 : \L_1 \to \L'_1$ with $\alpha_1 s_1 = s'_1$ and similarly $\alpha_2 : \L_2 \to \L'_2$.
 
We have already seen in 2. that the composition $B \to A \to B$ is isomorphic to the identity. For the other composition, start with $(\L_1,s_1,\L_2,s_2)$ in $A$ and construct $(\L,s) = (\L_1 \otimes \L_2,s_1 \otimes s_2)$. We have to prove that
\[E_1^{\otimes 2} \otimes E_2 \otimes \L^{\otimes -1} \rightrightarrows E_1 \xrightarrow{s_1} \L_1\]
is a coequalizer, where the two morphisms correspond to $E_1^{\otimes 2} \otimes E_2 \rightrightarrows E_1 \otimes \L$, $a \otimes c \otimes b \mapsto a \otimes s(c \otimes b) = a \otimes s_1(c) \otimes s_2(b)$ resp. $c \otimes s_1(a) \otimes s_2(b)$. So let $h : E_1 \to T$ be a morphism which coequalizes the two morphisms, i.e. we have $h(a) \otimes s_1(c) \otimes s_2(b) = h(c) \otimes s_1(a) \otimes s_2(b)$ for all $a,c \in E_1$ and $b \in E_2$. Since $s_2$ is an epimorphism, this means that $h(a) \otimes s_1(c) = h(c) \otimes s_1(a)$. By \autoref{goodepi} this means that $h$ factors uniquely through $s_1$. Of course the same argument works for $s_2$. This finishes the proof that $A \to B \to A$ is isomorphic to the identity.
\end{proof}

\begin{cor}[Segre embedding] \label{segre-emb}
Let $\C$ be a cocomplete tensor category and $E_1,E_2 \in \C$. Then there is an equivalence of cocomplete tensor categories (using element notation)
\[\P^\otimes_\C(E_1) \sqcup_\C \P^\otimes_\C(E_2) \simeq \P^\otimes_\C(E_1 \otimes E_2) / ((a \otimes b) \cdot (c \otimes d) = (a \otimes d) \cdot (c \otimes b))_{a,c \in E_1,\, b,d \in E_2}.\]
\end{cor}
 
\begin{ex}
The classical Segre embedding
\[\P_\mathds{C}^1 \times \P_\mathds{C}^1 \hookrightarrow \P_\mathds{C}^3,\, ([a:b],[c:d]) \mapsto [ac:ad:bc:bd]\]
whose image is cut out by $x_0 x_3 = x_1 x_2$ globalizes as follows:

There is an equivalence of cocomplete tensor categories
\[\P_\C^1 \sqcup_\C \P_\C^1 \simeq \P_\C^3 / (x_0 x_3 = x_1 x_2).\]
\end{ex}


\subsection{The Veronese embedding}

The usual $d$-uple Veronese embedding $\nu_d : \P^n \hookrightarrow \P^{N}$ (with $N= \binom{n+d}{d}-1$) has a coordinate-free description and generalization: If $E$ is a quasi-coherent module on a scheme $S$, then there is a closed immersion
\[\nu_d : \P(E) \hookrightarrow \P(\Sym^d E)\]
which maps on $T$-valued points an invertible quotient of $E|_T$ to its $d$th symmetric power, which is an invertible quotient of $\Sym^d E|_T$. A reference for a special case of this is Harris' book \cite[p. 25]{Har92}. There, the assumption on the characteristic is caused by a dualization process which is already built into Grothendieck's functorial definition of $\P(E)$ (\cite[9.7.5]{EGAI}), therefore we will not need any such assumption. We now proceed similarly to the Segre embedding and will therefore keep it shorter.

\begin{thm}[Veronese relations]
Let $\C$ be a cocomplete tensor category, $E \in \C$ and $d \in \N^+$. Then there is an equivalence between the category of line objects $\L$ with a regular epimorphism $s : E \twoheadrightarrow \L$ and the category of line objects $\K$ with a regular epimorphism $t : \Sym^d E \twoheadrightarrow \K$ such that the Veronese relations hold: The diagram
\[\xymatrix@C=40pt@R=10pt{E^{\otimes d} \otimes E^{\otimes d} \ar@{>>}[r]  \ar[dd]_{\cong} &  \Sym^d E \otimes \Sym^d E \ar@/^0.5pc/[dr]^-{t \otimes t} & \\ && \K \otimes \K \\ E^{\otimes d} \otimes E^{\otimes d} \ar@{>>}[r]  & \Sym^d E \otimes \Sym^d E \ar@/_0.5pc/[ur]^-{t \otimes t} }\]
commutes, where the vertical isomorphism is induced by the symmetry which interchanges the first tensor factors in the two $E^{\otimes d}$.
In element notation, this means that
\[t(v_1 \cdot v_2 \cdot \dotsc \cdot v_d) \otimes t(w_1 \cdot w_2 \cdot \dotsc \cdot w_d) = t(w_1 \cdot v_2 \cdot \dotsc \cdot v_d) \otimes t(v_1 \cdot w_2 \cdot \dotsc \cdot w_d).\]
This is also equivalent to the condition that $t^{\otimes 2} : \Sym^d E \otimes \Sym^d E \to \K \otimes \K$ extends to a morphism $\Sym^{2d} E \to \K \otimes \K$.
\end{thm}

\begin{proof}
Let $A$ be the category of all $(\L,s)$ and let $B$ be the category of all $(\K,t)$ satisfying the Veronese relations. Both are essentially discrete. There is an obvious functor $A \to B$ given by $\K = \Sym^d \L = \L^{\otimes d}$ and $t = \Sym^d s$. Notice that $t$ is a regular epimorphism by \autoref{sym-exakt}. In order to construct $B \to A$, let $(\K,t) \in B$. There are two morphisms $E \otimes E^{\otimes d} \rightrightarrows E \otimes \K$ defined by $w_1 \otimes v_1 \otimes \dotsc \otimes v_d \mapsto w_1 \otimes t(v_1 \cdot \dotsc \cdot v_d)$ resp. $v_1 \otimes t(w_1 \cdot \dotsc \cdot v_d)$. Let $s : E \to \L$ be the coequalizer of the corresponding morphisms $E \otimes E^{\otimes d} \otimes \K^{\otimes -1} \rightrightarrows E$, so that we have
\[s(w_1) \otimes t(v_1 \cdot \dotsc \cdot v_d) = s(v_1) \otimes t(w_1 \cdot \dotsc \cdot v_d).\]
In order to show that $\L$ is symtrivial, we calculate
\[s(w_1) \otimes s(w'_1) \otimes t(v_1 \cdot \dotsc \cdot v_d) = s(w_1) \otimes s(v_1) \otimes t(w'_1 \cdot \dotsc \cdot v_d)\]
\[= s(w'_1) \otimes s(v_1) \otimes t(w_1 \cdot \dotsc \cdot v_d) = s(w'_1) \otimes s(w_1) \otimes t(v_1 \cdot \dotsc \cdot v_d).\]
Since $t$ is an epimorphism, this means $s(w_1) \otimes s(w'_1) = s(w'_1) \otimes s(w_1)$, i.e. $\L$ is symtrivial. Next we claim that there is an isomorphism $\L^{\otimes d} \cong \K$ over $\Sym^d E$; in particular $\L$ will be invertible, hence a line object. Looking at the coequalizer diagrams (\autoref{goodepi} and \autoref{sym-exakt})
\[(\Sym^d E)^{\otimes 2} \rightrightarrows \Sym^d E \otimes \K \xrightarrow {t \otimes \K} \K \otimes \K\]
\[E \otimes E^{\otimes d} \otimes \Sym^{d-1} E \rightrightarrows \Sym^d E \otimes \K \xrightarrow{\Sym^d(s) \otimes \K} \L^{\otimes d} \otimes \K\]
it suffices to prove the relations
\[s(v_1) \otimes \dotsc \otimes s(v_d) \otimes t(w_1 \cdot \dotsc \cdot w_d) = s(w_1) \otimes \dotsc \otimes s(w_d) \otimes t(v_1 \cdot \dotsc \cdot v_d),\]
\[t(w_1 \cdot w_2 \cdot \dotsc \cdot w_d) \otimes t(v_1 \cdot v_2 \cdot \dotsc \cdot v_d) = t(v_1 \cdot w_2 \cdot \dotsc \cdot w_d) \otimes t(w_1 \cdot v_2 \cdot \dotsc \cdot v_d). \]
The first relation follows inductively from the definition of $s$. The second is the Veronese relation. This defines the functor $B \to A$ and shows that $B \to A \to B$ is isomorphic to the identity. Conversely, given $(\L,s)$ and defining $(K,t)$ by $(\L^{\otimes d}, \Sym^d s)$, we have to prove that
\[E \otimes E^{\otimes d} \otimes \K^{\otimes -1} \rightrightarrows E \xrightarrow{s} \L\]
is a coequalizer. If $h : E \to T$ is a morphism, then it coequalizes the two morphisms if and only if $h(w_1) \otimes s(v_1) \otimes \dotsc \otimes s(v_d) = h(v_1) \otimes s(w_1) \otimes \dotsc \otimes s(v_d)$. Since $s^{\otimes d-1}$ is an epimorphism, this condition reduces to
\[h(w_1) \otimes s(v_1) = h(v_1) \otimes s(w_1).\]
According to \autoref{goodepi} this means that $h$ factors uniquely through $s$. This finishes the proof that $A \to B \to A$ is isomorphic to the identity.
\end{proof}

\begin{cor}[Veronese embedding] \label{veronese}
Let $\C$ be a cocomplete tensor category, $E \in \C$ and $d \in \N^+$. Then there is an equivalence
\[\P_\C^\otimes(E) \simeq \P_\C^\otimes(\Sym^d(E))/((a_1 a_2 \dotsc a_d) \cdot (b_1 b_2 \dotsc b_d) = (b_1 a_2 \dotsc a_d) (a_1 b_2 \dotsc b_d)).\]
\end{cor}
 
\begin{ex}
The classical $2$-uple Veronese embedding
\[\P^1 \hookrightarrow \P^5,~[a:b:c] \mapsto [a^2:b^2:c^2:bc:ac:ab]\]
becomes
\[\P^1_\C \simeq \P^5_\C / (t_0 t_1=t_5^2,\,t_0 t_4^2,\,t_0 t_3 = t_4 t_5,\, t_1 t_2 = t_3^2,\, t_1 t_4 = t_3 t_5,\, t_2 t_5 = t_3 t_4).\]
\end{ex}


\subsection{The Plücker embedding}

Recall from \cite[9.8]{EGAI} the classical Plücker embedding from Grassmannians into projective spaces: If $S$ is a scheme, $d \geq 1$ and $E$ is a quasi-coherent module on $S$, then there is a closed immersion
\[\overline{\omega} : \Grass_d(E)  \hookrightarrow \P(\Lambda^d(E))\]
which maps on $T$-valued points a locally free quotient of rank $d$ of $E$ to its $d$-th exterior power (or determinant), which is an invertible quotient of $\Lambda^d(E)$. It is cut out by the \emph{Plücker relations} (\cite{KL72}). We would like to globalize this, because this will enable us to prove the existence of Grassmannians over cocomplete tensor categories. In the following we fix $d \geq 1$ and $R$ is a commutative $\QQ$-algebra (or at least $d! \in R^*$).
 
\begin{thm}[Plücker relations] \label{PR}
Let $\C$ be an $R$-linear cocomplete tensor category and $E \in \C$. Then the category of pairs $(V,t)$, where $V \in \C$ is a locally free object of rank $d$ and $t : E \to V$ is a regular epimorphism, is equivalent to the category of pairs $(\L,s)$, where $\L \in \C$ is a line object and $s : \Lambda^d(E) \to \L$ is a regular epimorphism which satisfies the Plücker relations: The morphism $\Lambda^{d-1} E \otimes \Lambda^{d+1} E \to \L \otimes \L$ which maps $a_1 \wedge \dotsc \wedge a_{d-1} \otimes b_0 \wedge \dotsc \wedge b_d$ to
\[\sum_{k=0}^{d} (-1)^k s(a_1 \wedge \dotsc \wedge a_{d-1} \wedge b_k) \otimes s(b_0 \wedge \dotsc \wedge \widehat{b_k} \wedge \dotsc \wedge b_d)\]
(in element notation) has to vanish.
\end{thm}

\begin{proof}
Let $A$ be the category of pairs $(V,t)$ and $B$ be the category of pairs $(\L,s)$ as in the claim. There is a functor $A \to B$ defined by mapping $(V,t)$ to $(\Lambda^d(V),\Lambda^d(t))$. Notice that $\Lambda^d(t)$ satisfies the Plücker relations because of \autoref{rel-frei} and that $\Lambda^d(V)$ is a line object by \autoref{det}. Conversely, let $\L$ be a line object and $s : \Lambda^d(E) \to \L$ be a regular epimorphism satisfying the Plücker relations. Consider the morphism (see \autoref{hilfs-omega} for the definition of $\omega$)
\[\Lambda^{d+1}(E) \xrightarrow{\omega} E \otimes \Lambda^d(E) \xrightarrow{E \otimes s} E \otimes \L\]
and the corresponding morphism $\Lambda^{d+1}(E) \otimes \L^{\otimes -1} \to E$. Let $t : E \to V$ be its cokernel. Thus, we have in element notation
\[(\star) ~ ~ 0 = \sum_{k=0}^{d} (-1)^k  t(a_k) \otimes s(a_0 \wedge \dotsc \wedge \widehat{a_k} \wedge \dotsc \wedge a_d)\]
for all $a_0 \wedge \dotsc \wedge a_d \in \Lambda^{d+1}(E)$. We claim that there is a unique isomorphism $\Lambda^d(V) \cong \L$ such that the diagram
\[\xymatrix{ & \Lambda^d(E) \ar[dr]^{s} \ar[dl]_{\Lambda^d(t)} & \\ \Lambda^d(V) \ar[rr]_{\cong} && \L}\]
commutes. By \autoref{goodepi} and \autoref{sym-exakt} we have two exact sequences
\[\xymatrix@C=44pt{\Lambda^d(E) \otimes \Lambda^d(E) \ar[r] & \Lambda^d(E) \otimes \L \ar@{=}[d] \ar[r]^-{s \otimes \L} & \L \otimes \L \ar[r] & 0 \\ 
\Lambda^{d+1}(E) \otimes \Lambda^{d-1}(E) \ar[r] & \Lambda^d(E) \otimes \L \ar[r]^-{\Lambda^d(t) \otimes \L} \ar[r] & \Lambda^d(V) \otimes \L \ar[r] & 0.}\]
Therefore, the claim will follow once we show that the two compositions
\begin{enumerate}
\item $\Lambda^d(E) \otimes \Lambda^d(E) \to \Lambda^d(E) \otimes \L \to \Lambda^d(V) \otimes \L$
\item $\Lambda^{d+1}(E) \otimes \Lambda^{d-1}(E) \to \Lambda^d(E) \otimes \L \to \L \otimes \L$
\end{enumerate}
vanish. For $1.$ this means
\[t(a_1) \wedge \dotsc \wedge t(a_d) \otimes s(b_1 \wedge \dotsc \wedge b_d) = t(b_1) \wedge \dotsc \wedge t(b_d) \otimes s(a_1 \wedge \dotsc \wedge a_d),\]
which follows from $(\star)$ and \autoref{symlem}. And $2.$ is precisely the content of the Plücker relations for $s$.

In particular, $\Lambda^d(V)$ is invertible. In order to show that $V$ is locally free of rank $d$, we  have to verify by \autoref{rel-frei} that
\[0 = \sum_{k=0}^{d} (-1)^k t(a_k) \otimes t(a_0) \wedge \dotsc \wedge \widehat{t(a_k)} \wedge \dotsc \wedge t(a_d).\]
Since we have already proven $\Lambda^d t \cong s$, this follows from $(\star)$. This finishes the construction of the functor $B \to A$ and the proof that $B \to A \to B$ is isomorphic to the identity. For the other direction, let $(V,t) \in B$ and $\L := \Lambda^d(V)$, $s := \Lambda^d(t)$. We have to prove that
\[\Lambda^{d+1}(E) \otimes \L^{\otimes -1} \to E \xrightarrow{t} V \to 0\]
is exact. First of all, the composition $\Lambda^{d+1}(E) \otimes \L^{\otimes -1} \to V$ resp. its dual $\Lambda^{d+1}(E) \to V \otimes \L$ maps $a_0 \wedge \dotsc \wedge a_d$ to
\[\sum_{k=0}^{d} (-1)^k t(a_k) \otimes t(a_0) \wedge \dotsc \wedge \widehat{t(a_k)} \wedge  \dotsc \wedge t(a_d), \]
which vanishes because $V$ is locally free of rank $d$ (\autoref{rel-frei}). Now let $h : E \to T$ be a morphism such that $\Lambda^{d+1}(E) \otimes \L^{\otimes -1} \to E \to T$ vanishes, i.e.
\[0 = \sum_{k=0}^{d} (-1)^k h(a_k) \otimes t(a_0) \wedge \dotsc \wedge \widehat{t(a_k)} \wedge \dotsc \wedge t(a_d).\]
We have to prove that $h$ factors uniquely through $t$. The argument is similar to \autoref{goodepi}: Since $t$ is a regular epimorphism, we find an exact sequence $K \xrightarrow{r} E \xrightarrow{t} V  \to 0$. For $c \in K$ we have $t(r(c))=0$ and therefore
\[h(r(c)) \otimes t(a_1) \wedge \dotsc \wedge t(a_d) = \sum_{k=1}^{d} (-1)^{k+1} h(a_k) \otimes t(r(c)) \wedge t(a_1) \wedge \dotsc \wedge \widehat{t(a_k)}  \wedge \dotsc\]
vanishes. This shows $hr \otimes s = 0$. By \autoref{epi-cancel} this means $hr=0$. Since $t$ is the cokernel of $r$, we see that $h$ factors uniquely through $t$.
\end{proof}

The following definition mimics Grothendieck's functorial definition of Grassmannians (\cite[9.7]{EGAI}).
 
\begin{defi}[Grassmannian tensor categories]
Let $\C$ be an $R$-linear cocomplete tensor category and $E \in \C$. We define the \emph{Grassmannian tensor category} $\Grass_d^\otimes(E)$ to be a cocomplete tensor category with the following universal property (if it exists): If $\D$ is another $R$-linear cocomplete tensor category, then $\Hom_{c\otimes/R}(\Grass_d^\otimes(E),\D)$ is equivalent to the category of triples $(F,V,t)$, where $F : \C \to \D$ is a cocontinuous tensor functor, $V \in \D$ is locally free of rank $d$ and $t : F(E) \to V$ is a regular epimorphism. For example, we have $\Grass_1^\otimes(E) = \P^\otimes(E)$.
\end{defi}

\begin{thm}[Existence via Plücker] \label{grassex}
If $\C$ is a locally presentable $R$-linear tensor category and $E \in \C$, then $\Grass_d^\otimes(E)$ exists and is given by
\[\P_\C^\otimes(\Lambda^d(E))/(0=\sum_{k=0}^{d} (-1)^k (a_1 \wedge \dotsc \wedge a_{d-1} \wedge b_k) \cdot (b_0 \wedge \dotsc \wedge \widehat{b_k} \wedge \dotsc \wedge b_d)).\]
\end{thm}

\begin{proof}
This follows immediately from \autoref{proj-ex} and \autoref{PR}.
\end{proof}

From \autoref{grassex} and \autoref{proj-compare} we deduce:
 
\begin{cor} \label{grass-compare} \marginpar{I've added this Corollary.}
If $\C=\Q(S)$ for some $R$-scheme $S$ and if $E \in \C$ is of finite presentation, then $\Grass_d^\otimes(E) \simeq \Q\bigl(\Grass_d(E)\bigr)$.
\end{cor}

\begin{ex}
For $d=2$ and $E = \mathcal{O}^{\oplus 4}$ we get the simplest non-trivial example of a Grassmannian. If $E$ has a basis indexed by $1,2,3,4$, then $\Lambda^2(E)$ has a basis indexed by $(1,2),(1,3),(1,4),(2,3),(2,4),(3,4)$. Hence, we have $\P^\otimes(\Lambda^2(E)) = \P^5$ with variables $X_{1,2},X_{1,3},X_{1,4},X_{2,3},X_{2,4},X_{3,4}$. There is essentially only one Plücker relation with $a_1=1$ and $b_1=2$, $b_2=3$, $b_3=4$. Thus,
\[\Grass_2(\O^4) \simeq \P^5/(0 = X_{1,2} X_{3,4} - X_{1,3} X_{2,4} + X_{1,4} X_{2,3})\]
is the universal cocomplete linear tensor category with a locally free object of rank $2$ with $4$ global generators.
\end{ex}

\begin{rem}
It would be interesting to investigate if Grothendieck's construction of the Hilbert scheme as a subscheme of a Grassmannian globalizes to the setting of tensor categories.
\end{rem}

\section{Products of schemes} \label{prodschemes}

It is a natural question if the $2$-functor
\[\Q : \Stack \to \Cat_{c\otimes}^{\op}\]
preserves $2$-limits. Notice that it preserves $2$-colimits for trivial reasons (\autoref{adjunction}), but preservation of $2$-limits is a nontrivial problem.

Sch\"appi has given an affirmative answer for finite limits of Adams stacks (\cite{Sch12b} and \cite{Sch13}). The proof is quite involved since it uses Sch\"appi's characterization of weakly Tannakian categories.

In this section, we draw our attention to schemes and to more direct proofs.

\subsection{Tensorial base change}

\begin{defi}[Tensorial base change]
We say that a morphism of schemes $f : X \to S$ has \emph{tensorial base change} if for all morphisms of schemes $g : Y \to S$ the diagram
\[\xymatrix{\Q(S) \ar[r]^{g^*} \ar[d]^{f^*} & \Q(Y) \ar[d]^{\pr_Y^*} \\ \Q(X) \ar[r]_-{\pr_X^*} & \Q(X \times_S Y)}\]
is a $2$-pushout in the $2$-category of cocomplete tensor categories.

In other words, we require that for every cocomplete tensor category $\C$ the canonical functor
\[\Hom_{c\otimes}(\Q(X \times_S Y),\C)\]
\[\downarrow\]
\[\Hom_{c\otimes}(\Q(X),\C) \times_{\Hom_{c\otimes}(\Q(S),\C)}  \Hom_{c\otimes}(\Q(Y),\C)\]
given by mapping $H : \Q(X) \to \C$ to
\[\bigl(H \pr_X^*, H \pr_Y^*, H \pr_X^*  f^* \cong H (f \pr_X)^* = H (g \pr_Y)^* \cong H \pr_Y^* g^*\bigr)\]
is an equivalence of categories. We say that $f$ has qc qs tensorial base change, this holds at least for all qc qs morphisms $g : Y \to S$ and $S$ is qc qs.
\end{defi}

\begin{lemma}[Closure properties]\noindent
\begin{enumerate}
\item Isomorphisms have tensorial base change.
\item Morphisms with tensorial base change are closed under composition.
\item Morphisms with tensorial base change are closed under base change with respect to morphisms with tensorial base change.
\end{enumerate}
The same holds for the qc qs variant.
\end{lemma}

The proof is just formal, therefore we omit it.
 
\begin{prop}[Closure under coproducts]
If $\{X_i \to S\}$ is a family of morphisms with tensorial base change, then $\coprod_i X_i \to S$ also has tensorial base change.
\end{prop}

\begin{proof}
This follows readily from \autoref{prodUE}.
\end{proof}

\begin{thm}
Affine morphisms have tensorial base change.
\end{thm}

\begin{proof}
This follows from \autoref{mod-BW} and \autoref{aff-global}.
\end{proof}

\begin{thm} \label{immtensbw}
Quasi-compact immersions between qc qs schemes have qc qs tensorial base change.
\end{thm}

\begin{proof}
Closed immersions are affine and therefore have tensorial base change. Thus it suffices to prove that every quasi-compact open immersion $i : U \hookrightarrow S$ has qc qs tensorial base change, where $S$ is a qc qs scheme. Let $f : X \to S$ be a qc qs morphism. Then the pullback $j : f^{-1}(U) = U \times_S X  \hookrightarrow X$ is a quasi-compact open immersion, too.

Using \autoref{open-immersion}, we have equivalences of categories
\begin{eqnarray*}
 & & \Hom_{c\otimes}(\Q(U),\C) \times_{\Hom_{c\otimes}(\Q(S),\C)} \Hom_{c\otimes}(\Q(X),\C) \\
 & \simeq & \{F \in \Hom_{c\otimes}(\Q(X),\C) : F f^* \text{ is $i$-local}\}\\
 & \stackrel{!}{=}  & \{F \in \Hom_{c\otimes}(\Q(X),\C) : F \text{ is $j$-local}\}\\
 & \simeq & \Hom_{c\otimes}(f^{-1}(U),\C) \simeq \Hom_{c\otimes}(\Q(U \times_S X),\C).
\end{eqnarray*}
We are left to prove the following: If $F : \Q(X) \to \C$ is a cocontinuous tensor functor, then $F f^*$ is $i$-local if and only if $F$ is $j$-local.

The direction ``$\Leftarrow$'' is trivial. For ``$\Rightarrow$'', assume that $F f^*$ is $i$-local. Let $I \subseteq \O_X$ be some quasi-coherent ideal such that $I|_{f^{-1}(U)} = \O_{f^{-1}(U)}$. It suffices to prove that $F$ maps $I \to \O_X$ to a split epimorphism (see \autoref{weak} and \autoref{split-epi}).
 
Since $f$ is qc qs, $f_* \O_X$ and $f_* I$ are quasi-coherent, hence the ideal
\[J := (f^\# : \O_S \to f_* \O_X)^{-1}(f_* I) \subseteq \O_S\]
is also quasi-coherent. Observe that $1 \in \Gamma(U,J)$ since $1 \in \Gamma(f^{-1}(U),I)$. Hence, by assumption, $F f^*$ maps $J \to \O_S$ to an isomorphism. The canonical homomorphism $J \to f_* I$ corresponds to a homomorphism $f^* J \to I$, which commutes with the morphisms to $f^* \O_S \cong \O_X$. Hence, $F$ maps $I \to \O_X$ to a split epimorphism.
\end{proof}

\begin{defi}
Let us call a morphism of schemes $X \to S$ \emph{projective} if it is isomorphic to $\Proj_S(A) \to S$ for some $\N$-graded quasi-coherent algebra $A$ on $S$ such that $A$ is generated by $A_1$ and $A_1 \in \Q(S)$ is of finite presentation.
\end{defi}

Notice that in \cite[Definition 5.5.2]{EGAII} it is only required that $A_1$ is of finite type, which is somewhat adapted to the special case that $S$ is locally noetherian. We suggest that quasi-coherent modules of finite presentation are better suited for the general case.

\begin{thm} \label{projtensbw}
Projective morphisms have tensorial base change.
\end{thm}

\begin{proof}
This follows from \autoref{proj-BW} and \autoref{proj-compare}.
\end{proof}
 
\begin{thm} \label{quasiproj}
Let $S$ be a qc qs base scheme. Let $f : X \to S$ be a quasi-projective morphism and $g : Y \to S$ be any qc qs morphism. Then
\[\xymatrix{\Q(S) \ar[r]^{g^*} \ar[d]^{f^*} & \Q(Y) \ar[d]^{\pr_Y^*} \\ \Q(X) \ar[r]_-{\pr_X^*} & \Q(X \times_S Y)}\]
is a $2$-pushout in $\Cat_{c\otimes}$.
\end{thm}

\begin{proof}
This follows from \autoref{immtensbw} and \autoref{projtensbw}.
\end{proof}

Notice that the proof of this result relies heavily on the universal properties of projective tensor categories.


\subsection{The case of qc qs schemes over a field}

Fix a commutative ring $R$ and two qc qs $R$-schemes $X,Y$. We would like to know if $X$ has tensorial base change with respect to $Y$, i.e. that for every cocomplete $R$-linear tensor category $\C$ the canonical functor
\[\Hom_{c\otimes/R}(\Q(X \times_R Y),\C) \to \Hom_{c\otimes/R}(\Q(X),\C) \times \Hom_{c\otimes/R}(\Q(Y),\C)\]
is an equivalence of categories. We then call $(X,Y)$ a \emph{product pair}. Note that although we have proven this when $X$ is quasi-projective, our proof does not give an explicit description of the inverse functor. The same remark applies to Sch\"appi's proof in the case of Adams stacks.

In order to remedy this, let us recall the definition of the \emph{external tensor product}: For $A \in \Q(X)$ and $B \in \Q(Y)$ we let
\[A \boxtimes B := \pr^*_X A \otimes \pr^*_Y B \in \Q(X \times_R Y).\]
Since $(A,B) \mapsto A \boxtimes B$ is $R$-bilinear, it induces an $R$-linear functor (see \autoref{colimits})
\[\boxtimes : \Q(X) \otimes_R \Q(Y) \to \Q(X \times_R Y),~ (A,B) \mapsto A \boxtimes B.\]
 
\begin{prop} \label{hombox}
Assume that $R$ is a field. Let $A,A' \in \Q(X)$ and $B,B' \in \Q(Y)$ and assume that $A'$ and $B'$ are of finite presentation. Then the canonical homomorphism
\[\Hom_{\O_X}(A',A) \otimes_R \Hom_{\O_Y}(B',B) \to \Hom_{\O_{X \times Y}}(A' \boxtimes B',A \boxtimes B)\]
is an isomorphism. In particular, the restriction of the external tensor product
\[\boxtimes : \Q_{\fp}(X) \otimes_R \Q_{\fp}(Y) \to \Q(X \times_R Y)\]
is fully faithful.
\end{prop}
 
Actually we only need that $R$ is absolutely flat, or just that all the involved $R$-modules are flat. The proof is standard, but we include it due to the lack of a reference. In a special case it appears as \cite[Example 1.2.3]{GK11} without proof.
 
\begin{proof}
First let $X$, $Y$ be affine. Observe that the class of all $A'$ (similarly $B'$) satisfying the claim is closed under finite direct sums, but also under cokernels, using that $\Hom_{\O_X}(-,A)$ and $\Hom_{\O_{X \times Y}}(-,A \boxtimes B)$ are left exact, $\otimes_R$ is exact and that $- \boxtimes B'$ is right exact. Besides the claim is clear for $(A',B')=(\O_X,\O_Y)$. Thus, it also holds when $A'$ and $B'$ are of finite presentation.

Now let $X$ be affine and $Y$ be a quasi-compact \emph{separated} scheme. Choose a finite open affine cover $\{Y_i \to Y\}$. Then we have an exact sequence
\[0 \to \Hom_{\O_Y}(B',B) \to \bigoplus_i \Hom_{\O_{Y_i}}(B'|_{Y_i},B|_{Y_i}) \]
\[\to \bigoplus_{i,j} \Hom_{\O_{Y_i \cap Y_j}}(B'|_{Y_i \cap Y_j},B|_{Y_i \cap Y_j}).\] 
Since $R$ is a field, it stays exact after tensoring with $\Hom_{\O_X}(A',A)$ over $R$. Since $X$, the $Y_i$ as well as their intersections $Y_i \cap Y_j$ are affine and we have already dealt with the affine case, the resulting sequence is isomorphic to
\[0 \to \Hom_{\O_X}(A',A) \otimes_R \Hom_{\O_Y}(B',B) \to \bigoplus_i \Hom_{\O_{X \times Y_i}}(A' \boxtimes B'|_{Y_i},A \boxtimes B|_{Y_i})~~~~\]
\[ \hspace{50mm} \to \bigoplus_{i,j} \Hom_{\O_{X \times Y_i \cap Y_j}}(A' \boxtimes B'|_{Y_i \cap Y_j},A \boxtimes B|_{Y_i \cap Y_j}).\]
Since $\{X \times Y_i \to X \times Y\}$ is an open cover, $\Hom_{\O_{X \times Y}}(A' \boxtimes B',A \boxtimes B)$ is also the kernel of the last homomorphism, which establishes the isomorphism.

If $X$ is affine and $Y$ is quasi-compact and quasi-separated, we can choose a finite open affine cover $\{Y_i \to Y\}$ and use that the $Y_i \cap Y_j$ are separated and also quasi-compact since $Y$ is quasi-separated to deduce as above the isomorphism from the case above. The same technique then works when we generalize to the case that $X$ is separated. Finally we get the case that $X$ is an arbitrary quasi-compact quasi-separated scheme.
\end{proof}

\begin{rem} \label{adjoint}
Let $A \in \Q(X)$, $B \in \Q(Y)$ and $M \in \Q(X \times Y)$. By the usual adjunctions a homomorphism $A \boxtimes B \to M$ on $X \times Y$ corresponds to a homomorphism
\[A \to (\pr_X)_* \underline{\Hom}_{\O_{X \times Y}}((\pr_Y)^* B,M)\]
on $X$. In general the latter is only a sheaf of modules on $X \times Y$. But it is quasi-coherent when $B$ is of finite presentation (using \cite[Proposition 6.7.1]{EGAI} and \autoref{qchom}).
\end{rem}

In the following, we assume that $R$ is a field and that $X,Y$ are qc qs $R$-schemes. We have just seen that $\Q_{\fp}(X) \otimes_R \Q_{\fp}(Y)$ can be seen as a full subcategory of $\Q(X \times_R Y)$. Clearly it is essentially small. Next, we will see that it is a dense subcategory (\autoref{density}).
 
\begin{lemma} \label{bifort}
Let $U \subseteq X$ and $V \subseteq Y$ be quasi-compact open subschemes and let $A' \boxtimes B' \to M|_{U \times V}$ be a homomorphism such that  $A' \in \Q_{\fp}(U)$ and $B' \in \Q_{\fp}(V)$. Then there are $A \in \Q_{\fp}(X)$ and $B \in \Q_{\fp}(Y)$ which extend $A'$ and $B'$ such that the homomorphism extends to a homomorphism $A \boxtimes B \to M$ on $X \times Y$.
\end{lemma}

\begin{proof}
We may assume $V=Y$ and therefore $B'=B$, since by symmetry this also gives the case $U=X$ and then we just apply the two cases in succession. By \autoref{adjoint} the given homomorphism corresponds to a homomorphism $A' \to (\pr_U)_* \underline{\Hom}_{\O_{U \times Y}}(\pr_Y^* B,M|_{U \times Y})$ on $U$. The latter is the restriction of the quasi-coherent module $(\pr_X)_* \underline{\Hom}_{\O_{X \times Y}}(\pr_Y^* B,M)$ from $X$ to $U$. By \cite[Lemme 6.9.10.1]{EGAI} there is some $A \in \Q_{\fp}(X)$ extending $A'$ and an extension of the homomorphism on $U$ to a homomorphism $A \to (\pr_X)_* \underline{\Hom}_{\O_{X \times Y}}(\pr_Y^* B,M)$ on $X$. But again this corresponds to a homomorphism $A \boxtimes B \to M$ on $X \times Y$.
\end{proof}

\begin{prop} \label{dense}
The external tensor product
\[\boxtimes : \Q_{\fp}(X) \otimes_R \Q_{\fp}(Y) \to \Q(X \times_R Y)\]
is dense, both in the $R$-linear and the usual sense. In other words, for every quasi-coherent module $M$ on $X \times Y$, we have
\begin{eqnarray*}
M & \cong &  \colim_{A \boxtimes B \to M} (A \boxtimes B) \\&& \\
&  \cong & \int^{(A,B) \in \Q_{\fp}(X) \otimes_R \Q_{\fp}(Y)} \Hom(A \boxtimes B,M) \otimes_R (A \boxtimes B).
\end{eqnarray*}
\end{prop}

\begin{proof}
We apply \autoref{qdense} have to check two conditions. For the first condition, let $\sigma : A \boxtimes B \to M$ and $\sigma' : A' \boxtimes B' \to M$ be two morphisms. Define
\[\tau : (A \oplus A') \boxtimes (B \oplus B') \cong A \boxtimes B \oplus A \boxtimes B' \oplus A' \boxtimes B \oplus A' \boxtimes B' \longrightarrow M\]
by $\tau|_{A \boxtimes B} = \sigma$, $\tau|_{A' \boxtimes B'} = \sigma'$, $\tau|_{A \boxtimes B'} = 0$ and $\tau|_{A' \boxtimes B} = 0$. Then we have a commutative diagram:
\[\xymatrix{ & A \boxtimes B \ar[dl] \ar[dr] & \\ (A \oplus A') \boxtimes (B \oplus B') \ar[rr] && M \\ & A' \boxtimes B' \ar[ul] \ar[ur] & }\]
For the second condition, we choose open affine coverings $X = \bigcup_i U_i$, $Y = \bigcup_j V_j$. Then $X \times Y = \bigcup_{i,j} U_i \times V_j$ is an open affine covering. Consider a local section $s \in \Gamma(U_i \times V_j,M)$. It induces a homomorphism $\O_{U_i} \boxtimes \O_{V_j} \to M|_{U_i \times V_j}$ via $1 \boxtimes 1 \mapsto s$. By \autoref{bifort} it extends to a homomorphism $A \boxtimes B \to M$. Since $s$ has a preimage, we are done.
\end{proof}

\begin{cor}
The canonical functor
\[\Hom_{c\otimes/R}(\Q(X \times_R Y),\C) \to \Hom_{c\otimes/R}(\Q(X),\C) \times \Hom_{c\otimes/R}(\Q(Y),\C)\]
is fully faithful.
\end{cor}
 
It remains to prove essentially surjectivity, i.e. that for all $F : \Q(X) \to \C$ and $G : \Q(Y) \to \C$ there is some $H : \Q(X \times_R Y) \to \C$ such that $H \pr_X^* \cong F$ and $H \pr_Y^* \cong G$.

\subsubsection*{Construction of $H$}

Given two $R$-linear cocontinuous tensor functors
\[F : \Q(X) \to \C,~ G : \Q(Y) \to \C,\]
\autoref{dense} tells us how the corresponding functor
\[H : \Q(X \times_R Y) \to \C\]
has to look like: We first get an $R$-linear tensor functor (not cocontinuous)
\[F \otimes_R G : \Q_{\fp}(X) \otimes_R \Q_{\fp}(Y) \to \C,~ A \boxtimes B \mapsto F(A) \otimes G(B)\]
and then define the $R$-linear functor $H : \Q(X \times_R Y) \to \C$ to be the left Kan extension of $F \otimes_R G$ along $\Q(X) \otimes_R \Q(Y) \hookrightarrow \Q(X \times_R Y)$. In other words:
\[H(M) := \int^{A \in \Q_{\fp}(X), B \in \Q_{\fp}(Y)} \Hom(A \boxtimes B,M) \otimes_R (F(A) \otimes G(B)).\]
Actually, by \autoref{denselin}, this simplifies to
\[H(M) = \colim_{A \boxtimes B \to M} (F(A) \otimes G(B)).\]
But this description somehow hides the linearity of $H$. In any case, it is not clear at all why $H$ should be a cocontinuous tensor functor! Though, some properties are easy to check: Since $H$ is $R$-linear, $H$ preserves finite direct sums. It is also clear that $H$ preserves directed colimits, essentially because $A \boxtimes B$ is of finite presentation on $X \times_R Y$ when $A$ resp. $B$ are of finite presentation on $X$ resp. $Y$. Hence, $H$ preserves arbitrary direct sums.

For $A' \in \Q_{\fp}(X)$ and $B' \in \Q_{\fp}(Y)$ we have
\begin{eqnarray*}
&& H(A' \boxtimes B')\\
 & = & \int\limits^{A,B} (\Hom(A,A') \otimes_R \Hom(B,B')) \otimes_R (F(A) \otimes G(B)) \\
& \cong & \int\limits^{A} \Hom(A,A') \otimes_R F(A)  ~ \otimes ~ \int\limits^{B} \Hom(B,B') \otimes_R G(B) \\
& \stackrel{\text{Yoneda}}{\cong} & F(A') \otimes G(B').
\end{eqnarray*}
Hence, $H$ extends $F \otimes_R G$. In particular, $H \pr_X^* \cong F$ and $H \pr_Y^* \cong G$ (just as functors). Next, we can endow $H$ with the structure of a lax monoidal functor: From the isomorphisms above we get $H(1) \cong 1$. If $M,\, M' \in \Q(X \times_R Y)$, then $H(M) \otimes H(M')$ is isomorphic to
\[\int\limits^{A,A',B,B'} (\Hom(A \boxtimes B,M) \otimes_R \Hom(A' \boxtimes B',M')) \otimes_R (F(A \otimes A') \otimes G(B \otimes B')).\]
Using $A'' = A \otimes A'$ and $B'' = B \otimes B'$, this admits a canonical morphism to
\[\int\limits^{A'',B''} \Hom(A'' \boxtimes B'',M \otimes M') \otimes_R (F(A'') \otimes G(B'')) = H(M \otimes M').\]
If we already knew that $H$ preserved cokernels and therefore was  cocontinuous, this morphism would be an isomorphism, because then (using Proposition \ref{dense}), it may be reduced to the case $M = A \boxtimes B$ and $M' = A' \boxtimes B'$ and we already know
\[H((A \boxtimes B) \otimes (A' \boxtimes B')) \cong  H((A \otimes A') \boxtimes (B \otimes B')) \cong  F(A \otimes A') \otimes G(B \otimes B')\]
\[\cong  (F(A) \otimes G(B)) \otimes (F(A') \otimes G(B')) \cong H(A \boxtimes B) \otimes H(A' \boxtimes B').\]
Thus, if $H$ preserved cokernels, then $H$ would be a cocontinuous tensor functor and $(F,G) \mapsto H$ would provide a functor which is inverse to the canonical one.

But showing that the functor $H$ preserves cokernels seems to be intractable. Therefore, we will use a more ad hoc definition below. It already seems to be hard to prove that $H$ preserves epimorphisms. In fact, already the special case of epimorphisms between external tensor products $A' \boxtimes B' \twoheadrightarrow A \boxtimes B$ remains unclear to the author. Although we know how to write down all these morphisms, we need a reformulation of the epi condition which treats $\Q(X)$ and $\Q(Y)$ separately in order to be translated by $F$ and $G$. At least the following special case is easy to deal with: A \emph{diagonal} homomorphism $\alpha \boxtimes \beta : A' \boxtimes B' \to A \boxtimes B$ is an epimorphism if and only if $A=0$ or $B=0$ or $A,B \neq 0$ and $\alpha,\beta$ are epimorphisms. It is then clear that $F(A') \otimes G(B') \to F(A) \otimes G(B)$ is also an epimorphism.

\subsubsection{Concrete construction of $H$}

For a qc qs scheme $Z$ we have
\[\Hom_{c\otimes}(\Q(Z),\C) \simeq \Hom_{fc\otimes}(\Q_{\fp}(Z),\C).\]
by \autoref{qcohind}. Hence, it suffices to construct a right exact $R$-linear tensor functor $H : \Q_{\fp}(X \times Y) \to \C$.
  
\begin{lemma} \label{resol}
Let $M \in \Q(X \times Y)$. Then there are $A \in \Q(X)$ and $B \in \Q(Y)$ with an epimorphism $A \boxtimes B \twoheadrightarrow M$. If $M$ is of finite type, then we can arrange $A \in \Q_{\fp}(X)$ and $B \in \Q_{\fp}(Y)$. If $M$ is of finite presentation, there is an exact sequence
\[A' \boxtimes B' \to A \boxtimes B \to M \to 0\]
with $A,A' \in \Q_{\fp}(X)$ and $B,B' \in \Q_{\fp}(Y)$.
\end{lemma}

\begin{proof}
By \autoref{dense} there is an epimorphism $\bigoplus_i (A_i \boxtimes B_i) \twoheadrightarrow M$ for certain $A_i \in \Q_{\fp}(X)$ and $B_i \in \Q_{\fp}(Y)$. If $M$ is of finite type, a finite direct sum already suffices. Now take $A = \bigoplus_i A_i$ and $B = \bigoplus_i B_i$. Their external tensor product $A \boxtimes B = \bigoplus_{i,j} A_i \boxtimes B_j$ projects onto $\bigoplus_i A_i \boxtimes B_i$. This proves the first and the second claim. For the third, apply the second to the kernel of $A \boxtimes B \twoheadrightarrow M$.
\end{proof}

Such an exact sequence will be called a \emph{$\boxtimes$-presentation} of $M$. We will fix one for every $M$.
 
\begin{defi}
Let $F : \Q(X) \to \C$ and $G : \Q(Y) \to \C$ be cocontinuous tensor functors. Given an object $M \in \Q_{\fp}(X \times Y)$ with its $\boxtimes$-presentation
\[A' \boxtimes B' \xrightarrow{\sigma} A \boxtimes B \to M \to 0,\]
we define $H(M) \in \C$ to be the cokernel of $(F \otimes_R G)(\sigma)$.

Therefore, if $\sigma = \sum_i (\alpha_i \boxtimes \beta_i)$ with $\alpha_i : A' \to A$ and $\beta_i : B' \to B$, this induces the morphism
\[(F \otimes_R G)(\sigma) = \sum_i F(\alpha_i) \otimes G(\beta_i) : F(A') \otimes G(B') \to F(A) \otimes G(B),\]
and there is an exact sequence in $\C$
\[F(A') \otimes G(B') \to F(A) \otimes G(B) \to H(M) \to 0.\]
\end{defi}

The advantage of this definition is that we can compute $H(M)$ once we have a $\boxtimes$-presentation of $M$. Of course we have to prove that $H(M)$ does not depend on the $\boxtimes$-presentation on $M$, that $H(-)$ is a functor, which is linear, right exact and besides carries the structure of a tensor functor. Unfortunately, this remains unclear to the author. At least, we have the following result:

\begin{thm} \label{prod1}
Let $X,Y$ be qc qs schemes over a field $R$ such that $X \times_R Y$ is noetherian. If $\C$ is a cocomplete $R$-linear tensor category, then
\[\Hom_{c\otimes/R}(\Q(X \times_R Y),\C) \to \Hom_{c\otimes/R}(\Q(X),\C) \times \Hom_{c\otimes/R}(\Q(Y),\C)\]
is fully faithful and the essential image consists of those pairs $(F,G)$ with the following property $(\star)$: Given an exact sequence
\[A' \boxtimes B' \to A \boxtimes B \to C \boxtimes D \to 0,\]
with $A',A,C \in \Q_{\fp}(X)$ and $B',B,D \in \Q_{\fp}(Y)$, then
\[F(A') \otimes G(B') \to F(A) \otimes G(B) \to F(C) \otimes G(D) \to 0\]
also is exact.
\end{thm}

\begin{proof}
We may apply \autoref{extend} to the full subcategory
\[\boxtimes : \Q_\fp(X) \otimes_R \Q_\fp(Y) \hookrightarrow \Q_\fp(X \times_R Y).\]
The assumptions are satisfied because of \autoref{resol} and the proof of \autoref{dense}.
\end{proof}

Notice that we only need $X \times_R Y$ noetherian so that $\Q_{\fp}(X \times_R Y)$ is abelian (which might be superfluous using Sch\"appi's notion of ind-abelian categories, see \cite{Sch12b}). In the setting considered by Sch\"appi, we can prove $(\star)$ directly:

\begin{thm} \label{prod2}
With the above notations, assume that $X,Y$ have the strong resolution property. If $\C$ has the property that directed colimits are exact and equalizers exists (for example when $\C$ is locally finitely presentable), then $(\star)$ is satisfied for all $F : \Q(X) \to \C$ and $G : \Q(Y) \to \C$. Hence\marginpar{Therefore->Hence}, $\Q(X \times_R Y)$ is the $2$-coproduct of $\Q(X)$ and $\Q(Y)$ in the $2$-category of locally finitely presentable $R$-linear tensor categories. \marginpar{coproduct -> 2-coproduct.}
\end{thm}

\begin{proof}
Choose presentations $p : P \to X$ and $q : Q \to Y$ (so that $P,Q$ are affine and $p,q$ are faithfully flat and affine). Then $p \times q : P \times Q \to X \times Y$ is a presentation. Let $A = p_* \O_P$ and $B = q_* \O_Q$. Then $A':=F(A)$ is a descent algebra by \autoref{adamsspec} and \autoref{specdes}. The same holds for $B':=G(B)$ and then for $A' \otimes B'$, too. The rest is formal: $F$ induces a cocontinuous tensor functor
\[\Q(P) \cong \M(A) \to \M(A') \to \M(A' \otimes B'). \]
Similarly, $G$ induces $\Q(Q) \to \M(A' \otimes B')$. Since $P,Q$ are affine, these correspond to a cocontinuous tensor functor $\Q(P \times Q) \to \M(A' \otimes B')$ (\autoref{mod-BW}). We claim that the composition
\[\Q(X \times Y) \to \Q(P \times Q) \to \M(A' \otimes B')\]
actually factors through the category of descent data. Since the external tensor product $\Q(X) \otimes \Q(Y) \to \Q(X \times Y)$ is fully faithful and dense, it suffices to do this for $\Q(X) \otimes \Q(Y) \to \M(A' \otimes B')$. But this even factors as
\[\Q(X) \otimes \Q(Y) \xrightarrow{F \otimes G} \C \to \M(A' \otimes B').\]

We conclude that $\Q(X \times Y) \to \M(A' \otimes B')$ lifts to a cocontinuous tensor functor $H : \Q(X \times Y) \to \C$. We claim that it extends $F \otimes G$, i.e. that there are natural isomorphisms $H(C \boxtimes D) \cong F(C) \otimes G(D)$. It suffices to prove this after tensoring with $A' \otimes B'$ (and verifying a compatibility condition), but then both sides factor through $\Q(P \times Q)$ and the isomorphism holds by construction. In particular, $H \pr_X^* \cong F$ and $H \pr_Y^* \cong G$, so that we are done by \autoref{prod1}.
\end{proof}
\section{Tangent tensor categories} \label{tangentt}

Let $S$ be a base scheme and let $X$ be an $S$-scheme. We abbreviate
\[X[\e] := X \times_{\Z} \Z[\e]/(\e^2)\]
and consider it as an $S$-scheme (not as an $S[\e]$-scheme). There is a closed immersion $X \hookrightarrow X[\e]$, whose corresponding ideal sheaf is nilpotent. Hence, $X[\e]$ has the same underlying topological space as $X$, only the structure sheaf differs and is given by $\O_X \oplus \O_X \cdot \e$. We obtain a \emph{thickening functor} $\Sch/S \to \Sch/S$, $X \mapsto X[\e]$.
    
Recall (\cite[16.5.12]{EGAIV}) that the \emph{tangent bundle} of an $S$-scheme $X$ is defined by
\[T(X/S) := \Spec \Sym \Omega^1_{X/S}.\]
This may not be locally trivial. Since $\Sym \Omega^1_{X/S}$ is a quasi-coherent $\O_X$-algebra, $T(X/S)$ is an affine $X$-scheme. But let us view it as an $S$-scheme. We obtain the \emph{tangent bundle functor} $T : \Sch/S \to \Sch/S$, which enjoys the following concise functorial characterization:
  
\begin{lemma} \label{tangent-adj}
The tangent bundle functor $T : \Sch/S \to \Sch/S$ is right adjoint to the thickening functor. Thus, for every two $S$-schemes $X,Y$ there is a canonical bijection
\[\Hom_S(Y,T(X/S)) \cong \Hom_S(Y[\e],X).\]
\end{lemma}

\begin{proof}
By the universal property of the relative spectrum (\cite[Definition 9.1.8]{EGAI}), an $S$-morphism $Y \to T(X/S)$ may be identified with an $S$-morphism $f : Y \to X$ together with a homomorphism of algebras $\Sym \Omega^1_{X/S} \to f_* \O_Y$. By the universal properties of the symmetric algebra and the module of differentials, the latter corresponds to a derivation $d : \O_X \to f_* \O_Y$ lying over $f^\#$. This corresponds to a homomorphism of algebras $\O_X \to f_* \O_Y \oplus f_* \O_Y \cdot \e$, defined by $s \mapsto f^\#(s) \oplus d(s) \e$, which in turn is a lift of $f$ along $Y \hookrightarrow Y[\e]$ to a morphism $Y[\e] \to X$.
\end{proof}

Using this adjunction, we will globalize tangent bundles. First, note that the thickening functor globalizes: Fix a cocomplete$R$-linear  tensor category $\C$. Applying \autoref{ringchange} to the $R$-algebra $R[\e] := R[\e]/\e^2$, we get an  cocomplete $R$-linear tensor category $\C[\e]$ whose objects are pairs $(M,\e_M)$, where $M \in \C$ is an object and $\e_M : M \to M$ is an endomorphism satisfying $\e_M^2=0$. Usually we abbreviate $\e := \e_M$. A morphism $(M,\e) \to (N,\e)$ is a morphism $M \to N$ which commutes with $\e$; we will also say $\e$-linear for this property.
 
There is an $R$-linear cocontinuous tensor functor $\C \to \C[\e]$ which is given by mapping $M \mapsto M[\e] := M \oplus M \cdot \e := (M \oplus M,\e)$, where $\e$ acts by the following matrix:
\[\begin{pmatrix} 0 & 0 \\ 1 & 0 \end{pmatrix}\]
It is left adjoint to the forgetful functor which we will occasionally denote by $T \mapsto T|_\C$, but often suppress from the notation. The section $R[\e] \to R,\, \e \mapsto 0$ induces an $R$-linear cocontinuous tensor functor $\C[\e] \to \C$, which maps $(M,\e)$ to $M/\e M := \coker(\e_M)$. It should not be confused with the forgetful functor, which is not a tensor functor at all. There is still another functor $\C \to \C[\e]$ which endows an object $M$ of $\C$ with the trivial action $\e_M := 0$. In order to differentiate the tensor product from $\C[\e]$ from the one in $\C$, we denote it by $- \otimes_{\e} -$. We will usually abbreviate $(M,\e) \in \C[\e]$ by $M$.
 
In the following, we will often omit $R$ from the notation and every tensor category and tensor functor is understood to be $R$-linear.
  
\begin{defi}[Tangent tensor categories] \label{tangdef}
Let $F : \C \to \D$ be a cocontinuous linear tensor functor between cocomplete linear tensor categories. The \textit{tangent tensor category} of $F$ is a cocomplete linear tensor category $T(F)=T(\D/\C)$ over $\C$ i.e. together with a cocontinuous linear tensor functor $\C \to T(F)$ and a natural equivalence of categories
 \[\Hom_{c\otimes/\C}(T(\D/\C),\E) \simeq \Hom_{c\otimes/\C}(\D,\E[\e])\] 
for every cocontinuous linear tensor functor $\C \to \E$. In other words, the tangent tensor category of $\C \to \D$ is a representation of the $2$-functor of \textit{thickenings} or \emph{deformations}
\[\Cat_{c\otimes/\C} \to \Cat,~ \E \mapsto \Hom_{c\otimes/\C}(\D,\E[\e]).\]
\end{defi}

\begin{rem}[Ansatz] \label{ansatz}
We do not know if the tangent tensor category exists in general. Looking at the scheme case, it is tempting to define an object $\Omega^1_F$ of $\D$ associated to $F$ and just define $T(F) := \M(\Sym \Omega^1_F)$. Then by \autoref{mod-UE} we have an equivalence of categories
\[\Hom_{c\otimes/\C}(T(F),\E) \simeq \{G \in \Hom_{c\otimes}(\D,\E),\, GF \cong H,\, G(\Omega^1_F) \to \O_\E\}\]
for every $H : \C \to \E$. But we do not know of any reasonable definition of $\Omega^1_F$ (see also \autoref{defo}). At least this approach works in some special cases, as we will see later.
\end{rem}

\begin{ex}
In the example $\D=\gr_\Z(\C)$ (see \autoref{grZ}) the speculative object $\Omega^1 \in \D$ (\autoref{ansatz}) would have to satisfy, for every line object $\L$ in $\E$, that $\Hom_{\E}(G_\L(\Omega^1),\O_\E) \simeq \{\text{line objects } \K \in \E[\e] \text{ with } \K/\e\K \cong \L\}$. Here, we have $G_\L(M) = \bigoplus_{n \in \Z} M_n \otimes \L^{\otimes n}$. Such an object $\Omega^1$ cannot exist (when $\C \neq 0$) since the category on the right is not essentially discrete. It would be interesting if $T(\gr_\Z(\C)/\C)$ nevertheless exists.
\end{ex}

\begin{rem}[Reformulation] \label{refo}
Assume that $T(\D/\C)$ exists. Then it comes equipped with a cocontinuous linear tensor functor (everything over $\C$)
\[U : \D \to T(\D/\C)[\e].\]
It has the universal property that for every $\D \to \E[\e]$ there is essentially a unique $T(\D/\C) \to \E$ such that the obvious diagram commutes up to unique isomorphism.

We may mod out $\e$ after $U$ to obtain a cocontinuous linear tensor functor
\[P : \D \to T(\D/\C).\]
This should be imagined as the ``tangent bundle projection''. The universal property can therefore also be formulated in $\Cat_{c\otimes/\D}$. Namely, require for every $G : \D \to \E$ that $U$ induces an equivalence of categories
\[\Hom_{c\otimes/\D}(T(\D/\C),\E) \simeq \{F \in \Hom_{c\otimes/\C}(\D,\E[\e]), F/\e \cong G\}.\]
\end{rem}

\begin{rem}
The tangent tensor category is stable under base change in the sense that $T(\D/\C) \sqcup_\C \C' \simeq T(\D \sqcup_\C \C' /\C')$ if these tangent tensor categories and $2$-pushouts exist.
\end{rem}

\begin{prop}[Affine case] \label{tangmod}
Let $A$ be a commutative algebra in a cocomplete linear tensor category $\C$. Then $T(\M(A) / \C)$ exists. It is given by $\M(\Sym_A \Omega^1_A)$, where $\Omega^1_A \in \M(A)$ is defined as in \autoref{diffdef}.
\end{prop}

\begin{proof}
Let $H : \C \to \E$ be a cocontinuous linear tensor functor. Then two applications of \autoref{mod-UE} yield an equivalence of categories
\[\Hom_{c\otimes/ \C}(\M(\Sym_A \Omega^1_A),\E) \simeq \Hom_{\CAlg(\E)}(H(A),\O_\E) \times \Hom_\E(H(\Omega^1_A),\O_\E).\]
Using \autoref{omegaH}, the right factor identifies with $\Der(H(A),\O_\E)$. By \autoref{mod-UE}  we have an equivalence of categories
\[\Hom_{c\otimes/ \C}(\M(A),\E[\e]) \simeq \Hom_{\CAlg(\E)}(H(A),\O_\E[\e]).\]
Therefore, the claim is just \autoref{eder}.
\end{proof}

\begin{ex}
For the $n$-dimensional tensorial affine space $\mathds{A}^n_\C$ with variables $T_1,\dotsc,T_n$ (see \autoref{affsp}) we have
\[T(\mathds{A}^n_\C / \C) \simeq \mathds{A}^{2n}_\C\]
with variables $T_1,\dotsc,T_n,d(T_1),\dotsc,d(T_n)$. This follows from \autoref{tangmod} and \autoref{deriv-sym}.
\end{ex}

The tangent bundle in algebraic geometry commutes with coproducts. We globalize this as follows:

\begin{prop} \label{prod-tangent}
Let $(\C \to \D_i)_{i \in I}$ be a family of cocontinuous linear tensor functors. If $T(\D_i / \C)$ exists for every $i$, then $T(\prod_i \D_i / \C)$ also exists and is given by
\[T\bigl(\prod_i \D_i / \C\bigr) \simeq \prod_i T(\D_i / \C)\]
\end{prop}

\begin{proof}
Let $\C \to \E$ be a cocontinuous linear tensor functor. We will omit $\C$ from the notation. The category $\Hom(\prod_i T(\D_i),\E)$ is equivalent to the category of o.i.d. $(e_i)_{i \in I}$ in $\E$ together with a family of cocontinuous linear tensor functors $T(\D_i) \to \E_{e_i}$ (\autoref{prodUE}). By definition, each $T(\D_i) \to \E_{e_i}$ corresponds to some $\D_i \to \E_{e_i}[\e]$. The cocontinuous linear tensor functor $\E \to \E[\e]$ yields the o.i.d. $(e_i[\e])_{i \in I}$ of $\E[\e]$ and one easily checks $\E_{e_i}[\e] = \E[\e]_{e_i[\e]}$. Applying \autoref{prodUE} once again, we arrive at a cocontinuous linear tensor functor $\prod_i \D_i \to \E[\e]$. Thus, we have only to show that $\E$ and $\E[\e]$ have the same o.i.d.

So assume that $(f_i)_{i \in I}$ is an o.i.d. of $\E[\e]$. Then each $f_i$ is an idempotent in the ring $\End(\O_{\E[\e]}) \cong \End(\O_\E)[\e]$, say $f_i = e_i + g_i \e$ with $e_i,g_i \in \End(\O_\E)$. Now, $f_i^2=f_i$ means $e_i = e_i^2$ and $2 e_i g_i = g_i$. It follows $2 e_i g_i = 2 e_i^2 g_i = e_i g_i$, thus $e_i g_i=0$, which implies $g_i = 2 e_i g_i = 0$. This means $f_i = e_i [\e]$ as morphisms $\O_\E[\e] \to \O_\E[\e]$. Taking cokernels in $\bigoplus_i \coker(f_i) \cong \O_\E[\e]$ with respect to $\e$, we get $\bigoplus_i \coker(e_i) \cong \O_\E$. Thus, $(e_i)$ is an o.i.d., the unique one which lifts $(f_i)$. 
\end{proof}

The tangent bundle of a closed subscheme $V(I) \subseteq X$ is given by a base change of $T(X) \to X$ to $V(I)$ and then one has to mod out the differentials of $I$ -- this is the usual conormal exact sequence (\cite[Theorem 21.2.12]{Vak13} or part 3 in \autoref{deriv-prop}). The next result is a globalization of this fact.

\begin{prop} \label{tangcl}
Let $\C \to \D$ be a cocontinuous linear tensor functor and let $I \to \O_\D$ be a morphism. If $T:=T(\D/\C)$ exists, then the same is true for $T(V^{\otimes}_\D(I)/\C)$. Explicitly, we can construct two morphisms $U(I)|_T \to \O_{T}$ and $P(I) \to \O_T$ and then realize
\[T(V^{\otimes}_\D(I)/\C) \simeq V^{\otimes}_T(U(I)|_T,P(I)).\]
\end{prop}

\begin{proof}
We will omit $\C$ from the notation. Recall the notations $U,P$ from \autoref{refo} and $V^\otimes_\D(I)$ from \autoref{closed-UE}. We have an $\e$-morphism
\[U(I) \to U(\O_\D) = \O_T[\e] = \O_T \oplus \O_T \cdot \e.\]
By modding out $\e$ we get a morphism $P(I) \to \O_T$ and the second projection gives a morphism $U(I)|_T \to \O_T$. We claim that $V^{\otimes}_T(U(I),P(I))$ satisfies the desired universal property. In fact, for every $\E$ we have
\begin{eqnarray*}
&& \Hom_{c\otimes}(V^{\otimes}_\D(I),\E[\e]) \\
& \simeq & \{F \in \Hom_{c\otimes}(\D,\E[\e]):\, F(I) \to \O_\E[\e] \text{ vanishes}\} \\
& \simeq & \{G \in \Hom_{c\otimes}(T,\E):\, G(P(I)) \to \O_\E[\e] \text{ vanishes}\} \\
& \simeq & \{G \in \Hom_{c\otimes}(T,\E):\, G(P(I)) \to \O_\E \text{ and } G(U(I)|_T) \to \O_\E \text{ vanish}\} \\
& \simeq & \Hom_{c\otimes}(V^{\otimes}_T(U(I)|_T,P(I)),\E). 
\end{eqnarray*}
\qedhere
\end{proof}

Our next goal is to prove that projective tensor categories have a tangent tensor category. For simplicity, we will restrict to projective bundles (for the general case one may use \autoref{tangcl}). We need several preparations first.

\begin{defi}[Cocycles] \label{cocycle}
Let $s : E \to \L$ be an epimorphism in a cocomplete linear tensor category $\C$, such that $\L$ is a line object. We say that a morphism $\lambda : \Lambda^2(E) \to \L^{\otimes 2}$ is a \emph{cocycle} (with respect to $s$) if we have
\[\lambda(a \wedge b) \otimes s(c) - \lambda(a \wedge c) \otimes s(b) + \lambda(b \wedge c) \otimes s(a) = 0\]
in element notation. Formally, this is of course an equality of two morphisms $\Lambda^2(E) \otimes E \rightrightarrows \L^{\otimes 3}$.

Actually we mean $\ASym^2(E)$ here instead of $\Lambda^2(E)$, which we have only defined when $2 \in R^*$ (\autoref{exterior-power}). But if we are in the case of modules, then the cocycle condition \emph{implies} that $\lambda$ is alternating and therefore lifts to a morphism on $\Lambda^2(E)$: Just set $a=b$ and use \autoref{epi-cancel}. This justifies to write $\Lambda^2(E)$ here also in the general case.

Note that we can define a morphism $\Lambda^3(E) \to \Lambda^2(E) \otimes \L$ by
\[a \wedge b \wedge c \mapsto (a \wedge b) \otimes s(c) - (a \wedge c) \otimes s(b) + (b \wedge c) \otimes s(a).\]
It corresponds to a morphism
\[\delta : \Lambda^3(E) \otimes \L^{\otimes -3} \to \Lambda^2(E) \otimes \L^{\otimes -2}.\]
Then a cocycle is nothing else than a morphism $\coker(\delta) \to \O_\C$.
\end{defi}

The following Theorem classifies line objects $\K$ in $\C[\e]$ with prescribed ``generators'' in terms of line objects $\L$ in $\C$ with prescribed ``generators'' and a cocycle which intuitively measures the obstruction that $\K$ is just the trivial extension $\L[\e]$.

\begin{thm} \label{tangent-main}
Let $\C$ be a cocomplete linear tensor category and $E \in \C$. Then there is a natural equivalence of categories
\[\begin{array}{c}
\{\L \in \C \text{ line object},~ s : E \twoheadrightarrow \L \text{ regular epimorphism},~\lambda : \Lambda^2 E \to \L^{\otimes 2} \text{ cocycle}\} \medskip \\
 |\wr \medskip  \\
\{(\K,t) : \K \in \C[\e] \text{ line object},~ t : E[\e] \to \K \text{ regular epimorphism}\}.
\end{array}\]
\end{thm}

Before proving this Theorem, let us explain how it enables us to deduce:

\begin{thm}[Projective case] \label{tangent-proj}
Let $\C$ be a cocomplete linear tensor category and $E \in \C$. Assume that the projective tensor category $\P^{\otimes}_\C(E)$ exists (for example when $\C$ is locally presentable). Then its tangent tensor category $T(\P^{\otimes}_\C(E) / \C)$ exists.

It is constructed as follows: Consider the universal triple consisting of a cocontinuous linear tensor functor $P : \C \to \P^{\otimes}_\C(E)$, a line object $\O(1)$ and a regular epimorphism $x : P(E) \to \O(1)$. Define
\[\delta : \Lambda^3(P(E)) \otimes \O(-3) \to \Lambda^2(P(E)) \otimes \O(-2)\]
as in \autoref{cocycle} and let $\Omega^1 \in \P^{\otimes}_\C(E)$ be the cokernel of $\delta$. Then
\[T(\P^{\otimes}_\C(E) / \C) \simeq \M(\Sym \Omega^1).\]
\end{thm}
 
Note that the definition of $\Omega^1$ is compatible with (and motivated by) the Euler sequence (\autoref{euler}, \autoref{Koszul}).
 
\begin{proof}
Let $H : \C \to \E$ be a cocontinuous linear tensor functor. We have to prove
\[\Hom_{c\otimes/\C}(\M(\Sym \Omega^1),\E) \simeq \Hom_{c\otimes/\C}(\P^\otimes_\C(E),\E[\e]).\]
By definition of $\P^{\otimes}(E)$ there is an equivalence between $\Hom_{c\otimes/\C}(\P^\otimes_\C(E),\E[\e])$ and the category of pairs $(\K,t)$ consisting of a line object $\K \in \E[\e]$ and a regular epimorphism $t : H(E)[\e] \to \K$. According to \autoref{tangent-main} this is equivalent to the category of triples $(\L,s,\lambda)$, where $\L \in \C$ is a line object, $s : H(E) \to \L$ is a regular epimorphism and $\lambda : \Lambda^2 H(E) \to \L^{\otimes 2}$ is a cocycle. The rest follows from \autoref{ansatz}. A bit more detailed, the pair $(\L,s)$ corresponds to a cocontinuous linear tensor functor $G : \P^\otimes_\C(E) \to \E$ lying over $H$. Then $\lambda$ corresponds to a morphism $G(\Omega^1) \to \O_\E$, hence corresponds to a lift of $G$ to a cocontinuous linear tensor functor $\M(\Sym \Omega^1) \to \E$ (\autoref{mod-UE}).
\end{proof}

For \autoref{tangent-main} we need the following Lemma.
 
\begin{lemma} \label{KL-exact}
Let $\C$ be a cocomplete linear tensor category and let $\K$ be a line object in $\C[\e]$. Then $\L := \K/\e\K$ is a line object in $\C$. If $p : \K \twoheadrightarrow \L$ is the canonical epimorphism in $\C$,
then there is a unique morphism ${i} : \L \to \K$ such that ${i} p = \e : \K \to \K$. Besides, the sequence
\[0 \to \L \xrightarrow{{i}} \K \xrightarrow{p} \L \to 0\]
is exact in $\C$. It is even exact in $\C[\e]$ when we endow $\L$ with the trivial $\e$-action.
\end{lemma}


\begin{proof}
Reduction modulo $\e$ is a  tensor functor $\C[\e] \to \C$, hence it preserves line objects. In particular, $\L$ is a line object in $\C$. The existence of ${i}$ follows from $\e^2=0$.

We define $\e=0$ on $\L$. Then ${i}$ is a morphism in $\C[\e]$ because $\e{i}p=\e \e = 0$, hence $\e {i} = 0$. Likewise, $p : \K \to \L$ is a morphism in $\C[\e]$ because $p \e = 0$. It is clear that
\[\L \xrightarrow{{i}} \K \xrightarrow{p} \L \to 0\]
is exact, since
\[\K \xrightarrow{\e} \K \xrightarrow{p} \L \to 0\]
is exact. In order to show that the sequence is exact in $\C[\e]$ and hence in $\C$, it is enough to prove that the sequence
\[0 \to \O_\C \xrightarrow{{i}} \O_\C[\e] \xrightarrow{p} \O_\C \to 0\]
is exact in $\C[\e]$, because then tensoring with $\K$ gives the general sequence. The sequence is clearly (split) exact in $\C$. Let $T \in \C[\e]$ and $h : T \to \O_\C$ be a morphism in $\C$ such that ${i} h$ is $\e$-linear, i.e. ${i} h \e = \e {i} h = 0$. Then $h \e = 0$, which means that $h$ is $\e$-linear.
\end{proof}

\begin{proof}[Proof of \autoref{tangent-main}]
Let $A$ denote the category of triples $(\L,s,\lambda)$ and $B$ be the category of pairs $(\K,t)$ as in the Theorem. Both are essentially discrete by \autoref{inv-diskret}. We will construct functors $A \to B$ and $B \to A$ and show that they are inverse to each other. We will make heavy use of element notation. We will not explicitly mention the ubiquitous applications of \autoref{epi-cancel}. As explained in \autoref{cocycle}, we will write $\Lambda^2(E)$ instead of $\ASym^2(E)$ (just for aesthetic reasons). If $t : E[\e] \to \K$, we will sometimes also write $t$ for the restriction to a morphism $E \to \K$.

\textbf{1. Step: Construction of $B \to A$.} We construct a functor $B \to A$ as follows: Given $(\K,t)$, applying the cocontinuous tensor functor $\C[\e] \to \C$ which mods out $\e$, we get a pair $(\L,s)$, where $\L = \K/\e \K$ is a line object and $s : E \to \L$ is a regular epimorphism induced by $t$. By \autoref{KL-exact} we have an exact sequence
\[0 \to \L \xrightarrow{{i}} \K \xrightarrow{p} \L \to 0,\]
both in $\C[\e]$ and in $\C$. It follows that it stays exact after tensoring with $\L$ in $\C$:
\[0 \to \L^{\otimes 2} \xrightarrow{\L \otimes {i}} \L \otimes \K \xrightarrow{\L \otimes p} \L^{\otimes 2} \to 0\]
Consider the morphism
\[\Lambda^2(E) \to \L \otimes \K,~ a \wedge b \mapsto s(a) \otimes t(b) - s(b) \otimes t(a).\]
Since $\L$ is symtrivial, the morphism vanishes when composed with $\L \otimes p$. Hence, by exactness there is a unique morphism
\[\lambda : \Lambda^2(E) \to \L^{\otimes 2}\]
which is characterized by the equation of morphisms $\Lambda^2(E) \to \L \otimes \K$
\begin{equation} \label{eq1}
{i} \cdot \lambda(a \wedge b) = s(a) \otimes t(b) - s(b) \otimes t(a)
\end{equation}
where we have abbreviated $(\L \otimes {i}) \circ -$ by ${i} \cdot$. Let us check the cocycle condition:
\[{i} \cdot \lambda(a \wedge b) \otimes s(c) - {i} \cdot \lambda(a \wedge c) \otimes s(b) + {i} \cdot \lambda(b \wedge c) \otimes s(a)\]
\[=s(a) \otimes t(b) \otimes s(c) - s(b)\otimes t(a) \otimes s(c) - s(a) \otimes t(c) \otimes s(b)\]
\[ ~~~~~~~~~~~~~+ s(c) \otimes t(a) \otimes s(b) + s(b) \otimes t(c) \otimes s(a) - s(c) \otimes t(b) \otimes s(a)=0\]
Here, we have used that $\L$ is symtrivial. This suffices since ${i}$ is a monomorphism. This defines $B \to A$ on objects. The action on morphisms (which are unique isomorphisms) is clear.

\textbf{2. Step: Proof of $\e t = {i} s$.} With the notations above, we claim that $\e t|_E = {i} s$ as morphisms $E \to \K$. Let $q : E[\e] \to E$ be the projection modulo $\e$. By construction, we have $s q = p t$. Besides, we have $\e = {i} q$. It follows
\[{i} s q = {i} p t = \e t = t \e = t {i} q = \e t|_E q.\]
Since $q$ is an epimorphism, this means ${i} s = \e t$. Applying this to \ref{eq1}, we get another characterization of the cocycle $\lambda$:
\begin{equation} \label{eq2}
\lambda(b \wedge c) \otimes \e t(a) = s(a) \otimes (s(b) \otimes t(c) - s(c) \otimes t(b))
\end{equation}
Here, $a,b,c$ run through $E$ in element notation.
 
\textbf{3. Step: Construction of $A \to B$.} We construct a functor $A \to B$ as follows: Given $(\L,s,\lambda)$, we will \emph{define} $t : E[\e] \to \K$ exactly in such a way that \ref{eq2} becomes true. For this, define the morphism $E^{\otimes 3} \to \L^{\otimes 2} \otimes E[\e]$ in $\C$ by
\[a \otimes b \otimes c \mapsto s(a) \otimes (s(b) \otimes c - s(c) \otimes b) - \lambda(b \wedge c) \otimes \e a.\]
It corresponds to a morphism
\[\phi : (\L^{\otimes -2} \otimes E^{\otimes 3})[\e] \to E[\e]\]
in $\C[\e]$. Let $t : E[\e] \to \K$ be the cokernel of $\phi$ in $\C[\e]$. Then \ref{eq2} holds by definition. Applying \autoref{goodepi} to the pair $(\L,s)$, it is easy to see that $\K/\e \cong \L$, with $s$ being induced by $t$ on the quotient. We postpone the proof that $\K$ is a line object to Steps 5 and 6. Note that the cocycle associated to $(\K,t)$ is $\lambda$ by construction. This defines on $A \to B$ on objects. The action on morphisms is clear.

\textbf{4. Step: The functors are inverse to each other.} We have already argued above that $A \to B \to A$ is isomophic to the identity. In order to show this for $B \to A \to B$, we have to show for every pair $(\K,t)$ with associated triple $(\L,s,\lambda)$ as in Step 1 that $t$ is the cokernel of $\phi$ as defined in the Step 3. \marginpar{I've replaced (K,s) by (K,t).} By construction we have $t \phi = 0$ and $t$ is an epimorphism. Now let $T \in \C[\e]$ and $E[\e] \to T$ be a morphism in $\C[\e]$, or equivalently $h : E \to T$ be a morphism in $\C$ which vanishes when composed with $\phi$. This comes down to
\begin{equation} \label{eq3}
s(a) \otimes (s(b) \otimes h(c) - s(c) \otimes h(b)) = \lambda(b \wedge c) \otimes \e h(a).
\end{equation}
We want to show that $h$ factors through $t$. By \autoref{goodepi} this means
\begin{equation} \label{eq4}
t(a) \otimes_\e h(b) = t(b) \otimes_\e h(a).
\end{equation}
This is an equation of morphisms $E^{\otimes 2} \to \K \otimes_\e T$. We calculate
\[s(a) \otimes (s(b) \otimes t(c) - s(c) \otimes t(b)) \otimes_\e h(d) \stackrel{\ref{eq2}}{=} \lambda(b \wedge c) \otimes \e t(a) \otimes_\e h(d)\]
\[=\lambda(b \wedge c) \otimes t(a) \otimes_\e \e h(d) \stackrel{\ref{eq3}}{=} s(d) \otimes (s(b) \otimes t(a) \otimes_\e h(c) - s(c) \otimes t(a) \otimes_\e h(b)).\]
It follows
\[s(d) \otimes s(c) \otimes t(a) \otimes_\e h(b) + s(a) \otimes s(b) \otimes t(c) \otimes_\e h(d)\]
\[ = s(d) \otimes s(b) \otimes t(a) \otimes_\e h(c) + s(a) \otimes s(c) \otimes t(b) \otimes_\e h(d).\]
Let us simplify the notation via $\langle a,b,c,d \rangle := s(a) \otimes s(b) \otimes t(c) \otimes_\e h(d)$. Then the equation becomes
\[\langle d,c,a,b \rangle + \langle a,b,c,d \rangle = \langle d,b,a,c \rangle + \langle a,c,b,d \rangle.\]
Since $\L$ is symtrivial, the order of $a,b$ in $\langle a,b,c,d \rangle$ does not matter. Therefore we may simplify the notation further by writing $\langle c,d \rangle := \langle a,b,c,d \rangle$ -- similarly $\langle a,b \rangle := \langle c,d,a,b \rangle$ etc. and the equation becomes
\begin{equation} \label{eq5}
\langle a,b \rangle + \langle c,d \rangle = \langle a,c \rangle + \langle b,d \rangle.
\end{equation}
Thus, in a sum of two brackets, we may interchange the two inner variables. Our claim \ref{eq4} becomes
\begin{equation} \label{eq6}
\langle a,b \rangle = \langle b,a \rangle.
\end{equation}
It is easy to establish $2 \langle a,b \rangle = 2 \langle b,a \rangle$: In \ref{eq5} the right hand side does not change when we interchange $a \leftrightarrow b$ and $c \leftrightarrow d$, but then left hand side becomes $\langle b,a \rangle + \langle d,c \rangle$. It follows
\[\langle a,b \rangle - \langle b,a \rangle = \langle d,c \rangle - \langle c,d \rangle.\]
The right hand side does not change when we interchange $a \leftrightarrow b$, but the left hand side changes its sign. It follows $\langle a,b \rangle - \langle b,a \rangle = \langle b,a \rangle - \langle a,b \rangle$, i.e. $2 \langle a,b \rangle = 2 \langle b,a \rangle$. In particular, we are done when $2 \in R^*$. The general case is more complicated: We (have to) introduce an additional parameter $e$ from $E$ and will abbreviate $s(a) \otimes s(b) \otimes s(c) \otimes t(d) \otimes_\e h(e)$ by $\langle a,b,c,d,e \rangle$, or just by $\langle d,e \rangle$ since we may again permute the first three factors freely. We have to deduce \ref{eq6} from \ref{eq5}.

The following linear combination has been found using the computer algebra system \textsc{Sage} (see \url{http://www.sagemath.org}). Observe that those brackets with the same color cancel each other out.
\definecolor{yell}{rgb}{1,0.8,0.2}
\begin{eqnarray*}
0 & \stackrel{\ref{eq5}}{=}& \phantom{+} (\textcolor{Aquamarine}{\langle a,b\rangle } + \textcolor{Bittersweet}{\langle c,e\rangle }) - (\textcolor{Blue}{\langle a,c\rangle } + \textcolor{CarnationPink}{\langle b,e\rangle }) \\
&&+(\textcolor{Orange}{\langle b,c\rangle } + \textcolor{DarkOrchid}{\langle a,e\rangle }) - (\textcolor{ForestGreen}{\langle b,a\rangle } + \textcolor{Bittersweet}{\langle c,e\rangle }) \\
&&+(\textcolor{LimeGreen}{\langle a,b\rangle } + \langle e,d \rangle) - (\textcolor{DarkOrchid}{\langle a,e\rangle } + \textcolor{yell}{\langle b,d\rangle }) \\
&&+(\textcolor{Red}{\langle a,d\rangle } + \textcolor{CarnationPink}{\langle b,e\rangle }) - (\textcolor{LimeGreen}{\langle a,b\rangle } + \langle d,e \rangle ) \\
&&+(\textcolor{ForestGreen}{\langle b,a\rangle } + \textcolor{Maroon}{\langle c,d\rangle }) - (\textcolor{Orange}{\langle b,c\rangle } + \textcolor{Red}{\langle a,d\rangle }) \\
&&+(\textcolor{Blue}{\langle a,c\rangle } + \textcolor{yell}{\langle b,d\rangle }) - (\textcolor{Aquamarine}{\langle a,b\rangle } + \textcolor{Maroon}{\langle c,d\rangle }) \\
&=&  \langle e,d \rangle - \langle d,e \rangle
\end{eqnarray*}
This finally proves \ref{eq6}.

\textbf{5. Step: $\K$ is symtrivial.} We still have to complete the proof of Step 3: Given a triple $(\L,s,\lambda) \in A$, we have defined $t: E[\e] \to \K$ to be the cokernel of $\phi : (\L^{\otimes -2} \otimes E^{\otimes 3})[\e] \to E[\e]$. In particular, \ref{eq2} holds. We have to prove that $\K$ is a line object. Let us first prove that $\K$ is symtrivial. Since $t$ is an epimorphism, it suffices to prove that $t(c) \otimes_\e t(d) = t(d) \otimes_\e t(c)$ in element notation. Using additional aprameters $a,b$, we compute:
\[s(a) \otimes s(b) \otimes t(c) \otimes_\e t(d) \stackrel{\ref{eq2}}{=} s(a) \otimes s(c) \otimes t(b) \otimes_\e t(d) + \lambda(b \wedge c) \otimes \e t(a) \otimes_\e t(d)\]
\[\stackrel{\ref{eq2}}{=} s(a) \otimes s(d) \otimes t(b) \otimes_\e t(c)+ \lambda(c \wedge d) \otimes t(b) \otimes_\e \e t(a) + \lambda(b \wedge c) \otimes t(a) \otimes_\e \e t(d)\]
\[\stackrel{\ref{eq2}}{=} s(a) \otimes s(b) \otimes t(d) \otimes_\e t(c) + \lambda(d \wedge b) \otimes \e t(a) \otimes_\e t(c) + \lambda(c \wedge d) \otimes t(b) \otimes_\e \e t(a) \]
\[+ \lambda(b \wedge c) \otimes t(a) \otimes_\e \e t(d).\]
Thus, we have to prove
\begin{equation} \label{eq7}
0 = \lambda(c \wedge d) \otimes t(b) \otimes_\e \e t(a) - \lambda(b \wedge d) \otimes t(a) \otimes_\e \e t(c) +  \lambda(b \wedge c) \otimes t(a) \otimes_\e \e t(d).
\end{equation}
Let us start with the cocycle condition
\[0 = \lambda(c \wedge d) \otimes s(b) - \lambda(b \wedge d) \otimes s(c) + \lambda(b \wedge c) \otimes s(d).\]
We define ${i}$ as before and have ${i} s = \e t$ as in Step 2 (this did not use that $\K$ is a line object). Thus, applying $\L^{\otimes 2} \otimes {i}$ to the equation above, we get
\[0 = \lambda(c \wedge d) \otimes \e t(b) - \lambda(b \wedge d) \otimes \e t(c) + \lambda(b \wedge c) \otimes \e t(d).\]
Tensoring with $t$ in the middle, we get
\[0 = \lambda(c \wedge d) \otimes t(a) \otimes_\e \e t(b) - \lambda(b \wedge d) \otimes t(a) \otimes_\e \e t(c) + \lambda(b \wedge c) \otimes t(a) \otimes_\e \e t(d).\]
Comparing this with \ref{eq7}, we are left to prove that
\[t(a) \otimes_\e \e t(b) = t(b) \otimes_\e \e t(a),\]
which is a special case of the claimed symtriviality of $\K$. In order to prove this, we make the same calculations as in the beginning of this Step:
\[s(a) \otimes s(b) \otimes t(c) \otimes_\e \e t(d) \stackrel{\ref{eq2}}{=} s(a) \otimes s(c) \otimes t(b) \otimes_\e \e t(d) + \lambda(b \wedge c) \otimes \e t(a) \otimes_\e \e t(d)\]
\[=s(a) \otimes s(c) \otimes t(b) \otimes_\e \e t(d) \stackrel{\ref{eq2}}{=} s(a) \otimes s(d) \otimes t(b) \otimes_\e \e t(c)+ \lambda(c \wedge d) \otimes t(b) \otimes_\e \e^2 t(a)\]
\[=s(a) \otimes s(d) \otimes t(b) \otimes_\e \e t(c) \stackrel{\ref{eq2}}{=} s(a) \otimes s(b) \otimes t(d) \otimes_\e \e t(c) + \lambda(d \wedge b) \otimes \e t(a) \otimes_\e \e t(c)\]
\[=s(a) \otimes s(b) \otimes t(d) \otimes_\e \e t(c)\]
This proves $t(c) \otimes_\e \e t(d) = t(d) \otimes_\e \e t(c)$, as desired.
 
\textbf{6. Step: $\K$ is invertible.} We have to construct an object and show that it is $\otimes$-inverse to $\mathcal{K}$. Consider the morphism
\[\psi : (\L^{\otimes -4} \otimes E^{\otimes 3})[\e] \to (\L^{\otimes -2} \otimes E)[\e]\]
in $\C[\e]$ which is induced by the morphism $E^{\otimes 3} \to (\L^{\otimes 2} \otimes E)[\e]$ which maps $a \otimes b \otimes c$ to
\[s(c) \otimes (s(b) \otimes a - s(a) \otimes b) - \lambda(a \wedge b) \otimes \varepsilon c.\]
Note that $\psi$ is almost like $\phi$, but we have twisted with $\L^{\otimes -2}$ and have changed the sign of the cocycle $\lambda$. Let $\K^{\otimes -1}$ be the cokernel of $\psi$. Consider the morphism
\[s \otimes s + \e \lambda : E \otimes E \to \L^{\otimes 2}[\e]\]
in $\C$. It induces a morphism
\[\beta : (\L^{\otimes -2} \otimes E)[\e] \otimes_\e E[\e] \cong (\L^{\otimes -2} \otimes E \otimes E)[\e] \to \O_\C[\e]\]
in $\C[\e]$. We claim that $\beta \circ ((\L^{\otimes -2} \otimes E)[\e]  \otimes \phi) = 0$. This composition is dual to the morphism $E^{\otimes 4}[\e] \to E^{\otimes 2} \otimes \L^{\otimes 2}[\e] \to  \L^{\otimes 4}[\e]$ which maps $u \otimes a \otimes b \otimes c$ in $E^{\otimes 4}$ first to
\[u \otimes c \otimes s(a) \otimes s(b) - u \otimes b \otimes s(a) \otimes s(c) - u \otimes \e a \otimes \lambda(b \wedge c)\]
in $E^{\otimes 2} \otimes \L^{\otimes 2}[\e]$ and then to
\[(s(u) \otimes s(c) + \e \lambda(u \wedge c)) \otimes s(a) \otimes s(b) - (s(u) \otimes s(b) + \e \lambda(u \wedge b)) \otimes s(a) \otimes s(c)\]
\[ - \e s(u) \otimes s(a) \otimes \lambda(b \wedge c)\]
in $\L^{\otimes 4}[\e]$. This is easily seen to be zero, using the cocycle condition and that $\L$ is symtrivial. This proves the claim.

Since $(\L^{\otimes -2} \otimes E)[\e]  \otimes t$ is a cokernel of $(\L^{\otimes -2} \otimes E)[\e]  \otimes \phi$, there is a unique morphism $\alpha : (\L^{\otimes -2} \otimes E)[\e] \otimes_\e \K \to \O_\C[\e]$ such that $\alpha \circ ((\L^{\otimes -2} \otimes E)[\e] \otimes t) = \beta$. We claim that
\[(\L^{\otimes -4} \otimes E^{\otimes 3})[\e] \otimes_\e \K \xrightarrow{\psi \otimes \K} (\L^{\otimes -2} \otimes E)[\e]  \otimes_\e \K \xrightarrow{\alpha} \O_C[\e] \to 0\]
is exact -- this will imply $\K^{\otimes -1} \otimes_\e \K \cong \O_\C[\e]$. It suffices to prove that
\[(\L^{\otimes -2} \otimes E^{\otimes 3})[\e] \otimes_\e \K \to E[\e] \otimes_\e \K \to \L^{\otimes 2}[\e] \to 0\]
is exact. So let $T \in \C[\e]$ and $E[\e] \otimes_\e \K \to T$ be a morphism in $\C[\e]$ which vanishes on $(\L^{\otimes -2} \otimes E^{\otimes 3})[\e]$. Since $E[\e] \otimes t : (E \otimes E)[\e] \to E[\e] \otimes_\e \K$ is the cokernel of $E[\e] \otimes \phi$, the morphism corresponds to a morphism $h : E \otimes E \to T$ in $\C$ such that
\begin{equation} \label{eq8}
s(b) \otimes s(c) \otimes h(a \otimes d) - s(b) \otimes s(d) \otimes h(a \otimes c) = \lambda(c \wedge d) \otimes \e h(a \otimes b)
\end{equation}
and
\begin{equation} \label{eq9}
s(b) \otimes s(c) \otimes h(a \otimes d) - s(a) \otimes s(c) \otimes h(b \otimes d) = \lambda(a \wedge b) \otimes \e h(c \otimes d).
\end{equation}
We have to prove that there is a unique morphism $\tilde{h} : \L^{\otimes 2} \to T$ in $\C$ such that $\tilde{h}(s \otimes s + \e \lambda) = h$, i.e.
\begin{equation} \label{eq10}
\tilde{h}(s(a) \otimes s(b)) + \e \tilde{h}(\lambda(a \wedge b)) = h(a \otimes b).
\end{equation}
We check uniqueness first: Given \ref{eq10}, the cocycle condition implies
\[0 = s(a) \otimes \e \tilde{h}(\lambda(b \wedge c)) - s(b) \otimes \e \tilde{h}(\lambda(a \wedge c)) + s(c) \otimes \e \tilde{h}(\lambda(a \wedge b))\]
\[=s(a) \otimes (h(b \otimes c) - \tilde{h}(s(b) \otimes s(c))) - s(b) \otimes (h(a \otimes c) - \tilde{h}(s(a) \otimes s(c))\]
\[ + s(c) \otimes (h(a \otimes b) - \tilde{h}(s(a) \otimes s(b))).\]
Because of $s(a) \otimes \tilde{h}(s(b) \otimes s(c)) = s(b) \otimes \tilde{h}(s(a) \otimes s(c))$ (use that $\L$ is symtrivial) this simplifies to
\[s(c) \otimes \tilde{h}(s(a) \otimes s(b)) = s(a) \otimes h(b \otimes c) - s(b) \otimes h(a \otimes c) + s(c) \otimes h(a \otimes b).\]
This shows that $\tilde{h}$ is unique, if it exists. But it also tells us how to construct $\tilde{h}$ from $h$: Consider the morphism
\[\theta : E^{\otimes 3} \to \L \otimes T, c \otimes a \otimes b \mapsto s(a) \otimes h(b \otimes c) - s(b) \otimes h(a \otimes c) + s(c) \otimes h(a \otimes b).\]
We would like to lift it to $\L^{\otimes 3}$. By \autoref{coeq-tensor} and \autoref{goodepi} it suffices to check the equations 
\begin{eqnarray*}
s(d) \otimes \theta(c \otimes a \otimes b) & = s(c) \otimes \theta(d \otimes a \otimes b)\\
&  = s(a) \otimes \theta(c \otimes d \otimes b)\\ 
&  = s(b) \otimes \theta(c \otimes a \otimes d).
\end{eqnarray*} 
Let us first verify 
\begin{equation} \label{eq12}
\e h(a \otimes b) = \e h(b \otimes a).
\end{equation}
Note that \ref{eq8} and \ref{eq9} imply:
\begin{equation} \label{eq13}
s(c) \otimes \e h(a \otimes d) = s(d) \otimes \e h(a \otimes c)
\end{equation}
\begin{equation} \label{eq14}
s(b) \otimes \e h(a \otimes d) = s(a) \otimes \e h(b \otimes d)
\end{equation} 
As in the proof of the Segre embedding (\ref{segre}), we get \ref{eq12} as follows:
\[s(c) \otimes s(d) \otimes \e h(a \otimes b) \stackrel{\ref{eq13}}{=} s(c) \otimes s(b) \otimes \e h(a \otimes d)\]
\[  = s(b) \otimes s(c) \otimes \e h(a \otimes d) \stackrel{\ref{eq14}}{=} s(b) \otimes s(a) \otimes \e h(c \otimes d)\]
\[ = s(a) \otimes s(b) \otimes \e h(c \otimes d) \stackrel{\ref{eq14}}{=} s(a) \otimes s(c) \otimes \e h(b \otimes d)\]
\[=s(c) \otimes s(a) \otimes \e h(b \otimes d) \stackrel{\ref{eq13}}{=} s(c) \otimes s(d) \otimes \e h(b \otimes a)\]
Now we can check the equations for $\theta$:
\[s(d) \otimes \theta(c \otimes a \otimes b)\]
\[=s(d) \otimes s(c) \otimes h(a \otimes b) - s(d) \otimes s(b) \otimes h(a \otimes c) + s(d) \otimes s(a) \otimes h(b \otimes c)\]
\[\stackrel{\ref{eq9}}{=} s(d) \otimes s(c) \otimes h(a \otimes b) + \lambda(b \wedge a) \otimes \e h(d \otimes c)\]
stays invariant after interchanging $d \leftrightarrow c$ (because of  \ref{eq12}). Likewise,
\[s(d) \otimes \theta(c \otimes a \otimes b) \stackrel{\ref{eq8}}{=} \lambda(c \wedge b) \otimes \e h(a \otimes d) + s(d) \otimes s(a) \otimes h(b \otimes c)\]
stays invariant after interchanging $d \leftrightarrow a$. For the third variable, we calculate:
\[s(d) \otimes \theta(c \otimes a \otimes b) - s(b) \otimes \theta(c \otimes a \otimes d)\]
\[=s(d) \otimes s(c) \otimes h(a \otimes b) - s(b) \otimes s(c) \otimes h(a \otimes b)\]
\[+s(d) \otimes s(a) \otimes h(b \otimes c) - s(d) \otimes s(a) \otimes h(d \otimes c)\]
\[\stackrel{\ref{eq8},\,\ref{eq9}}{=} \lambda(d \wedge b) \otimes \e h(a \otimes c) + \lambda(b \wedge d) \otimes \e h(a \otimes c) = 0.\]
 
We have proven that $\theta$ lifts uniquely to a morphism $\L^{\otimes 3} \to \L \otimes T$. Hence, there is a unique morphism $\tilde{h} : \L^{\otimes 2} \to T$ such that
\[s(c) \otimes \tilde{h}(s(a) \otimes s(b)) = s(a) \otimes h(b \otimes c) - s(b) \otimes h(a \otimes c) + s(c) \otimes h(a \otimes b).\]
We want to prove \ref{eq10}. This comes down to
\[s(c) \otimes \e \tilde{h}(\lambda(a \wedge b)) = s(b) \otimes h(a \otimes c) - s(a) \otimes h(b \otimes c).\]
Let us introduce additional parameters $d,e$ from $E$. We compute:
\[s(e) \otimes s(d) \otimes s(c) \otimes \e \tilde{h}(\lambda(a \wedge b)) = \lambda(a \wedge b) \otimes s(e) \otimes \e \tilde{h}(s(d) \otimes s(c))\]
\[=\lambda(a \wedge b) \otimes \e (s(d) \otimes h(e \otimes c) - s(c) \otimes h(e \otimes d) + s(e) \otimes h(d \otimes e))\]
Observe that \ref{eq8} implies $\e(s(d) \otimes h(e \otimes c) - s(c) \otimes h(e \otimes d))=0$. Thus, the expression simplifies to
\[=\lambda(a \wedge b) \otimes s(e) \otimes \e  h(d \otimes e).\]
By \ref{eq9} this expands to
\[=s(e) \otimes (s(b) \otimes s(d) \otimes h(a \otimes c) - s(a) \otimes s(d) \otimes h(a \otimes c))\]
\[=s(e) \otimes s(d) \otimes (s(b) \otimes h(a \otimes c) - s(a) \otimes h(a \otimes c)).\]
This finishes the proof.
\end{proof}

For the sake of completeness, let us mention the following variant of \autoref{tangent-main} which gets rid of the ``generator object'' $E$. It generalizes the well-known classification of deformations of invertible sheaves (\cite[Proposition 2.6]{Har10}), which says: If $X$ is a scheme, then there is an isomorphism of groups
\[\Pic(X[\e]) \cong \Pic(X) \oplus H^1(X,\mathcal{O}_X).\]
 
\begin{thm} \label{deform}
Let $\C$ be a cocomplete linear tensor category. There is an equivalence of categories between the category of line objects in $\C[\e]$ and the category of extensions $0 \to \L \to \K \to \L \to 0$ in $\C$, where $\L \in \C$ is a line object.
\end{thm}

Here, a morphism from an extension $(\L,\K,\L)$ to another $(\L',\K',\L')$ is defined to be a pair consisting of a morphism $\L \to \L'$ and a morphism $\K \to \K'$ in $\C$ such that the obvious diagram commutes.

\begin{proof}
Let $A$ be the category of line objects of $\C[\e]$ and $B$ be the category of extensions as above. Then \autoref{KL-exact} produces a functor $A \to B$. In order to construct $B \to A$, let
\[0 \to \L \xrightarrow{i} \K \xrightarrow{p} \L \to 0\]
be an exact sequence in $\C$, where $\L$ is a line object. Define $\e \in \End(K)$ by $\e = ip$. Note that $p\e=0$ and $\e i = 0$. In particular, $\e^2=0$, thus $\K$ becomes an object of $\C[\e]$. Let $\overline{\K}$ be the object of $\C[\e]$ with the same underlying object of $\C$, but $\e$ changes its sign on $\overline{\K}$. We claim
\[\K \otimes_\e \overline{\K} \cong \L^{\otimes 2}[\e].\]
This will prove that $\K$ is invertible in $\C[\e]$ with inverse $\L^{\otimes -2} \otimes \overline{\K}$. Consider the morphism
\[\K^{\otimes 2} \to \K \otimes \L,~ a \otimes b \mapsto a \otimes p(b) - b \otimes p(a).\]
Since $\L$ is symtrivial, it vanishes when composed with $p \otimes \L$. Since $i \otimes \L$ is the kernel of $p \otimes \L$, there is a unique morphism
\[\lambda : \K^{\otimes 2} \to \L^{\otimes 2}\]
such that
\[(i \otimes\L) \lambda(a \otimes b) = a \otimes p(b) - b \otimes p(a).\]
Clearly we have $\lambda(a \otimes b) = - \lambda(b \otimes a)$ and one checks
\begin{equation} \label{equ1}
\lambda(\e a \otimes b) = p(a) \otimes p(b) \text{ and } \lambda(a \otimes \e b) = - p(b) \otimes p(a) = - p(a) \otimes p(b).
\end{equation}
Now consider the morphism
\[\alpha : \K \otimes \K \to \L^{\otimes 2}[\e],~ a \otimes b \mapsto p(a) \otimes p(b) + \e \lambda(a \otimes b).\]
It satisfies
\[\alpha(\e a \otimes b) = \e (p(a) \otimes p(b)) = \e \alpha(a \otimes b),\]
\[\alpha(a \otimes \e b) = - \e (p(a) \otimes p(b)) = - \e \alpha(a \otimes b).\]
Hence, it lifts to a morphism $\beta : \K \otimes_\e \overline{\K} \to \L^{\otimes 2}[\e]$ in $\C[\e]$ such that
\[\beta(\e a \otimes b) = \e (p(a) \otimes p(b)).\]
We construct an inverse morphism as follows: The morphism $\K^{\otimes 2} \to \K \otimes_\e \overline{\K}$ mapping $a \otimes b \mapsto \e a \otimes_\e b$ kills $\e a \otimes b$ and $a \otimes \e b$. Since $p : \K \to \L$ is the cokernel of $\e$, using \autoref{coeq-tensor} there is a unique morphism $\gamma : \L^{\otimes 2} \to \K \otimes_\e \overline{\K}$ such that
\[\gamma(p(a) \otimes p(b)) = \e a \otimes_\e b.\]
Observe that $\e \gamma = 0$. Next, we consider the morphism $\K^{\otimes 2} \to \K \otimes_\e \overline{\K}$ defined by
\[a \otimes b \mapsto a \otimes_\e b - \gamma(\lambda(a \otimes b)).\]
By \ref{equ1} it maps $\e a \otimes b$ to
\[\e a \otimes_\e b - \gamma(p(a) \otimes p(b)) = 0\]
and $a \otimes \e b$ to
\[a \otimes_\e \e b - \gamma(-p(a) \otimes p(b))=- \e a \otimes_\e b + \e a \otimes_\e b = 0.\]
By the same argument as above, it follows that there is a unique morphism $\delta : \L^{\otimes 2} \to  \K \otimes_\e \overline{\K}$ such that
\[\delta(p(a) \otimes p(b)) = a \otimes_\e b - \gamma(\lambda(a \otimes b)).\]
Since $\e \gamma = 0$, we have $\e \delta(p(a) \otimes p(b)) = \e a \otimes_\e b = \gamma(p(a) \otimes p(b))$. Since $p \otimes p$ is an epimorphism, this proves
\begin{equation} \label{equ2}
\e \delta = \gamma : \L^{\otimes 2} \to \K \otimes_\e \overline{\K}.
\end{equation}
Besides, we have
\begin{equation} \label{equ3}
\beta \gamma = \e : \L^{\otimes 2} \to \L^{\otimes 2}[\e]
\end{equation}
because of $\beta(\gamma(p(a) \otimes p(b))) = \beta(\e a \otimes_\e b) = \e (p(a) \otimes p(b))$. Let
\[\pi : \L^{\otimes 2}[\e] \to \K \otimes_\e \overline{\K}\]
be the morphism in $\C[\e]$ induced by $\delta$. We have $\pi \circ \beta = \id_{\K \otimes_\e \overline{\K}}$ because of
\[\pi(\beta(a \otimes_\e b)) = \delta(p(a) \otimes p(b)) + \e \delta(\lambda(a \otimes b))\]
\[ = a \otimes_\e b - \gamma(\lambda(a \otimes b))  + \e \delta(\lambda(a \otimes b)) \stackrel{\ref{equ2}}{=} a \otimes_\e b.\]
And we have $\beta \circ \pi = \id_{\L^{\otimes 2}[\e]}$ because of
\[\beta(\pi(p(a) \otimes p(b))) = \beta(a \otimes_\e b) - \beta(\gamma(\lambda(a \otimes b)))\]
\[\stackrel{\ref{equ3}}{=} p(a) \otimes p(b) + \e \lambda(a \otimes b) - \e \lambda(a \otimes b) = p(a) \otimes p(b).\]
This finishes the proof that $\K \otimes_\e \overline{\K} \cong \L^{\otimes 2}[\e]$. Therefore, $\K$ is invertible in $\C[\e]$. If we knew that $\K$ is symtrivial, this would finish the construction of the functor $B \to A$. It is straightforward to check that it is inverse to the functor $A \to B$.

In order to prove that $\K$ is symtrivial, let us check first the cocycle condition
\[p(a) \otimes \lambda(b \otimes c) - p(b) \otimes \lambda(a \otimes c) + p(c) \otimes \lambda(a \otimes b) = 0.\]
In fact, it suffices to check this after applying $\L \otimes i \otimes \L$, but then this becomes
\[p(a) \otimes (b \otimes p(c) - c \otimes p(b)) - p(b) \otimes (a \otimes p(c) - c \otimes p(a))\]
\[+p(c) \otimes (a \otimes p(b) - b \otimes p(a)) = 0,\]
which follows easily from $\L$ being symtrivial. Applying $i \otimes \L \otimes \L$ to the cocycle condition, we obtain
\[\e(a) \otimes \lambda(b \otimes c) - \e(b) \otimes \lambda(a \otimes c) + \e(c) \otimes \lambda(a \otimes b) = 0.\]
Since $\L$ is symtrivial, we have
\[\e(b) \otimes \lambda(a \otimes c) = (i \otimes \L \otimes \L)(p(b) \otimes \lambda(a \otimes c)) = (i \otimes \L \otimes \L)(\lambda(a \otimes c) \otimes p(b))\]
\[=(a \otimes p(c) - c \otimes p(a)) \otimes p(b).\]
It follows that
\[c \otimes p(a) \otimes p(b) + \e(c) \otimes \lambda(a \otimes b) = a \otimes p(c) \otimes p(b) + \e(a) \otimes \lambda(c \otimes b).\]
But this precisely means that the morphism $\K \otimes \L^{\otimes 2} \to \K \otimes \L^{\otimes 2}$ which is determined by the commutative diagram
\[\xymatrix@C=50pt{\K \otimes_\e \K \otimes_\e \overline{\K} \ar[r]^{S_{\K,\K} \otimes_\e \overline{\K}} \ar[d]^{\cong}_{\K \otimes_\e \beta} & \K \otimes_\e \K \otimes_\e \overline{\K} \ar[d]^{\K \otimes_\e \beta}_{\cong} \\ \K \otimes \L^{\otimes 2} \ar[r] & \K \otimes \L^{\otimes 2} }\]
equals the identity. Since $\overline{\K}$ is invertible, it follows that $S_{\K,\K}$ also equals the identity.
\end{proof}

\begin{rem}
If $\C=\Q(X)$ for some scheme $X$, then of course \autoref{deform} is easier to prove: That $\K$ is invertible may be checked locally; but locally the exact sequence splits and everything is easy. But even in this case our proof gives some additional information, namely how the inverse of $\K$ looks like \emph{globally} in terms of $\K$ and $\L$:
\[\K^{\otimes -1} = \L^{\otimes -2} \otimes \overline{\K}\]
\end{rem}

\begin{thm} \label{tangsch}
If $f : X \to S$ is a morphism of schemes which is affine or projective, then
\[T\bigl(\Q(X)/\Q(S)\bigr) \simeq \Q\bigl(T(X/S)\bigr).\]
\end{thm}

\begin{proof}
This is a consequence of \autoref{tangmod}, \autoref{tangcl} and \autoref{tangent-proj} as well as the corresponding results in algebraic geometry, notably the Euler sequence (\autoref{euler}).
\end{proof}

We may use the classical lifting properties of smooth, unramified and \'{e}tale morphisms (\cite[\para 17]{EGAIV}) to get corresponding notions for tensor functors. Let us illustrate this for the notion of being formally unramified.

\begin{defi}[Formally unramified tensor functors]
A cocontinuous linear tensor functor $F : \C \to \D$ is called \emph{formally unramified} if $T(\D/\C)=\D$, i.e. for every cocontinuous linear tensor functor $\C \to \E$ the canonical functor
\[\Hom_{c\otimes/\C}(\D,\E) \to \Hom_{c\otimes/\C}(\D,\E[\e])\]
is an equivalence of categories.
\end{defi}

Then formally unramified cocontinuous linear tensor functors are closed under composition and cobase changes. We have the following comparison to algebraic geometry:

\begin{cor} \label{fr}
Let $f : X \to S$ be a morphism of schemes which is affine or projective. If $f : X \to S$ is formally unramified, then its pullback functor $f^* : \Q(S) \to \Q(X)$ is formally unramified. The converse holds when $X$ is quasi-separated.
\end{cor}

\begin{proof}
Let $p : T(X/S) \to X$ denote the tangent bundle projection.  Then $f$ is formally unramified if and only if $p$ is an isomorphism (\cite[Proposition 17.2.1]{EGAIV}) and $f^*$ is formally unramified if and only if $p^*$ is an equivalence (by \autoref{tangsch}). Therefore, the claim follows from \autoref{isofunk}.
\end{proof}

\chapter{Monoidal monads and their modules} \label{monoidalmonads}
This chapter is quite independent from the rest of this thesis. \marginpar{Added ``is''.}It is devoted to the study of monoidal monads, their module categories as well as certain universal properties of these module categories. This provides a machinery for constructing cocomplete tensor categories.


\section{Overview} \label{monadoverview}

Let $\C$ be a cocomplete tensor category. Given a commutative algebra $A$ in $\C$, we already know that its category of modules $\M(A)$ becomes a cocomplete tensor category (\autoref{modA}). The forgetful functor $U : \M(A) \to \C$ creates colimits and has a left adjoint $F : \C \to \M(A)$, mapping $X$ to $X \otimes A$ with the obvious right $A$-action. In fact, $U$ is monadic and $F$ is a cocontinuous tensor functor. Our goal is to generalize this construction to certain monads on $\C$, the above being the special case of the monad $- \otimes A$ associated to $A$. This will include lots of other specific examples. See \cite[VI]{ML98} for an introduction to monads.
 
This construction was already done by Lindner (\cite{Lin75}), Guitart (\cite{Gui80}) and recently by Seal (\cite{Sea13}), all influenced by the pioneering work of Kock (\cite{Kock71}). However, we would like to extend this construction 
and specialize to the realm of \emph{cocomplete} tensor  categories. Besides, our exposition is slightly different. Many preparatory results of this chapter can be found almost verbatim in \cite{Sea13}, which coincidentally appeared at the same time the author studied monoidal monads.
  
Let us describe our goal more precisely, thereby motivate or sketch the following definitions and constructions \emph{backwards}. Let $(T,\mu,\eta)$ be a monad (usually abbreviated by $T$) on the underlying category of a cocomplete (symmetric) monoidal category $\C$. We want to endow the category $\M(T)$ of $T$-modules (often called $T$-algebras) with a tensor product $\otimes_T$ such that $\M(T)$ becomes a cocomplete (symmetric) monoidal category. This extra structure should be reflected in an extra structure on $T$.

We want the free functor $F : \C \to \M(T)$ also to carry the structure of a symmetric monoidal functor, which means in particular that the free module $F(1)$ is a unit for $\otimes_T$ and that there are natural isomorphisms of free $T$-modules
\[F(X) \otimes_T F(Y) \cong F(X \otimes Y)\]
since our monoidal functors are required to be strong. It also implies that the forgetful functor $U$ is lax monoidal (\autoref{adj-lax}), so there are natural morphisms $U(F(X)) \otimes U(F(Y)) \to U(F(X) \otimes_T F(Y))$, i.e.
\[d_{X,Y} : T(X) \otimes T(Y) \to T(X \otimes Y),\]
which should be compatible with the monad structure. This already enables us to compute the tensor product in general: Given $T$-modules $(A,a),(B,b)$, tensoring the canonical presentations (\cite[VI.7]{ML98})
\[F(T(A)) \rightrightarrows F(A) \to (A,a),~F(T(B)) \rightrightarrows F(B) \to (B,b)\]
yields a coequalizer
\[F(T(A) \otimes T(B)) \rightrightarrows F(A \otimes B) \to (A,a) \otimes_T (B,b).\]
In general, two coequalizers in a cocomplete monoidal category cannot be tensored pointwise, but this works for \emph{reflexive coequalizers}. The two homomorphisms of $T$-modules $F(T(A) \otimes T(B)) \rightrightarrows F(A \otimes B)$ correspond to $\C$-morphisms $T(A) \otimes T(B) \rightrightarrows T(A \otimes B)$, the first one being $d_{A,B}$ and the other one being $\eta_{A \otimes B} \circ a \otimes b$.

Conversely, given a monad $T$ with a suitable natural transformation $d$ as above, we may construct $(A,a) \otimes_T (B,b)$ as the coequalizer above, provided it exists. There is an alternative description using \emph{bihomomorphisms}. This notion has been studied by Kock (\cite{Kock71bilin}) and Banaschewski-Nelson (\cite{Ban76}). Thus we have to impose that $\M(T)$ has reflexive coequalizers.

In order to get associativity of $\otimes_T$, we also have to impose the nontrivial condition that $\otimes_T$ preserves reflexive coequalizers in each variable. Namely, tensoring the canonical presentation $F(T(A)) \rightrightarrows F(A) \to (A,a)$ with the reflexive coequalizer $F(T(B) \otimes T(C)) \rightrightarrows F(B \otimes C) \to (B,b) \otimes_T (C,c)$ gives a coequalizer
\[F(T(A) \otimes (T(B) \otimes T(C))) \rightrightarrows F(A \otimes (B \otimes C)) \to (A,a) \otimes_T ((B,b) \otimes_T (C,c))\]
and then we use the associator in $\C$. These two conditions lead to the definition of a \emph{coherent monoidal monad}.


\section{Reflexive coequalizers}
 
If a functor $F : \C \times \C' \to \D$ preserves coequalizers in each variable and $\D$ has coproducts, then coequalizer diagrams
\[\xymatrix@!{A \ar@<0.2pc>[r]^{f} \ar@<-0.2pc>[r]_{g} & B \ar[r]^{p} & C},~\xymatrix@!{A' \ar@<0.2pc>[r]^{f'} \ar@<-0.2pc>[r]_{g'} & B' \ar[r]^{p'} & C'}\]
in $\C$ resp. $\C'$ induce a coequalizer diagram
\[\xymatrix@C=50pt{F(A,B') \oplus F(B,A') \ar@<0.2pc>[rr]^-{F(f,\id),F(\id,f')} \ar@<-0.2pc>[rr]_-{F(g,\id),F(\id,g')}&  & F(B,B') \ar[r]^{F(p,p')} & F(C,C')}\]
in $\D$. This applies in particular to the tensor product of a cocomplete monoidal category. Intuitively, this describes how we can tensor two presentations (of $C,C'$) with generators ($B,B'$) and relations ($A,A'$), which we have already observed in \autoref{coeq-tensor}. But for \emph{reflexive} coequalizers this simplifies (we can simply use $F(A,A')$ instead of the direct sum) -- this is the main motivation for their use in our context.
 
\begin{defi}
Recall that a pair of morphisms $f,g : A \to B$ in a category is called \emph{reflexive} if they have a common section, i.e. there is some morphism $s : B \to A$ with $fs=gs=1_B$. A coequalizer of a reflexive pair is called a \emph{reflexive coequalizer}.
\end{defi}

\begin{lemma}[{\cite[Lemma A.1.2.11]{John02}}] \label{reflexdiag}
Given a commutative diagram
\[\xymatrix{A_1 \ar@<0.2pc>[r] \ar@<-0.2pc>[r]  \ar@<0.2pc>[d] \ar@<-0.2pc>[d] & B_1  \ar@<0.2pc>[d] \ar@<-0.2pc>[d] \ar[r] & C_1 \ar@<0.2pc>[d] \ar@<-0.2pc>[d] \\
A_2 \ar@<0.2pc>[r] \ar@<-0.2pc>[r] \ar[d] & B_2 \ar[d] \ar[r] & C_2 \ar[d]  \\ 
A_3 \ar@<0.2pc>[r] \ar@<-0.2pc>[r]  & B_3 \ar[r] & C_3 }\]
in a category such that each row and each column is a reflexive coequalizer diagram, then the diagonal $\xymatrix@!{A_1 \ar@<0.2pc>[r] \ar@<-0.2pc>[r] & B_2 \ar[r] & C_3}$ is also a reflexive coequalizer.
\end{lemma}

\begin{cor}[{\cite[Corollary A.1.2.12]{John02}}] \label{bifunktor}
Suppose that $F : \C_1 \times \C_2 \to \D$ is a functor which preserves reflexive coequalizers in each variable. Then $F$ preserves reflexive coequalizers.
\end{cor}

\begin{rem} \label{reflexivex}
The coequalizer of a pair
\[\xymatrix{A \ar@<0.2pc>[r]^{f} \ar@<-0.2pc>[r]_{g} & B}\]
is the same as the coequalizer of the reflexive pair
\[\xymatrix{A \sqcup B \ar@<0.2pc>[r]^-{f,\id_B} \ar@<-0.2pc>[r]_-{g,\id_B} & B.}\]
Thus, in a category with coproducts, every coequalizer can be reduced to a reflexive coequalizer. For example in the category of sets
\[\xymatrix{\N \sqcup \N \ar@<0.2pc>[r]^-{\id,\id} \ar@<-0.2pc>[r]_-{\id,S} & \N \ar[r] & \{0\}}\]
is a reflexive coequalizer, where $S$ denotes the successor function.
\end{rem}
 
\begin{nota}
A monad $(T,\mu,\eta)$ on a category $\C$ will be abbreviated by $T$. The category of $T$-modules (often also called $T$-algebras) will be denoted by $\M(T)$. These are pairs $(A,a)$ consisting of an object $A \in \C$ and a morphism $a : T(A) \to A$ such that the two evident diagrams commute, expressing compatibility with $\mu$ and $\eta$. We have a forgetful functor $U : \M(T) \to \C$, $(A,a) \mapsto A$ with left adjoint $F : \C \to \M(T)$, $X \mapsto (T(X),\mu_X)$ (the free $T$-module on $X$). We have the fundamental relation $T=UF$. In order to avoid confusions we will try not to ignore the forgetful functor $U$ i.e. will \emph{not} abbreviate $(A,a)$ by $A$.
\end{nota}
  
\begin{rem}[Canonical presentations] \label{canpres}
Let $T$ be a monad on a category $\C$ and let $(A,a) \in \M(T)$. Then $T(a),\mu_A : T^2(A) \to T(A)$ are homomorphisms of $T$-modules $F(T(A)) \to F(A)$, reflexive with common section $T(\eta_A)$. Their coequalizer is $a : F(A) \to (A,a)$ in $\M(T)$. This is the \emph{canonical presentation} of the $T$-module $(A,a)$ (cf. \cite[VI.7]{ML98}).
\end{rem}
 
\begin{lemma}[{\cite[Corollary 2]{Linton}}] \label{modcomplete}
Let $T$ be a monad on a category $\C$ with coproducts. Then the following are equivalent:
\begin{enumerate}
\item $\M(T)$ is cocomplete.
\item $\M(T)$ has reflexive coequalizers.
\end{enumerate}
\end{lemma}

\begin{proof}
Only one direction is non-trivial. Assume that $\M(T)$ has reflexive coequalizers. Since $F : \C \to \M(T)$ is left adjoint to the forgetful functor, it preserves arbitrary coproducts. Therefore, coproducts of free $T$-modules exist. Let $((A_i,a_i))_{i \in I}$ be a family of $T$-modules. The coequalizer of the reflexive pair $F(T(A_i)) \rightrightarrows F(A_i)$ is $(A_i,a_i)$. Hence, the coequalizer of the reflexive pair $\coprod_i F(T(A_i)) \rightrightarrows \coprod_i F(A_i)$ is a coproduct of the $(A_i,a_i)$. Hence, coproducts exist in $\M(T)$. Finally, every coequalizer can be reduced to a reflexive one.
\end{proof}

The proof of the Lemma also shows:
 
\begin{lemma} \label{modcoco}
Let $T$ be a monad on a category $\C$ with coproducts such that $\M(T)$ is cocomplete. For a functor $H : \M(T) \to \D$ the following are equivalent:
\begin{enumerate}
\item $H$ is cocontinuous.
\item $HF : \C \to \D$ preserves coproducts and $H$ preserves reflexive coequalizers.
\end{enumerate}
\end{lemma}

\begin{prop}[{\cite[Proposition 4]{Linton}}] \label{modepi}
Let $\C$ be a complete and co-well-powered category with an epi-mono factorization system (for example if $\C$ is locally presentable). Let $T$ be a monad on $\C$ which preserves the epi-mono factorizations up to isomorphism. Then $\M(T)$ has coequalizers. Thus, if $\C$ is cocomplete, then $\M(T)$ is cocomplete.
\end{prop}

\begin{ex} \label{setcomp}
For $\C=\Set$ the assumptions in \autoref{modepi} are always satisfied (see \cite[Example, pp. 89-90]{Linton}). Thus, the category of modules over \emph{any monad} on $\Set$ is cocomplete (and of course also complete).
\end{ex}

\begin{prop}[Crude Monadicity Theorem,{\cite[Theorem 1.1.2]{John02}}] \label{crude}
Let $\D$ be a category with reflexive coequalizers. Let $U : \D \to \C$ be a conservative right adjoint functor which creates reflexive coequalizers. Then $U$ is monadic.
\end{prop}
 
\begin{lemma} \label{create}
Let $\C$ be a category with reflexive coequalizers. Consider a monadic functor $U : \D \to \C$ with associated monad $T : \C \to \C$. The following are equivalent:
\begin{enumerate}
\item $U$ creates reflexive coequalizers.
\item $U$ preserves reflexive coequalizers.
\item $T$ preserves reflexive coequalizers.
\end{enumerate}
\end{lemma}

\begin{proof}
$1. \Rightarrow 2.$ follows since $U$ is conservative and $\C$ has reflexive coequalizers. $2. \Rightarrow 3.$ follows from $T=UF$. The remaining direction $3. \Rightarrow 1.$ follows from \cite[Proposition 3]{Linton}.
\end{proof}

I have learned the following result from Todd Trimble.

\begin{prop} \label{finmon}
Every finitary monad on $\Set$ preserves reflexive coequalizers.
\end{prop}

\begin{proof}
Let $T$ be a finitary monad on $\Set$. Then we have the following coend expression (\cite[Proposition 4.1.3]{Dur07}):
\[T = \int^{n \in \N} T(n) \times \Hom(n,-)\]
Thus, it suffices to prove that $\Hom(n,-)$ preserves reflexive coequalizers. But this functor is isomorphic to the composite of the diagonal functor $\Set \xrightarrow{\Delta} \Set^n$, which is left adjoint to the product functor, and the $n$-fold product functor $\Set^n \to \Set$, which preserves reflexive coequalizers according to \autoref{bifunktor}.
\end{proof}

\begin{ex} \label{power}
Not every monad on $\Set$ preserves reflexive coequalizers. Consider the example from \autoref{reflexivex}.
\begin{enumerate}
\item If $I$ is an infinite set, then the coequalizer of
\[\Hom(I,\N + \N) \rightrightarrows \Hom(I,\N)\]
equals $\Hom(I,\N) / {\sim}$, where $f \sim g$ if $\sup_{i \in I} |f(i)-g(i)| < \infty$. This set is infinite, but $\Hom(I,\{0\})$ just has one element. Thus, $\Hom(I,-)$ does not preserve coequalizers.
\item Consider the power set monad $\wp$. The coequalizer of $\wp(\N+\N) \rightrightarrows \wp(\N)$ is $\wp(\N)/{\sim}$, where $\sim$ is the equivalence relation generated by
\[A \cup B \sim A \cup (B+1).\]
Then the class of $\emptyset$ consists only of $\emptyset$ and the class of a finite subset consists only of finite subsets. It follows that $\wp(\N)/{\sim}$ has at least three elements, namely $[\emptyset],[\{0\}],[\N]$, so that it cannot coincide with $\wp(\{0\})$.
\end{enumerate}
\end{ex}

We have the following universal property of $\M(T)$.

\begin{prop} \label{modT}
Let $\C$ be a cocomplete category. Let $T$ be a monad on $\C$ preserving reflexive coequalizers. Then $\M(T)$ is a cocomplete category and $F : \C \to \M(T)$ induces, for every cocomplete category $\D$, an equivalence of categories between the category $\Hom_c(\M(T),\D)$ and the category of pairs $(G,\rho)$, where $G \in \Hom_c(\C,\D)$ and $\rho : G T \to G$ is a right action.
\end{prop}

\begin{proof}
The category $\M(T)$ is  cocomplete by \autoref{modcomplete} and \autoref{create}. The multiplication of the monad $\mu : T^2 \to T$ induces a right action $FT \to F$. This induces, for every cocontinuous functor $H : \M(T) \to \D$, a right action $HFT \to HF$ on the cocontinuous functor $HT : \C \to \D$. Conversely, consider a cocontinuous functor $G : \C \to \D$ with a right action $\rho : GT \to G$. We want to find a cocontinuous functor $H : \M(T) \to \D$ with $HF \cong G$ (commuting with the $T$-actions). According to \autoref{canpres} we have no choice but to define $H(A,a)$ to be the coequalizer of $G(T(A)) \rightrightarrows G(A)$, the one morphism being $G(a)$ and the other one $\rho_A$. Then $H : \M(T) \to \D$ is a functor. For $A \in \C$ we have a coequalizer diagram
\[\xymatrix{G(T^2(A)) \ar@<1ex>[r]^{\rho_{T(A)}} \ar@<-1ex>[r]_{G(\mu_A)} & G(T(A)) \ar[r]^{\rho_A} & G(A).}\]
In fact, it is a \emph{split} coequalizer via the sections $G(\eta_A) : G(A) \to G(T(A))$ and $G(T(\eta_A)) : G(T(A)) \to G(T^2(A))$. This proves $HF \cong G$. We still have to prove that $H$ is cocontinuous. By \autoref{modcoco} it suffices to check that $H$ preserves reflexive coequalizers. Let $(A_2,a_2) \rightrightarrows (A_1,a_1) \to (A_0,a_0)$ be a reflexive coequalizer. Since $T$ preserves reflexive coequalizers, so does $U$ by \autoref{create}. Therefore, $A_2 \rightrightarrows A_1 \to A_0$ is a reflexive coequalizer. This remains true after applying $G$ or $GT$. We get the following commutative diagram:
\[\xymatrix{G(T(A_2)) \ar@<1ex>[r] \ar@<-1ex>[r]  \ar@<1ex>[d] \ar@<-1ex>[d] & G(A_2)  \ar@<1ex>[d] \ar@<-1ex>[d] \ar[r] & H(A_2,a_2) \ar@<1ex>[d] \ar@<-1ex>[d] \\
G(T(A_1)) \ar@<1ex>[r] \ar@<-1ex>[r] \ar[d] & G(A_1) \ar[d] \ar[r] & H(A_1,a_1) \ar[d]  \\ 
G(T(A_0)) \ar@<1ex>[r] \ar@<-1ex>[r]  & G(A_0) \ar[r] & H(A_0,a_0).}\]
The first and the second column are coequalizers. The rows are coequalizers by definition of $H$. Hence, the third column is also a coequalizer. This finishes the proof.
\end{proof}

\begin{rem}
We could not find this universal property in the literature. According to Mike Shulman it is known when $T$ is cocontinuous (which is a too strong assumption for our purposes). If $T$ preserves reflexive coequalizers, then the corresponding universal property dealing with functors preserving reflexive coequalizers is also known.
\end{rem}


\section{Monoidal monads}

For the convenience of the reader, we review the notion of monoidal monads (see for example \cite{GLN}). In order to simplify the notation, we will often suppress the associators of monoidal categories. Besides, the notation in complicated diagrams will be quite sloppy: The morphisms are labelled only by their primary constituents. For example we will often abbreviate something like $T(f) \otimes \mathrm{id}$ by $f$. For a component $\sigma_A$ of a natural transformation $\sigma$ we will often just write $\sigma$. The precise meaning will be clear from the context.
 
\begin{defi}[Monoidal monads] \label{monoidal-monad}
Let $\C$ be a monoidal category (not assumed to be symmetric). A \emph{monoidal monad} on $\C$ is a monad object in the $2$-category of lax monoidal endofunctors of $\C$. That is, it consists of a monad $(T,\mu,\eta)$ on the underlying category and the structure of a lax monoidal functor on $T$, such that $\mu$ and $\eta$ are lax monoidal transformations.
 
Let us write down what this means in concrete terms. The lax monoidal structure is given by a morphism $u : 1 \to T(1)$ together with a natural transformation $d_{A,B} : T(A) \otimes T(B) \to T(A \otimes B)$, such that the diagrams

\[\xymatrix@C=40pt{T(A) \otimes T(B) \otimes T(C) \ar[r]^-{d_{A,B}} \ar[d]_{d_{B,C}} & T(A \otimes B) \otimes T(C) \ar[d]^{d_{A \otimes B,C}} \\ T(A) \otimes T(B \otimes C) \ar[r]_{d_{A,B \otimes C}} & T(A \otimes B \otimes C) }\]

\[\xymatrix{T(A) \otimes 1 \ar[r]^-{u} \ar[d]_{\cong} & T(A) \otimes T(1) \ar[d]^{d_{A,1}} \\ T(A) \ar[r]^-{\cong} & T(A \otimes 1)} ~~~~~~~~ \xymatrix{1 \otimes T(A) \ar[r]^-{u} \ar[d]_{\cong} & T(1) \otimes T(A) \ar[d]^{d_{1,B}} \\ T(A) \ar[r]^-{\cong} & T(1 \otimes A)}\]

commute. That $\eta : \id_\C \to T$ is a lax monoidal transformation means that $u=\eta_1$ and that 
\[\xymatrix@R=30pt@C=40pt{A \otimes B \ar[d]_{\eta_A \otimes \eta_B} \ar[dr]^{\eta_{A \otimes B}} \\ T(A) \otimes T(B) \ar[r]_-{d_{A,B}} & T(A \otimes B)}\]
commutes. In particular $u$ may be removed from the data and replaced by $\eta_1$. The following commutative diagrams express that $\mu : T^2 \to T$ is a lax monoidal transformation.
\[\xymatrix@C=40pt{T^2(A) \otimes T^2(B) \ar[rr]^{\mu_A \otimes \mu_B} \ar[d]_{d_{T(A),T(B)}} && T(A) \otimes T(B) \ar[d]^{d_{A,B}} \\ T(T(A) \otimes T(B)) \ar[r]_-{T(d_{A,B})} & T^2(A \otimes B)  \ar[r]_{\mu_{A \otimes B}} & T(A \otimes B)}\]
\[\xymatrix{1 \ar[r]^-{u} \ar[d]_{u} & T(1) \\ T(1) \ar[r]_{T(u)} & T^2(1) \ar[u]_{\mu_1}}\]
Actually the second diagram is superfluous, it already follows from $u=\eta_1$ and the monad axioms. In the following, we will use these diagrams without further mentioning.

There is an obvious notion of a morphism between monoidal monads on $\C$. Usually we will abbreviate $(T,\mu,\eta,d)$ by $T$.
\end{defi}

Notice that if $T$ is a monoidal monad on $\C$, then (since $T$ is lax monoidal) $T$ lifts to an endofunctor of the category of monoid objects (aka algebra objects) in $\C$.
  
\begin{rem}[Decomposition] \label{decompose}
The natural transformation $d$ can be divided into two parts: Define $\sigma_{A,B} : A \otimes T(B) \to T(A \otimes B)$ to be the composition
\[\xymatrix{A \otimes T(B) \ar[r]^-{\eta_A} & T(A) \otimes T(B) \ar[r]^-{d_{A,B}} & T(A \otimes B).}\]
In a similar way, one defines $\sigma'_{A,B} : T(A) \otimes B \to T(A \otimes B)$. Then $d_{A,B}$ equals the composition
\[\xymatrix@C=40pt{T(A) \otimes T(B) \ar[r]^{\sigma_{T(A),B}} & T(T(A) \otimes B) \ar[r]^{T(\sigma'_{A,B})} & T^2(A \otimes B) \ar[r]^{\mu_{A \otimes B}} & T(A \otimes B).}\]
This follows from the following commutative diagram:
\[\xymatrix@C=45pt@R=35pt{T(A) \otimes T(B) \ar[d]_{\eta_{T(A)}} \ar@{=}[r] & T(A) \otimes T(B) \ar[rr]^{d_{A,B}}   & & T(A \otimes B) \\
T^2(A) \otimes T(B) \ar[r]^-{\eta_B} \ar[d]_{d_{T(A),B}} & T^2(A) \otimes T^2(B) \ar[d]^{d_{T(A),T(B)}} \ar[u]_{\mu_{A} \otimes \mu_{B}} & & \\
T(T(A) \otimes B) \ar[r]^-{\eta_B} & T(T(A) \otimes T(B)) \ar[rr]^-{d_{A,B}} & &  T^2(A \otimes B) \ar[uu]_{\mu_{A \otimes B}}}\]
In the same way one checks that $d_{A,B}$ equals the composition
\[\xymatrix@C=40pt{T(A) \otimes T(B) \ar[r]^{\sigma'_{A,T(B)}} & T(A \otimes T(B)) \ar[r]^{T(\sigma_{A,B})} & T^2(A \otimes B) \ar[r]^{\mu_{A \otimes B}} & T(A \otimes B).}\]
Thus, $d$ can be recovered from $\sigma$ and $\sigma'$. One can translate the defining diagrams of a monoidal monad in terms of $\sigma$ and $\sigma'$. This results in the following equivalent description of monoidal monads.
\end{rem}

\begin{defi}[Strengths] \label{strength}
Let $T$ be a monad on the underlying category of a monoidal category $\C$. A \emph{strength} is a natural transformation
\[\sigma_{A,B} : A \otimes T(B) \to T(A \otimes B)\]
such that the following diagrams commute, expressing compatibility with the monoidal structure of $\C$ as well as the monad structure of $T$:
\[\xymatrix{1 \otimes T(A) \ar[r]^-{\sigma_{1,A}} \ar[dr]_{\cong} & T(1 \otimes A) \ar[d]^{\cong} \\ & T(A)}
~~~~~~ \xymatrix{A \otimes B \ar[r]^-{\eta_B} \ar[dr]_{\eta_{A \otimes B}} & A \otimes T(B) \ar[d]^{\sigma_{A,B}} \\ & T(A \otimes B) } \]
\[\xymatrix{A \otimes B \otimes T(C) \ar[rr]^{\sigma_{A \otimes B,C}} \ar[dr]_{\sigma_{B,C}} & & T(A \otimes B \otimes C) \\
 & A \otimes T(B \otimes C) \ar[ur]_{~\sigma_{A,B \otimes C}} & }\]
\[\xymatrix@C=40pt{A \otimes T^2(B) \ar[d]_{\mu_B} \ar[r]^-{\sigma_{A,T(B)}} & T(A \otimes T(B)) \ar[r]^-{\sigma_{A,B}} & T^2(A \otimes B) \ar[d]^{\mu_{A \otimes B}} \\ A \otimes T(B) \ar[rr]_{\sigma_{A,B}} & & T(A \otimes B)}\]
There is an obvious notion of a \emph{costrength} $\sigma'_{A,B} : T(A) \otimes B \to T(A \otimes B)$, which is a natural transformation satisfying the analoguous four diagrams. We say that $\sigma$ and $\sigma'$ are \emph{compatible} if the following diagrams commute:
\[\xymatrix@C=10pt@R=10pt{
A \otimes T(B) \otimes C \ar[r]^{\sigma_{A,B}} \ar[dd]_{\sigma'_{B,C}} & T(A \otimes B) \otimes C \ar[dd]^{\sigma'_{A \otimes B,C}} \\
~~~~~~~~~~~~~~~~~~~~~~~~~~~~~~~~~~~~~~~~~~~~~~~~~~~ & ~~~~~~~~~~~~~~~~~~~~~~~~ ~~~~~~~~~~~~~~~( \heartsuit )\\
A \otimes T(B \otimes C) \ar[r]_{\sigma_{A,B \otimes C}} & T(A \otimes B \otimes C)}\]
\[\xymatrix@C=15pt{ & T(T(A) \otimes B) \ar[rr]^-{\sigma'_{A,B}} & & T^2(A \otimes B) \ar[dr]^{\mu_{A \otimes B}} & \\
~~~~~ T(A) \otimes T(B) \ar[ur]^{\sigma_{T(A),B}} \ar[dr]_{\sigma'_{A,T(B)}} &  & & & T(A \otimes B)    ~~~~~~ ( \lozenge  ) \\
& T(A \otimes T(B)) \ar[rr]_{\sigma_{A,B}} & & T^2(A \otimes B) \ar[ur]_{\mu_{A \otimes B}}  &} \]

\end{defi}

\begin{ex} \label{monoid-monad}
Let $M$ be a monoid (i.e. algebra) object in $\C$ and consider the functor $T = - \otimes M : \C \to \C$ with its usual monad structure, so that $T$-modules are precisely the objects with a right $M$-action. Then the associator $A \otimes (B \otimes M) \cong (A \otimes B) \otimes M$ is a strength for $T$. But in general there is no costrength.
\end{ex}
 
\begin{prop}[{\cite[A.1 and A.2]{GLN}}] \label{monoidal-monad-equivalence}
Let $T$ be a monad on the underlying category of a monoidal category $\C$. Then there is a 1:1 correspondence between natural transformations $d : T(-) \otimes T(-) \to T(- \otimes -)$ making $T$ a monoidal monad and pairs consisting of a strength $\sigma$ and a costrength $\sigma'$ such that $\sigma$ and $\sigma'$ are compatible with each other.
\end{prop}

\begin{prop} \label{enriched}
Let $\C$ be a closed monoidal category. Let $T$ be a monad on the underlying category. Then there is a bijection between strengths of $T$ and natural transformations $\tau_{A,B} : \HOM(A,B) \to \HOM(T(A),T(B))$ of functors $\C^{\op} \times \C \to \C$ making $T$ a $\C$-enriched functor.
\end{prop}

\begin{proof}
We only sketch the construction of the bijection. The details just involve diagram chases and can be found in \cite{Kock70} and \cite{Kock72} (at least for the symmetric case). Let us denote by $c$ the counit and by $u$ the unit of the hom-tensor adjunction. Given a strength $\sigma$, we define $\tau$ by the following commutative diagram:
\[\xymatrix{\HOM(A,B) \ar[d]_{u} \ar[r]^{\tau} & \HOM(T(A),T(B)) \\ \HOM(T(A),\HOM(A,B) \otimes T(A)) \ar[r]_{\sigma} & \HOM(T(A),T(\HOM(A,B) \otimes A)) \ar[u]_{c}} \]
Conversely, if $T$ has an enrichment $\tau$, define the strength $\sigma$ by the following commutative diagram:
\[\xymatrix{A \otimes T(B) \ar[d]_{u} \ar[r]^{\sigma} & T(A \otimes B) \\  \HOM(B,A \otimes B) \otimes T(B) \ar[r]_-{\tau} & \HOM(T(B),T(A \otimes B)) \otimes T(B)  \ar[u]_{c}}\]
\end{proof}

\begin{rem} \label{costrength-enriched}
In the setting of \autoref{enriched} one can also find a bijection between costrengths and natural transformations
\[\tau'_{A,B} : T(\HOM(A,B)) \to \HOM(A,T(B))\]
satisfying the obvious compatibility diagrams (cf. \cite[Section 1]{Kock71}). Given a costrength $\sigma'$, one defines $\tau'$ by the following commutative diagram:
\[\xymatrix{T(\HOM(A,B)) \ar[d]_{u} \ar[r]^{\tau'} & \HOM(A,T(B))  \\ \HOM(A,T(\HOM(A,B)) \otimes A) \ar[r]_{\sigma'} & \HOM(A,T(\HOM(A,B) \otimes A)) \ar[u]_{c}}\]
Given $\tau'$, one defines $\sigma'$ by the following commutative diagram:
\[\xymatrix{T(A) \otimes B \ar[r]^{\sigma'} \ar[d]_{u} & T(A \otimes B) \\ T(\HOM(B,A \otimes B)) \otimes B \ar[r]_{\tau'} & \HOM(B,T(A \otimes B)) \otimes B \ar[u]_{c}}\]
\end{rem}

  

\begin{rem} \label{Thom}
Let $T$ be a monoidal monad on a monoidal closed category. If $A \in \C$ and $(B,b) \in \M(T)$, then $\HOM(A,B) \in \C$ carries the structure of a $T$-module as follows:
\[T(\HOM(A,B)) \xrightarrow{\tau'} \HOM(A,T(B)) \xrightarrow{b} \HOM(A,B).\]
By abuse of notation this $T$-module will be also denoted by $\HOM(A,B)$.
\end{rem}
  
\begin{defi}[Symmetric monoidal monads] \label{symmetric-monoidal-monad}
A \emph{symmetric monoidal monad} on a symmetric monoidal category $\C$ is a monad in the $2$-category of symmetric lax monoidal endofunctors of $\C$. Specifically this means that $T$ is a monoidal monad on (the underlying monoidal category of) $\C$ as in \autoref{monoidal-monad} such that moreover the diagram
\[\xymatrix@C=40pt{T(A) \otimes T(B) \ar[d]_{S_{T(A),T(B)}} \ar[r]^-{d_{A,B}} & T(A \otimes B) \ar[d]^{T(S_{A,B})} \\
T(B) \otimes T(A) \ar[r]_-{d_{B,A}} & T(B \otimes A)}\]
commutes.
\end{defi}

Notice that $T$ lifts to an endofunctor of the category of commutative monoid objects of $\C$ (using the symmetric lax monoidal structure).
 
\begin{prop}[{\cite[A.3]{GLN}}]
Let $T$ be a monad on the underlying category of a symmetric monoidal category $\C$. Then there is a 1:1 correspondence between natural transformations $d : T(-) \otimes T(-) \to T(- \otimes -)$ making $T$ a symmetric monoidal monad and strengths $\sigma : - \otimes T(-) \to T(- \otimes -)$ such that, for the natural transformation $\sigma' : T(-) \otimes - \to T(- \otimes -)$ defined by
\[\xymatrix@C=40pt{T(A) \otimes B \ar[r]^{S_{T(A),B}} & B \otimes T(A) \ar[r]^{\sigma_{B,A}} & T(B \otimes A) \ar[r]^{S_{B,A}} & T(A \otimes B)},\]
the diagram $( \lozenge  )$ as in \autoref{strength} commutes.
\end{prop}

Note that the diagram $( \heartsuit )$ in \autoref{strength} commutes automatically by naturality of the symmetry and the properties of the strength $\sigma$.

\begin{ex}[The case of sets] \label{generalized-ring1}
Let $T$ be a monad on $\Set$. By \autoref{enriched} $T$ has a unique strength. It may be described as follows: For $a \in A$ the map $B \to A \times B, b \mapsto (a,b)$ induces a map $T(B) \to T(A \times B)$. By varying $a$ we get a map $\sigma_{A,B} : A \times T(B) \to T(A \times B)$. The corresponding costrength $\sigma'$ can be described similarly. Then $T$ is a symmetric monoidal monad if and only if for all $T$-modules $(A,a)$ and all sets $X,Y$ and elements $u \in T(X)$, $v \in T(Y)$ with associated maps $\tilde{u} : A^X \to A$, $\tilde{v} : A^Y \to A$ (defined by $\tilde{u}(f)=a(T(f)(u))$), the diagram
\[\xymatrix{A^{X \times Y} \ar[r]^{\cong} \ar[d]_{\cong} & (A^X)^Y \ar[r]^{\tilde{u}} & A^Y \ar[dd]^{\tilde{v}} \\
(A^Y)^X \ar[d]_{\tilde{v}} & & \\ A^X \ar[rr]_{\tilde{u}} & & A }\]
commutes. This coincides with the usual notion of a \emph{commutative monad}, used for example in Durov's thesis \cite[Chapter 5]{Dur07}. It is equivalent to the condition that for all $T$-modules $(A,a)$, $(B,b)$ the set of homomorphisms $\Hom_T(A,B)$ is a submodule of the $T$-module $\Hom(A,B)$.
 
For a specific class of examples, let $R$ be ring. Then, the forgetful functor $\M(R) \to \Set$ is monadic (for example by \autoref{crude}). The corresponding monad $T$ on $\Set$ maps a set $X$ to the set of formal finite linear combinations $\sum_{x \in X} x \cdot \lambda_x$ with coefficients $\lambda_x \in R$. Then one checks that $T$ is commutative if and only if $R$ is commutative. Moreover, this yields a fully faithful functor of the category of commutative rings into the category of commutative monads on $\Set$. This is one of Durov's motivations to define a \emph{generalized commutative ring} to be a finitary commutative monad on $\Set$. 
\end{ex}

\begin{ex} \label{monbsp}
Let $\C$ be a symmetric monoidal category and $M \in \C$ be a monoid object (aka algebra object) of the underlying monoidal category. The monad $- \otimes M$ has a strength (\autoref{monoid-monad}), with corresponding costrength
\[(A \otimes M) \otimes B \cong B \otimes (A \otimes M) \cong (B \otimes A) \otimes M \cong (A \otimes B) \otimes M.\]
With these data $- \otimes M$ is a symmetric monoidal monad on $\C$ if and only if the diagram $(\lozenge)$ commutes, which simplifies (using $A=B=1_\C)$ to the commutativity of the diagram
\[\xymatrix{M \otimes M \ar[rr]^{\mu} \ar[dr]_{S_{M,M}} & & M \\ & M \otimes M \ar[ur]_{\mu} & }\]
i.e. to the commutativity of $M$. This induces a functor from the category of commutative monoids in $\C$ to the category of symmetric monoidal monads on $\C$, which is fully faithful and has a right adjoint, mapping $T$ to the commutative monoid $T(1)$. For example, every commutative monoid $M$ in $\Set$ gives a generalized commutative ring. Every commutative $R$-algebra yields a symmetric monoidal monad on $\M(R)$.
\end{ex}


\section{Tensor product of modules} \label{tpmonad}

The idea of the tensor product of modules classifying bihomomorphisms over a monoidal monad goes back to Linton (\cite[\para 1, Remark]{Linton}) and has been studied by Kock (\cite{Kock71bilin}, \cite{Kock12}). For \emph{concrete} categories the corresponding notions were studied in detail by Banaschewski and Nelson (\cite{Ban76}). In the following the reader may find it helpful to have in mind \autoref{generalized-ring1} and \autoref{monbsp}.

\begin{defi}[Bihomomorphisms]
Let $\C$ be a monoidal category and let $T$ be a monoidal monad on $\C$. Let $(A,a),(B,b),(C,c)$ be $T$-modules. A morphism $f : A \otimes B \to C$ in $\C$ is called a \emph{bihomomorphism} (with respect to the actions $a,b,c$) if it is a homomorphism in each variable, i.e. the diagrams
\[\xymatrix{T(A) \otimes B \ar[r]^{\sigma}  \ar[d]_{a \otimes B} & T(A \otimes B) \ar[r]^-{T(f)} & T(C) \ar[d]^{c} \\ A \otimes B \ar[rr]^{f} && C}\]
\[\xymatrix{A \otimes T(B) \ar[r]^{\sigma'}  \ar[d]_{A \otimes b} & T(A \otimes B) \ar[r]^-{T(f)} & T(C) \ar[d]^{c} \\ A \otimes B \ar[rr]^{f} && C}\]
commute. Clearly, if $f :A \otimes B \to C$ is a bihomomorphism, then the same is true for $h \circ f : A \otimes B \to C'$ for every homomorphism $h : (C,c) \to (C',c')$. Hence, we obtain a functor
\[\BiHom((A,a),(B,b),-) : \M(T) \to \Set\]
which sends a $T$-module $(C,c)$ to the set of bihomomorphisms $A \otimes B \to C$ with respect to $a,b,c$.
\end{defi}

\begin{defi}[Tensor product]
If the functor $\BiHom((A,a),(B,b),-)$ is representable, we call a representing object a \emph{tensor product} of $(A,a)$ and $(B,b)$ and denote it by $(A,a) \otimes_T (B,b)$. We will abbreviate the underlying object $U((A,a) \otimes_T (B,b))$ by $A \otimes_T B$. It comes equipped with a universal bihomomorphism $A \otimes B \to A \otimes_T B$, also denoted by $\otimes$.
\end{defi}

\begin{prop}[{\cite[Proposition 2.1.2]{Sea13}}] \label{bihom}
Let $(A,a),(B,b),(C,c)$ be $T$-modules. A morphism $f : A \otimes B \to C$ is a bihomomorphism if and only if the diagram
\[\xymatrix@C=30pt{T(A) \otimes T(B) \ar[r]^-{d_{A,B}} \ar[d]_{a \otimes b} & T(A \otimes B) \ar[r]^-{T(f)} & T(C) \ar[d]^{c} \\
A \otimes B \ar[rr]^{f}  && C}\]
commutes.
\end{prop}

\begin{proof}
If the diagram commutes, then precomposing with $\eta_A \otimes T(B)$ implies that $f$ is a homomorphism in $B$. Similarly, we see that $f$ is a homomorphism in $A$ by precomposing with $T(A) \otimes \eta_B$. Now assume that $f$ is a bihomomorphism. Then each part in the following diagram commutes (we will use sloppy notation again).

\[\xymatrix{T(A) \otimes T(B)  \ar[rr]^{d} \ar[dr]^{\sigma'} \ar[dd]^{b} & & T(A \otimes B) \ar[rr]^{f} && T(C) \ar[ddd]^{c} \\
 & T(A \otimes T(B)) \ar[r]^{\sigma} \ar[d]^{b} & T^2(A \otimes B) \ar[u]^{\mu} \ar[r]^-{f} & T^2(C) \ar[ur]^{\mu} \ar[d]^{c} &  \\
 T(A) \otimes B \ar[r]^{\sigma'} \ar[d]^{a} & T(A \otimes B) \ar[rr]^{f} && T(C) \ar[dr]^{c} & \\
 A \otimes B \ar[rrrr]^{f} &&&& C }\]

Hence, the outer rectangle also commutes, which is precisely the claim.
\end{proof}

\begin{cor}[Existence of tensor products] \label{tensorex}
If $\M(T)$ has reflexive coequalizers, then the tensor product of two $T$-modules $(A,a) \otimes_T (B,b)$ always exists, namely it is given by the coequalizer in $\M(T)$ of the two homomorphisms $F(T(A) \otimes T(B)) \rightrightarrows F(A \otimes B)$ defined as $F(a \otimes b)$ resp. as the extension of $d_{A,B} : T(A) \otimes T(B) \to T(A \otimes B)$.
\end{cor}

In fact, these homomorphisms are reflexive with section $T(\eta_A \otimes \eta_B)$. Next, we check that the tensor product satisfies the expected properties.
 
\begin{prop}[{\cite[Theorem 2.5.5]{Sea13}}] \label{eins}
Let $T$ be a monoidal monad on a monoidal category $\C$. Then for every $T$-module $(A,a)$ there is a unique isomorphism $(A,a) \otimes_T F(1) \cong (A,a)$ such that
\[\xymatrix@C=35pt{A \otimes_T T(1) \ar[r]^-{\cong} & A \\ A \otimes T(1) \ar[u]^{\otimes} & A \otimes 1  \ar[l]^-{A \otimes \eta_1} \ar[u]_{\cong}}\]
commutes. A similar statement holds for $F(1) \otimes_T (A,a) \cong (A,a)$.
\end{prop}
 
\begin{proof}
Let $(C,c)$ be some $T$-module. We have to show that there is a natural bijection between bihomomorphisms $A \otimes T(1) \to C$ and homomorphisms $A \to C$. Our notation will be quite sloppy, for example we will abbreviate the morphism $A \otimes \eta_1 : A \otimes 1 \to A \otimes T(1)$ by $\eta$.

Given a bihomomorphism $\tilde{f} : A \otimes T(1) \to C$, define $f : A \to C$ as the composition
\[f : A \xrightarrow{\rho^{-1}} A \otimes 1 \xrightarrow{\eta} A \otimes T(1) \xrightarrow{\tilde{f}} C.\]
The following diagram commutes:
\[\xymatrix@R=35pt@C=40pt{
T(A) \ar[r]^{\rho^{-1}} \ar[dr]_{\rho^{-1}} \ar[dd]^{a} & T(A \otimes 1) \ar[r]^-{\eta} & T(A \otimes T(1)) \ar[r]^-{\tilde{f}} & T(C) \ar[dd]^{c} \\
 & T(A) \otimes 1 \ar[d]^{a} \ar[r]^-{\eta} \ar[u]_{\sigma'} & T(A) \otimes T(1) \ar[u]_{\sigma'} \ar[d]^{a} & \\
 A \ar[r]^{\rho^{-1}} & A \otimes 1 \ar[r]^-{\eta} & A \otimes T(1) \ar[r]^-{\tilde{f}} & C
}\]
The commutativity of the outer rectangle says that $f$ is a homomorphism.

Conversely, let $f : A \to C$ be a homomorphism. Define $\tilde{f} : A \otimes T(1) \to C$ as the composition
\[\tilde{f} : A \otimes T(1) \xrightarrow{\sigma} T(A \otimes 1) \xrightarrow{\rho} T(A) \xrightarrow{f} T(C) \xrightarrow{c} C,\]
or equivalently as the composition
\[\tilde{f} : A \otimes T(1) \xrightarrow{\sigma} T(A \otimes 1) \xrightarrow{\rho} T(A) \xrightarrow{a} A \xrightarrow{f} C.\]
We claim that $\tilde{f}$ is a bihomomorphism. We show first that $\tilde{f}$ is a homomorphism in $A$, i.e. that the diagram (D)
\[\xymatrix@R=35pt{T(A) \otimes T(1) \ar[r]^{\sigma} \ar[d]^{a} & T(A \otimes T(1)) \ar[r]^-{\tilde{f}} & T(C) \ar[d]^{c} \\
A \otimes T(1) \ar[rr]^-{\tilde{f}} && C}\]
commutes. Observe that the following two diagrams commute:
\[\xymatrix@R=35pt{
T(A) \otimes T(1) \ar[r]^{\sigma} \ar[d]^{a} & T(T(A) \otimes 1) \ar[r]^-{\rho} \ar[d]^{a} & T^2(A) \ar[r]^{\mu} \ar[d]^{a} & T(A) \ar[d]^{a}\\
A \otimes T(1) \ar[r]^{\sigma} \ar[d]^{f} & T(A \otimes 1) \ar[r]^-{\rho} \ar[d]^{f} & T(A) \ar[r]^{a} \ar[d]^{f} & A \ar[d]^{f} \\
C \otimes T(1) \ar[r]^{\sigma} & T(C \otimes 1) \ar[r]^-{\rho} & T(C) \ar[r]^{c} & C
}\]
\[\xymatrix{
T(A) \otimes T(1) \ar[rrr]^{\sigma'} \ar[dd]^{\sigma} &&& T(A \otimes T(1)) \ar[d]^{\sigma} \ar@/^4pc/[dddd]^{\tilde{f}} \\
 & T^2(A \otimes 1) \ar[r]^{\mu} \ar[d]^{\rho} & T(A \otimes 1) \ar[d]^{\rho} & T^2(A \otimes 1) \ar[l]^{\mu} \ar[d]^{\rho} \\
T(T(A) \otimes 1) \ar[r]^-{\rho} \ar[ur]^{\sigma'} & T^2(A) \ar[r]^{\mu} & T(A) \ar[d]^{a} & T^2(A) \ar[l]^{\mu} \ar[d]^{a} \\
 & & A \ar[d]^{f} & T(A) \ar[d]^{f} \ar[l]^{a} \\
 & & C & T(C) \ar[l]^{c}
 }\] 
The first diagram implies (with sloppy notation) $\tilde{f} a = c f \rho \sigma a = f a \mu \rho \sigma$ and the second one implies $f a \mu \rho \sigma = c \tilde{f} \sigma'$, thus (D) commutes. The following commutative diagram shows that $\tilde{f}$ is a homomorphism in $T(1)$:
\[\xymatrix@C=20pt{
A \otimes T^2(1) \ar[r]^-{\sigma} \ar[d]^{\mu} & T(A \otimes T(1)) \ar[r]^{\sigma} \ar@/^2pc/[rrrr]^{\tilde{f}} & T^2(A \otimes 1) \ar[r]^-{\rho} \ar[d]^{\mu} & T^2(A) \ar[d]^{\mu} \ar[r]^{a} & T(A) \ar[d]^{a} \ar[r]_{f} & T(C) \ar[d]^{c} \\
A \otimes T(1) \ar[rr]^{\sigma}  \ar@/_2pc/[rrrrr]_{\tilde{f}} && T(A \otimes 1)  \ar[r]^-{\rho} & T(A) \ar[r]^{a}   & A \ar[r]^{f} & C
}\]
Thus, $\tilde{f}$ is a bihomomorphism. Let us show that these constructions are inverse to each other: Given a homomorphism $f : A \to C$, the associated bihomomorphism $\tilde{f}$ yields the homomorphism
\[A \xrightarrow{\rho^{-1}} A \otimes 1 \xrightarrow{\eta} A \otimes T(1) \xrightarrow{\tilde{f}} C\]
and the following commutative diagram shows that it equals $f$:
\[\xymatrix{
 & A \otimes 1 \ar[r]^-{\eta} \ar[dr]_{\eta} \ar[dd]^{\rho} & A \otimes T(1) \ar[d]^{\sigma} \ar[r]^-{\tilde{f}} & C \\
A \ar[ur]^{\rho^{-1}} \ar@{=}[dr] & & T(A \otimes 1) \ar[d]^{\rho} & \\
 & A \ar[r]^{\eta} \ar@{=}[dr] & T(A) \ar[d]^{a} \ar[r]^{f} & T(C) \ar[uu]^{c} \ar[d]^{c} \\ 
 & & A \ar[r]^{f} & C
}\]
Conversely, let $\tilde{f} : A \otimes T(1) \to C$ be a bihomomorphism, with associated homomorphism $f$, which gives the bihomomorphism $c f \rho \sigma$. The following commutative diagram shows that it equals $\tilde{f}$.

\[\xymatrix{
T(A) \ar[r]^-{f} & T(C) \ar[r]^{c} & C \\
T(A \otimes 1) \ar[u]^{\rho} \ar[r]^-{\eta} & T(A \otimes T(1)) \ar[u]^{\tilde{f}} & \\
A \otimes T(1) \ar[u]^{\sigma} \ar[r]^-{\eta} \ar@/_1.5pc/@{=}[rr] & A \otimes T^2(1) \ar[u]^{\sigma} \ar[r]^{\mu} & A \otimes T(1) \ar[uu]^{\tilde{f}}
}\]
\end{proof}

\begin{prop}[{\cite[Proposition 2.5.2]{Sea13}}] \label{freiesTP}
Let $X,Y \in \C$. Then there is a unique isomorphism of $T$-modules $F(X \otimes Y) \cong F(X) \otimes_T F(Y)$ such that
\[\xymatrix{T(X \otimes Y) \ar[rr]^{\cong} &&  T(X) \otimes_T T(Y) \\ & T(X) \otimes T(Y) \ar[ul]^{d_{X,Y}} \ar[ur]_{\otimes} & }\]
commutes.
\end{prop}

\begin{proof}
We prove that $d_{X,Y} : T(X) \otimes T(Y) \to T(X \otimes Y)$ satisfies the universal property of the tensor product. Notice that $d_{X,Y}$ is a bihomomorphism because of \autoref{bihom} and the commutative diagram
\[\xymatrix@C=40pt{T^2(X) \otimes T^2(Y) \ar[r]^-{d} \ar[d]^{\mu \otimes \mu} & T(T(X) \otimes T(Y)) \ar[r]^-{d} & T^2(X \otimes Y) \ar[d]^{\mu} \\
T(X) \otimes T(Y) \ar[rr]^{d} && T(X \otimes Y)}\]
which is part of the definition of a monoidal monad.

Now let $(A,a)$ be some $T$-module and $f : T(X) \otimes T(Y) \to A$ be a bihomomorphism. We have to show that there is a unique homomorphism of $T$-modules $\tilde{f}: F(X \otimes Y) \to (A,a)$ such that $\tilde{f} \circ d_{X,Y} = f$. From this equation it follows
\[\tilde{f} \circ \eta_{X \otimes Y} = \tilde{f} \circ d_{X,Y} \circ (\eta_X \otimes \eta_Y) = f \circ (\eta_X \otimes \eta_Y)\]
which describes $\tilde{f}$ completely since $F(X \otimes Y)$ is the free $T$-module on $X \otimes Y$. This establishes uniqueness and for existence we have to define $\tilde{f}$ to be the extension of $f \circ (\eta_X \otimes \eta_Y) : X \otimes Y \to A$, i.e. as the homomorphism of $T$-modules
\[T(X \otimes Y) \xrightarrow{T(\eta_X \otimes \eta_Y)} T(T(X) \otimes T(Y)) \xrightarrow{T(f)} T(A) \xrightarrow{a} A.\]
We are left to prove that $\tilde{f} \circ d = f$. Since $f$ is a bihomomorphism, the diagram
\[\xymatrix@C=50pt{T^2(X) \otimes T^2(Y) \ar[r]^-{d_{T(X),T(Y)}} \ar[d]_{\mu \otimes \mu} & T(T(X) \otimes T(Y)) \ar[r]^-{T(f)} & T(A) \ar[d]^{a} \\
T(X) \otimes T(Y) \ar[rr]^{f} &&  A}\]
commutes (\autoref{bihom}). Precomposing with $T(\eta_X) \otimes T(\eta_Y)$ yields the commutative diagram
\[\xymatrix@C=50pt{T^2(X) \otimes T^2(Y) \ar[r]^-{d_{T(X),T(Y)}} & T(T(X) \otimes T(Y)) \ar[r]^-{T(f)} & T(A) \ar[d]^{a} \\
T(X) \otimes T(Y) \ar[u]^{T(\eta_X) \otimes T(\eta_Y)} \ar[rr]^{f} &&  A.}\]
By naturality of $d$ with respect to $\eta_X$ and $\eta_Y$, this simplifies to
\[\xymatrix@C=50pt{T(X \otimes Y) \ar[r]^-{T(\eta_X \otimes \eta_Y)} & T(T(X) \otimes T(Y)) \ar[r]^-{T(f)} & T(A) \ar[d]^{a} \\
T(X) \otimes T(Y) \ar[u]^{d_{X,Y}} \ar[rr]^{f} &&  A.}\]
But this means exactly $\tilde{f} \circ d_{X,Y} = f$, which finishes the proof.
\end{proof}

\begin{lemma} \label{symmetry}
Let $T$ be a symmetric monoidal monad on a symmetric monoidal category $\C$. Let $(A,a),(B,b)$ be two $T$-modules. Assuming that $(A,a) \otimes_T (B,b)$ exists, then $(B,b) \otimes_T (A,a)$ exists as well and there is a unique isomorphism $(A,a) \otimes_T (B,b) \cong (B,b) \otimes_T (A,a)$ such that the following diagram commutes.
\[\xymatrix@R=15pt{A \otimes_T B \ar[r]^{\cong} & B \otimes_T A \\ A \otimes B \ar[u]^{\otimes} \ar[r]^{S} & B \otimes A \ar[u]_{\otimes}}\]
\end{lemma}

\begin{proof}
We have to show that if $f : A \otimes B \to C$ is a bihomomorphism, then $f \circ S_{B,A} : B \otimes A \to C$ is a bihomomorphism, too. This follows from the following commutative diagram:
\[\xymatrix@R=15pt{ & T(B \otimes A)  \ar[dr]^{S} &&  \\ T(B) \otimes T(A) \ar[ru]^{d} \ar[dd]_{b \otimes a} \ar[dr]^{S} & & T(A \otimes B) \ar[r]^-{f} & T(C) \ar[dd]^{c} \\
& T(A) \otimes T(B) \ar[d]^{a \otimes b} \ar[ur]^{d} && \\ B \otimes A \ar[r]^{S} & A \otimes B \ar[rr]^{f} & & C}\]
\end{proof}

In order to get associativity of $\otimes_T$, we need a coherence property of $T$.
 
\begin{defi}[Coherent monoidal monads]
A (symmetric) monoidal monad $T$ on a (symmetric) monoidal category $\C$ is called \emph{coherent} if $\M(T)$ has reflexive coequalizers and $- \otimes_T -$ preserves them in each variable.
\end{defi}

\begin{prop}[{\cite[Theorem 2.5.5]{Sea13}}] \label{associator}
Let $T$ be a coherent monoidal monad on a monoidal category $\C$. Then for all $T$-modules $(A,a),(B,b),(C,c)$ there is a unique isomorphism of $T$-modules
\[(A,a) \otimes_T ((B,b) \otimes_T (C,c)) \cong ((A,a) \otimes_T (B,b)) \otimes_T (C,c)\]
such that the diagram
\[\xymatrix{A \otimes_T (B \otimes_T C) \ar[r]^-{\cong} & (A \otimes_T B) \otimes_T C \\
T(A) \otimes_T T(B \otimes C) \ar[u] & T(A \otimes B) \otimes_T T(C) \ar[u] \\
T(A \otimes (B \otimes C)) \ar[u] \ar[r]^{\cong} & T((A \otimes B) \otimes C) \ar[u]}\]
commutes.
\end{prop}

\begin{proof}
By \autoref{tensorex} and \autoref{canpres} we have two coequalizer diagrams
\[\xymatrix@C=40pt{F(T(B) \otimes T(C)) \ar@<1ex>[r]^-{\mu T d} \ar@<-1ex>[r]_-{F(b \otimes c)} & F(B \otimes C) \ar[r] & (B,b) \otimes_T (C,c)}\]
\[\xymatrix@C=40pt{F(T(A)) \ar@<1ex>[r]^-{\mu} \ar@<-1ex>[r]_-{T(a)} & F(A) \ar[r]^-{a} & (A,a)}\]
in $\M(T)$. Using \autoref{bifunktor} (which is applicable since $T$ is coherent) they may be tensored to a coequalizer diagram
\[\xymatrix@C=42pt{F(T(A)) \otimes_T F(T(B) \otimes T(C)) \ar@<1ex>[r]^-{\mu \otimes \mu T d} \ar@<-1ex>[r]_-{F(a) \otimes_T F(b \otimes c)} & F(A) \otimes_T F(B \otimes C) \ar[d] \\
& (A,a) \otimes_T ((B,b) \otimes_T (C,c))}\]
in $\M(T)$. By \autoref{freiesTP} this identifies with
\[\xymatrix@C=42pt{F(T(A) \otimes (T(B) \otimes T(C))) \ar@<1ex>[r]^-{x} \ar@<-1ex>[r]_-{F(a \otimes (b \otimes c))} & F(A \otimes (B \otimes C)) \ar[d]\\ &(A,a) \otimes_T ((B,b) \otimes_T (C,c)),}\]
where $x$ corresponds to
\[T(A) \otimes (T(B) \otimes T(C)) \xrightarrow{T(A) \otimes d_{B,C}} T(A) \otimes T(B \otimes C) \xrightarrow{d_{A,B \otimes C}} T(A \otimes (B \otimes C)).\]
Since $T$ is a monoidal monad with respect to $d$, this identifies under the associator isomorphisms in $\C$ with
\[(T(A) \otimes T(B)) \otimes T(C) \xrightarrow{d_{A,B} \otimes T(C)} T(A \otimes B) \otimes T(C) \xrightarrow{d_{A \otimes B,C}} T((A \otimes B) \otimes C)).\]
Likewise, $T(a \otimes (b \otimes c))$ gets identified with $T((a \otimes b) \otimes c)$. We conclude that $(A,a) \otimes_T ((B,b) \otimes_T (C,c))$ is the coequalizer of two morphisms whose coequalizer is also, by the same reasoning as above, $((A,a) \otimes_T (B,b)) \otimes_T (C,c)$.
\end{proof}

\begin{ex}[Smash products] \label{smashp}
Consider the cartesian monoidal category of topological spaces $\Top$ as well as the category of pointed spaces $\Top_*$. The forgetful functor $\Top_* \to \Top$ is monadic, the corresponding monad is given by the functor $T : \Top \to \Top$, $X \mapsto X \sqcup *$ with the obvious monad structure. Also, $T$ has a canonical monoidal structure. If $(X,*),(Y,*),(Z,*)$ are pointed spaces, a continuous map $f : X \times Y \to Z$ is a bihomomorphism if and only if $f(x,*)=*=f(*,y)$ for all $(x,y) \in X \times Y$. It follows that the tensor product $\otimes_T$ on $\Top_*$ is exactly the \emph{smash product}
\[(X,*) \wedge (Y,*) = \bigl((X \times Y) / (x,*) \simeq (*,y),(*,*)\bigr).\]
However, the smash product of pointed spaces is not associative in general! For example, $\mathds{N} \wedge (\mathds{Q} \wedge \mathds{Q})$ is not isomorphic to $(\mathds{N} \wedge \mathds{Q}) \wedge \mathds{Q}$. The canonical continuous bijection is not an isomorphism (\cite[Section 1.5]{May06}). In particular, $T$ is not coherent. This pathology disappears when working with a \emph{convenient category of spaces} (\cite{St67}) such as the category of compactly generated (weak) Hausdorff spaces. The latter is also a cocomplete tensor category in contrast to $\Top$ (for example $\mathds{Q} \times -$ does not preserve quotient maps). See also \autoref{add-point} below.
\end{ex}
The coherence condition of a monad might be hard to check in general, but in the closed case it comes for free:

\begin{rem}[Internal $T$-homs]
Let $\C$ be a closed symmetric monoidal category with equalizers and let $T$ be a symmetric monoidal monad on $T$. For $T$-modules $(A,a),(B,b)$ we define $\HOM_T((A,a),(B,b)) \in \M(T)$ as a submodule of $\HOM(A,B)$ (\autoref{Thom}) as follows (see \cite[Section 2]{Kock71} for details): The underlying object is the equalizer of
\[\HOM(A,B) \xrightarrow{\tau} \HOM(T(A),T(B)) \xrightarrow{b_*} \HOM(T(A),B)\]
and
\[\HOM(A,B) \xrightarrow{a^*} \HOM(T(A),B).\]
There is a unique morphism $T(\HOM_T(A,B)) \to \HOM_T(A,B)$ lying over the action $T(\HOM(A,B)) \to \HOM(A,B)$. This endows $\HOM_T(A,B)$ with the structure of a $T$-module. This construction provides an internal hom-object for $\otimes_T$: For another $T$-module $(C,c)$ we have
\begin{eqnarray*}
\Hom_{T}((C,c) \otimes_T (A,a),(B,b)) & \cong & \BiHom((C,c),(A,a),(B,b)) \\
 && ~~~~~~~~~~\rotatebox{-90}{$\subseteq$} \\  \\
   \Hom(C,\HOM(A,B)) &\cong &\Hom(C \otimes A,B).
\end{eqnarray*}
The image of the set of morphisms $C \otimes A \to B$ which are homomorphisms in $A$ is $\Hom(C,U \,\HOM_T((A,a),(B,b)))$. The image of the set of morphisms which are homomorphisms in $C$ is $\Hom_T((C,c),\HOM((A,a),(B,b)))$. This implies
\[\Hom_{T}\bigl((C,c) \otimes_T (A,a),(B,b)\bigr) \cong \Hom_{T}\bigl((C,c),\HOM_T((A,a),(B,b))\bigr).\]
Hence, $- \otimes_T (A,a)$ is left adjoint and therefore cocontinuous. This proves:
\end{rem}

\begin{cor} \label{cloherent}
Let $\C$ be a closed symmetric monoidal category with equalizers and let $T$ be a symmetric monoidal monad on $\C$. If $\M(T)$ has reflexive coequalizers, then $T$ is coherent.
\end{cor}

See \autoref{modepi} for a criterion that $\M(T)$ has coequalizers. As a special case, we get (using \autoref{setcomp}):
 
\begin{cor} \label{setauto}
Every symmetric monoidal monad on $\Set$ is coherent.
\end{cor}

The following criterion is useful.
 
\begin{prop} \label{refcoherent}
Let $\C$ be a finitely cocomplete tensor category. Let $T$ be a monoidal monad on $\C$ which preserves reflexive coequalizers. Then $T$ is coherent.
\end{prop}

\begin{proof}
By \autoref{create} $\M(T)$ has reflexive coequalizers, created and therefore preserved by the forgetful functor to $\C$. Now let $(A,a)$ be some $T$-module and $(P,p) \rightrightarrows (Q,q) \rightarrow (B,b)$ be a reflexive coequalizer in $\M(T)$. Then $P \rightrightarrows Q \to P$ is also a reflexive coequalizer. It stays a reflexive coequalizer after tensoring with $A$ in $\C$, but also after applying $T$. Hence, in the following commutative diagram, we have reflexive coequalizers in every column and in the first two rows:
\[\xymatrix{F(T(A) \otimes T(P)) \ar@<1ex>[d] \ar@<-1ex>[d]  \ar@<1ex>[r] \ar@<-1ex>[r] & F(T(A) \otimes T(Q)) \ar@<1ex>[d] \ar@<-1ex>[d] \ar[r] & F(T(A) \otimes T(B)) \ar@<1ex>[d] \ar@<-1ex>[d] \\
F(A \otimes P) \ar[d] \ar@<1ex>[r] \ar@<-1ex>[r] & F(A \otimes Q) \ar[d] \ar[r] & F(A \otimes B) \ar[d] \\
(A,a) \otimes_T (P,p) \ar@<1ex>[r] \ar@<-1ex>[r] & (A,a) \otimes_T (Q,q) \ar[r] & (A,a) \otimes_T (B,b)}\]
But then the bottom row also has to be a coequalizer.
\end{proof}
 

\section{Module categories}

Let us gather the results from \autoref{tpmonad}.

\begin{thm}[{\cite[Corollary 2.5.6]{Sea13}}] \label{cohmod}
Let $T$ be a coherent (symmetric) monoidal monad on a (symmetric) monoidal category $\C$. Then $\M(T)$ becomes a (symmetric) monoidal category with tensor product $\otimes_T$ and unit $T(1)$, unit constraint as in \autoref{eins} and associator as in \autoref{associator} (and symmetry as in \autoref{symmetry}).
\end{thm}

\begin{proof}
Because of the uniqueness statements in \autoref{eins}, \autoref{symmetry} and \autoref{associator} the coherence diagrams in the definition of a monoidal category follow immediately from those for $\C$.

For example, the pentagon identity for $T$-modules $(A,a),(B,b),(C,c),(D,d)$ follows from the following commutative diagram:

{\scriptsize   \[\xymatrix@C=20pt@R=35pt{((A \bigotimes B) \bigotimes C) \bigotimes D \ar[dr] & &&(A \bigotimes (B \bigotimes C)) \bigotimes D \ar[lll] \ar[dl]\\ 
 & ((A \bigotimes\limits_T B) \bigotimes\limits_T C) \bigotimes\limits_T D & (A \bigotimes\limits_T (B \bigotimes\limits_T C)) \bigotimes\limits_T D \ar[l] & \\
(A \bigotimes B) \bigotimes (C \bigotimes D) \ar[r] \ar[uu] & (A \bigotimes\limits_T B) \bigotimes\limits_T (C \bigotimes\limits_T D) \ar[u] & & \\
& A \bigotimes\limits_T (B \bigotimes\limits_T (C \bigotimes\limits_T D)) \ar[u] \ar[r] & A \bigotimes\limits_T ((B \bigotimes\limits_T C) \bigotimes\limits_T D) \ar[uu] & \\
A \bigotimes (B \bigotimes (C \bigotimes D)) \ar[uu] \ar[ur] \ar[rrr] &&& A \bigotimes ((B \bigotimes C) \bigotimes D) \ar[ul] \ar[uuuu]}\]  }
\end{proof}

\begin{prop} \label{freef}
In the setting of \autoref{cohmod}, the free $T$-module functor $F : \C \to \M(T)$ carries the structure of a strong (symmetric) monoidal functor.
\end{prop}

\begin{proof}
This follows from \autoref{eins} and \autoref{freiesTP}. The coherence diagrams in the definition of a strong (symmetric) monoidal functor are readily checked. For example, in order to show that the diagram which expresses compatibility with the associator

\[\xymatrix{F(X) \otimes_T (F(Y) \otimes_T F(Z)) \ar[r] \ar[d] & (F(X) \otimes_T F(Y)) \otimes_T F(Z) \ar[d] \\
F(X) \otimes_T F(Y \otimes Z) \ar[d] & F(X \otimes Y) \otimes_T F(Z) \ar[d] \\
F(X \otimes (Y \otimes Z)) \ar[r] & F((X \otimes Y) \otimes Z)}\]
is commutative, we may precompose with
\[T(X) \otimes (T(Y) \otimes T(Z)) \to T(X) \otimes (T(Y) \otimes_T T(Z)) \to T(X) \otimes_T (T(Y) \otimes_T T(Z))\]
and reduce the diagram to
\[\xymatrix{T(X) \otimes (T(Y) \otimes T(Z)) \ar[r] \ar[d] & (T(X) \otimes T(Y)) \otimes T(Z) \ar[d] \\
T(X) \otimes T(Y \otimes Z) \ar[d] & T(X \otimes Y) \otimes T(Z) \ar[d] \\
T(X \otimes (Y \otimes Z)) \ar[r] & T((X \otimes Y) \otimes Z)}\]
which simply belongs to the definition of a monoidal monad.
\end{proof}

The following result is not covered in \cite{Sea13}. It is a really a \emph{machine} which produces lots of cocomplete tensor categories.

\begin{thm} \label{maschine}
Let $T$ be a coherent (symmetric) monoidal monad on a cocomplete (symmetric) monoidal category $\C$. Then $\M(T)$ carries the structure of a cocomplete (symmetric) monoidal category. Moreover, the free functor $F : \C \to \M(T)$ carries the structure of a cocontinuous (symmetric) monoidal functor.
\end{thm}

\begin{proof}
Since $T$ is coherent, $\M(T)$ has reflexive coequalizers. Then $\M(T)$ is cocomplete by \autoref{modcomplete}. Because of \autoref{cohmod} and \autoref{freef}, we are left to prove that
\[(A,a) \otimes_T - : \M(T) \to \M(T)\]
is cocontinuous for each $T$-module $(A,a)$ (surely the same proof will work for $- \otimes_T (A,a)$). Using canonical presentations (\autoref{canpres}), the assumption that $- \otimes_T (B,b)$ preserves reflexive coequalizers for all $T$-modules $(B,b)$, as well as the fact that colimits commute with colimits, we may reduce to the case that $(A,a)$ is a free  $T$-module, say $(A,a)=F(X)$ for some $X \in \C$. Then, by \autoref{modcoco} it suffices to prove that
\[F(X) \otimes_T - : \M(T) \to \M(T)\]
preserves coproducts of free modules. But this follows easily from \autoref{freiesTP} and the fact that $X \otimes - : \C \to \C$ commutes with coproducts.
\end{proof}

\begin{cor} \label{closed-maschine}
Let $T$ be a symmetric monoidal monad on a cocomplete closed symmetric monoidal category $\C$ with equalizers. Assume that $\M(T)$ has reflexive coequalizers. Then $\M(T)$ becomes a cocomplete closed symmetric monoidal category with equalizers. Besides, the free $T$-module functor $F : \C \to \M(T)$ becomes a cocontinuous symmetric monoidal functor.
\end{cor}

\begin{proof}
This follows from \autoref{maschine} and \autoref{cloherent}.
\end{proof}

\begin{rem}
By the discussion in \autoref{monadoverview} we see that the converse of \autoref{maschine} is also true: If $T$ is a monad on the underlying category of a cocomplete tensor category $\C$, then the structure of a cocomplete tensor category on $\M(T)$ and a tensor functor on $F : \C \to \M(T)$ induces a coherent symmetric monoidal monad structure on $T$.
\end{rem}

\begin{prop} \label{modfunk}
Let $T \to T'$ be a morphism of coherent monoidal monads on a monoidal category $\C$. Then, the forgetful functor $\M(T') \to \M(T)$ has a left adjoint functor $\M(T) \to \M(T')$, denoted by $M \mapsto M \otimes_T T'$, which is a cocontinuous monoidal functor. The same holds in the symmetric case.
\end{prop}

\begin{proof}
The existence of the left adjoint follows from \cite[Proposition 3.3.19]{Dur07}. The idea is to extend $F(X) \mapsto F'(X)$ using canonical presentations. The forgetful functor is lax monoidal. In fact, given $T'$-modules $M=(A,a)$, $N=(B,b)$, the universal $T'$-bihomomorphism $A \otimes B \to A \otimes_{T'} B$ is also a $T$-bihomomorphism, hence lifts to a homomorphism $M \otimes_T N \to M \otimes_{T'} N$ of $T$-modules. It follows (\autoref{adj-lax}) that the left adjoint is oplax monoidal. It is actually strong because this may be reduced to free $T$-modules for which it is clear.
\end{proof}

Next, we will provide (possibly new) universal properties of module categories. We do not really need symmetric monoidal structures, monoidal structures would work equally well.

\begin{thm} \label{monadUE1}
Let $T$ be a symmetric monoidal monad on a cocomplete tensor category $\C$, which preserves reflexive coequalizers. Then $\M(T)$ is a cocomplete tensor category and $F : \C \to \M(T)$ is a cocontinuous tensor functor with a lax symmetric monoidal right action $FT \to F$, inducing for every cocomplete tensor category $\D$ an equivalence of categories between $\Hom_{c\otimes}(\M(T),\D)$ and the category of pairs $(G,\rho)$, where $G \in \Hom_{c\otimes}(\C,\D)$ and $\rho : GT \to G$ is a lax symmetric monoidal right action.
\end{thm}

\begin{proof}
That $\M(T)$ is a cocomplete tensor category follows from \autoref{maschine} and \autoref{refcoherent}. For the universal property we use \autoref{modT}. The action $FT \to F$ is a morphism of lax tensor functors basically because by definition $\mu : T^2 \to T$ is one. Let $G : \C \to \D$ be a cocontinuous tensor functor equipped with an action $GT \to G$. We consider its cocontinuous extension $H : \M(T) \to \D$ and endow $H$ with the structure of a tensor functor as follows: For $(A,a) \in \M(T)$ the coequalizer
\[G(T(A)) \rightrightarrows G(A) \rightarrow H(A,a)\]
is reflexive with section $G(\eta_A)$. It follows (\autoref{bifunktor}) that for $T$-modules $(A,a),(B,b)$ we have a coequalizer
\[G(T(A)) \otimes G(T(B)) \rightrightarrows G(A) \otimes G(B) \rightarrow H(A,a) \otimes H(B,b).\]
The morphisms on the left are isomorphic to $G(T(A) \otimes T(B)) \rightrightarrows G(A \otimes B)$, whose coequalizer is $H((A,a) \otimes_T (B,b))$ by \autoref{tensorex}.
\end{proof}

With the above notations, if $G : \C \to \D$ has a right adjoint $G_* : \D \to \C$, we get a symmetric monoidal monad $T_G$ on $\C$ (with underlying functor $G_* G$) and a lax symmetric monoidal right action $GT \to G$ may be identified with a morphism $T \to T_G$ of symmetric monoidal monads on $\C$. Therefore, we get the following:

\begin{thm} \label{monadUE2}
Let $T$ be a symmetric monoidal monad on a cocomplete tensor category $\C$ which preserves reflexive coequalizers. If $\D$ is a cocomplete tensor category, then the category of left adjoint tensor functors $\M(T) \to \D$ is equivalent to the category of left adjoint tensor functors $G : \C \to \D$ equipped with a morphism of symmetric monoidal monads $T \to T_G$.
\end{thm}

\begin{rem}
One can show this more directly without the use of \autoref{monadUE1} and then does not need the assumption that $T$ preserves reflexive coequalizers. By \autoref{saft} we may replace ``left adjoint'' by ``cocontinuous'' if $\C$ and $\M(T)$ are locally presentable.
\end{rem}

\begin{ex}
Let us look at the special case $\C=\Set$. Here $G : \C \to \D$ is essentially unique and is given by $X \mapsto X \otimes 1_\D = 1_\D^{\oplus X}$ with right adjoint $\Hom(1_\D,-)$. The symmetric monoidal monad $T_\D := T_G$ on $\Set$ is given by
\[T_\D(X)=\Hom(1_\D,1_D^{\oplus X})\]
on objects. Then, by \autoref{finmon} and \autoref{monadUE2}, we get the following result, which already appeared as \cite[Proposition 4.2.3]{BC14}.
\end{ex}

\begin{cor} \label{setcase} \label{monadUE3}
Let $T$ be a finitary commutative monad on $\Set$ (i.e. a generalized commutative ring in the sense of Durov \cite{Dur07}). Then $\M(T)$ is a locally presentable tensor category with the following universal property: If $\D$ is a cocomplete tensor category and $T_\D$ is the associated symmetric monoidal monad on $\Set$, then there is a natural equivalence of categories
\[\Hom_{c\otimes}(\M(T),\D) \simeq \Hom(T,T_\D).\]
\end{cor}

Now let us discuss some examples.

\begin{ex}[Basic modules]
Let $R$ be a commutative ring. The machinery constructs, starting from the basic example $\Set$, the locally presentable tensor category $\M(R)$ of $R$-modules  (in particular $\Ab$ for $R=\Z$) as follows: We have already observed in \autoref{generalized-ring1} that the monad $T$ defined by $\M(T)=\M(R)$ is commutative. Now we may apply \autoref{monadUE3} and obtain a tensor product on $\M(R)$. If $(A,a),(B,b),(C,c)$ are $R$-modules, then a map $f : A \times B \to C$ is a bihomomorphism with respect to $T$ if and only if
\[\xymatrix@C=30pt{T(A) \times T(B) \ar[r]^-{d_{A,B}} \ar[d]_{a \times b} & T(A \times B) \ar[r]^-{T(f)} & T(C) \ar[d]^{c} \\
A \times B \ar[rr]^{f}  && C}\]
commutes, which means that for all $\sum_{x \in A} x \cdot \lambda_x \in T(A),~\sum_{y \in B} y \cdot \mu_y \in T(B)$ we have 
\[f\left(\sum_{x \in A} x \cdot \lambda_x,\sum_{y \in B} y \cdot \mu_y\right) = \sum_{x \in A, y \in B} \lambda_x \mu_y  \cdot f(x,y).\]
This means that $f$ is $R$-bilinear in the usual sense. Therefore the tensor product on $\M(R)$ is the usual one. Notice that \autoref{modfunk} includes the statement that every homomorphism $R \to S$ of commutative rings induces a cocontinuous tensor functor $\M(R) \to \M(S)$. By the way, a similar construction works if $R$ is just a commutative semiring. Here an $R$-module is a commutative monoid with an $R$-action. For $R=\N$ this gives $\CMon$, the tensor category of commutative monoids.
\end{ex}

\begin{ex}[Modules over algebras]
Let $\C$ be a cocomplete tensor category and let $A$ be a commutative algebra in $\C$. Then $ - \otimes A : \C \to \C$ carries the structure of a symmetric monoidal monad (see \autoref{monbsp}). It preserves reflexive coequalizers (in fact all colimits), hence it is coherent by \autoref{refcoherent}. It follows by \autoref{maschine} that $\M(A) := \M(- \otimes A)$ becomes a cocomplete tensor category. Of course we already understand this well (\autoref{modA}). But we would like to emphasize that this is just a special case of general results about monoidal monads. We have also proven a universal property of $\M(A)$ (\autoref{mod-UE}) which now follows directly from \autoref{monadUE2} at least when $\C$ is locally presentable. In fact, if $\D$ is a cocomplete tensor category and $G : \C \to \D$ is a cocontinuous tensor functor with right adjoint $G_*$, then a morphism of symmetric monoidal monads $- \otimes A \to G_* G$ corresponds (\autoref{monbsp}) to a morphism of algebras $A \to (G_* G)(\O_\C)=G_*(\O_D)$ in $\C$, i.e. to a morphism of algebras $G(A) \to \O_D$ in $\D$.
\end{ex}

\begin{ex}[Pointed objects] \label{add-point}
Let $\C$ be a cocomplete tensor category. The reader may think of $\Set$ or any (cartesian) convenient category of topological spaces. Let $*$ be a terminal object (if it exists). Let us denote by $\C_* := * / \C$ the category of pointed objects of $\C$, i.e. pairs $(X,p)$ where $X \in \C$ and $p : * \to X$ is a morphism, imagined as a base point. The forgetful functor $\C_* \to \C$ is monadic, the monad $T$ sends $X$ to $X + *$. It has a symmetric monoidal structure as follows: Given pointed objects $(X,*),(Y,*)$, we define
\[d : (X+*) \otimes (Y+*) = (X \otimes Y) + (X \otimes *) + (* \otimes Y) + (* \otimes *) \to (X \otimes Y) + *\]
by the canonical inclusion on $X \otimes Y$ and by the unique morphism to $*$ on the rest. It is clear that $T$ preserves all coequalizers, hence $T$ is coherent (\autoref{refcoherent}). We may apply \autoref{maschine} to endow $\C_*$ with the structure of a cocomplete tensor category together with a cocontinuous tensor functor $\C \to \C_*$.

Let us describe the tensor structure. The unit is $T(1)=1+*$. Given pointed objects $(X,p),(Y,q),(Z,r)$, a morphism $X \otimes Y \to Z$ is a bihomomorphism if and only if it maps the points $p \otimes Y$, $X \otimes q$ and $p \otimes q$ to the point $r$, i.e. when it is bipointed. It follows that $(X,*) \otimes (Y,*)$ is given by a kind of smash product, namely the quotient of $(X \otimes Y + *,*)$ which identifies the points $p \otimes Y$, $X \otimes q$ with $*$. If $\C$ is cartesian, this simplifies to the corresponding quotient of $X \times Y$ which already has the point $(p,q)$.

If $\C$ is locally presentable, it follows from \autoref{monadUE2} that $\Hom_{c\otimes}(\C_*,\D)$ is equivalent to the category of pairs consisting of a cocontinuous tensor functor $F : \C \to \D$ and a morphism $F(*) \to 0$. The proof also works if $\C$ is arbitrary. In particular, if we restrict on both sides to functors preserving the terminal object, it follows that $\C_*$ is the universal solution to making the initial object $0 \in \C$ to a zero object. Following Durov's theory the symmetric monoidal monad $T$ could be seen as the ``field with one element'' relative to $\C$.
\end{ex}

\begin{ex}[Complete partial orders]
Consider the power set monad $\wp$ on $\Set$. Here $\mu : \wp^2 \to \wp$ is the union operator and $\eta : \id_{\Set} \to \wp$ the singleton operator. It is well-known (\cite{Man67}) that $\M(\wp)$ is isomorphic to the category of complete partial orders, namely a complete partial order $(A,\leq)$ corresponds to the $\wp$-module $(A,\sup)$. Although $\wp$ is neither finitary nor does it preserve reflexive coequalizers (\autoref{power}), it is coherent by \autoref{setauto}, so that \autoref{maschine} is applicable and we obtain a well-behaved tensor product of complete partial orders. It classifies maps which preserve suprema in each variable. This tensor product has been studied (\cite{JT84}).
\end{ex}

\begin{ex}[Commutative algebras]
The category $\CMon = \CAlg(\Set)$ of commutative monoids  has a tensor product which looks like the tensor product of abelian groups. We can generalize its construction as follows: Let $\C$ be a \emph{cartesian} cocomplete tensor category. By \autoref{symalg} the forgetful functor $U : \CAlg(\C) \to \C$ has a left adjoint $\Sym : \C \to \CAlg(\C)$ and $U$ creates reflexive coequalizers -- the usual proof of this works (\cite[VI.8, Theorem 1]{ML98}). Hence, $U$ is monadic (\autoref{crude}) and the corresponding monad $T$ on $\C$ is coherent (\autoref{refcoherent}). We have $T(M)=\bigoplus_{n \in \N} \Sym^n(M)$ with $\Sym^n(M) = M^n/\S_n$. We endow $T$ with a symmetric monoidal structure, where $d : T(-) \otimes T(-) \to T(- \otimes -)$ is induced by
\[\Sym^p(M) \times \Sym^q(N) \to \Sym^{p*q}(M \times N),~\bigl(\prod_i m_i,\prod_j n_j\bigr) \mapsto \prod_{i,j} (m_i,n_j).\]
By \autoref{maschine} $\CAlg(\C)$ becomes a cocomplete tensor category in such a way that $\Sym(M) \otimes \Sym(N) = \Sym(M \times N)$. The unit is $\N \otimes 1_\C$.
For example, if $\C$ is a convenient category of topological spaces, this produces a cocomplete tensor category of commutative topological monoids.

We mention without proof that a similar construction works for the category of commutative Hopf algebras (i.e. abelian group objects) in $\C$ (when $\C$ has equalizers) by generalizing the classical Grothendieck group.
\end{ex}

\begin{rem}
Our goal is to construct and explicitly describe cocomplete tensor categories generated by objects and morphisms subject to given relations (in the language of cocomplete tensor categories). This is similar to the theory of classifying topoi (\cite[Chapter VIII]{Mac92}). \marginpar{I've added the classifying-topoi remark.} For categories of quasi-coherent sheaves this should correspond to the construction of moduli spaces and stacks. This has been done in a lot of special cases in \autoref{Cons}. By \autoref{monadUE2} the general (monadic) case is reduced to find coherent symmetric monoidal monads described by generators and relations. Over $\Set$ this theory has been developed by Durov (\cite{Dur07}). We do not know if such a theory exists in general, although certainly \cite{Kel80} and \cite{Hyland} go in that direction.
\end{rem}




 \nocite{*} 
\bibliography{literature}{}
\bibliographystyle{alpha}



\end{document}